\documentclass{article}

\usepackage{cttools}
\usepackage{imakeidx}
\usepackage{endnotes}
\pgfplotsset{compat=1.18}

\author{Alexander Prähauser}
\date{\today}
\title{Inter-Categorical Subobject Analysis}

\makeindex

\begin{document}
\emergencystretch 3em

\maketitle

\begin{abstract}
  We introduce a systematic way to use the subobject functor on any category to set up a framework for generalized topological and measure-theoretic notions and inferences.
  Concretely, we develop two connected calculi:

  \vspace{1 mm}
  \begin{itemize}
  \item a fourfold calculus derived from notions of smallness and density that generalizes, among other things, topological density and measure-theoretic negligibility,

  \item a twofold calculus that generalizes topological and measure-theoretic inference procedures.
  \end{itemize}
  \vspace{1.4 mm}

  The resulting theory is distinguished in that it is \emph{inter-categorical} in character: its definitions can be applied not just to describe smallness and density of subobjects
  in any category and classify morphisms by how they handle them, but also functors and natural transformations by how they transform them, somewhat like the theory of limits but
  with a simple and intuitive conceptual meaning that applies to each incarnation. We provide examples of these notions from many areas of mathematics and outline their unfolding
  throughout category theory as we add assumptions that bring them closer to classical settings, which we complete by generalizing the topological duality of open and closed subsets and
  outlining a possible framework for a categorical measure theory.

  Central to this interaction is the subobject 2-functor and the 2-category of upward-bounded preorders, so in the last part of the text we provide a study of their function.
  Concretely we derive representation theorems for this 2-functor and 2-category and a conjecture about how to generalize these theorems to \(n\)-categories.
\end{abstract}

\tableofcontents

\pagebreak

\section{Introduction}
\label{sec:introduction}

We ask the reader to picture a simple shape, say, a cone \(C\). ``The same'' cone is an object in quite a lot of categories: it is a combinatorial structure, a continuous one, a
topological one, an algebro-geometric one, and many more, but it also is a subobject of \(ℝ^{3}\), which can have, say, a linear structure that casts its own shadows onto its subsets.
Picture the subobjects of \(C\) in each kind of category, or the equivalences of \(C\) in its topological, smooth, and combinatorial incarnations, or how the different embeddings of \(C\)
cast shadows through outside shapes, such as the sections of \(C\) by planes in \(ℝ^{3}\), or perhaps even how \(C\) degenerates into itself, like a film passing by, illuminating \(C\)
in a different color and shade at each moment.

\begin{quote}
  \label{qo_:2}
  [Reason] is often taken ... [as] something separate from positive understanding. But in truth reason is spirit... It is the negative which characterizes the quality of both dialectical reason and
  understanding: it negates the simple, and thus sets the difference of understanding, but resolves it just as much, thus it is dialectic. It does however not remain in the
  nothingness of this result but restores the first simple through it, but as something general, which is in itself concrete; under this is not something
  particular subsumed, but in this determination and its resolution has the particular already determined itself.
\end{quote}
- G. W. F. Hegel, \cite{WissenschaftLogikLehre}, personally shortened and translated.\endnote{The unshortened quote is:

  \begin{quote}
    \label{qo_:3}
    Der Verstand bestimmt und hält die Bestimmungen fest; die Vernunft ist negativ und dialektisch, weil sie die Bestimmungen des Verstands in nichts auflöst; sie ist positiv, weil sie das
    Allgemeine erzeugt und das Besondere darin begreift. Wie der Verstand als etwas Getrenntes von der Vernunft überhaupt, so pflegt auch die dialektische Vernunft als etwas
    Getrenntes von der positiven Vernunft genommen zu werden. Aber in ihrer Wahrheit ist die Vernunft Geist, der höher als beides, verständige Vernunft oder vernünftiger Verstand ist. Er ist
    das Negative, dasjenige, welches die Qualität sowohl der dialektischen Vernunft als des Verstandes ausmacht; — er negiert das Einfache, so setzt er den bestimmten Unterschied des
    Verstandes; er löst ihn ebensosehr auf, so ist er dialektisch. Er hält sich aber nicht im Nichts dieses Resultates, sondern ist darin ebenso positiv und hat so das erste Einfache damit
    hergestellt, aber als Allgemeines, das in sich konkret ist; unter dieses wird nicht ein gegebenes Besonderes subsumiert, sondern in jenem Bestimmen und in der Auflösung desselben hat
    sich das Besondere schon mit bestimmt.
  \end{quote}}

So what we are after is a kind of re-connection of what appears to be very separate, and a kind of abstraction that is different from that which is more commonplace in category
theory; where we put the focus onto the object again and how it externalizes itself in different categories. For this, we proceed from an elementary but fundamental observation:
the subobjects of each object in each category form a preorder, and not just any preorder but one with an upper bound. That already means that even though all these objects live in
different categories and are thus of different types, they have something in common, which is a subobject preorder. In other words, each \(X{:} \mathcal{C}\) is a \emph{composite} in the
loosest possible sense, consisting of a \emph{whole} \(\top = \mathbf{id}_{X}\) and \emph{parts}, i.e. subobjects of \(X\). So when we consider, say, the subobjects of a cone, these
may live in very different categories depending on the kind of subobject, but they \emph{also} all live in the same 2-category \(\mathbf{UBP}\) of upward-bounded preorders (but
with not necessarily top-preserving order-morphisms). Thus, if we define our notions in \(\mathbf{UBP}\), they are \emph{inter-categorical}: they work in any category, similar to
limits and colimits but on a purely monic level.

What can we do with this? Surprisingly much. In particular, we introduce two synergistic calculi that are formulated in \(\mathbf{UBP}\), ``under'' all categories, which is the
source of their wide applicability.

\subsection{The Calculus of Smallness \& Density, Reducedness \& Satedness}
\label{sec:calculus-smallness-}

This calculus consists of a square of notions:

\begin{center}
  \begin{tabular}[H]{cc}
  Small & Dense \\
  Reduced & Sated
\end{tabular}
\end{center}

These are the four basic notions whose incarnations we study throughout mathematics.

\begin{itemize}
\item The notions in a row are abstractly dual, though this duality breaks when the interaction with category theory is considered, as we will see in \cref{sec:dual-break-point}.

\item The notions in the second row are derived from and in a sense orthogonal to those in the first row.

\item the notions in the first column are applicable in any category while those in the second column are applicable in any category with a strict initial object.
\end{itemize}

We call this overall system the \emph{calculus of smallness \& density, reducedness \& satedness} or the \emph{fourfold calculus} for short.

\paragraph{Smallness}

Consider the most elementary notion that drove this investigation, that of \emph{smallness}. What does it mean for a part to be small? The
answer we provide is simple enough to be stated with barely any conceptual baggage:

\begin{quote}
  \label{qo_:6}
  A part is \emph{small} if it cannot be made whole by a proper part.
\end{quote}

or, in more succinct mathematical language:

\[  a \text{ small} {:}\Leftrightarrow ∀b{:} (a ∨ b = \top \Rightarrow b = \top).\]

In the lattice of closeds of a topological space this is the case if and only if \(a\) is a nowhere dense subset, but our definition can be applied in any preorder with a top
element \(\top\) — an upward-bounded preorder.\footnote{Note that \cref{qo_:6} is definitional whether or not all joins exist, so that it can indeed be applied to the subobjects of
  any object in any category.} In particular, each object of each category features such a structure. The resulting system of smallness, among other things,
generalizes:

\begin{enumerate}
\item \label{item:23} the topological notion of nowhere density (see \cref{xmp:7}),

\item\label{item:24} lower-dimensionality of manifolds with boundary (see \cref{xmp:63}),

\item\label{item:144} a variant of Frattini theory (see \cref{sec:algebraic}),

\item\label{item:27} and provides a sensible notion of smallness for schemes and semialgebraic varieties (see \cref{xmp:18} and \cref{xmp:14}).
\end{enumerate}
 
If we extend the definition in a natural way, it also generalizes finitude (see \cref{sec:monoidal-smallness}).
A further result of ours that will be is that smallness induces a stratification on (the subobjects of) each object of any category (see \cref{■pn_:strat-welldef}). To visualize smallness, it makes sense to think of the
two-dimensional topological case, where for instance a line is lower-dimensional and thus small as a manifold with boundary (see \cref{xmp:60}):

\begin{figure}[H]
  \centering
  \begin{tikzpicture}
  \draw[black] (0, 0) -- (5, 0);
\end{tikzpicture}
  % \caption{\label{fi♭_:19}}
\end{figure}

\paragraph{Density}

What happens when we dualize smallness? We call this \emph{density}:

\[a \text{ dense} {:}\Leftrightarrow ∀b{:} (a ∧ b = \bot \Rightarrow b = \bot).\]

But this dualization changes the relation of the notion to categories and with that the intuitive interpretation slightly but fundamentally:

\begin{quote}
  A part is \emph{dense} if part of it is contained in any other non-trivial part.
\end{quote}

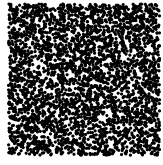
\begin{figure}[H]
  \centering
  \begin{tikzpicture}
    \pgfmathsetseed{42}%
    \foreach \x in {-30,...,30}{% Two indices running over each
      \foreach \y in {-30,...,30}{% node on the grid we have drawn
        \fill (rand,rand) circle[radius=0.6pt];
              } }
  \end{tikzpicture}
\end{figure}

This imposes some rather strong restrictions on dense subobjects. In particular, \cref{■pn_:8} tells us that if each subobject contains some atom, then a subobject is dense
if and only if it contains all atoms. So if each subobject is made of atoms (i.e. if the subobject preorder is atomistic), then density is trivial.
But if a preorder generally does not have atoms, such as the order \(\mathbf{O}(X)\) of opens of a topological space, then density can be non-trivial, and indeed specializes to the
topological notion of density. On the other hand, if the subobjects have structure outside of atoms, then there often exists a smallest dense subobject, given by the union of all
atoms, such as the dual of the Jacobson radical in an affine scheme (\cref{xmp:18}) or the subcategory of some category \(C\) on all objects and no non-trivial morphisms, i.e. the smallest wide
subcategory (\cref{xmp:4}) of \(C\).

Something that has to be heeded in the following is that the most illustrative example, that of topological spaces, is so well-behaved that it can be misleading: each dense open
subset of a topological space is dual to a small (i.e. nowhere dense) closed one, so we might think a similar correspondence holds
in general or is common. But this is far from the truth: generally, the concrete behavior of small and dense subobjects is highly non-dual, even though they are abstractly dual.
Moreover, smallness and density are not exclusionary, so the same ``part'', i.e. subobject can be both small and dense, such as a gas (as in \cref{xmp_:1}). Their contrast is more
conceptual, and nicely illustrated by the fact that dense subobjects form a coverage in any regular category with a strict initial object (see \cref{crl:4}) — smallness stratifies
into an ever-descending sequence, density remains globally stable. The sheaves for such dense covers contain in particular rational maps of varieties (see \cref{xmp_:7}).

\paragraph{Reducedness}

Reducedness is orthogonal to smallness in that a small subobject is to a reduced subobject as it is to the whole \(\top\). So we might say:

\begin{quote}
  \label{qo_:4}
  A part is \emph{reduced} if it cannot be divided into a small and a proper part.
\end{quote}

E.g. this figure should not be reduced (as a subobject of the topological plane):

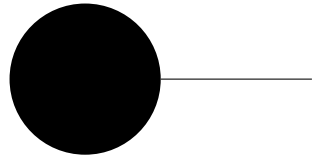
\begin{figure}[H]
  \centering
  \begin{tikzpicture}
  \node[fill,circle,minimum size=2cm] at (-1, 0) {};
  \draw[black] (0, 0) -- (2, 0);
  \end{tikzpicture}
  % \caption{\label{fi♭_:21}}
\end{figure}

because it has a small line sticking out, whereas the left disk itself would be reduced. In this way, reduced subobjects generalize regular closed subsets. But \cref{qo_:4} is actually not
precisely the right definition for reducedness. The issue is that for the top element \(\top\) it is equivalent to say that ``it cannot be divided into a small part and a proper part''
and that ``it is not \emph{in} a small part and a proper part'', but for elements other than \(\top\) these conditions are different and it is actually the laxer covering condition that
is the correct one. So the correct defining condition for reducedness is

\[c \text{ reduced} {:}\Leftrightarrow ∀a \text{ small}, b{:} (a ∨ b \geq c \Rightarrow b \geq c).\]

In a distributive setting we can reduce each such \(a, b\) to a pair \(c = (a ∨ b) ∧ c = (a ∧ c) ∨ (b ∧ c)\) such that \(a ∧ c\) is small, so that \(c = b\) and both definitions
align. Moreover, in a co-Heyting algebra each subobject decomposes into a reduced part \(¬¬ a\) and a small part \(a ∖ ¬¬ a\). But even for a general subobject \(ι\) in any category it
holds that \(ι\) is reduced if and only if it reflects smallness, which is the way the two are ``orthogonal''.

\paragraph{Satedness}

Density itself has a notion analogous to reducedness called \emph{satedness}, which is again an abstract dual of reducedness:

\[c \text{ sated} {:}\Leftrightarrow ∀a \text{ dense}, b{:} (a ∧ b \leq c \Rightarrow b \leq c).\]

So

\begin{quote}
  \label{qo_:5}
  A part is \emph{sated} if it contains each part that it contains a dense part of.
\end{quote}

This generalizes regular open subsets of topological spaces but also e.g. radical ideals (see \cref{xmp:54}). So geometrically we can for instance say that the submodule of
\(ℝ^{n}_{\geq 0}\) generated by the positive unit circle quadrant is not sated as a submonoid:

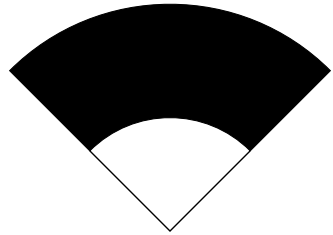
\begin{figure}[H]
  \centering
  \begin{tikzpicture}[radius=1.5cm]
    \draw[name path=A] (0,0) -- ({1.5*cos(45)},{+1.5*sin(45)}) arc [start angle=45, delta angle=90] -- cycle;
    \draw[name path=B] (0,0) -- ({3*cos(45)},{+3*sin(45)}) arc [start angle=45, delta angle=90, radius=3cm] -- cycle;
    \tikzfillbetween[of=A and B]{black};
  % \draw (0,0) -- (2,0) arc (-45:45:2cm) -- cycle;
  \end{tikzpicture}
\end{figure}

Because it does not contain the part of the monoid under it, i.e. the interior of the positive unit circle quadrant.\footnote{This also holds for the submonoid generated by the
  positive unit circle quadrant, but the generated submodule is more illustrative. Using tablets we can extend these results to the positive unit circle quadrant itself (as well as
  the circle).}

The abstract duality of satedness with reducedness is again broken by the interaction
with category theory and thus the principal way its orthogonality with density manifests itself is different, such as through \cref{■pn_:11}: a part \(b\) is sated in another part \(a\) if and only
if it is thick enough on \(a\) to contain the intersection with \(a\) of each part it contains a dense part of. With some tablets and scales we can see that this implies
meanings of satedness fairly close to applied mathematics, such as in \cref{xmp_:22}.

\subsection{The Calculus of Tablets and Scales} 
\label{sec:calc-tabl-scal}

The fourfold calculus we have just outlined synergizes with another calculus based on the notion of a subobject, that of \emph{tablets} and \emph{scales}. Both notions are
initially elementary, the main technical difficulty is to capture their relation to smallness, reducedness, density and satedness in such a general setting, which is somewhat
subtle but can be done through contemplation of the provided examples.

\paragraph{Tablets}
\label{sec:tablets-1}

The principal idea of tablets is very simple: if we have a family of nice subobjects of an object \(τ{:} A \hookrightarrow X\) in a category (such as open or closed subobjects), then we can use \(τ\) to
make inferences about all subobjects through their relation to nice subobjects. On a more abstract level, we can make inferences about \(X\) by mapping \emph{into it} with nice
objects — an old category-theoretic idea, but here specialized to subobjects. This generalizes how properties of arbitrary subsets in a topological space are inferred, but does not
assume any topological structure.

The introduction of tablets greatly extends the power of the prior calculus. For instance, they allow us to describe non-open dense subsets, but they also enable notions of \emph{radial
  density} (see \cref{xmp:monoid-action-density}) and a specialization to universal points (see \cref{xmp_:25}), among other things.

\paragraph{Scales}
\label{sec:scales-1}

Dual to tablets, \emph{scales} are a way to infer about the structure of subobjects of an object by mapping \emph{out of it} into some nice object, classically the real numbers. So
we have a scale \(σ{:} P → S\) on a preorder \(P\) that maps the elements of \(P\) into \(S\) according to some ``measure''. We consciously say ``tablets'' and ``a scale'': an element of
the source \(T\) of the tablets \(τ\)  of an object \(X\) is one tablet, whereas an element of \(S\) is a scale degree.

Scales allow us to specify the definitions in the fourfold calculus to hold at a quantitative abstraction of \(P\) itself. In particular, if \(σ\) is the measure of a (real-valued) measure space \(X\), then:

\begin{itemize}
\item a subset \(A\) of \(X\) is dense according to \(σ\) if and only if its complement is negligible,

\item if \(σ\) is a finite measure, then \(A\) is \(σ\)-small if and only if it is negligible; if \(σ\) has infinite measure, then \(A\) is small if and only if it is co-infinite,

\item each measurable subset of \(X\) is \(σ\)-small or \(σ\)-reduced and both if and only if it is negligible,

\item \(σ\)-satedness is trivial.
\end{itemize}

So the introduction of scales also allows us to encompass measure-theoretic smallness, i.e. negligibility (see \cref{xmp:28} for the details),
but also many other phenomena such as rounding (see \cref{xmp_:12}) and nonstandard infinitesimality (see \cref{xmp_:1}). We can use both at the same time, which yields
notions of the fourfold calculus that are relativized both with regard to nice subobjects \emph{in} \(X\) and a scale \emph{on} \(X\). So when we combine the fourfold calculus with
that of tablets and scales we obtain four notions that can be specialized along each part of a pair of tablets and a scale:

\[T \xhookrightarrow{τ} P \xhookrightarrow{σ} S\]

providing us with a total of \(2 * 4 = 8\) parameter combinations, which are inter-categorical in nature in that they can be applied directly at the level of subobject preorders.

\paragraph{Heyting Spaces, Open-Closed Dualities and Categorical Measures}
\label{sec:heyting-spaces-open}

The calculus of tablets and scales can loosely be identified with the duality of space and quantity specialized to a monic
context.\cite{lawvereCategoriesSpaceQuantity1992}\cite{SpaceQuantityNLab} So its usefulness is by no means exhausted by the interaction with the fourfold calculus. We illustrate
this by closing the connection to topology and measure theory. This we do by introducing:

\begin{enumerate}
\item\label{item:81} The notion of a \emph{Heyting space}. We can provide a minimalistic of a Heyting space \(X\) as a lattice embedding \(ω\) of a Heyting
  algebra of \emph{opens} \(\mathbf{O}(X)\) into a Boolean algebra \(\mathbf{B}(X)\). But this already provides us with the basic notions necessary to do a good deal of generalized
  topology. In particular,

  \begin{itemize}
  \item each topological space \(X\) is a Heyting space with \(\mathbf{B}(X)\) the powerset \(\mathbf{P}(X)\) of \(X\),

  \item but semialgebraic sets also form Heyting spaces (see \cref{xmp_:11})
  \end{itemize}
   
  So Heyting spaces unify both the infinitary structures of topology and the calculable ones of algebra. Moreover, the signature of Heyting spaces only consists of binary joins,
  meets and three interrelated types of binary complements, which can be tracked with standard first-order-logic. So the theory of Heyting spaces is all about retaining the
  inference procedures about subsets of topological spaces in a generally finitistic setting. In this work, we use Heyting spaces to derive a general and well-behaved notion of
  constructible closure and to provide a variant of the semialgebraic spaces of Delfs-Knebusch suited for our purposes (see \cref{xmp:112}).

\item\label{item:85} The notion of an \emph{open-closed duality}. This abstracts both from Heyting spaces and from the category of topological spaces and puts them into a larger
  context, assigning to each object \(X\) in a suitable category \(\mathcal{C}\) dual collections of \emph{opens} \(\mathbf{O}(X)\) and \emph{closeds} \(\mathbf{C}(X)\). What is most
  pleasing about this notion is that it applies both to spaces with opens and closeds and to categories of such spaces. In particular, each subset of each topological space \(X\)
  features collections of opens and closeds given by restriction, i.e. pullback, defining an open-closed duality on each topological space (or more generally each Heyting space).
  But we can also associate to each object of the category of topological spaces its opens and closeds, and the definition of a continuous function is precisely such that it preserves opens
  and closeds under pullback. So the structure on the subset lattice of a topological space is mirrored on \(\mathbf{Top}\) itself. But our definition also describes e.g. the open
  and closed subschemes of a scheme (see \cref{xmp_:19}) and the category of schemes (see \cref{xmp_:19}).

\item\label{item:86} A setup for a categorical measure theory. The main idea here is to \emph{not} define some category or type of category formed by measurable spaces, but to use
  the subobject functor to define measures \emph{under} a category \(\mathcal{C}\), on its subobjects. \(\mathcal{C}\) can again be either the algebra of measurable subsets of a measure space itself
  or a category of measure spaces. Both work, as long as the morphisms in \(\mathcal{C}\) interact well with the measure. We describe what it means to ``interact well'' in
  \cref{sec:categ-meas-theory}, and we also show that these measures interact nicely with nice tablets (see \cref{prop:borel-heyting-extension}).
\end{enumerate}
 
So again we have \emph{tablets} (opens or closeds) and a \emph{scale} (a measure), so something(s) that we map \emph{into} our objects and something that we map \emph{them} into;
but now on a categorical level, and thus these two sides are slightly different. That last point is worth ruminating on: each category \(\mathcal{C}\) has a canonical image in
\(\mathbf{UBP}\) through the functor \(\mathbf{SOb}_{\mathcal{C}}\), so that \(\mathcal{C}\) is, in a sense, over \(\mathbf{UBP}\), and this is how the measures we define in \cref{sec:categ-meas-theory}
are under \(\mathcal{C}\): as natural transformations from \(\mathbf{SOb}_{\mathcal{C}}\) into some constant object. On the other hand, we describe tablets in \cref{sec:open-clos-dual} as monic
natural transformations \emph{into} \(\mathbf{SOb}_{\mathcal{C}}\). So there is an overall asymmetry in the structure of tablets and a scale under a category:

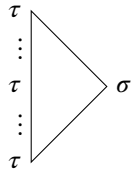
\begin{figure}[H]
  \centering
  \begin{tikzpicture}
    \draw[black] (0, 0) -- (0, -2) -- (1, -1) -- cycle;
    \coordinate [label=left:{\(τ\)}]  (A) at (0,0);
    \coordinate [label=left:{\(τ\)}]  (B) at (0,-1);    
    \coordinate [label=left:{\(τ\)}]  (B) at (0,-2);    
    \coordinate [label=right:{\(σ\)}]  (B) at (1,-1);    
    \coordinate [label=left:{\rotatebox{90}{\(…\)}}]  (A) at (0,-0.5);
    \coordinate [label=left:{\rotatebox{90}{\(…\)}}]  (B) at (0,-1.5);    
  \end{tikzpicture}
  % \caption{\label{fi♭_:22}}
\end{figure}

This kind of asymmetry arises repeatedly in category theory and our case can nicely be described in terms of subobjects of the subobject hyperdoctrine (the tablets) in a slice category
over a base object (the scale).

\subsection{Examples of the Calculi}
\label{sec:examples-2}

The gamut between topology and measure theory is one of the main themes of this work and concretely worked out in \cref{sec:topological}, \cref{xmp_:27}, \cref{xmp_:26},
\cref{xmp:28} and \cref{xmp:2}. But the notions of the fourfold calculus really can be applied wherever there are parts and wholes, and can produce very nice results: a particular
early highlight (and also a main source of difficulty in \cref{sec:calculus-smallness-}) is a complete description of smallness and density in a universal-algebraic setting that
also connects it to Frattini theory. Another interesting thread of examples that has not yet been mentioned is a description of smallness and density in a general setting for
convexity that we describe in \cref{sec:other-examples} and \cref{xmp_:15}; here we need to set up some notions for our statement that are of independent interest. There are also patterns in the
examples that cannot (yet) just be described, such as the behavior of smallness with regard to dimensionality (see mainly the examples in \cref{sec:topological} and \cref{xmp:14}).
Finally, the examples in \cref{sec:category-theoretic} will be re-used and extended in \cref{sec:covar-repr-theor}, and \cref{■pn_:64} will be used in later work.

Tablets and scales meanwhile depend on an element of choice that has to be given either by the context or the  mathematician. Sometimes, such as with topological or measure spaces,
or with subobjects in a concrete category as in \cref{xmp:monoid-action-density} or \cref{xmp_:15}, this context is clear. but often we have to pick the tablets and scales
ourselves. But this is also what allows us to be creative, such as in \cref{xmp:21}. Moreover, it is these relativizations that get us to some examples that are very close to
applied mathematics, such as \cref{xmp:1}, \cref{xmp_:1} and \cref{xmp_:22}.

\subsection{The Inter-Categoricity of the Calculus}
\label{sec:inter-dimens-calc}

With these calculi in hand we can return to the intuition from the first paragraph of our introduction and provide it with a rigorous example, as their establisment in
\(\mathbf{UBP}\) allows us to do precisely what we sketched there: \emph{fix an object and vary its categories}, i.e. look at its incarnations throughout different categories and
how they reflect back on the base structure. In particular, \(\mathbf{Set}\) maps into \(\mathbf{UBP}\) through the subobject functor \(\mathbf{SOb}_{\mathbf{UBP}}\), so each
concrete category \(\mathcal{C}\) whose concretification functor \(\mathopen{∣}-\mathclose{∣}\) preserves monomorphisms induces an embedding \(\mathbf{SOb}_{\mathcal{C}}(X) \hookrightarrow \mathbf{SOb}_{\mathbf{Set}}(\mathopen{∣}X\mathclose{∣})\) which we can
use as tablets, and a basic strategy we employ in \cref{sec:examples-1} is to use the subobjects of different concrete categories to infer properties of their underlying subsets.
on the other hand, since \(\mathbf{SOb}_{\mathcal{C}}\) goes into \(\mathbf{UBP}\) for each category \(\mathcal{C}\), the same scale can be used to measure the subobjects of any category.

So if we fix for instance \(ℝ\) as an object, then
\cref{xmp_:26,xmp:monoid-action-density,xmp_:25,xmp:28,xmp_:12,xmp:6,xmp_:15,xmp:30,xmp_:22,xmp_:23} all induce notions
of smallness, density, reducedness and satedness for subsets of \(ℝ\), each illuminating a slightly different aspect of the notions in the same context (the subset preorder
\(\mathbf{P}(ℝ) := \mathbf{SOb}_{\mathbf{Set}}(ℝ)\)), but with respect to different bases for inference and thus a different shade of meaning, and with their own relations between
them. This kind of inter-categorical analysis, where we use the interaction between subobjects of different categories in \(\mathbf{UBP}\) or between the subobjects of a
category and an object in \(\mathbf{UBP}\), is more or less explicitly woven throughout this text, and to a small extent formalizes the usage of subobjects in everyday mathematics,
where categories are often not kept track of and subobjects of different kinds (subsets, algebraic, topological, order-theoretic subobjects…) are mixed with at most implicit
regard for the functorial system connecting these different subobjects into ``the same'' object.

\subsection{The Development of the Fourfold Calculus}
\label{sec:devel-fourf-calc}

The parameters of the combined calculus are not exhausted by our \(2 * 4 = 8\) calculation where the four possible properties of elements of a preorder \(P\) vary along two
parameters: we can also vary \(P\) itself. This raises the question of the behavior of the combined calculus under preorder morphisms, which we investigate in
\cref{sec:smalln-under-morph}.

It is a
simple but fundamental result of inter-categorical subobject analysis that this last parameter, variation under morphisms, splits in half when we consider the interaction with category
theory through the subobject functor:

\begin{itemize}
\item a monomorphism \(m{:} X \hookrightarrow Y\) in a category \(\mathcal{C}\) induces an ideal embedding \(ι_{*}{:} \mathbf{SOb}_{\mathcal{C}}(X) \hookrightarrow \mathbf{SOb}_{\mathcal{C}}(Y)\) in \(\mathbf{UBP}\),

  \item a monomorphism-preserving functor induces for each \(X\) a top-preserving morphism \(F_{*, X}{:} \mathbf{SOb}_{\mathcal{C}}(X) → \mathbf{SOb}_{\mathcal{D}}(X)\).
\end{itemize}

Thus:

\begin{itemize}
\item we can study the behavior of each notion of the fourfold calculus under embeddings \emph{in} categories,

\item but we can also study their behavior under (monomorphism-preserving) functors \emph{between} categories.\footnote{Since natural transformations are
  componentwise given by morphisms in \(\mathcal{D}\), they reduce to the first case.}
\end{itemize}

These two morphism types, ideal embeddings and top-preserving morphisms, are orthogonal in \(\mathbf{UBP}\), so that we can study the notions of the fourfold calculus both in and
between categories as they develop when we add assumptions on them. Moreover, if \(\mathcal{C}\) has pullbacks, we can also study the behavior of the calculus under them, and if
\(\mathcal{C}\) is regular, then non-monomorphisms \(f{:} X → Y\) also induce direct images \(f_{*}{:} \mathbf{SOb}_{\mathcal{C}}(X) → \mathbf{SOb}_{\mathcal{C}}(Y)\). We trace out the resulting development of the fourfold calculus in \cref{sec:inter-categ-size} when we assume, in increasing specificity:

\begin{enumerate}
\item \label{item:31} strict initial objects,

\item\label{item:32} pullbacks,

\item\label{item:33} regularity,

\item\label{item:34} coherence.
\end{enumerate}

This is where the brokenness of the duality between smallness and density comes into full light, which is perhaps the most aesthetically pleasing aspect of the current theory.

\subsection{Diamond Structures}
\label{sec:diamond-structures}

So the subobject 2-functor \(\mathbf{SOb}\) has both the monic part of categories and the monomorphism-preserving functors and monic natural
transformations cast shadows into \(\mathbf{UBP}\), in a way that some things are forgotten while others are cast into even starker relief — like any good projection, with the
crucial difference that \emph{both} the \emph{internal structure in} each category and the \emph{external structure between} categories are on roughly the same footing — a rather unique
feature. The only rigorous point of comparison is to the 2-functor of slicing.

To sketch out this connection, let us quickly remember this 2-functor. The 2-category \(\mathbf{Cat}\) has a property somewhat akin to closure under taking powers: for each category
\(\mathcal{C}\) there exists a 2-functor \(\mathcal{C}_{/(-)}{:} \mathcal{C} → \mathbf{Cat}\) which maps \(X{:} \mathcal{C}\) to the slice category \(\mathcal{C}_{/X}\) and \(f{:} X → Y\) to the functor \(f_{*}\) that acts
through postcomposition, so a kind of representation of \(\mathcal{C}\) in \(\mathbf{Cat}\). So each object of \(\mathbf{Cat}\) is also represented in \(\mathbf{Cat}\), or rather its full
sub-2-category \(\mathbf{Cat}_{*}\) of categories with terminal objects (since \(X\) is terminal in \(\mathcal{C}_{/X}\)). These
representations interact well with functors and natural transformations and thus lift to one \(\mathbf{Cat}\)-parameterized family of representations of categories in
\(\mathbf{Cat}_{*}\), and through that to a representation of \(\mathbf{Cat}\) in some 2-category \(\mathbf{Cat}_{◆ \mathbf{Cat}_{*}}\) of \(\mathbf{Cat}_{*}\)-valued representations of
objects in \(\mathbf{Cat}\) (we pronounce this construction ``diamond-over'') and suitable morphisms and 2-morphisms between them. It is not that strange that the 2-category of
categories should have such a reflection-property, but if \(\mathcal{C}\) is monic, then the functor \(\mathcal{C}_{/(-)}\) factors through the \(\mathbf{UBP}\)-valued functor \(\mathbf{SOb}\).\footnote{Note
that this 2-functor is covariant though, in difference to the 2-functors in indexed category theory or the theory of hyperdoctrines. In fact, the representation theorems which we
derive would not work for the contravariant subobject functor — from this perspective it seems that the universal property of the contravariant subobject functor among other
hyperdoctrines \emph{arises} through its adjointness to the covariant one.}

If \(\mathcal{A}\) is a sub-2-category of \(\mathbf{Cat}\) consisting of monic categories, then it itself parameterizes a family of
representations of objects of \(\mathcal{A}\) in \(\mathbf{UBP}\) with natural transformations between them, meaning that there exists a 2-functor
\(\mathbf{SOb}_{(-)}{:} \mathcal{A} → \mathcal{A}_{◆ \mathbf{UBP}}\) which assigns to each \(\mathcal{C}{:} \mathcal{A}\) the functor \(\mathbf{SOb}_{\mathcal{C}}{:} \mathcal{C} → \mathbf{UBP}\). Now, note that
\(\mathbf{SOb}_{\mathcal{C}}\) takes values in the monic part \(\overset{\hookrightarrow}{\mathbf{UBP}}\) of \(\mathbf{UBP}\) (i.e. its subcategory of monomorphisms). In fact it maps each monomorphism
\(ι{:} A \hookrightarrow X\) in \(\mathcal{C}\) to an ideal embedding \(ι_{*}{:} \mathbf{SOb}_{\mathbf{C}}(A) \hookrightarrow \mathbf{SOb}_{\mathbf{C}}(X)\), which is principal because \(\mathbf{SOb}_{\mathbf{C}}(A)\) has a top \(\mathbf{id}_{A}\) and so designates an element \(ι := ι_{*}(\top)\) of
\(\mathbf{SOb}_{\mathbf{C}}(X)\). So in \(\mathbf{SOb}_{\mathbf{C}}(X)\) the images of subobjects of \(X\) are as densely packed as possible without identifying any of them. If now
\(F{:} \mathcal{C} → \mathbf{UBP}\) is some other representation of \(\mathcal{C}\), then this maps \(ι\) to some other morphism \(F(ι)\), which also designates an element \(F(ι)\big(\top_{F(A)}\big)\) of
\(F(X)\), so that we can define a morphism \(F_{★, X}{:} \mathbf{SOb}_{\mathbf{C}}(X) → F(X){:}ι ↦ F(ι)\big(\top_{F(A)}\big)\). Thus, \(\mathbf{SOb}_{\mathcal{C}}\) is the initial representation of \(\mathcal{C}\) in
\(\mathbf{UBP}\). Taking this together we find that for each such \(\mathcal{A}\) the subobject 2-functor \(\mathbf{SOb}_{(-)}{:} \mathcal{A} → \mathcal{A}_{◆ \mathbf{UBP}}\) is adjoint to the 2-functor
\(\mathbf{s}{:}\mathcal{A}_{◆ \mathbf{UBP}} → \mathcal{A}\), which assigns to each functor
\(F{:} \mathcal{C} → \mathbf{UBP}\) its source, and so is an initial reflection of \(\mathcal{A}\) in \(\mathcal{A}_{◆ \mathbf{UBP}}\).

Proving this is the easy part, the hard part is to show that \(\mathbf{UBP}\), or rather its monic part \(\overset{\hookrightarrow}{\mathbf{UBP}}\), is actually universal as a 2-category
allowing for this kind of reflection property, which we do in \cref{sec:univ-char-mathbf}. What makes our proof particularly interesting is that it works essentially by
re-constructing the involved structure one level higher: what we use is that, if there is another 2-category \(\mathcal{B}\) that can play the same role as \(\mathbf{UBP}\) and if monic
2-cells in \(\mathcal{B}\) are stable under left whiskering (e.g. if \(\mathcal{B}\) consists of monomorphism-preserving functors), then we can \emph{also}
define a 2-functor \(\mathbf{SOb}_{\mathcal{B}}\) from the monic part of \(\mathcal{B}\) that we can use to get back to \(\mathbf{UBP}\). So if we transition to some other monic 2-category \(\mathcal{B}\), we can
use the \(\mathbf{SOb}_{\mathcal{B}}\)-2-functor as if \(\mathcal{B}\) were a category, by suitably truncating the local structure between monomorphisms into objects in \(\mathcal{B}\). We close the paper with
a conjecture that this is because \(\mathbf{UBP}\) is the first of a tower of \((n + 1)\)-categories which can at each level of the categorical ladder provide a common home for the
monic parts of \(n\)-categories through a suitable subobject \((n+1)\)-functor.

\section{Preamble}
\label{sec:preamble}

\subsection{Acknowledgements}
\label{sec:acknowledgements}

I would like to thank my advisor Tobias Fritz for his general support as well as numerous suggestions and observations, but above all for giving me the freedom to pursue the mathematics I
wanted. I would also like to thank the University of Innsbruck for allowing me to write this text as part of its doctoral program,
 as well as Kevin Carlsen and Morgan Rodgers for their feedback. Finally I would like to thank my parents.

This paper was written with the help of various LLMs. Their main contributions were:

\begin{itemize}
\item writing out examples,

\item proving pre-defined statements, often from a proof outline,

\item proof checking.
\end{itemize}

However, the majority of all labor has been undertaken by the author, including the definition of all major concepts, the statements of propositions (possibly not all lemmata), the
entirety of the large-scale architecture, all exposition and the orientation and perspective of the text.

\subsection{Discussion of Prior Literature}
\label{sec:disc-prior-liter}

Our definition of smallness and density generalizes a few notions currently in use:

\begin{itemize}
\item Our notion of a dense element is a generalization of a notion in lattice theory (see e.g. \cite{kumarFILTERSDISTRIBUTIVELATTICES}), where we use partially defined joins
  instead of a globally defined pseudocomplement. That notion in turn generalizes that of density in topological spaces. Similarly, smallness generalizes nowhere density.

\item Our notion also generalizes superfluous and essential modules in algebra as described in \cite{andersonRingsCategoriesModules1992}, which already noticed a similarity to dense
  and nowhere dense subsets in topology. In particular, the definition of an essential element as given in Definition 1 of \cite{grzeszczukGoldieDualGoldie} is essentially
  the same definition as that of a dense element, but restricted to a modular lattice.

\item In the algebraic context, our theory possesses significant overlap with \emph{Frattini theory}, i.e. with extensions of the group-theoretic notion of a Frattini subgroup
  \cite{zotero-item-4568} to 
  other areas of algebra. There are several such extensions such as \cite{kohFrattiniSublatticeLattice1971}, \cite{jinNilpotencyFrattiniTheory2026} and
  \cite{mcalisterFRATTINIPROPERTIESNILPOTENCY2019} and there is an understanding that these are instantiations of a universal-algebraic notion. The paper most significant to us is
  \cite{lippariniNongeneratorsExtensionsInfinitary2022}, which develops a theory of generators in a universal setting. We will investigate the relation of our notions to Frattini
  theory in \cref{sec:algebraic}. 
\end{itemize}

Another connection arises when we extend our considerations to the categorical context provided by the subobject-functor \(\mathbf{SOb}_{\mathcal{C}}\) on some category \(\mathcal{C}\). This functor
is used more or less explicitly in a lot of categorical definitions and constructions. For instance, the definition of a coherent category works by describing
\(\mathbf{SOb}_{\mathcal{C}}(X)\) for each \(X{:} \mathcal{C}\) and how the morphisms act on subobject preorders by pullbacks. The \(\mathbf{SOb}_{\mathcal{C}}\)-functor is a hyperdoctrine in a
general sense, as was already observed for subsets by Lawvere in the introduction of hyperdoctrines \cite{lawvereAdjointnessFoundations1969}, and also the monic part of the
self-indexing of \(\mathcal{C}\) given by slicing (see, e.g. Section B1.2 of \cite{johnstoneSketchesElephantTopos2002}). It was also investigated in \cite{NorthHollandMathematicalLibrary},
starting in Ch. I 1.4. Moreover, a classical result in topos theory (see, e.g. 1.912 in ibid.) states
that a subobject classifier is equivalently an object representing the subobject functor. However, it should be noted that the subobject-functor \(\mathbf{SOb}_{\mathcal{C}}\) comes in a
covariant variant \(\mathbf{SOb}_{\mathcal{C}, *}\) for any category \(\mathcal{C}\) and, if \(\mathcal{C}\) has pullbacks, a contravariant one \(\mathbf{SOb}_{\mathcal{C}}^{*}\), which can be bundled together into one
functor \(\mathbf{SOb}_{\mathcal{C}}\) that maps morphisms in \(\mathcal{C}\) to adjoint pairs of preorder morphisms, and most attention on \(\mathbf{SOb}_{\mathcal{C}}\) has been on its \emph{contravariant}
part. For instance, both hyperdoctrines and indexed categories are taken as contravariant 2-functors.
But to us the \emph{covariant} part is more important because it exists on every category and our considerations in \cref{sec:covar-repr-theor} function solely for the covariant
subobject functor, while the contravariant part seems ``anchored'' mainly through its connection to the covariant one (see also \cref{rmk:4}). 

\cref{sec:inter-categ-size} also builds on a general categorical background. The main imported definitions, those of regular and coherent categories, are well-established. A reference
for regular categories that fits well with our material is Ch. I 1.5 of  \cite{NorthHollandMathematicalLibrary}, while Section A.1.4 of \cite{johnstoneSketchesElephantTopos2002}
provides a dedicated treatment of coherent categories.

A major impetus for the notion of Heyting spaces we introduce in \cref{sec:open-clos-dual} is the theory of semialgebraic spaces developed in \cite{delfsSemialgebraicTopologyReal1981}.

\subsection{Prerequisites}
\label{sec:prerequisites}

The prerequisites of this work vary widely depending on the section:

\begin{enumerate}
\item All of \cref{sec:smallness} can be understood with only minimal category-theoretic knowledge, \emph{except for} the examples in \cref{sec:examples}, which are grouped by area
  and of which some are more difficult to understand than the surrounding text. We provide the relevant definitions and try to introduce the context for each of these areas, but some of the
  significance might be lost without prior understanding of that area. These examples are built upon in the subsequent section but each ``thread'' of examples is largely
  self-contained and so can in principle be ignored, except for the category-theoretic examples, which re-appear in \cref{sec:univ-char-mathbf}. \cref{sec:tablets-spans-scales} is
  largely the same.

\item\label{item:14} \cref{sec:smalln-under-morph} contains some tedious proofs but does not require anything but elementary order-theory, but it is largely setup for
  \cref{sec:inter-categ-size} and \cref{sec:applications}, which don't strictly require but certainly benefit from a solid category-theoretic background. The beginning (up to the
  preservation and reflection statements) is also important for \cref{sec:covar-repr-theor}.

\item\label{item:7} \cref{sec:covar-repr-theor}, and \cref{sec:univ-char-mathbf} in particular, is by far the hardest part of the text and requires at least an elementary
  familiarity with (strict) 2-categories. With that said, the difficulty of \cref{sec:covar-repr-theor} lies not so much in its prerequisites and more in the cumulative complexity of the
  involved 2-cells.
\end{enumerate}

\subsection{Notation}
\label{sec:notation}

Throughout this article we adhere to the following notational conventions:

\begin{itemize}
  \item \textbf{Categories and objects.}
  We use calligraphic font for variables of a higher-categorical level and bold font for particular categories and functors. So we denote an arbitrary category \(\mathcal{C}\) but write
  \(\mathbf{PrO}\) for the 2-category of preorders and \(\mathbf{P}\) for the powerset functor.
    An object \(X\) belonging to a category \(\mathcal{C}\) is written \(X {:} \mathcal{C}\) (type-theoretic style). In preorders \(P\) we generally write \(a \in P\), which allows us to keep preorders notationally apart from
    outside categories.
    The category of small categories is \(\mathbf{Cat}\).
    We denote the monic part of a category \(\mathcal{C}\), i.e. its category of monomorphisms, as \(\overset{\hookrightarrow}{\mathcal{C}}\).

  \item \textbf{Morphisms and composition.}
    A morphism \(f\) from \(X\) to \(Y\) is written \(f {:} X \to Y\).
    Composition is always written in diagrammatic order: \(f ⨾ g\) denotes ``\(f\) then \(g\)''.  If we vertically compose 2-cells, we use the unadorned semicolon: \(φ; ψ\).
    Identities are \(\mathbf{id}_X\) or \(\top_X\) if we reason internally to the subobject preorder of \(X\).

  \item \textbf{Subobjects.}
    We denote subobjects as \(ι {:} A \hookrightarrow X\) and \(κ, λ...\).
    For a category \(\mathcal{C}\), \(\mathbf{SOb}_{\mathcal{C}}(X)\) denotes the preorder of monomorphisms into \(X\) (subobjects of \(X\)).
    When \(\mathcal{C}\) is clear we simply write \(\mathbf{SOb}(X)\).

  \item \textbf{Preorders and lattices.}
    A preorder is a category with at most one arrow between any two objects.
    The top element of a bounded preorder is written \(\top\) (or \(\top_{P}\) to stress the preorder), the bottom \(\bot\) (or \(\bot_P\)).
    Joins are \(∨\), meets \(∧\). If we are taking unions and intersections of sets or set-like subobjects we write \(\cup\) and \(\cap\).
    \(\mathbf{UBP}\) is the 2-category of upward-bounded preorders and all monotone maps; \(\mathbf{UBP}_{w}\) is its wide sub-2-category of top-preserving maps.
    Dually, \(\mathbf{DBP}\) denotes the 2-category of downward-bounded preorders.
    The image of a preorder top \(\top_{P}\) under a morphism \(f{:} P → Q\) is written as \(\top_{f}\) as described in \cref{sec:smalln-under-morph}.

  \item \textbf{Miscellaneous.}
    The Yoneda embedding of a category \(\mathcal{C}\) is \(\widehat{(-)} {:} \mathcal{C} \hookrightarrow \widehat{\mathcal{C}}\).
    The set-theoretic powerset of a set \(X\) is \(\mathbf{P}(X)\).
    The subobject classifier in \(\mathbf{Set}\) is \(\mathbb{T} = \{\top,\bot\}\) with \(\top\) as ``true''.
    We denote the underlying function of a morphism \(f\) as \(\mathopen{∣}f\mathclose{∣}\).
\end{itemize}

All other notation is either standard or introduced locally and explained at its first occurrence.

\subsection{Categorical Assumptions}
\label{sec:size-assumptions}
In \cref{sec:covar-repr-theor} we need to consider 2-categories other than \(\mathbf{UBP}\) as well, which raises the
question of weakness. Our considerations do not use unitors and associators at all, so that our proofs would be needlessly complicated if we used an explicit weak foundation, i.e.
what are called bicategories. But if we were to strictify all 2-categories in question, this would create a disconnect with our use of preorders, which are only strict up to
homotopy. The way we solve this is to implicitly assume univalent foundations,\footnote{See \cite{HomotopyTypeTheory} for an introduction on univalent foundations and homotopy type
  theory more generally.} so that a ``2-category'' is already strict only up to homotopy. Consequently, we use equality signs for considerations in preorders, not equivalence
signs. It should be underlined that this does not matter in any way to our arguments, it merely simplifies their exposition.

\subsection{Subobject Embeddings}
\label{sec:subobject-embeddings}

Part of what allows our investigations is an attention on what it means to be a subobject that might in many situations be argued to be excessive. In particular, we always
separate the domain of a subobject from its embedding morphism. So for a general subobject we write \(ι{:} A \hookrightarrow X\), even if \(ι\) might be clear from the context. This can be
slightly cumbersome but offers several advantages:

\begin{itemize}
\item most pertinently to us, the subobject preorder \(\mathbf{SOb}_{\mathcal{C}}(X)\) of an object \(X\) in a category \(\mathcal{C}\), as the monic part of the slice category over
  \(X\), is defined in terms of the monomorphisms \(ι\) into \(X\).
    
\item Keeping track of a embeddings and their source independently separates properties of \(ι\) from those of \(A\). For instance, assume we study an embedding
  \(ι{:} A \hookrightarrow X\) of manifolds with boundary and want to say that the sub-manifold \(ι\) has no boundary as a manifold. This is actually a property of \(A\), not
  \(ι\), and any other embedding of \(A\) into \(X\) would have the same property. By contrast, if we were to say that \(ι\) has no boundary as a subset of \(X\), i.e. that its
  closure is its interior, this would be a statement about how \(A\) is embedded into \(X\) and is thus dependent on \(ι\). We will frequently encounter such situations in this and
  subsequent work, and distinguishing them notationally will help us navigate them safely.

\item The category that forms the center of our investigations, the category of upward-bounded preorders \(\mathbf{UBP}\), is canonically a (locally thin) 2-category. Thus, there
  is a further structure \(\leq\) between embeddings \(ι, κ{:} A \hookrightarrow P\) of (upward-bounded) preorders that will become increasingly important to us. This is most cleanly expressed
  through the typing we indicated.
\end{itemize}

We only stray from this rule of writing subobjects with their embedding morphisms in examples that concern the interaction of subsets and (more strictly defined) subobjects in a concrete category, where we identify a subset with its
canonical embedding (given set-theoretically by the identity on elements). E.g. in \cref{sec:algebraic} we often consider subobjects generated from subobjects in algebraic structures, so we write
\(A \subseteq X\) for a subset, leaving the embedding implicit, and \(⟨A⟩\) for the subobject generated by \(A\) to simplify notation.

\subsection{Duality}
\label{sec:duality}

The results from \cref{sec:smallness} up to \cref{sec:inter-categ-size} have pure dualizations. But in light of \cref{sec:inter-categ-size} some proficiency has to be built up with thinking on either side
of the duality; to use Grothendieck's vocabulary, to do a yoga it does not suffice to know that the asanas can be dualized, one actually has to do the other side. To deal with this, we provide lots of
examples, do some proofs on either side and provide explicit dualizations of the preservation and reflection results in \cref{sec:smalln-under-morph}.

\section{Smallness \& Density}
\label{sec:smallness}

Here we develop the most elementary notions we are using in more detail, starting with assumptions, or rather their absence:

\begin{Remark}
  \label{rmα_:20}
  In a general preorder, joins and meets don't always exist, but just as not all limits in a category need to exist so that we can talk about them, so too can we still use the notions of joins and
  meets in general preorders. So we don't assume they exist, but do assume that they fulfill their universal property if they do.
  
  Note how this is similar to several category-theoretic notions, e.g. separated presheaves. On a deeper level, the elements of a plain monoid are only connected through its monoid
  multiplication, so
  if we make this operation partial, we suffer a big loss.\footnote{Unless, of course, we add further structure, such as a source-target typing.} But the elements of a preorder have a categorical structure that partial operations can be described through using
  universal properties. This is why it is not particularly problematic to state our definitions without existence assumptions.
\end{Remark}
\begin{Definition}
  \label{defn_:29}
\index{downward-bounded preorder}
\index{upward-bounded preorder}
  A preorder \(P\) is \emph{upward-bounded} if it has a largest element \(\top\). \(P\) is \emph{downward-bounded} if it contains a smallest element \(\bot\). \(P\) is \emph{bounded} if it is
  upward- and downward-bounded.
\end{Definition}

Let us consider some elementary notions in this setting:

\begin{Definition}
  \label{defn_:58}
\index{dense element}
\index{small element}
  Given an upward-bounded preorder \(P\), an element \(a ∈ P\) is \emph{small} if the only \(b\) such that \(a ∨ b = \top\) is \(\top\).
  Dually, an element \(d\) in a downward-bounded preorder \(P\) is \emph{dense} if the only \(c\) such that \(c ∧ d = \bot\) is \(\bot\).
\end{Definition}

\begin{Example}
  \label{xmp:55}
  In each bounded preorder, \(\top\) is dense and \(\bot\) is small.
\end{Example}
\begin{Terminology}
  \label{trm:1}
  If \(a ∨ b = \top\) we say that ``\(a\) joins with \(b\)''. More generally, if \(a ∨ b = c\) we say that ``\(a\) joins with \(b\) to \(c\)''. Dually for meets.
\end{Terminology}

\begin{Proposition}
  \label{prp:1}
\index{small element!ideal property}
  The small objects of an upward-bounded preorder are stable under downward closure and binary unions, i.e. an ideal in the lattice-theoretic sense.
\end{Proposition}
\begin{proof}
  \label{prf:20}
  Given \(a \leq b\) with \(b\) small and \(c ∈ P\) such that \(a ∨ c = \top\), then by monotonicity of joins \(b ∨ c = \top\), so that \(c = \top\)  by smallness.

  Now let \(a, b\) be small, \(c\) such that \((a ∨ b) ∨ c = \top\) and \(d\) be some common upper bound of \(b\) and \(c\) (one exists because \(\top\) is one). Then each common bound
  \(e\) of \(a\) and \(d\) is \(\geq (a ∨ b) ∨ c = \top\) and thus equal to \(\top\), so that \(a ∨ d\) exists and is \(\top\). So by smallness of \(a\), \(d = \top\), so \(b ∨ c = \top\), so
  \(c = \top\) by smallness of \(b\).
\end{proof}

\subsection{Examples}
\label{sec:examples}

Let us now begin to study how these notions play out over mathematics using examples. To
avoid disturbing the exposition we will provide some proofs in lemmata after the main example, which introduces the context.

The first thing to note is that if there is too much decomposability, size becomes trivial:

\begin{Example}
  \label{xmp:9}
\index{Boolean algebra!smallness}
  No non-trivial element \(a\) of any Boolean algebra is small or dense, since it has a negation \(¬a\) such that \(a ∨ ¬ a = \top\) and \(a ∧ ¬ a = \bot\). In particular, smallness and density
  in the subobject lattice of every set are trivial.
\end{Example}

So at this point we will not get anything interesting from \(\mathbf{Set}\). What we need is structure.

\subsubsection{Combinatorial}
\label{sec:combinatorial}

Let us start with some simple cases:

\begin{Definition}
  \label{defn_:61}\index{unreachable preorder}
  An upward-bounded preorder \(P\) is \emph{unreachable} if \(\top\) is the only element in \(P\) that is not small.
\end{Definition}
\begin{Example}
  \label{xmp:75}
  Any total order with a top element is unreachable.
\end{Example}

\begin{Observation}
  \label{obs_:1}
\index{globe category}
\index{cube category}
\index{simplex category}
  Let \(⬢\) be any of
  \begin{enumerate}
  \item \label{item:22} the simplex category \(▲\), given by the order-preserving functions between finite ordinals,

  \item\label{item:35} the cube category \(■\), given by the free strict monoidal category generated by an object \(■^{1}\) with two maps \(ι_{0}, ι_{1}{:} ■^{0} \hookrightrightarrows ■^{1}\) from the monoidal
    unit \(■^{0}\), (see e.g. \cite{grandisCubicalSetsTheir})

  \item\label{item:41} the globe category \(●\), given by objects \(●^{n}, n ∈ ℕ\) with maps generated by \(σ, τ{:} ●^{n} → ●^{n + 1}\) subject to the relations \(σ⨾ σ = σ⨾ τ\) and \(τ⨾ σ = τ⨾ τ\),

  \item\label{item:42} the category of finite abstract polytopes and lattice-preserving maps \(⬟\), i.e.:

    \begin{enumerate}
    \item \label{item:43} lattices \(L\) with a top \(\top\) and a bottom \(\bot\),

    \item\label{item:44} such that all total sub-orders between \(\top\) and \(\bot\) have the same length, meaning that there exists a unique ranking function \(\mathbf{rank}{:} L → N\),

    \item\label{item:46} \(L\) is \emph{strongly connected}: for any two elements \(a, b ∈ L\) it holds that in the full sub-lattice \([a, b]\) of elements in \(L\) between \(a\) and \(b\) there exists a
      sequence connecting any two proper faces.

    \item\label{item:45} satisfies the \emph{diamond law}: if \(\mathbf{rank}(a) = \mathbf{rank}(b) - 2\) and \(a \leq b\), then there exist precisely two elements between \(a\) and \(b\).
    \end{enumerate}
  \end{enumerate}

  Then each object of \(⬢\) has only finitely many subobjects.
\end{Observation}

All of these are both atomic and coatomic. Remember the definition of atomic orders:

\begin{Definition}
  \label{defn_:25}
  An upward-bounded preorder \(P\) is \emph{atomic} if for each non-top element \(x\) there exists an atom \(α ∈ P\) such that \(α \leq x\).
  Dually \(P\) is \emph{coatomic} if for each non-bottom element \(x\) there exists a co-atom \(ω ∈ P\) such that \(x \leq ω\).
  \(P\) is \emph{atomistic} if each element of \(P\) is the join of the atoms below it and \emph{coatomistic} dually.
 \end{Definition}
\begin{Example}
  \label{xmp:11}
  Each finite preorder \(P\) is atomic and coatomic.
\end{Example}

\begin{Proposition}
  \label{■pn_:8}
  \index{small element!co-atom characterization} Let \(P\) be a coatomic upward-bounded preorder. Then \(a ∈ P\) is small if and only if it is under all co-atoms in \(P\).
\end{Proposition}
\begin{proof}
  Assume \(a\) is small and let \(ω\) be a co-atom. If \(a \nleq ω\), then \(ω ∨ a\) is larger than \(ω\) and thus \(\top\). By smallness, \(ω = \top\), contradicting \(ω < \top\). Hence \(a \leq ω\).

  Conversely, suppose \(a\) lies below every co-atom and let \(b\) be such that  \(a ∨ b = \top\). If \(b \neq \top\), there exists a co-atom \(ω\) with \(b \leq ω\); then \(a ∨ b \leq ω < \top\), a
  contradiction.
  Thus \(b = \top\) and \(a\) is small.
\end{proof}

\begin{Corollary}
  \label{coR_:2}
\index{unreachable preorder!criterion}
Let \(P\) be an upward-bounded preorder bounded by exactly one co-atom. Then \(P\) is unreachable.
\end{Corollary}
\begin{proof}
  \label{prf_:16}
  Since \(P\) has a co-atom \(ω\), there exists an element \(a\) below \(\top\). for any such \(a\) let \(b\) be such that \(a ∨ b = \top\). If \(b < \top\) then \(a, b \leq ω\) so that
  \(a ∨ b \leq ω < \top\). Thus, \(b = \top\).
\end{proof}
\begin{Corollary}
  \label{coR_:3}
\index{dense element!atomistic}
  Let \(P\) be atomistic with an atom. Then no non-trivial element of \(P\) is dense.
\end{Corollary}
\begin{proof}
  \label{prf_:4}
  This follows immediately from the dual of \cref{■pn_:8} and \(a = ∨_{α \leq a} α\) with each \(α\) being an atom.
\end{proof}

So a subobject of any of the objects we introduced in \cref{obs_:1} is small if and only if it is contained in all faces, which only happens in trivial cases. Let's put a pin in
that.

\subsubsection{Topological}
\label{sec:topological}

To see how \cref{defn_:58} works in a non-trivial setting we can return to topological spaces:

\begin{Example}
  \index{topological density}\index{nowhere dense subset}
  \label{xmp:7}
  A closed subset \(A\) is small in the lattice (more precisely co-frame) of closeds \(\mathbf{C}(X)\) of a topological space \(X\) if and only if the only closed subset \(B\) such that
  \(A \cup B = X\) is \(X\), so if the closure of its set-theoretic complement is \(X\) itself, and hence if it is
  nowhere dense. Dually, an open subset is dense in our sense if and only if it is dense as a topological subset.

  We can also ask for dense elements of \(\mathbf{C}(X)\). If \(X\) is Hausdorff, this can only happen in trivial cases, but it is common in algebraic geometry. For
  instance, in the spectrum \(X\) of a local ring \(R\) with a maximal ideal \(m\), the closed subscheme corresponding to \(R/m\) is contained in every other closed subscheme and
  thus ``closed-dense''.
\end{Example}
\begin{Remark}
  \label{rmα_:40}
  So \cref{defn_:58} generalizes the topological notion of density. But the reason we use the name ``density'' instead of the formal dual, ``largeness'', is a categorical one: if a subobject
  \(ι{:} A \hookrightarrow X\) in a category \(\mathcal{C}\) is dense in \(\mathbf{SOb}_{\mathcal{C}}(X)\), then per definition \(ι ∧ κ\) is non-empty for each non-initial \(κ{:} B \hookrightarrow X\), which implies that the
  pullback \(κ^{*}(ι)\) of \(ι\) along \(κ\) exists and is non-initial. So a dense subobject is one which is non-empty in every other subobject. We will investigate such categorical
  phenomena in \cref{sec:inter-categ-size}.
\end{Remark}

A lot of examples we will be looking at are variants of this basic one, but we have to be careful: for one, this notion of size depends on the granularity of our preorder. For instance:

\begin{Example}
  \label{xmp:63}
  \index{manifold boundary!smallness in disk category} \index{small disk} Let us study smallness in the category \(\mathbf{TDisk}\) whose objects are \(n\)-dimensional (closed)
  disks \(●^{n}\) and whose morphisms are closed continuous functions. The join \(ι ∨ ι'{:} ●^{n'} \hookrightarrow ●^{n}\) of two closed embeddings
  \(ι{:} ●^{m} \hookrightarrow ●^{n}, ι'{:} ●^{m'} \hookrightarrow ●^{n}\) is the smallest \(n'\)-dimensional disk in \(●^{n}\) containing both \(ι\) and \(ι'\), if it exists. It follows from \cref{le↘_:1}
  below that the image \(\mathbf{Im}(∂ ●^{n'})\) of the manifold-boundary of \(●^{n'}\) is contained in the set-theoretic union of \(ι\) and \(ι'\). So in particular, if
  \(ι ∨ ι' = \top_{●^{n}}\) then \(∂ ●^{n} \subseteq ι \cup ι'\), and conversely, \(\top_{●^{n}}\) is the only disk in \(●^{n}\) containing \(∂ ●^{n}\). Thus,
  \(ι ∨ ι' = \top_{●^{n}}\) if and only if \(ι \cup ι'\) contains \(∂ ●^{n}\).

  Given now \(ι{:} ●^{m} \hookrightarrow ●^{n}\), if \(ι \cap ∂ ●^{n}\) is nowhere dense in \(∂ ●^{n}\), then every closed subset \(ι'\) of \(●^{n}\) such that
  \(ι \cup ι' = ∂ ●^{n}\) has to contain \(∂ ●^{n}\), so that it has to be \(\top_{●^{n}}\). Conversely, if \(ι \cap ∂ ●^{n}\) is not nowhere dense in \(∂ ●^{n}\), then it contains an
  \((n-1)\)-dimensional disk \(κ\). Then the closure \(\overline{∂ ●^{n} ∖ κ}\) of the complement of \(κ\) in \(∂ ●^{n}\) is an \((n - 1)\)-dimensional disk \(ι'\) such that
  \(ι ∨ ι' = \top_{●^{n}}\).

  % Embedding \(●^{n}\) into \(ℝ^{n}\) in such a way that \(∂ κ\) is a flat circle we see that \(ι'\) can be completed to a sub-disk \(ι''\) of \(●^{n}\) which does not contain
  % \(ι\) (by separating \(●^{n}\) along \(∂ κ\)). Thus, \(ι ∨ ι'' = \top_{●^{n}}\), so that \(ι\) is not small.

  So since everything is a disk the size measure becomes one on the boundary.
\end{Example}
\begin{Lemma}
  \label{le↘_:1}
  \index{manifold boundary} Let \(●^{n} \subseteq ℝ^{n}\) be the closed \(n\)-disk, and let \(ι{:} ●^{m} \hookrightarrow ●^{n}\), \(ι'{:} ●^{m'} \hookrightarrow ●^{n}\) be closed embeddings whose join
  \(ι ∨ ι'{:} ●^{n'} \hookrightarrow ●^{n}\) exists in \(\mathbf{SOb}_{\mathbf{TDisk}}(●^{n})\). Then the image of the manifold boundary \(\partial ●^{n'} \subseteq ●^{n'}\) under
  \(ι ∨ ι'\) is contained in the set-theoretic union of the images of \(ι\) and \(ι'\).
\end{Lemma}
\begin{proof}
Write \(D := (ι ∨ ι')(●^{n'})\) for the image of the join and \(S := ι(●^{m}) ∪ ι'(●^{m'})\) for the union of the two images, so that \(D\) is a closed subset of \(●^{n}\)
homeomorphic to the closed \(n'\)-disk; we must show \(\partial D ⊆ S\). Note first that \(S\) is closed in \(●^{n}\): by the closed map lemma, \(ι\) and \(ι'\) are closed maps, since
their domains are compact and codomains are Hausdorff, so both images, and hence their union, are closed.

Suppose, for contradiction, that some boundary point \(x ∈ \partial D ∖ S\) exists. Since \(D\) is a manifold with boundary, \(x\) admits a neighborhood basis in \(D\) by sets homeomorphic to closed half-balls of \(ℝ^{n'}\). Because \(x ∉ S\) and \(S\) is closed, we may therefore choose a closed cap neighborhood \(C ⊆ D\) of \(x\) with the following properties:

\begin{enumerate}
\item \(x ∈ C^{∘}\), the interior of \(C\) in \(D\);
\item \(C ∩ S = ∅\);
\item \(C\) is strictly contained in \(D\).
\end{enumerate}

Concretely, let \(φ\) be a chart carrying a neighborhood \(U\) of \(x\) in \(D\) homeomorphically onto a neighborhood \(V\) of the origin in a closed half-space of \(ℝ^{n'}\), with \(φ(x) = 0\). Then \(φ(U ∩ S)\) is closed in \(V\) and does not contain the origin; since \(ℝ^{n'}\) is \(T_{4}\), a small closed half-ball about the origin avoids it, and pulling this half-ball back along \(φ\) yields \(C\).

Now cut the cap off: set \(D' := D ∖ C^{∘}\). By (2) we have \(S ⊆ D'\), while \(x ∈ C^{∘}\) ensures that \(D'\) is a proper subset of \(D\) — geometrically, we have removed from \(D\) a small boundary cap at a point lying in neither image. Moreover, \(D'\) is closed in \(D\), hence in \(●^{n}\), and up to homeomorphism it is still an \(n'\)-disk; we may therefore regard it as the image of a closed embedding \(k{:} ●^{n'} ↪ ●^{n}\). Since the images of \(ι\) and \(ι'\) both lie within \(D'\), the two embeddings factor through \(k\) as subobjects of \(\mathbf{TDisk}\). Thus \(k\) is an \(n'\)-dimensional disk of \(●^{n}\) containing \(ι\) and \(ι'\) yet strictly smaller than \(ι ∨ ι'\), contradicting the defining property of the join as the smallest such disk.

The contradiction shows that no boundary point of \(D\) lies outside \(S\); that is,
\[\partial D = (ι ∨ ι')(∂ ●^{n'}) ⊆ S = ι(●^{m}) ∪ ι'(●^{m'}),\]
as claimed.
\end{proof}

\begin{Remark}
  \label{rmα_:33}
  In \cite{schreiberDifferentialCohomologyCohesive}, Schreiber uses the category \(\mathbf{CartSp}_{top}\), whose objects are the Cartesian spaces \(ℝ^{n}\) and whose morphisms are continuous
  functions, to construct a topological model for cohesion as developed by Lawvere in e.g. \cite{lawvereCohesiveToposesCantor1994}. The objects of \(\mathbf{TDisk}\) are compactifications of the spaces \(ℝ^{n}\) and the morphisms between them are also continuous
  functions, which are closed since each closed disk is compact Hausdorff. So \(\mathbf{TDisk}\) is the compact, i.e. universally closed, analogue of \(\mathbf{CartSp}_{top}\).
  \(\mathbf{TDisk}\) will be of importance in later texts due to its site structure. From that perspective, what \cref{xmp:63} shows to us is that there is an inherent conflict between a
  rich geometry and a manageable site structure, which was known before, albeit in different contexts.
\end{Remark}

But if the subobject preorder is sufficiently rich, closed and homogeneous, smallness reduces to lower-dimensionality. This is in particular the case for manifolds with boundary. For
this we only need local homeomorphy to a real vector space, Hausdorffness and second-countability are inessential to the question we are studying so we don't assume them:

\begin{Definition}[Topological manifold with boundary]
  \label{def:top-manifold-with-boundary}
  \index{manifold with boundary} Given \(n ∈ ℕ\), a \emph{(not necessarily Hausdorff or second-countable) topological \(n\)-manifold with boundary} is a topological space \(M\)
  such that for every point \(x ∈ M\) there exists

  \begin{itemize}
  \item an open neighborhood \(U \subseteq M\) of \(x\), and
  \item a homeomorphism (a \emph{chart}) \(φ {:} U \xrightarrow{\;\cong\;} V\) onto an open subset \(V \subseteq H^{n}\),
  \end{itemize}

  where the \emph{closed upper half-space} \(H^{n} \subseteq ℝ^{n}\) is defined by
  \[
    H^{n} := \{ (x_1, ..., x_n) ∈ ℝ^{n} \mid x_n ≥ 0 \}
  \]

  with the induced topology. The collection of such charts is called an \emph{atlas} for \(M\).

  A point \(x ∈ M\) is called an \emph{interior point} of \(M\) if there exists a neighborhood \(U_{x}\) of \(x\) and a homeomorphism \(φ{:} U_{x} \cong ℝ^{n}\) and a \emph{boundary
    point} otherwise.

  An \emph{embedding of topological manifolds} \(ι{:} A \hookrightarrow M\) is a regular monomorphism of topological spaces, i.e. a continuous injection such that each open subset of
  \(A\) is the restriction of some open of \(M\). \(ι\) is \emph{closed} if its image in \(M\) is.
\end{Definition}
\begin{Remark}
  \label{rmα_:31}
  The set of all interior points of a manifold with boundary \(M\) is denoted \(M^{\circ}\); it is an open subset of \(M\) and carries the structure of a (boundaryless) \(n\)-manifold.
\end{Remark}

\begin{Example}
  \label{xmp:60}
  \index{small submanifold}
  For a manifold with boundary \(X\), write \(\mathbf{SOb}_{∂ \mathbf{Mfld}}(X)\) for the preorder of closed embedded submanifolds of \(X\), ordered by inclusion.
  Then a closed submanifold \(ι {:} A \hookrightarrow X\) is small in \(\mathbf{SOb}_{∂ \mathbf{Mfld}}(X)\) precisely when \(\mathbf{dim} A < \mathbf{dim} X\).
\end{Example}
\begin{proof}
  Write \(n = \mathbf{dim} X\).

  If \(\mathbf{dim} ι < n\) we have to show that a closed submanifold of lower dimension is nowhere dense in \(X\).
  Suppose \(κ \hookrightarrow X\) is another closed submanifold with \(ι ∨ κ = X\). Then the set-theoretic union \(ι \cup κ\) has to consist of all of \(X\). To see this, assume otherwise. Then there exists a
  \(p ∈ X ∖ ι \cup κ\). Since \(X\) is a manifold, there exists a local homeomorphism \(φ_{p}\) from a neighborhood \(U\) of \(p\) to \(D\), which is either \(ℝ^{n}\) or \(ℝ^{n}_{\geq 0}\) mapping \(p\) to \(0\). Since \(ι\) and \(κ\) are closed,
  \(ι \cup κ\) is as well, so that \(φ_{p}(ι \cup κ)\) is closed in \(D\) and does not contain \(0\). Since \(ℝ\) is regular, there exists a closed disk or half disk \(●_{p}\) around \(0\) that is disjoint
  from \(φ_{p}(ι \cup κ)\), so that \(φ_{p}^{-1}(●_{p})\) is disjoint from \(ι \cup κ\). Assume now that \(Y\) is a closed submanifold with boundary containing \(ι \cup κ\) and \(p\). Then
  \(\overline{Y ∖ ●_{p}}\) is also a closed submanifold with boundary that contains \(ι \cup κ\) but does not contain \(p\). Thus, \(p\) is outside of \(ι ∨ κ\) by the minimality of the join,
  so no such \(p\) can exist.

  Now, because \(ι\) is lower-dimensional, it is nowhere dense in each coordinate chart around its points and thus nowhere dense. Thus, the interior of \(ι ∪ κ\) coincides with the interior of \(κ\).
  Hence \(κ^{\circ}\) is dense in \(X\).  As \(κ\) is closed, this forces \(κ = X\).
  Thus the only subobject joining with \(ι\) to give \(X\) is \(X\) itself, so \(ι\) is small.

  \emph{If \(\mathbf{dim} ι = n\):}
  Then \(ι\) is a closed submanifold of full dimension. Choose a closed \(n\)-disk \(κ{:} ●^{n} → ι^{\circ}\). Then \(\overline{X ∖ κ}\) is a closed submanifold with boundary and \(ι ∨ \overline{X ∖ κ} = X\), showing that \(ι\) is not small.
\end{proof}

So here, unlike in \cref{xmp:63} we have enough regularity to make our definition of smallness local.

\subsubsection{Graph-Theoretic}
\label{sec:graph-theoretic}

Another decently simple but important example is that of graphs:

\begin{Example}
  \label{xmp:17}
\index{dense subgraph}
\index{small subgraph}
  Let \(G\) be a (directed) graph with vertex set \(V(G)\) and edge set \(E(G)\). Let \(\mathbf{SOb}(G)\) denote the preorder of (directed) subgraphs of \(G\), ordered by inclusion of vertices and
  edges. This is an upward-bounded preorder with top element \(\top_G = G\) and bottom element \(\bot_G = ∅\) (the empty graph on the empty vertex set). Joins of finite families of subgraphs are given
  by set-theoretic unions on both vertices and edges:
  \[
    ι_1 ∨ ι_2 \;\;=\;\; ι_1 \cup ι_2,
  \]
  and meets are intersections \(ι_1 ∧ ι_2 = ι_1 \cap ι_2\).

  Applying \cref{defn_:58} to \(\mathbf{SOb}(G)\) gives concrete notions of small and dense subgraphs:

  \begin{enumerate}
  \item \textbf{Small subgraphs.} A subgraph \(ι \hookrightarrow G\) is small in \(\mathbf{SOb}(G)\) if the only subgraph \(κ \hookrightarrow G\) with
    \[
      ι ∨ κ \;=\; ι \cup κ \;=\; G
    \] is \(G\). Assume now that \(ι\) contains no edges or isolated vertices in \(G\). Then, since \(ι \cup κ\) has to contain all edges and isolated vertices in \(G\), \(κ\) has to
    contain them as well, so that \(κ = G\) and
    \(ι\) is small. Conversely, if \(ι\) contains an edge or an isolated vertex \(a\), then \(ι \cup (G ∖ a) = G\), so that \(ι\) is not small.

  \item \textbf{Dense subgraphs.} It follows from the dual of \cref{■pn_:8} that a subgraph is dense if and only if it contains all vertices of \(G\).
  \end{enumerate}
\end{Example}

\subsubsection{Algebraic}
\label{sec:algebraic}

We can give general statements about smallness and density for a whole slew of algebraic theories. For this, we work in a category \(\mathcal{C}\) that is monadic over \(\mathbf{Set}\). Monads provide
a categorical framework for universal algebra and the categories of most algebraic structures are monadic. Let us quickly recount the definitions:

\begin{Definition}[Monad]
  \label{def:monad}
\index{monad}
  A \emph{monad} on a category \(\mathcal{C}\) consists of
  \begin{itemize}
    \item an endofunctor \(T {:} \mathcal{C} \to \mathcal{C}\),
    \item a natural transformation \(η {:} \mathbf{id}_{\mathcal{C}} \Rightarrow T\) (the \emph{unit}), and
    \item a natural transformation \(\mu {:} T^{2} \Rightarrow T\) (the \emph{multiplication})
  \end{itemize}
  such that the following diagrams commute for every \(X {:} \mathcal{C}\):
  \[
  \begin{tikzcd}
    T^{3}(X) \ar[r, "\mu_{T(X)}"] \ar[d, "T(\mu_{X})"'] & T^{2}(X) \ar[d, "\mu_{X}"] \\
    T^{2}(X) \ar[r, "\mu_{X}"'] & T(X)
  \end{tikzcd}
  \qquad
  \begin{tikzcd}
    T(X) \ar[r, "T(η_{X})"] \ar[dr, "\mathbf{id}_{T(X)}"'] & T^{2}(X) \ar[d, "\mu_{X}"] & T(X) \ar[l, "η_{T(X)}"'] \ar[dl, "\mathbf{id}_{T(X)}"] \\
    & T(X) &
  \end{tikzcd}
  \]
  Equivalently, a monad is a monoid in the monoidal category \(([\mathcal{C}, \mathcal{C}], ⨾, \mathbf{id}_{\mathcal{C}})\) of endofunctors with composition.
\end{Definition}
\begin{Definition}[Algebra of a monad]
  \label{def:algebra-monad}
\index{Eilenberg-Moore algebra}
\index{algebra of a monad}
  Let \((T, η, \mu)\) be a monad on \(\mathcal{C}\). A \emph{\(T\)-algebra} (or \emph{Eilenberg-Moore algebra}) is a pair \((X, \xi)\) consisting of an object \(X {:} \mathcal{C}\) and a morphism \(\xi {:} T(X) \to X\) (the \emph{structure map}) such that the diagrams
  \[
  \begin{tikzcd}
    T^{2}(X) \ar[r, "T(\xi)"] \ar[d, "\mu_{X}"'] & T(X) \ar[d, "\xi"] \\
    T(X) \ar[r, "\xi"'] & X
  \end{tikzcd}
  \qquad
  \begin{tikzcd}
    X \ar[r, "η_{X}"] \ar[dr, "\mathbf{id}_{X}"'] & T(X) \ar[d, "\xi"] \\
    & X
  \end{tikzcd}
  \]
  commute.

  A morphism of \(T\)-algebras from \((X, \xi)\) to \((Y, θ)\) is a morphism \(f {:} X \to Y\) in \(\mathcal{C}\) satisfying \(\xi ⨾ f = T(f) ⨾ θ\). This forms the \emph{Eilenberg-Moore category} \(\mathcal{C}^{T}\).
\end{Definition}
\begin{Definition}[Monadic category]
  \label{def:monadic-category}
\index{monadic category}
  A category \(\mathcal{A}\) is called \emph{monadic} over \(\mathcal{C}\) (or \emph{monadic} with respect to a forgetful functor \(U {:} \mathcal{A} \to \mathcal{C}\)) if there exists a monad \(T\) on \(\mathcal{C}\) and an equivalence of
  categories \(\mathcal{A} \simeq \mathcal{C}^{T}\) that commutes with the forgetful functors to \(\mathcal{C}\). A functor \(U {:} \mathcal{A} \to \mathcal{C}\) is \emph{monadic} if it has a left adjoint \(F\) and the comparison functor \(\mathcal{A} \to \mathcal{C}^{T}\) (where \(T\) is the induced monad) is an equivalence.
\end{Definition}
\begin{Remark}
  \label{rmk:2}
  The categories of many algebraic structures such as groups, rings, modules over a fixed ring, etc. are monadic over \(\mathbf{Set}\). \(F(S)\) is given by the free algebra generated by \(S\), \(U(M)\)
  by the projection onto the underlying set and an epimorphism \(F(S) → X\) a generalized quotient by some system of algebraic relations between the elements of \(X\). So what monads really
  describe from our perspective is the idea of generation, and this is what we can use to describe smallness and density.
\end{Remark}

The most crucial consequence of assuming monadicity for our purposes is that \(U\) as a right adjoint creates limits. Hence, monomorphisms are lifted from \(\mathbf{Set}\) and are
thus injective functions, so that subobjects are special subsets (those that are algebraically closed). Moreover, the meet in the subobject lattice is computed as the set-theoretic
intersection, which is again a subalgebra. Colimits, i.e. joins, are more delicate however: suppose we have two subobjects \(ι{:} A \hookrightarrow M\) and
\(κ{:} B \hookrightarrow M\). Their join \(ι ∨ κ\) is the smallest subobject of \(M\) that contains both \(A\) and \(B\). In the category \(\mathcal{C}\) this is obtained by taking the image of the
canonical map \([ι, κ]{:} A + B → M\) from the coproduct \(A + B\) to \(M\) induced by the two inclusions. For \(\mathcal{C} = \mathbf{Set}\) the image of the function underlying that map is
precisely the union \(A ∪ B\) inside \(M\). For the same reason, \(F\) preserves the initial object since it is a colimit, so each subobject preorder
\(\mathbf{SOb}(M)\) has a bottom \(\bot_{M}\). Thus, density is easy to analyse:

\paragraph{Density}

Because the meet is just the intersection, the defining condition can be rephrased as:
\[
  \text{\(A\) is dense} \;\Longleftrightarrow\;
  \bigl(B \neq \bot \;\Rightarrow\; A \cap B \neq \bot\bigr).
\]
\begin{Example}
  \label{xmp:49}
\index{dense submonoid}
  In the lattice of submonoids of a monoid \(M\), \(ι\) is dense
  if and only if for each non-identity element \(m\) of \(M\) a non-identity power \(m^{n}\) of \(m\) is in \(ι\) (or \(M\) is trivial). This is because each non-trivial submonoid of \(M\) has to contain
  a non-identity element and thus its powers, and each submonoid \(B\) generated by elements \(b_{i}\) contains the submonoid generated by each \(b_{i}\).
\end{Example}

Now, the same idea as in \cref{xmp:49} generalizes as soon as the category \(\mathcal{C}\) has free algebras on one generator.
Recall that for a monad \(T\) on \(\mathbf{Set}\), the free \(T\)-algebra on a singleton set is written \(F(1)\).
A \emph{cyclic subalgebra} of \(M\) is the image of a morphism \(F(1) \to M\); equivalently it is the subalgebra \(\langle x \rangle\) generated by a single element \(x \in M\).
In every algebraic theory that is monadic over \(\mathbf{Set}\), every non-trivial subalgebra contains a non-trivial cyclic subalgebra, which we can construct by taking the
generalized orbit of some element \(M\) that is not in the intersection of all subalgebras. Consequently we obtain a fully general ``orbit'' criterion:

\begin{Proposition}[Cyclic-intersection criterion for density]
  \label{prp:density-cyclic}
\index{density!cyclic-intersection criterion}
\index{dense subalgebra!cyclic-intersection criterion}
  Let \(M\) be an algebra in a monadic category over \(\mathbf{Set}\).
  A subalgebra \(A \le M\) is dense in \(M\) if and only if for every element \(x \in M\) that does not lie in the minimal subalgebra \(\bot_{M}\),
  the cyclic subalgebra \(\langle x \rangle\) has non-trivial intersection with \(A\):
  \[
    \langle x \rangle \cap A \;\neq\; \bot_M .
  \]
\end{Proposition}
\begin{proof}
  If \(A\) is dense and \(x \notin \bot_M\), then \(\langle x \rangle \neq \bot_M\), so \(\langle x \rangle \cap A \neq \bot_M\).
  Conversely, let \(B \le M\) be a non-trivial subalgebra.  Choose any \(x \in B \setminus \bot_M\); then \(\langle x \rangle \le B\).
  The hypothesis gives \(\langle x \rangle \cap A \neq \bot_M\); hence \(B \cap A \neq \bot_M\).  Thus \(A\) is dense.
\end{proof}
\begin{Remark}
  \label{rmα_:29}
  Note that this proof works whether or not \(M\) has enough atoms. If it does however, then the dual of \cref{■pn_:8} shows that a subalgebra is dense if and only if it contains
  all atoms, which can be easier to check.
\end{Remark}

\paragraph{Smallness/Frattini Theory}

However, \(A ∪ B\) need not be a subalgebra; the categorical image closes the union under the algebraic structure it inherits from \(M\)
— that is, it takes the subalgebra generated by the union:

\[ι ∨ κ {:} ⟨A ∪ B⟩ \hookrightarrow M.\]

where \(⟨A \cup B⟩\) is the image of the canonical map \(F(U(A) \cup U(B)) → M\). So if we abuse notation for a moment by identifying a subobject with its image we can make the requirement for smallness a little more concrete:

\begin{Definition}
  \label{dfn:20}
\index{anchored element}
  Let \(A \le M\) be a subalgebra.
  An element \(a \in A\) is called \emph{anchored} (in \(A\)) if there exists a proper subalgebra \(B < M\) such that
  \[
    ⟨A ∪ B⟩ = M \qquad\text{and}\qquad a \notin B .
  \]
\end{Definition}
\begin{Proposition}
  \label{prp:3}
\index{small element!anchor criterion}
  \(A \le M\) is small in \(\mathbf{SOb}(M)\) if and only if \(A\) contains no anchored element.
\end{Proposition}
\begin{proof}
  \label{prf:2}
  \((\Rightarrow )\): If \(A\) is small and \(a \in A\) exists with \(⟨A ∪ B⟩ = M\), then smallness forces \(B = M\) so \(a ∈ B\). Hence no element of \(A\) is anchored.

  \((\Leftarrow )\): Suppose \(A\) is not small. Then there exists a proper subalgebra \(B < M\) with \(⟨A ∪ B⟩ = M\).
  Since \(B \neq M\), \(A\) cannot be contained in \(B\); pick \(a \in A \setminus B\).
  This \(B\) witnesses that \(a\) is anchored.
  Thus \(A\) contains an anchored element.
\end{proof}

In the finitary and finitely generated case we can further reduce smallness. This is where the connection to Frattini theory comes in:

\begin{Definition}[Finitary monad]
  \label{def:finitary-monad}
\index{finitary monad}
  A monad \(T\) on a category \(\mathcal{C}\) is called \emph{finitary} (or \emph{filtered-colimit preserving}) if its underlying endofunctor \(T {:} \mathcal{C} \to \mathcal{C}\) preserves filtered colimits.
\end{Definition}

Note the following standard fact about finitary monads:

\begin{Lemma}
  \label{le↘_:2}
  Let \(\mathcal{C}\) be a finitary monadic category and \(M\) an object in \(\mathcal{C}\). Then every subalgebra in \(M\) is the join of its finitely generated subalgebras.
\end{Lemma}
\begin{proof}
  \label{prf_:14}
  Let \(U {:} \mathcal{C} \to \mathbf{Set}\) be the forgetful functor, with left adjoint \(F\). The monad \(T = F⨾ U\) is finitary, hence \(T\) preserves filtered colimits. A standard result about finitary
  monads on \(\mathbf{Set}\) states that \(U\) creates (and therefore preserves) filtered colimits (see Proposition 4.3.2 of \cite{borceuxHandbookCategoricalAlgebra1994}).

  Let \(\mathcal{F}\) be the set of all finitely generated subalgebras of \(M\) contained in some subalgebra \(A\). Order \(\mathcal{F}\) by inclusion. If \(F_{1}, F_{2} \in \mathcal{F}\) are generated by the finite sets \(S_{1}, S_{2} \subseteq U(A)\), then the
  subalgebra generated by \(S_{1} \cup S_{2}\) is again in \(A\), finitely generated and contains both \(F_{1}\) and \(F_{2}\); hence \(\mathcal{F}\) is a directed set.

  Consider the directed diagram \(D {:} \mathcal{F} \hookrightarrow \mathbf{SOb}_{\mathcal{C}}(M)\) that sends each finitely generated subalgebra to its inclusion in \(M\). Since \(U\) creates filtered colimits, the colimit
  \(A\) of \(D\) exists and is created by the set-theoretic colimit of the underlying sets. The underlying sets of the subalgebras in \(\mathcal{F}\) are subsets of \(U(A)\) and the union of the
  underlying sets of all \(F \in \mathcal{F}\) is exactly \(U(A)\) (every element \(a \in U(A)\) lies in the subalgebra it generates, which is finitely generated). Thus the colimit
  \(\varinjlim_{F \in \mathcal{F}} U(F)\) is equivalent to \(U(A)\) through the map induced by the canonical inclusions.
\end{proof}

By \cref{prp:1} the diagram of small subobjects of \(M\) is filtered and thus has a colimit \(\mathbf{S}(M)\). This object is often the same as the Frattini subalgebra:

\begin{Definition}
  \label{dfn:22}
  \index{non-generator}\index{Frattini subalgebra}
  In an algebraic category \(\mathcal{C}\) an element \(g \in M\) is a \emph{non-generator} if for every subset \(S \subseteq M\)

  \[
    \langle S \cup \{g\} \rangle = M \quad\Longrightarrow\quad \langle S \rangle = M .
  \]

  The diagram of all finitely generated subalgebras of non-generators of \(M\) is finitely directed, so if \(\mathcal{C}\) is finitary, then its colimit is the union of its underlying sets. This is
  called the \emph{Frattini subalgebra} \(Φ(M)\) of \(M\).
\end{Definition}

This result is also standard (see for instance the classical group-theoretic treatment of the Frattini subgroup):

\begin{Proposition}
  \label{prp:frattini-nongenerator}
\index{non-generator!characterization}
In any finitary algebraic category \(\mathcal{C}\), let \(M\) be an object.  An element \(g \in M\) is a non-generator if and only if \(g\)
  belongs to every maximal proper subalgebra of \(M\).
\end{Proposition}
\begin{proof}
  Denote by \(\mathbf{SOb}(M)\) the lattice of subalgebras of \(M\); its meet is intersection and the join of a set of subalgebras is the subalgebra generated by their union. Because
  \(\mathcal{C}\) is finitary, every subalgebra is the join of its finitely generated subalgebras.

  Let \(g\) be a non-generator and \(N < M\) a maximal proper subalgebra. If \(g \notin N\), then \(N \subsetneq \langle N \cup \{g\} \rangle\). By maximality the only larger subalgebra is
  \(M\), hence \(\langle N \cup \{g\} \rangle = M\). Since \(g\) is a non-generator, this forces \(\langle N \rangle = M\), contradicting that \(N\) is proper. Thus \(g\) is in every maximal proper subalgebra.

  Conversely, assume \(g\) lies in every maximal proper subalgebra of \(M\). For every subset \(S \subseteq M\) with \(\langle S \cup \{g\} \rangle = M\) but \(\langle S \rangle \neq M\) the subalgebra
  \(H := \langle S \rangle\) is proper and does not contain \(g\) (otherwise \(H = M\)). If \(H\) were in a maximal subalgebra \(H'\) then \(\langle S \cup \{g\} \rangle \subseteq H'\), so \(H\) cannot lie in any
  maximal subalgebra. Thus, there has to exist an infinite chain \(H < H_{1} < ... < M\) such that \(g\) is not in any \(H_{i}\). But by Zorn's lemma, the chain \(H_{i}\) has an
  upper bound, so any maximal such chain is bounded by a maximal subalgebra, which does not contain \(g\) since \(\mathcal{C}\) is finitary and is thus proper. Consequently no such \(S\)
  exists, so that \(g\) must be a non-generator.

  Note that this argument does not require the \emph{a
    priori} existence of any maximal proper subalgebra; if there are none, the condition is vacuously true for every \(g\), and the proof
  above shows that then every element is a non-generator — which matches the intersection over the empty family being the whole algebra.
\end{proof}

So the notion of a non-generator is very close to the smallness of its orbit, but for a perfect correspondence we need to assume finite generation:

\begin{Definition}[Finitely generated algebra]
  \label{def:finitely-generated-algebra}
\index{finitely generated algebra}
  Let \(T\) be a finitary monad on \(\mathbf{Set}\). A \(T\)-algebra \((X, \xi)\) is called \emph{finitely generated} if there exists a finite set  \(K\) and an epimorphism of \(T\)-algebras \(F(K) \twoheadrightarrow (X, \xi)\) from the free \(T\)-algebra \(F(K)\) onto \((X, \xi)\).
\end{Definition}

In many algebraic theories this reduces to the standard notion: an algebra is finitely generated if it is generated by finitely many elements in the usual sense.
% In particular they fit nicely into our theory of smallness:

\begin{Lemma}
  \label{lem:enough-coatoms-fg}
  Let \(T\) be a finitary monad on \(\mathbf{Set}\) and \((X,\xi)\) a finitely generated \(T\)-algebra.
  Then the subalgebra lattice \(\mathbf{SOb}\bigl((X,\xi)\bigr)\) is coatomic.
\end{Lemma}
\begin{proof}
  Because \(T\) is finitary, the forgetful functor \(U {:} \mathbf{Set}^{T} \longrightarrow \mathbf{Set}\) creates filtered colimits; in particular, for any filtered family of subalgebras the underlying set
  of the colimit is the union of the underlying sets. Hence a chain of subalgebras of \((X,\xi)\) has a subalgebra as its union.

  Since \((X,\xi)\) is finitely generated, there exists an epimorphism \(e{:} F(K) \twoheadrightarrow (X,\xi)\). Since \(F\) is free, \(e\) is uniquely induced by a function
  \(\mathopen{∣}e\mathclose{∣}{:} K → \mathopen{∣}X\mathclose{∣}\) to the underlying set \(\mathopen{∣}X\mathclose{∣}\), so that \((X,\xi)\) itself is the only subalgebra that contains the set-theoretic image \(\mathopen{∣}e\mathclose{∣}(K)\) of \(K\).

  Let \(A \subsetneq X\) be a proper subalgebra. Consider the poset
  \[
    \mathcal{P} \;=\; \{\, B \le X \;|\; A \le B \;\text{and}\; B \neq X \,\},
  \]
  ordered by inclusion. That poset is non-empty because \(A \in \mathcal{P}\). Let \((B_{i})_{i \in I}\) be a chain in \(\mathcal{P}\) and set \(B := \bigcup_{i} B_{i}\) as the set-theoretic union. As noted,
  \(B\) is a subalgebra of \(X\) and clearly contains \(A\). If \(B = X\) then each of the finitely many generators \(k \in K\) lies in some \(B_{i}\); because the chain is totally ordered and
  \(K\) is finite, there exists an index \(i_{0}\) with \(K \subseteq B_{i_{0}}\). Since \(B_{i_{0}}\) is a subalgebra containing a generating set, we obtain \(B_{i_{0}} = X\), contradicting
  \(B_{i_{0}} \in \mathcal{P}\). Therefore \(B \neq X\) and \(B \in \mathcal{P}\); in other words, every chain in \(\mathcal{P}\) has an upper bound in \(\mathcal{P}\).

  By Zorn's Lemma, \(\mathcal{P}\) contains a maximal element \(M\). This \(M\) is a proper subalgebra of \(X\) that contains \(A\) and admits no proper intermediate subalgebra; hence \(M\) is a
  maximal proper subalgebra and so a co-atom of \(\mathbf{SOb}((X,\xi))\).
\end{proof}

\begin{Proposition}[Smallness and the Frattini subalgebra]
  \label{prp:5}\index{smallness and the frattini subalgebra}
  Let \(\mathcal{C}\) be a finitary algebraic category and \(M{:}\mathcal{C}\). Let \(A\) be a subalgebra of \(M\).

  \begin{enumerate}
  \item \label{item:108}   \(\mathbf{S}(M) \subseteq Φ(M)\).

  \item\label{item:132}    If \(M\) is coatomic (e.g. if \(M\) is finitely generated), then the converse holds and thus
    \[Φ(M) \subseteq \mathbf{S}(M)\].
  \end{enumerate}
\end{Proposition}
\begin{proof}
  \label{prf:7}
  Suppose \(A\) is small. Let \(N\) be any maximal proper subalgebra. If \(A \nsubseteq N\), then \(A ∨ N = M\) (by maximality of \(N\)), but \(N \neq M\), contradicting smallness of
  \(A\). Hence \(A ⊆ N\) for every maximal proper subalgebra, so \(A ⊆ Φ(M)\).

  \cref{item:132} follows from \cref{■pn_:8}.
\end{proof}

\begin{Example}
  \label{xmp:8}
  \index{small submonoid}
    In the lattice of submonoids of a monoid \(M\), the join \(ι ∨ κ\) is given by the completion of the set-theoretic union \(ι \cup κ\) under multiplication. By \cref{prp:3}, \(ι\) is small if
  and only if it contains no anchored element, and if \(M\) is finitely generated, this is the same as being in the intersection of all maximal subalgebras.
\end{Example}

When \(M\) does not have enough co-atoms, one can have a subalgebra contained in the Frattini subalgebra that is \emph{not} small.  Here is a simple counterexample:

\begin{Counterexample}
  \label{cxmp:7}
  Work in the finitary algebraic category of abelian groups and let \(M = \mathbb{Q}\).
  \begin{itemize}
  \item \(\mathbb{Q}\) is divisible, hence so is every quotient. The only simple abelian groups are the cyclic groups \(\mathbb{Z}/p\mathbb{Z}\) for primes \(p\), and none of these is divisible
    (multiplication by \(p\) is the zero map). Consequently \(\mathbb{Q}\) admits no maximal proper subgroup, and the Frattini subgroup \(Φ(\mathbb{Q})\) is the intersection over the empty family
    and thus \(\mathbb{Q}\) itself.

  \item Fix a prime \(p\) and set
  \[
    A \;=\; \mathbb{Z}_{(p)} \;=\; \left\{\frac{a}{b} \in \mathbb{Q} \;\middle|\; \gcd(b,\, p) = 1\right\}, \qquad
    B \;=\; \mathbb{Z}\!\left[\tfrac{1}{p}\right] \;=\; \left\{\frac{a}{p^{k}} \;\middle|\; a \in \mathbb{Z},\; k \geq 0\right\}.
  \]
  Both are proper subgroups of \(\mathbb{Q}\): \(\tfrac{1}{p} \notin A\) and \(\tfrac{1}{q} \notin B\) for any prime \(q \neq p\). Yet \(A ∨ B = A + B = \mathbb{Q}\). Indeed, any element of \(\mathbb{Q}\) can be written
  \(\frac{a}{p^{k}m}\) with \(\gcd(m,\, p) = 1\), and Bézout's identity supplies integers \(r,\, s\) with \(rm + sp^{k} = 1\), so that
  \[
    \frac{1}{p^{k}m} \;=\; \frac{r}{p^{k}} + \frac{s}{m} \;\in\; B + A.
  \]

\item Thus \(A\) joins with the proper subgroup \(B\) to the top element \(\top = \mathbb{Q}\), so that \(A\) is not small, even though \(A \subseteq \mathbb{Q} = Φ(\mathbb{Q})\).
  \end{itemize}
\end{Counterexample}

However, this difference of our theory to that of Frattini algebras is not necessarily a disadvantage: note that in \cref{cxmp:7}, our notion captures the idea of
relative smallness better, since the maximal subobject should not be small relative to itself.

\begin{Example}
  \label{xmp:50}
  \index{small submodule}\index{dense submodule}
  In the module theory of rings, the joins are sums of submodules and meets are set-theoretic intersections. There, smallness is also called
  \emph{superfluousness} while density is called \emph{essentiality}.
\end{Example}

\subsubsection{Algebro-Geometric}
\label{sec:algebro-geometric}

We can also do this in algebro-geometric settings:

\begin{Example}
  \label{xmp:18}
  \index{small closed subscheme} Let \(X\) be an affine scheme represented by a ring \(R\) and \(\mathbf{C}(X)\) its lattice of closed subschemes. This is dual to the lattice of quotient
  rings of \(R\) and thus that of ideals in \(R\). Since these are modules, this is actually a special case of \cref{xmp:50}. So for any proper ideal \(ι{:} I \hookrightarrow R\) the
  corresponding subscheme \(\mathbf{Sch}(ι): \mathbf{Sch}(R / I) \hookrightarrow \mathbf{Sch}(X)\) is small if and only if \(ι\) is dense. Nevertheless, a general formula is still hard to pin down. The
  most important cases are clear however:

  \begin{enumerate}
  \item\label{item:18}
    If \(ι\) contains a non-zero-divisor \(a\), then for each \(κ{:} K \hookrightarrow R\) such that \(ι ∧ κ = ι \cap κ = \{0\}\) and each \(k \in K \setminus \{0\}\), it holds that \(ak \neq 0\) since \(a\) is a non-zero-divisor, yet \(ak \in ι \cap κ = \{0\}\), a contradiction. So \(κ =
    \{0\}\). So \(ι\) is dense.

    These ideals are of Krull height \(> 0\) and thus their subschemes are lower-dimensional, in line with the case of manifolds. The difficulty is with the zero-divisors.

  \item\label{item:19} If \(R\) is reduced, \(ι\) consists of zero-divisors and there exists a \(b ∈ \mathbf{Ann}(ι)\) that is not in \(ι\), then the intersection \((b) \cap ι\) of the ideal
    generated by \(b\) with \(ι\) consists of elements \(br = i\) with \(r ∈ R, i ∈ ι\). But also \(bi = 0\) since \(i\) is in \(ι\), so that \(bbr = 0\) and in particular
    \((br)^{2} = brbr = (bbr)r = 0\) by commutativity. So \(br = 0\) since \(R\) is reduced, and so \(i = 0\), which means that \((b) \cap ι = (b) \cap \{0\}\). Thus, \(ι\) is not dense.

    But not all ideals of zero-divisors have a non-empty annihilator. A ring in which every finitely generated ideal of zero-divisors has a non-empty annihilator is said to have
    \emph{Property A}. This includes all Noetherian rings but also many others, see \cite{azarpanahIdealsConsistingEntirely2000}. Geometrically, non-nilpotent zero-divisors correspond to
    pieces of \(X\) with the same dimension, so at least in the finite-dimensional case the pieces of \(X\) outside of  the subscheme \(\mathbf{Sch}(ι)\) can be brought into a
    common ideal. So an equidimensional piece is never small (if \(R\) is reduced and Noetherian, of course).

  \item\label{item:20} If on the other hand \(R\) is not reduced and \(ι\) contains its annihilator \(\mathbf{Ann}(ι)\), then for any other non-zero ideal \(κ\), \(ικ \subseteq ι \cap κ\). If \(ικ ≠ \{0\}\) then \(ι \cap κ ≠ \{0\}\); if \(ικ = \{0\}\) then \(κ \subseteq \mathbf{Ann}(ι) \subseteq ι\), so \(ι \cap κ = κ ≠ \{0\}\). Either way, \(ι\) is dense. This is the case for instance for \((x)\) in \(R[x] / x^{n}\) (or more generally \(R[x, \overline{y}] / x^{n}\)). Here, \((x)\) corresponds to the
    quotient \(R[x]/x^{n} → R{:} x ↦ 0\) and thus the embedding of the reduced part of an infinitesimally thickened point.

  \item\label{item:131} We can also ask for what happens when an infinitesimally thickened point is embedded into a line. This amounts to asking whether the ideal \((x^{n})\) is dense in
    \(R[x]\) and thus is a special case of \cref{item:18}. So even though an infinitesimal thickening appears higher-dimensional than its reduction in itself (\((x)\) has Krull height 1 in
    \(R[x] / x^{n}\)), it seems lower-dimensional within \(R[x]\).
  \end{enumerate}

  So the general pattern is clear: if an embedding is lower-dimensional or (relatively) reduced in an infinitesimally thickened scheme, then it is small, if it is some equidimensional
  equally infinitesimally extended part it is not, except perhaps with infinite-dimensional shenanigans.

  \index{Jacobson radical} On the other hand, note that by Zorn's Lemma every ideal is bounded by some maximal ideal, so that the collection of small elements is the set of those elements
  that are below every maximal ideal by \cref{■pn_:8}. Since ideals are stable under intersections, this set is an ideal, which is also known as the \emph{Jacobson radical} (at least in the
  commutative case). Note that this means that the dual of the Jacobson-radical is the smallest ``closed-dense'' subscheme of \(X\). By duality, this is the smallest subscheme of \(X\) that
  contains all points of \(X\) which are minimal in that their corresponding ideal is maximal (not the points of minimal primes!).
\end{Example}
\begin{Remark}
  \label{rmk:15}
  We can use the results of \cref{sec:inter-categ-size} to extend smallness and density into general schemes.
\end{Remark}

\begin{Example}
  \label{xmp:14}
  \index{semialgebraic set!smallness} Let \(\mathbf{SAlg}(R^{n})\) be the lattice of semialgebraic subsets of \(R^{n}\), so of finite unions of finite intersections of polynomial
  equalities \(p = q\) and strict polynomial inequalities \(p < q\). It is well-known that \(\mathbf{SAlg}(R^{n})\) is closed under (Boolean) complementation. As a consequence,
  (absolute) size is trivial in \(\mathbf{SAlg}(R^{n})\). But we can define a semialgebraic set as

  \begin{enumerate}
  \item \label{item:109} \emph{open} if it can be defined solely through strict inequalities \(p < q\),

  \item\label{item:110} \emph{closed} if it can be defined solely through equalities \(p = q\) or inequalities \(p \leq q\).
  \end{enumerate}

  In particular, a closed semialgebraic subset \(ι\) is small in the lattice of closed semialgebraic subsets \(\mathbf{CSSAlg}(R^{n})\) if and only if it can be covered by a union
  of intersections of polynomial equalities, so if and only if it is lower-dimensional or equivalently if it can be contained in a finite number of algebraic varieties by
  \cref{prop:semialg-small-algebraic} below. Density meanwhile is trivial in \(\mathbf{CSSAlg}(R^{n})\) since it is atomistic, but of course the considerations for smallness in
  \(\mathbf{CSSAlg}(R^{n})\) are purely dual to those for density in the lattice of open semialgebraic subsets \(\mathbf{OSSAlg}(R^{n})\), which are only given through strict
  polynomial inequalities. Thus, \(\mathbf{SAlg}(R^{n})\) is kind of like a version of the topology of \(R^{n}\) that is small enough to be algebraically tractable. This is not a
  coincidence — semialgebraic subsets are a smaller but in some ways better-behaved version of the semi-analytic subsets that Grothendieck cites as a possible model for his vision
  of a tame topology in \emph{Esquisse
    d'un Programme}\cite{grothendieckSketchProgramme}:

  \begin{quote}
    It soon appeared that there could be no question of getting such an ambitious statement in the framework of topological spaces, because of the sempiternal ``wild'' phenomena… It was also clear that the contexts of the most rigid structures which existed then, such as the
    ``piecewise linear'' context were equally inadequate — one common disadvantage consisting in the fact that they do not make it possible, given a pair (U, S) of a ``space''
    \(U\) and a closed subspace \(S\), and a glueing map \(f {:} S → T\) , to build the corresponding amalgamated sum. Some years later, I was told of Hironaka's theory of what he
    calls, I believe, (real) ``semi-analytic'' sets which satisfy certain essential stability conditions (actually probably all of them), which are necessary to develop a usable
    framework of ``tame topology''.
  \end{quote}

  Both semialgebraic and semianalytic sets feature an instance of an open-closed duality in the sense that we will develop in \cref{sec:open-clos-dual}.
\end{Example}
\begin{Proposition}
  \label{prop:semialg-small-algebraic}
\index{semialgebraic set!smallness characterization}
  Let \(\mathbf{CSSAlg}(R^{n})\) be the preorder of closed semialgebraic subsets of \(R^{n}\) ordered by inclusion, regarded as an upward-bounded preorder with top \(\top = R^{n}\).
  For \(A ∈ \mathbf{CSSAlg}(R^{n})\), the following are equivalent:
  \begin{enumerate}
  \item \(A\) is \emph{small} in the sense of \cref{defn_:58}.
  \item \(A\) has empty interior in \(R^{n}\).
  \item \(\mathbf{dim} A < n\).
  \item \(A\) is covered by a finite union of (real) proper algebraic subvarieties of \(R^{n}\), i.e. sets of the form
    \[
      V(p_{1},\dots,p_{k}) = \{x ∈ R^{n} \mid p_{1}(x)=\dots=p_{k}(x)=0\}.
    \]
  \end{enumerate}
\end{Proposition}
\begin{proof}
  We proceed in steps.

  \emph{(1) ⇔ (2).}
  In \(\mathbf{CSSAlg}(R^{n})\), finite joins are given by set-theoretic unions: if \(A,B\) are closed semialgebraic, then \(A ∨ B = A ∪ B\), since the union of closed
  semialgebraic sets is again closed semialgebraic and minimal with this property. Thus (1) says that the only closed semialgebraic \(B\) with \(B ⊇ R^{n} \setminus A\) is \(B=R^{n}\).
  This is equivalent to
  \[
    \overline{R^{n} \setminus A} = R^{n},
  \]
  where the closure is in the Euclidean topology. But \(\overline{R^{n} \setminus A} = R^{n}\) if and only if \(A^{\circ} = ∅\), since
  \[
    A^{\circ} = R^{n} \setminus \overline{R^{n} \setminus A}.
\]

  \emph{(2) ⇔ (3).}
  A semialgebraic subset \(A ⊆ R^{n}\) has nonempty interior if and only if
  \(\mathbf{dim} A = n\). This is a standard fact of semialgebraic geometry. For \(\Rightarrow\), take \(x ∈ A^{\circ}\). Then there exists a ball \(●_{x} \subseteq A^{\circ}\) since \(R^{n}\) is regular Hausdorff, so that \(\mathbf{dim}(A) = n\).
  Conversely, it follows immediately from the Euclidean topology that if \(\mathbf{dim}(A) < n\) then \(A\) has an empty interior.

  \emph{(3) ⇔ (4).}
  Suppose \(\mathbf{dim} A < n\) and \(A\) is closed semialgebraic. We may write \(A\) as a finite union of semialgebraic subsets \(P_{i}\) each of which is of codimension at least \(1\). Each such piece lies inside an
  algebraic hypersurface, meaning the zero set of a nonzero polynomial. Thus, there exist finitely many polynomial maps \(f_{1},\dots,f_{r}\) such that
  \[
    A \subseteq \bigcup_{j=1}^{r} V(f_{j}).
  \]

  Conversely, if \(A\) is covered by a finite union of algebraic varieties in \(R^{n}\), each such variety is a real algebraic set of dimension at most \(n\), and if \(A\) had dimension \(n\), one of its
  components would contain a nonempty open subset of \(R^{n}\), which is impossible unless that component is all of \(R^{n}\) since a proper real algebraic subvariety of \(R^{n}\) has dimension strictly less than \(n\). Thus if \(A\) is a
  proper closed subset and covered by a finite union of algebraic varieties, then \(\mathbf{dim} A < n\). Hence (4) implies (3) in the nontrivial case \(A ≠ R^{n}\); for \(A=R^{n}\) all four conditions fail.

  Combining the implications, we obtain (1) ⇔ (2) ⇔ (3) ⇔ (4).
\end{proof}

In particular, the informal sentence in \cref{xmp:14} can be read precisely as: a closed semialgebraic subset \(A\) of \(R^{n}\) is small in \(\mathbf{CSSAlg}(R^{n})\) if and only if it has empty interior,
equivalently if and only if it is of strictly lower dimension than \(R^{n}\), equivalently if and only if it can be covered by a finite union of algebraic varieties.

\subsubsection{Order- and Category-Theoretic}
\label{sec:category-theoretic}

We can even apply this definition to the subobject lattices of preorders themselves, where we see a mix of the algebraic and graph-theoretic cases going on:

\begin{Example}
  \label{xmp:3}
\index{subobject preorder!of preorders}
  Let \(\mathbf{PrO}\) be the category whose objects are preorders \((P, \le_P)\) and morphisms \(ι {:} P → Q\) are monotone (order-preserving) maps. Then for each \(P\), the subobject
  preorder \(\mathbf{SOb}(P)\) has as objects monomorphisms into \(P\), so order-preserving injections. This is a lattice where, given \(ι{:} Q \hookrightarrow P, κ{:} R \hookrightarrow P\):

  \begin{enumerate}
  \item \label{item:4} the meet \(ι ∧ κ = ι \cap κ\) is given on elements by the intersection of the images of \(ι\) and \(κ\), such that for \(a, b ∈ ι ∧ κ\) it holds that \(a \leq b\) if and
    only if \(a \leq_{ι} b\) and \(a \leq_{κ} b\),

  \item\label{item:5} the join \(ι ∨ κ\) is the completion of the set-theoretic union under composition, so is given on elements by the set-theoretic union and such that for
    \(a, b ∈ ι ∨ κ\) it holds that \(a \leq b\) if and only if there exist families of elements \(c_{i} ∈ ι ∧ κ\) such that
    \(a \leq_{S} c_{0} \leq_{T} c_{1} ... \leq_{S} c_{n} \leq_{T} b\) with \(S = Q\) if \(a ∈ \mathbf{Im}(ι)\) and \(R\) otherwise, and similar for \(T\) and \(b\).
  \end{enumerate}

 We can describe the size of each
  \(ι{:} Q \hookrightarrow P\).

  \begin{enumerate}
  \item\label{item:134} \textbf{Density} Since the meets in \(\mathbf{SOb}(P)\) are given by set-theoretic intersections, density is similar to \cref{xmp:17}. In particular,
    \(\bot\) is the empty embedding and \(\mathbf{SOb}(P)\) has enough atoms. So by the dual of \cref{■pn_:8}, \(ι\) is dense in \(\mathbf{SOb}(P)\) if and only if it contains all
    atoms, i.e. one-element-embeddings into \(P\).

  \item \label{item:133} \textbf{Smallness} Since the join is given objectwise by the set-theoretic union and sequences of inequalities, any \(κ\) such that \(ι ∨ κ = \top\) has to
    generate all inequalities with \(ι\). Thus, \(κ\) has to contain all inequalities \(a \leq_{P} b\) with \(b ∈ ι\) such that there is no \(c ∈ ι\) such that
    \(a \leq c <_{ι} b\), so it has to contain each such \(b\). Thus, if \(ι\) does not contain any non-trivial inequalities and \(P\) does, then \(ι\) is small. More generally, let
    \(a < c < b\) in \(P\) be such that \(a <_{ι} b\) but not \(c ∈ ι\). Then \(κ\) has to contain all of \(a < c < b\) and thus \(a <_{κ} b\). Thus, if \(ι\) contains only
    morphisms in its image that are compositions of morphisms outside its image, then \(ι\) is small.

    Conversely, assume that \(ι\) contains a morphism \(a < b\) such that there
    does not exist a factorization of this morphism into morphisms outside the image of \(ι\) and let \(κ := \top ∖ (a < b)_{ι}\) be given by removing all morphisms from \(\top\) that
    \(a < b\) factors through and that are in the image of \(ι\). Then \(κ\) does not contain the chain of inequalities between \(a\) and \(b\) and thus \(a \cancel{<_{κ}} b\) but
    \(ι ∨ κ = \top\), so \(ι\) is not small.
  \end{enumerate}
\end{Example}

So here we see a mixture of algebraic and geometric smallness going on, where smallness both encompasses lower-dimensionality and factorizability. This is because preorders are special
categories, and we can (and have to) extend our considerations to general categories. For this let us first remember some functor types that will feature throughout this text:

\begin{Definition}[Faithful, full and fully faithful functors]
  \label{def:ff-es}
\index{fully faithful functor}
\index{full functor}
\index{faithful functor}
  Let \(\mathcal{C},\mathcal{D}\) be categories and \(F {:} \mathcal{C} \to \mathcal{D}\) a functor.

  \begin{enumerate}
  \item \(F\) is \emph{faithful} if for all objects \(X,Y {:} \mathcal{C}\) the function
    \[
      F_{X,Y} {:} \mathcal{C}(X,Y) \to \mathcal{D}(F(X),F(Y)),\quad f \mapsto F(f)
    \]
    is injective.  Equivalently, whenever \(f,g {:} X \to Y\) satisfy \(F(f) = F(g)\), we must have \(f = g\).

  \item \(F\) is \emph{full} if for all objects \(X,Y {:} \mathcal{C}\) the function \(F_{X,Y}\) is surjective.  Equivalently, every morphism \(h {:} F(X) \to F(Y)\) in \(\mathcal{D}\) is of the form \(h = F(f)\) for some \(f {:} X \to Y\) in \(\mathcal{C}\).

  \item \(F\) is \emph{fully faithful} if it is both full and faithful, i.e. for every pair \(X,Y {:} \mathcal{C}\) the map
    \[
      F_{X,Y} {:} \mathcal{C}(X,Y) \to \mathcal{D}(F(X),F(Y))
    \]
    is a bijection.
  \end{enumerate}
\end{Definition}

\begin{Example}[Smallness for faithful subcategories]
  \label{xmp:4}\index{subcategory}
  Let \(\mathcal{C}\) be a category and \(\mathbf{SOb}'_{\mathopen{∣}f\mathclose{∣}}(\mathcal{C})\) the preorder whose elements are faithful functors
  \(ι {:} \mathcal{D} \hookrightarrow \mathcal{C}\), ordered by \(ι \le κ\) if \(ι\) factors through \(κ\) up to isomorphism. The top element is \(\mathbf{id}_{\mathcal{C}}\); the bottom is the empty embedding.

  Joins in this preorder exist and are generated by the set-theoretic union by closure under morphisms. It follows again from the dual of \cref{■pn_:8} that a subcategory is dense if and only
  if it contains all objects of \(\mathcal{C}\).

  In general, \(\mathbf{SOb}'_{\mathopen{∣}f\mathclose{∣}}\) is a little pathological: for instance, the function between cardinals \([2] → [1]\) is a faithful functor (though it is not full) and
  thus in \(\mathbf{SOb}'_{\mathopen{∣}f\mathclose{∣}}\). But the defining condition makes it so that each faithful functor \(ι\) is equivalent to its image, literally given in
  terms of object and morphism sets. So for instance \(\mathbf{SOb}'_{\mathopen{∣}f\mathclose{∣}}([1])\) is not even a small category but equivalent to the one-object category. We
  can avoid this pathology by restricting to faithful functors that are also isomorphism-reflecting. So we define \(\mathbf{SOb}_{\mathopen{∣}f\mathclose{∣}}\) to consist of the
  isomorphism-reflecting faithful functors.
\end{Example}

Such an isomorphism-reflecting faithful functor is what we use as the definition of a subcategory:

\begin{Definition}
  \label{defn_:12}
  A\emph{subcategory} \(ι{:} \mathcal{C} \hookrightarrow \mathcal{D}\) is an isomorphism-reflecting faithful functor.
\end{Definition}

This is also known as a pseudomonic (see \cite{carboniModulatedBicategories1994}) and will factor in later.

We can mimic the arguments of \cref{sec:algebraic} for smallness step-by-step by using morphisms as generators. However, note that it is rare that a category of interest is generated by a
finite set of morphisms. Moreover, it is certainly not clear that a category like \(\mathbf{Top}\) is bounded by maximal proper subcategories. But we can still use our anchor criterion:

\begin{Definition}[Anchored morphism]
  \label{def:anchored-cat}
\index{anchored morphism}
  Let \(\mathcal{C}\) be a category and let \(\mathbf{SOb}_{\mathopen{∣}f\mathclose{∣}}(\mathcal{C})\) as in \cref{xmp:4}. A morphism \(f\) of \(\mathcal{C}\) is \emph{anchored} relative
  to a subcategory \(ι{:} \mathcal{A} \hookrightarrow \mathcal{C}\) if there exists a subcategory \(κ\) such that in \(\mathbf{SOb}_{\mathopen{∣}f\mathclose{∣}}(\mathcal{C})\)
      \[
        ι ∨ κ = \mathcal{C}
        \qquad\text{and}\qquad
        f \notin κ .
      \]
\end{Definition}

It is straightforward to show that a subcategory is small if and only if it does not contain an anchored morphism. Moreover, one simple criterion for smallness is that a subcategory \(ι\)
contains only morphisms that are composites of morphisms that are not factored by morphisms in \(ι\):

\begin{Definition}
  \label{dfn:21}
\index{irreducible morphism}
  A morphism \(f\) of a category \(\mathcal{C}\) is called \emph{\(ι\)-irreducible} for a subcategory \(ι {:} \mathcal{A} \hookrightarrow \mathcal{C}\) if \(f\) is either an identity or cannot be written as a
  composite of the form  \(f = g⨾ h⨾ i\) where at least one of \(g, h, i\) is a non-identity in \(\mathcal{A}\).
\end{Definition}

\begin{Proposition}[Categorial smallness criterion]
  \label{prp:7}\index{categorial smallness criterion}
  Let \(ι{:} \mathcal{A} \hookrightarrow \mathcal{C}\) be a faithful subcategory. Assume that every non-identity morphism \(f\) in \(ι\) admits a factorization
\[
  f \;=\; g_1 ⨾ g_2 ⨾ \dots ⨾ g_n
\]
into finitely many morphisms \(g_i\) that are \(ι\)-irreducible and not in \(ι\). Then \(ι\) is small in the
upward-bounded preorder of faithful subcategories of \(\mathcal{C}\).
\end{Proposition}
\begin{proof}
  \label{prf:5}
  We prove that under this hypothesis, \(ι\) cannot contain any anchored morphism.

  Given \(f\) in \(ι\) and a proper faithful subcategory \(κ \subsetneq \mathcal{C}\) such that \(ι ∨ κ \;=\; \mathcal{C}\) we may write \(f \;=\; g_1 ⨾ g_2 ⨾ \dots ⨾ g_n\) where each \(g_i\) is
  \(ι\)-irreducible. Hence \(g_i \notin ι\).

  Now, because \(ι ∨ κ = \mathcal{C}\), every morphism of \(\mathcal{C}\) can be expressed as a composition of morphisms taken from \(ι \cup κ\). In particular, for each \(i\) we can choose one such representation
  of \(g_i\):
  \[
    g_i \;=\; h_{i1} ⨾ h_{i2} ⨾ \dots ⨾ h_{ik_i}, \qquad h_{ij} \in ι \cup κ .
  \]

  Since morphisms are closed under composition and \(g_i\) is \(ι\)-irreducible, its representation cannot contain a non-identity factor from \(ι\). Therefore none of the \(h_{ij}\) belong to \(ι\). Consequently every \(h_{ij}\) lies in \(κ\).

  Thus each \(g_i\) is a composition of morphisms of \(κ\) and therefore belongs to the subcategory \(κ\). Finally \(f = g_1 ⨾ \dots ⨾ g_n\) is also a
  composition of morphisms in \(κ\), whence \(f \in κ\). Thus, no morphism of \(ι\) is anchored, so \(ι\) is small.
\end{proof}
\begin{Remark}
  \label{rmα_:4}
  This works great for preorders or deloopings of monoids without inverses but does not work for deloopings of groups or categories with non-trivial automorphism categories.
\end{Remark}

Both \cref{xmp:3} and \cref{xmp:4} are somewhat unsatisfying because they do not consider the categorical structure of the target category at all. We can define subobject preorders that do however:

\begin{Example}
  \label{xmp:5}
\index{subobject preorder!2-categorical}
Given a category \(\mathcal{C}\), let \(\mathbf{SOb}_{f}\) be the preorder whose elements are subcategories and  \(ι \leq κ\) if and only if there exists a functor \(λ{:} \mathcal{D} → \mathcal{E}\) and a monic natural transformation \(Φ{:} ι \leq λ⨾ κ\):

\begin{figure}[H]
    \centering
    \begin{tikzcd}[row sep=small]
      \mathcal{E} \ar[dr, "κ", hook] & \\
      & \mathcal{C} \\
      \mathcal{D} \ar[uu, "λ"]\ar[ur, "ι"', ""{name=S}, hook] \arrow[Rightarrow, from=S, to=1-1, "Φ", shorten >=6pt]
    \end{tikzcd}
    % \caption{\label{fi♭_:12}}
  \end{figure}

  This is a less truncated structure than \cref{xmp:4}, but its smallness relations stay the same: \(\top\) is still \(\mathbf{id}_{\mathcal{C}}\), \(\bot = ∅\) and the join is given by the closure of the union of
  the subcategories under composition. The meet is interesting however: \(ι ∧ κ\) is the full subcategory on objects that are simultaneously subobjects of something in \(\mathcal{C}_{ι}\) and of something
  in \(\mathcal{C}_{κ}\) in some way. So the dense elements are different from those in \cref{xmp:4} : \(ι\) is dense if and only if it contains for each object \(X\) in \(\mathcal{C}\) some object \(A\) that
  is a subobject of \(X\). In particular, if \(\mathcal{C}\) contains an initial object whose universal morphisms are monic, then a subcategory is dense if and only if it contains that.

  Note that the
  dual statement is not true for the terminal object even if the morphisms into it are monic, which is also a much rarer assumption. This is because of the variance of the 2-cells in it
  (which would be necessary to define filters, for instance). \(\mathbf{SOb}_{f}\) will be central in \cref{sec:covar-repr-theor} but
  \cref{rmα_:9} would not work for either \(\mathbf{SOb}_{\mathopen{∣}f\mathclose{∣}}\) or the 2-dual of \(\mathbf{SOb}_{f}\). So here we can already begin to see that the entire structure
  we are describing is of one piece and it is all \emph{covariant} on the lowest level (the variance on higher levels will grow quite complex in \cref{sec:univ-char-mathbf}).
\end{Example} 

We also have some specializations of this preorder:

\begin{Example}
  \label{xmp:29}
  Let \(\mathbf{SOb}_{ff}(\mathcal{C})\) be the quotient of \(\mathbf{SOb}_{f}(\mathcal{C})\) by the equivalence relation \(ι \sim κ\) if and only if \(ι\) and \(κ\) have the same fully faithful image. Then for each \(ι\) it
  holds that \(ι\) is equivalent to the embedding \(\mathopen{∣}ι\mathclose{∣}\) of the discrete subcategory on the objects in the image of \(ι\), and for \(\mathopen{∣}ι\mathclose{∣}\) and \(\mathopen{∣}κ\mathclose{∣}\) it holds that \(\mathopen{∣}ι\mathclose{∣} \leq \mathopen{∣}κ\mathclose{∣}\) if and only if there
  exists a function \(\mathopen{∣}ι\mathclose{∣} \hookrightarrow \mathopen{∣}κ\mathclose{∣}\) and monomorphism \(I{:} \mathopen{∣}ι\mathclose{∣}(X) \hookrightarrow \mathopen{∣}κ\mathclose{∣}(X)\). So \(ι \leq κ\) in \(\mathbf{SOb}_{ff}(\mathcal{C})\) holds if each \(X\) in the image of \(ι\) is the subobject of some \(Y\) in the image of
  \(κ\). We can similarly define a quotient \(\mathbf{SOb}_{\mathopen{∣}ff\mathclose{∣}}(\mathcal{C})\) of \(\mathbf{SOb}_{\mathopen{∣}f\mathclose{∣}}(\mathcal{C})\), where \(ι \leq κ\) if and only if the objects of \(ι\) are contained in the objects of \(κ\) up to
  isomorphism.

  The joins and meets of \(\mathbf{SOb}_{\mathopen{∣}ff\mathclose{∣}}(\mathcal{C})\) can be identified with set-theoretic unions and intersections, so that both smallness and density are trivial, but
  \(\mathbf{SOb}_{ff}(\mathcal{C})\) has a more interesting structure. The joins \(ι ∨ κ\) can still be identified with set-theoretic unions of the object sets and thus the only subcategory that is
  small in \(\mathbf{SOb}_{ff}(\mathcal{C})\) is \(∅\). But the meet \(ι ∧ κ\) is again the full subcategory on objects that are simultaneously subobjects of something in \(\mathcal{C}_{ι}\) and of
  something in \(\mathcal{C}_{κ}\), which is not generally the set-theoretic intersection.
\end{Example}
\begin{Example}[Wide subcategories of a category]
  \label{xmp_:5}
\index{wide subcategory}
  Let \(\mathcal{C}\) be a category and write \(\mathbf{Wide}(\mathcal{C})\) for the preorder of wide subcategories of \(\mathcal{C}\) (i.e. subcategories having all objects of
  \(\mathcal{C}\)), ordered by inclusion of morphism sets. We can construct this category as the preorder of faithful isomorphism-reflecting factorizations of the functor
  \(\mathbf{Disc}(\mathcal{C}){:} \mathbf{Ob}(\mathcal{C}) \hookrightarrow \mathcal{C}\). The top element is \(\mathcal{C}\) itself; the bottom is the discrete wide subcategory
  \(\mathbf{Disc}(\mathcal{C})\) containing only identities. A wide subcategory \(ι\) is again small in \(\mathbf{Wide}(\mathcal{C})\) if and only if it does not contain an anchor. \(ι\) is dense in \(\mathbf{Wide}(\mathcal{C})\) if and only if

  \begin{enumerate}
  \item every non-endomorphism of \(\mathcal{C}\) belongs to \(ι\), and
  \item for each non-identity endomorphism \(e {:} X → X\) of \(\mathcal{C}\) there exists an integer \(n \ge 1\) such that \(e^n\) is in \(ι\) with \(e^{n}\) not an identity.
  \end{enumerate}

  To see this, note that taking the wide subcategory generated by a single non-endomorphism forces all non-endomorphisms into \(ι\); taking those generated by an endomorphism \(e\) forces some
  power of \(e\) into \(ι\). The converse is a straightforward verification using the closure properties of wide subcategories.

  So on the level of wide subcategories density starts behaving as for monoids. Note that this is similar to how the periodic table for \(k\)-tuply \(n\)-monoidal categories only works when
  we take them as pointed.\cite{baezLecturesNCategoriesCohomology2007}
\end{Example}

\subsubsection{Topos-Theoretic}
\label{sec:topos-theoretic}

An important but perhaps somewhat surprising class of examples of unreachability arises in topos theory:

\begin{Proposition}
  \label{■pn_:64}
\index{unreachable preorder!representable presheaf}
\index{representable presheaf!unreachability}
  The subobject preorder of every representable presheaf is unreachable.
\end{Proposition}
\begin{proof}
  \label{prf_:83}
  Let \(\mathcal{C}\) be a category, \(X\) an object of \(\mathcal{C}\) and \(F \hookrightarrow \widehat{X}\) a proper subfunctor of the functor \(\widehat{X}\) represented by \(X\). We have to show that the only \(G \hookrightarrow \widehat{X}\) such that
  \(F ∨ G = \widehat{X}\) is \(\widehat{X}\) itself. For this remember that joins of presheaves are given by objectwise set-theoretic unions \(F(Y) \cup G(Y)\). Now \(\mathbf{id}_{X}{:} X → X\) is
  the only element of \(\widehat{X}(X)\) whose presence in a sub-functor forces that sub-functor to be all of \(\widehat{X}\). Since \(F\) is proper, it does not contain
  \(\mathbf{id}_{X}\) at \(X\). But \(F(X) \cup G(X) = \widehat{X}(X)\), so \(\mathbf{id}_{X} \in G(X)\). By Yoneda closure, this forces \(G = \widehat{X}\).
\end{proof}

\begin{Remark}
  \label{rmα_:13}
  At least if \(\mathcal{C}\) is monic and \(X\) is Dedekind-finite, i.e. has no proper self-embeddings, we can construct a co-atom \(∂ X \hookrightarrow \widehat{X}\) as the presheaf with \((∂ X)(Y)\) given by the set of morphisms \(f{:} Y → X\) such that there exists a
  \(ι{:} Z \hookrightarrow X\) with \(Z \cancel{\simeq} X\) and a \(g{:} Y → Z\) with \(f = g⨾ ι\). So we can freely construct a boundary.
\end{Remark}
\begin{Remark}
  \label{rmα_:41}
  This behavior is in contrast to that of density on representable presheaves: in \cref{rmk:1} we find that for a thin subcanonical site with an initial object the Yoneda embedding is a density-equivalence. The
  behavior on general categories is currently unknown but the mere fact that the structure of the base category matters independently of the coverage is already different from the case of
  smallness, where the coverage completely mediates what is small.
\end{Remark}
\begin{Remark}
  \label{rmk:5}
  \cref{■pn_:64} recontextualizes \cref{obs_:1}: generally we would think that faces, being lower-dimensional, would be small. And they are — once we add enough unions to construct a
  boundary. It is the mark of the combinatorial that there are no subcanonical coverages on a combinatorial site \(⬢\) that would make some objects in the sheaf topos \(\mathbf{Sh}(⬢)\)
  small and others not: there is only the \emph{boundary coverage}, which identifies an object with its boundary — but this recreates the original site.
\end{Remark}

\subsubsection{Convex}
\label{sec:other-examples}

We also have two examples related to the work of Tobias Fritz. In \cite{PostulatesBarycentricCalculus}, Stone defines an abstract convex space, which Fritz takes up in \cite{fritzConvexSpacesDefinition2015}. One thing that can be gleaned from the definition
is the idea of a convex space as a kind of continuous category: it has generalized associativity and unitality laws. We take up that theme and make it fully explicit, but in a simplified
setting where we forget the combinatorics of convexity and just remember the generated geodesics, somewhat similar to taking the path groupoid of a topological space. Note that the
material in this section requires some setup, which also hints at a categorical theory of convexity. Consequently, we have to introduce some new definitions for our statements that
will only really be taken up in later work, and impatient readers might want to skip this section.

Central to our theory is the (well-known) notion of laxification and in particular of a lax functor:

\begin{Definition}[Lax 2-functor]
  \label{def:lax-2-functor}
\index{sub-2-category}
\index{lax 2-functor}
  Let \(\mathcal{C}\) and \(\mathcal{D}\) be 2-categories. A \emph{lax 2-functor} \(F {:} \mathcal{C} \to \mathcal{D}\) consists of the following data:
  \begin{enumerate}
    \item A function on objects, mapping \(X {:} \mathcal{C}\) to \(F(X) {:} \mathcal{D}\).
    \item for every pair of objects \(X, Y {:} \mathcal{C}\), a functor on hom-categories
      \[
        F_{X,Y} {:} \mathcal{C}(X, Y) \to \mathcal{D}(F(X), F(Y)),
      \]
      mapping a 1-cell \(f {:} X \to Y\) to \(F(f) {:} F(X) \to F(Y)\) and a 2-cell \(α {:} f \Rightarrow g\) to \(F(α) {:} F(f) \Rightarrow F(g)\).
    \item for every object \(X {:} \mathcal{C}\), a 2-cell (the \emph{lax unitor})
      \[
        ν_{X} {:} \mathbf{id}_{F(X)} \Rightarrow F(\mathbf{id}_X).
      \]
    \item for every pair of composable 1-cells \(f {:} X \to Y\) and \(g {:} Y \to Z\) in \(\mathcal{C}\), a 2-cell (the \emph{lax compositor})
      \[
        \mu_{f,g} {:} F(f) ⨾ F(g) \Rightarrow F(f ⨾ g).
      \]
  \end{enumerate}

  Subject to the same coherence conditions as pseudofunctors. An \emph{oplax 2-functor} is defined in the same way, but with dual unitors and compositors.
\end{Definition}

With this definition, we can laxify the notion of a subcategory itself:

\begin{Definition}
  \label{defn_:15}
  Given a 2-category \(\mathcal{B}\), a \emph{sub-2-category} \(ι{:} \mathcal{A} \hookrightarrow \mathcal{B}\) is a 2-functor which is injective and isomorphism-reflecting on objects, 1-morphisms and 2-morphisms. A \emph{lax
    sub-2-category} is a lax 2-functor with the same properties and dually for oplaxity.
\end{Definition}
\begin{Remark}
  \label{rmk:lax-vs-pseudo}
  A pseudofunctor is precisely a lax 2-functor for which the lax unitors \(ν_X\) and lax compositors \(\mu_{f,g}\) are invertible 2-cells (i.e. equivalences in the
  hom-categories). If these 2-cells are further constrained to be identity 2-cells, \(F\) reduces to a strict 2-functor. Laxifying a pseudofunctor thus amounts to replacing
  the requirement that composition and identities are preserved up to coherent isomorphism with the weaker requirement that they are preserved only up to coherent morphisms
  pointing in a specific direction. This process is known as \emph{laxification}. Now, consider this little-appreciated fact: laxification is evil.\footnote{To the knowledge of the
    author, this was first noticed by Chris Schommer-Pries \cite{schommer-priesLaxFunctorsEquivalence2009}.}

  Let \(I\) be the 2-category with two elements \(X, Y\), morphisms generated by \(f{:} X → Y, g{:} Y → X\) and 2-cells generated by \(\mathbf{id}_{X} < f⨾ g, \mathbf{id}_{Y} < g⨾ f\).
  Let \(I_{0} = \{X, Y\}\) be the set of objects of \(I\), and let \(\mathbf{CDisc}(I_{0})\) be the codiscrete category on \(X, Y\), consisting of two morphisms \([X, Y]{:} X → Y, [Y, X]{:} Y → X\) that compose to the identities. Then there
  exists an oplax functor \(F{:} \mathbf{CDisc}(I_{0}) → I{:} [X, Y] ↦ f, [Y, X] ↦ g, \mathbf{id}_{X} ↦ \mathbf{id}_{X}, \mathbf{id}_{Y} ↦ \mathbf{id}_{Y}\) that is not equivalent to
  any constant functor, even though \(\mathbf{CDisc}(I_{0})\) is equivalent to the point.

  What should we make of that?

  Should we use lax \emph{anafunctors} instead of lax functors? Or does laxification detect structure that is invisible to weakness? The following exposition suggests the latter.
  It also has to be noted that it is known that Lawvere had a negative opinion of \((∞, 1)\)-category theory. The exact reason why is less well-known but can to some extent be
  reconstructed: in \cite{lawvereWilliamLawvereAre}, Lawvere explicitly \emph{denies} that homotopy types are the same as infinitesimal skeleta because to him, they are an instance
  of an identity of opposites. Understanding this \emph{only} to an identity would seem like an undialectical collapse from this perspective. And we do have to ask ourselves: if,
  say, the codiscrete categories on different sets are just the same, \emph{then why are they different}? Because they do have different object cardinalities and, more saliently,
  different lax functors out of them.
\end{Remark}
\begin{Example}
  \label{xmp:12}
\index{geodesic!in metric space}
\index{metric space!as 2-category}
  Given a Lawvere metric space \((X, d)\) with its induced topology generated by \(U_{x, r} := \{y ∈ X | d(x, y) < r\}\) we can define a 2-categorical structure on \(X\) whose

  \begin{itemize}
  \item objects are the elements of \(X\),

  \item morphisms are the contracting functions \(f{:} [0, l] → X\) with \(l_{f} := l ∈ ℝ_{\geq 0}\) where \(ℝ_{\geq 0}\) has its usual metric and \([0, l]\) inherits it,

  \item \(f \leq g\) if \(l_{f} \leq l_{g}\),

  \item \(f⨾ g\) is the morphism \([0, l_{f} + l_{g}] → X\) given by \(f\) on \([0, l_{f}]\) and \(g\) on \([l_{f}, l_{f} + l_{g}]\),

  \item \(\mathbf{id}_{x} : [0, 0] → X\) is given by \(x{:} * \hookrightarrow X\) (since \([0, 0]\) is canonically equivalent to the terminal set),

  \item associativity and identity laws follow from pointwise evaluation.
  \end{itemize}

  This is kind of like the path groupoid of a topological space except that we don't assume any identifications whatsoever  and associativity holds without restraint because we
  don't normalize path lengths at all (i.e. not every path is a morphism from \([0, 1]\)).  Note
  also that no object of \(X\) has a non-contractible automorphism 2-group, so each natural isomorphism between paths is equivalent to an identity.

  Let us now consider the example of \(ℝ^{n}\) and the embedding of the codiscrete groupoid

  \[γ{:} \mathbf{CDisc}(ℝ^{n}) \hookrightarrow ℝ^{n}{:} x ↦ x, (f{:} x → y) ↦ [x, y]\]

 which maps the (unique) path from \(x\) to \(y\) to the (affine) line from \(x\) to \(y\). Then \(γ\) maps \(\mathbf{id}_{X}\) to \([x, x]\), which is the identity in \(ℝ^{n}\), and maps
  \(x \xrightarrow{f} y \xrightarrow{g} z\) to \([x, z]\), which is minimal in \(X(x, z)\), and is thus an oplax embedding of \(\mathbf{CDisc}(ℝ^{n})\) into \((ℝ^{n})^{[0, 1]}\). In
  other words, \(γ\) is the embedding of the geodesics between points in \(ℝ^{n}\).
\end{Example}

This is the example we want to generalize with the next definition:

\begin{Definition}
  \label{dfn:19}
\index{convex sub-2-category}
\index{geodesic 2-category}
  An oplax sub-2-category \(ι{:} \mathcal{A} \hookrightarrow \mathcal{B}\) of a 2-category \(\mathcal{B}\) is \emph{convex} if for each 2-isomorphism
  \(σ{:} ι(f) \overset{\sim}{\Rightarrow} g⨾ h\) it holds that \(σ, g\) and \(h\) are in \(ι\).

  \(ι\) is \emph{convexly full} if for each \(f\) in \(\mathcal{B}\) such that  \(\mathbf{s}(f)\) and \(\mathbf{t}(f)\) are in (the image of) \(ι\) it holds that for each 2-isomorphism
  \(σ{:} f \simeq g⨾ h\) with \(\mathbf{t}(g) = \mathbf{s}(h)\) the cells \(σ\), \(g\) and \(h\) are in \(ι\).

  An object \(∅\) \emph{is sup-initial} if it has at most one morphism to every object (so if it is the dual of a subterminal object). \(∅\) is \emph{strict sup-initial} if it is
  sup-initial and each morphism with codomain \(∅\) is uniquely 2-isomorphic to the identity.

  A \emph{2-category with paths} is a pair of a 2-category \(\mathcal{B}\) and a wide convex oplax subcategory \(ι{:} \mathcal{A} \hookrightarrow \mathcal{B}\).

  A \emph{2-category with unique paths} is a 2-category with paths
  such that \(\mathcal{A}\) is the codiscrete groupoid \(\mathbf{CDisc}(\mathcal{B})\) on the objects of \(\mathcal{B}\).

  A \emph{2-category with geodesics} is a 2-category with paths
  \(ι\) such that each morphism \(p{:} X → Y\) in the image of \(ι\) is strict sup-initial in \(\mathcal{B}(X, Y)\).

  A \emph{geodesic 2-category} is a 2-category such that there
  exists a convex strict initial object \(γ{:} \mathcal{G} \hookrightarrow \mathcal{B}\) in the 2-category of wide oplax sub-2-categories of \(\mathcal{B}\) on
  \(\mathbf{CDisc}(\mathcal{B})\). Equivalently, \(\mathcal{B}\) is geodesic if each hom-category \(\mathcal{B}(X, Y)\) contains a strict initial object \(γ_{X, Y}\) such that each 2-isomorphism
  \(σ{:} γ_{X, Y} \simeq f⨾ g\) is an identity (since it is in the image of \(\mathbf{CDisc}(\mathcal{B})\)) and \(f = γ_{X, \mathbf{t}(f)}, g = γ_{\mathbf{s}(g), Y}\).
\end{Definition}
\begin{Remark}
  \label{rmk:16}
  So convex fullness is like fullness but only applies to morphisms that don't just go between objects in \(ι\) but also never factor through objects outside \(ι\) — no detours allowed!
\end{Remark}

And this we can apply to Stone's convex spaces:

\begin{Example}
  \label{xmp:19}
  \index{convex space!2-categorical structure} In \cite{fritzConvexSpacesDefinition2015}, Fritz describes a convex space as a set \(X\) with a family of operations
  \(cc_{λ}{:} X × X → X\) for \(λ ∈ [0, 1]\) such that \(cc_{0}(x, y) = x\) and \(cc_{1}(x, y) = y\) (by commutativity). This is equivalent to a family of paths
  \(γ_{x, y}∈ X^{[0, 1]}{:} γ_{x, y}(λ) := cc_{λ}(x, y)\) where we don't assume any structure on \(X\) or \([0, 1]\). Fritz moreover introduces a commutativity axiom which ensures that
  the inverse of some path is literally the same path backwards, and an associativity axiom which regulates the continuous structure through combinatorial approximation. We reduce
  this to a discrete, i.e. categorical structure. This loses some geometry and there should be a continuous version of category theory that builds on that structure, but starting
  from a more basic perspective allows us to seamlessly integrate our theory with standard category theory before introducing more genuinely continuous considerations.

  Each convex space in the sense of \cite{fritzConvexSpacesDefinition2015} has an underlying convex 2-category \(\mathopen{∣}∏\mathclose{∣}(X)\) whose

  \begin{itemize}
  \item objects are the elements of \(X\),

  \item\label{item:38} morphisms are generated by \(γ_{x, y}\) (freely except for the 2-morphisms below),

  \item\label{item:59} 2-morphisms are the horizontal and vertical composites of the following two kinds:

    \begin{itemize}
    \item\label{item:62} 2-cells \(γ_{x, y}⨾ γ_{y, z} \geq γ_{x, z}\) for all \(x, y, z\),

    \item inverses \(γ_{x, z} \sim γ_{x, y}⨾ γ_{y, z}\) if \(y\) is in the image of \(γ_{x, z}\).
    \end{itemize}
  \end{itemize}

  This is a 2-category whose identities are given by constant paths, which is both more rough and finer than the standard fundamental groupoid: it does not consider any
  higher-categorical structure (and there isn't really any unless it is induced by paths), but it also does not identify the composite \(γ⨾ \overline{γ}\) of a path
  \(γ{:} [0, 1] → X\) and its reflection \(\overline{γ}{:} [0, 1] \xhookrightarrow{\overline{x} := 1 - x} [0, 1] \xhookrightarrow{γ} X\) with \(\mathbf{id}(\mathbf{s}(x))\). It
  follows directly from the definition that each \(γ_{x, y}\) is the minimal path from \(x\) to \(y\), and that \(γ\) is convex (up to contractible homotopy).
  Thus, \(γ{:} \mathbf{CDisc}(X) → X\) is the \emph{underlying geodesic 2-category} of \(X\).
\end{Example}

These examples might seem weird from the point of view of homotopy theory but they work as convexity should, where it is not that structures unfold along higher cells but that
higher cells are added in by recursively connecting points. And a pleasant feature of geodesic 2-categories is that they allow for a notion of convex closure:

\begin{Definition}
  \label{dfn:25}
\index{convex closure}
  Given a convex wide oplax sub-2-category \(λ{:} \mathcal{A} \hookrightarrow \mathcal{B}\) and a family of objects \(\mathcal{F}\) in \(\mathcal{B}\), the \emph{\(λ\)-closure} of
  \(\mathcal{F}\) in \(\mathcal{B}\) is the smallest convex and convexly full (weak) sub-2-category of \(\mathcal{B}\) containing the restriction of \(λ\) to the objects in \(\mathcal{F}\). A sub-2-category
  \(κ\) of \(\mathcal{B}\) is \emph{\(λ\)-convex} if it is equivalent to the \(λ\)-closure of a family of objects in \(\mathcal{B}\). If \(\mathcal{B}\) is geodesic, then \(κ\) is \emph{geodesically convex} if it is
  \(γ\)-convex (where \(γ\) is the oplax sub-2-category of geodesics of \(\mathcal{B}\)).
\end{Definition}
\begin{Remark}
  \label{rmk:19}
  Note in particular that we can identify each \(ι\)-convex sub-2-category of \(\mathcal{B}\) with its subset of objects. In particular, for a family
  \(s{:} S \hookrightarrow \mathcal{B}\) the geodesic closure \(\hat{s}\) has to be the smallest convex and convexly full weak subcategory of \(\mathcal{B}\) containing \(γ|_{s}{:} \mathbf{CDisc}(S) \hookrightarrow \mathcal{B}\).
\end{Remark}
\begin{Example}
  \label{xmp:23}
  In a real vector space \(ℝ^{n}\) with its usual metric, let \(\mathcal{F}\) be a family of points. Applying \cref{xmp:12} we get that \(γ_{x, y}\) is the shortest path \([0, 1] → ℝ^{n}\) between
  \(x\) and \(y\), up to the points it contains. So the geodesic closure of \(\mathcal{F}\) contains \(\mathcal{F}\), \(γ_{x, y}\) for each \(x, y\) in \(\mathcal{F}\) and the shortest path between each pair of points in
  \(\mathcal{F}\), so concretely can be obtained from \(\mathcal{F}\) by adding each \(γ_{x, y}\), then adding all points on each obtained line, then adding the geodesics between each pair of new points and so
  on, then adding all paths that only cross points in that closure. Thus, \(\overline{\mathcal{F}}\) is indeed the convex closure of \(\mathcal{F}\) as traditionally understood.

  For instance if
  \(\mathcal{F}\) is given by function \(s{:} S \hookrightarrow ℝ^{n}\), with \(S\) a finite set, then \(\overline{s}{:} \mathbf{CDisc}(S) → ℝ^{n}\) is a (possibly degenerated) convex simplex in \(ℝ^{n}\) with vertices
  given by the images of the elements of \(S\). Note in particular that \(\mathbf{CDisc}([n])\) with \([n]\) a finite cardinal is the localization of the total order on
  the same elements, which is also known as the abstract \((n - 1)\)-simplex.

  Note also that a geodesic weak natural transformation \(η: ι_{0} \Rightarrow ι_{1}: I \hookrightrightarrows ℝ^{n}\) between geodesic
  embeddings of the directed interval category, i.e. arrow category \(0 \hookrightarrow 1\), is a square with edges \(ι_{0}, ι_{1}, η_{0}, η_{1}\) such that

  \[\mathopen{∣}η_{0}⨾ ι_{1}\mathclose{∣} = \mathopen{∣}η_{0}\mathclose{∣} + \mathopen{∣}ι_{1}\mathclose{∣} = \mathopen{∣}ι_{0}\mathclose{∣} + \mathopen{∣}η_{1}\mathclose{∣} = \mathopen{∣}ι_{0} + η_{1}\mathclose{∣},\]

  i.e. an ex-tangential quadrilateral.
\end{Example}

The point is now that we can develop a very general theory of smallness for convexity:

\begin{Example}
  \label{xmp:convex-small}
\index{convex cone!smallness}
\index{tablets!for convexity}
\index{small subobject!convex}
  Let \(\mathcal{B}\) be a geodesic \(2\)-category with distinguished geodesic morphisms
  \(γ_{x,y}{:} x → y\) for all objects \(x,y\). The geodesically convex sub-2-categories of \(\mathcal{B}\) form an upward-bounded preorder \(\mathbf{SOb}_{gcc}(\mathcal{B})\)
  under inclusion, with top element \(\mathcal{B}\) and bottom element the empty subcategory.
  Meets are given by intersections if they exist (which they often do not), and joins are
  geodesic closures of unions: \(ι ∨ κ \;=\; \overline{\mathbf{Ob}(ι) ∪ \mathbf{Ob}(κ)}\).

  Thus, smallness is not trivial in \(\mathbf{SOb}_{gcc}(\mathcal{B})\).
  A geodesically convex subcategory \(ι\) is not small if and only if there exists a proper geodesically convex \(κ\) such that \(\overline{ι ∪ κ} = \mathcal{B}\). Since both \(ι\) and \(κ\) are geodesic
  closures of objects \(ι_{0}, κ_{0}\), this closure is equivalent to \(\overline{ι_{0} \cup κ_{0}}\). Thus, \(ι\) is not small if there exists a \(κ_{0}\) such that each object of \(\mathcal{B}\)
  can be obtained by repeatedly taking geodesics of points on geodesics... between \(ι_{0}\) and \(κ_{0}\), but there is a point in \(ι_{0}\) that cannot be obtained from repeated geodesics between
  points in \(κ_{0}\).

  So for instance in a non-trivial convex cone in \(ℝ^{n}\) with the convex 2-categorical structure from \cref{xmp:12}, the one-point subset on the tip is not small while all other points are small.
\end{Example}
\begin{Remark}
  \label{rmα_:1}
  We might expect smallness in \cref{xmp:convex-small} to be given by ``A is contained in the closure of its set-theoretic complement'', but the notion we derived is generally
  stronger than that. This is roughly because smallness is derived using both the fewer subobjects that are convex subspaces and their more expansive join. In
  \cref{xmp_:15} we will learn how to use only one of these at a time, recovering a weaker notion of smallness.
\end{Remark}

\subsubsection{Categorical Probability Theoretic}
\label{sec:synth-prob-theory}

The final example is also one from Fritz but in synthetic probability theory:

\begin{Example}
  \label{xmp_:9}
\index{stochastic matrix!simplicial structure}
  \index{stochastic matrix!smallness}
  Let us look at the category \(\mathbf{FStoch}\) whose objects are finite sets and whose morphisms are stochastic matrices. Recall that
  stochastic matrix \(M{:} X → Y\) is a function \(f(- \mid -){:} X × Y → [0, 1]\) such that for each \(x ∈ X\) it holds that

  \[∑_{y ∈ Y} f(y \mid x) = 1.\]

  \(f\) here is a synthetic descriptor of conditional probability, where the defining equation says that each conditional probability is a probability function. However, \(M\) is
  also the descriptor of an affine transformation that maps a unit simplex \(▲^{X}_{1} \subseteq ℝ^{X}\), which is the convex hull of the unit vectors \(δ_{i}\) in \(ℝ^{X}\), into
  \(▲^{Y}_{1} \subseteq ℝ^{Y}\), given by multiplying \(δ_{i}\) with \(M\). Thus, we can identify \(X{:} \mathbf{FStoch}\) with the \((\mathopen{∣}X\mathclose{∣} - 1)\)-dimensional unit simplex \(▲^{X}_{1}\) and the
  subobjects \(ι{:} A \hookrightarrow X\) with the affine sub-simplices of \(▲^{X}_{1}\).\footnote{This geometric interpretation was pointed out to me by Fritz.} Since each point is an affine
  sub-simplex, \(\mathbf{SOb}(▲^{X}_{1})\) is atomistic by \cref{coR_:3}, so density is trivial; the question is again smallness. For joins, \(σ ∨ τ = ▲^{X}_{1}\) holds if and only if
  \(▲^{X}_{1}\) is the only subsimplex of \(▲^{X}_{1}\) containing the union \(σ \cup τ\).

  Dimension 0 is clear. In dimension 1 the only sub-simplex of \(▲^{1} = [0, 1]\) that contains both \(0\) and \(1\) is \(\top = \mathbf{id}_{[0, 1]}\) itself, so the join \(ι ∨ ι'\) of
  two affine sub-simplices \(ι, ι'\), i.e. sub-lines, is \(\top\) if and only if both \(0\) and \(1\) are in the union \(ι \cup ι'\). Thus, \(ι\) is small if and only if it does not
  contain either \(0\) or \(1\).

  For dimension \(> 1\)  we claim that the only small subobject in \(\mathbf{SOb}(X)\) is \(\bot\). To prove this, we treat \(▲^{X}_{1}\) as embedded into \(ℝ^{X}\) and construct for each
  \(p ∈ ▲^{X}_{1}\) a full-dimensional, maximal
  proper subsimplex that does not contain \(p\). For this, take some maximal face \(τ\) such that \(p\) is not in \(τ\). Since each simplex is \(T_{4}\) and both \(p\) and \(τ\)
  are closed and disjoint in \(▲^{X}_{1}\) there exists some open neighborhood \(U_{τ}\) of \(τ\) that does not contain \(p\) and thus, by the definition of the restricted Euclidean
  topology, an \(ε > 0\) such that the \(ε_{p}\)-thickening \(ε_{p}τ \cap ▲^{X}_{1}\) of \(τ\) does not contain \(p\). Let now \(τ_{p}\) be the sub-simplex of \(▲^{X}_{1}\) parallel to \(τ\)
  in \(▲^{X}_{1}\) whose orthogonally measured distance from \(τ\) is \(ε_{p}\). This is a shrunken but otherwise undeformed copy of \(τ\) and thus each of its vertices is linearly
  independent and independent of each vertex in the interior of \(τ\). Thus, the simplex \(τ_{p}'\) spanned by \(τ_{p}\) and the barycenter of \(τ\) is maximal (since it is
  equidimensional and each of its vertices in on the boundary) and proper in \(▲^{X}_{1}\) and does not contain \(p\), so that \(\{p\} ∨ τ'_{p} = \top\).
\end{Example}
\begin{Remark}
  \label{rmα_:32}
  \cref{xmp_:9} would seem like a trivialization of our theory in a part of synthetic probability theory; and it is — on its own. But it is re-contextualized by \cref{■pn_:64}, which shows
  that smallness in the sheaf topos of a site \(\mathcal{S}\) only depends on the coverage on \(\mathcal{S}\). So for a site, having rough smallness relations is in a sense welcome as long as the coverages on
  it are interesting. So the key point of \cref{xmp_:9} is not that smallness in \(\mathbf{FStoch}\) is trivial but that its structure is very similar to that of \(\mathbf{TDisk}\) from
  \cref{xmp:63}, which is quite important, as mentioned in \cref{rmα_:33}. We will show in later work that this similarity of site structures descends to one of topoi, which opens up a
  connection of synthetic probability theory to cohesion and closed geometric logic that could be of interest.
\end{Remark}

\subsection{Reducedness \& Satedness}
\label{sec:reducedness}

We can already describe what it means for an element to not contain any small elements, except within non-small ones.

\begin{Definition}
  \label{defn_:7}
  \index{reduced element}\index{sated element}
  An element \(a\) of an upward-bounded preorder \(P\) is \emph{reduced} if for any small \(b\)  and arbitrary \(c\) such that \(a \leq b ∨ c\) it holds that \(a \leq c\).

  Dually, \(a\) is \emph{sated} if for every dense \(b\) and \(c\) such that \(a \geq b ∧ c\) it holds that \(a \geq c\).
\end{Definition}

So \(a\) is small if the only \(b\) such that \(a ∨ b = \top\) is \(\top\) and \(a\) is reduced if whenever \(a \leq b ∨ c\) with \(b\) small, then \(a \leq c\). Since \(\top\) is at the top, \(a ∨ b \geq \top\)
implies \(a ∨ b = \top\), so for \(\top\) this inequality simplifies to the defining equality for smallness. So an element is reduced if it is in the same relation to small elements that
\(\top\) has. Dually, an element is sated if it contains any subobject that it contains a dense subobject of.

\begin{Example}
  \label{xmp_:27}
\index{reduced element!topological}
A paradigmatic example are the closed subsets \(\mathbf{C}(X)\) of a topological space: each closed \(C\) is equal to the union \(\mathopen{∣}C\mathclose{∣} \cup \mathbf{s}(C)\) where \(\mathopen{∣}C\mathclose{∣}\) is the
reduction of \(C\), i.e. the closure of its interior, and \(\mathbf{s}(C)\) is the closure in \(C\) of the complement \(C ∖ \mathopen{∣}C\mathclose{∣}\). Thus, if \(C\) is reduced, then
\(C \leq \mathopen{∣}C\mathclose{∣}\) and hence \(C = \mathopen{∣}C\mathclose{∣}\). Conversely, if \(C \leq D \cup E\) with \(D\) small, then the interior \(C°\) is contained in the interior \((D \cup E)°\), which is equal to \(E\) by
smallness. Thus, in the lattice of closeds of a topological space an element is reduced if and only if it is the closure of an open subset, so if it is regular. Dually, an open subset is sated if and only if it
  is the interior of some closed subset, so also if it is regular.
\end{Example}

\begin{Example}
  \label{xmp:54}
\index{radical ideal!satedness}
  In \cref{xmp:49} we saw that a submonoid \(C\) is dense if and only if it contains a power of each element in \(M\). So \(A\) is sated if and only if whenever \(A \geq B \cap C\) for some dense
  \(C\), then \(A \geq B\), or equivalently if and only if it contains each root of its elements. In particular each (ring-theoretic) ideal in a ring \(R\) is sated as a multiplicative
  submonoid of \(R\) if and only if it is a radical
  ideal.\footnote{Regrettably for our nomenclature this is also called a ``reduced ideal'', though our terminology is arguably more accurate since a radical ideal is ``complete under radicals'' and
    thus ``large enough to be sated over
    everything it contains'', not ``reduced'', i.e. obtained by shaving away all small parts.}
\end{Example}

Let us note that at least in bounded preorders smallness and reducedness are exclusive in non-trivial cases:

\begin{Proposition}
  \label{■pn_:5}
\index{small element!exclusivity with reducedness}
  Let \(P\) be a bounded preorder. Then \(a\) is both small and reduced if and only if \(a = \bot\).
\end{Proposition}
\begin{proof}
  Because \(P\) is bounded, \(\bot\) exists.
  \(\bot\) is small: \(\bot ∨ b = b\), so \(\bot ∨ b = \top \Rightarrow b = \top\).
  \(\bot\) is reduced: if \(b ∨ c \geq \bot\) with \(b\) small, then also \(c \geq \bot\) by the minimality of \(\bot\).

  For the converse, assume \(a\) is both small and reduced.  Choose \(b = a\) (which is small) and \(c = \bot\).  Then \(a \leq a ∨ \bot\), so by reducedness \(a \leq \bot\), so \(a = \bot\).
\end{proof}

\subsection{The Smallness Stratification}
\label{sec:smalln-strat}

We can use smallness to stratify every upward-bounded preorder:

\begin{Definition}
  \label{dfn:7}
  \index{smallness fibration} \index{n-small element}
  Given an upward-bounded preorder \(P\), every element \(a ∈ P\) is \emph{\(0\)-small}. \(a\) is \emph{\(n\)-small} for \(n\) a successor ordinal if there is an \((n - 1)\)-small element \(b \geq a\) such that \(a\) is small in the (restriction of
  \(P\) to the) principal ideal \((b)\) of \(b\). \(a\) is \emph{purely \(n\)-small} if \(a\) is \(n\)-small but not \((n + 1)\)-small. \(a\) is \emph{\(ω\)-small} for \(ω\) a limit ordinal if it is \(n\)-small
  for all \(n < ω\).
\end{Definition}

It follows immediately from the definition of \(n\)-smallness and the monotony of smallness that this assignment forms a stratification:

\begin{Definition}
  \label{dfn:8}
  \index{smallness stratification}
  Let \(\mathbf{Ord}\) be the large preorder of ordinals. Given an upward-bounded preorder \(P\) the \emph{smallness stratification} \(\mathbf{small}{:} P → \mathbf{Ord}\) is the order morphism given by mapping \(a ∈ P\) to
  \(n\) if and only if \(a\) is purely \(n\)-small, and to \(\omega\) if \(a\) is \(\omega\)-small. The \emph{smallness fibration} is the morphism \(∐_{n ∈ \mathbf{Ord}} \mathbf{small}^{-1}(n) → P\) classified by
  \(\mathbf{small}\). We define the \emph{density stratification} and \emph{density fibration} dually.
\end{Definition}

\begin{Proposition}[Well-definedness of the smallness stratification]
  \label{■pn_:strat-welldef}
  \index{smallness stratification!well-definedness}
  Let \(P\) be a (small) upward-bounded preorder. For every \(a ∈ P\), there exists an ordinal \(n\) such that \(a\) is purely \(n\)-small.
\end{Proposition}
\begin{proof}
  \label{prf:1}
  Assume that \(a\) is not purely \(n\)-small for any \(n\). By definition \(a\) is 0-small. To not be purely 0-small, \(a\) must be 1-small, meaning \(a\) is small in some 0-small principal ideal. Proceeding by induction, \(a\) is \((n+1)\)-small for each \(n\). For each limit ordinal \(ω\), if \(a\) is \(n\)-small for all \(n < ω\), then \(a\) is \(ω\)-small by definition. Thus, by transfinite induction, \(a\) is \(α\)-small for every ordinal \(α\). But \(α\)-smallness requires \(a\) to be small within some \((α-1)\)-small principal ideal, which itself must be small within some \((α-2)\)-small principal ideal, and so on. This creates a strictly ascending chain of principal ideals of length \(α\) for each ordinal \(α\), so there exists an injection \(\mathbf{Ord} \hookrightarrow P\). Consequently, \(P\) cannot be small.
\end{proof}

So smallness stratifies each upward-bounded preorder into sections of the same size. This is very similar to how dimension stratifies various subobject lattices, and indeed these two stratifications
often overlap:

\begin{Example}
  \label{xmp:107}
  In \cref{xmp:60} we saw that a (topological) sub-manifold with boundary \(ι{:} A \hookrightarrow X\) is small in a manifold with boundary \(X\) if and only if
  \(m := \mathbf{dim}(A) < n := \mathbf{dim}(X)\). If \(A\) is non-empty then there exists \(p ∈ A\), and thus a local isomorphism \(σ_{p}\) from some open \(U_{p}\) in \(X\) to either \(ℝ^{n}\) or
  \(ℝ^{n}_{\geq 0}\), which without loss of generality we can choose so that \(ι|_{U_{p}}⨾ σ\) is isomorphic to \(ℝ^{m}\). In particular we can for each \(m < n' < n\) find an
  \(n'\)-dimensional disk or half-disk \(D_{p}\) around \(ι|_{U}⨾ σ(p)\). Let us take the union \(κ := \bigcup_{p ∈ A} D_{p}\) in \(X\). Then the \(σ_{p}|_{D_{p}}\) provide an atlas for
  \(κ\), so that \(κ\) is an \(m\)-dimensional submanifold with boundary of \(X\). Thus, the smallness fibration in \(\mathbf{SOb}_{∂ \mathbf{Mfld}}(X)\) agrees with the dimensional fibration.
\end{Example}

However, note that for instance in \cref{xmp:63} the smallness stratification does not align with the dimensional one at all.

\subsection{Monoidal Smallness}
\label{sec:monoidal-smallness}

Building on \cref{rmα_:20} we can extend \cref{defn_:58} in another way, namely by replacing joins and meets by some other monoidal product. We describe these notions here,
although we will only develop them further in future texts.

\begin{Definition}
  \label{defn_:24}
\index{monoidal smallness}
An element \(a\) of an ordered monoid \(M\) with a product \(+\) and a top element \(\top\) is \emph{\(+\)-small} if and only if for each \(b ∈ M\) such that \(a + b = \top\)
  it holds that \(b = \top\). Dually, an element \(a\) of an ordered monoid \(M\) with a contravariantly monotone monoidal product \(\otimes\) and a bottom element \(\bot\) is \emph{\(\otimes\)-dense} if and
  only if for each \(b ∈ M\) such that \(a \otimes b = \bot\) it holds that \(b = \bot\).
\end{Definition}

One of these notions is infinitely more useful than the other. Can you guess which one?

Correct, it is \(+\)-smallness. This is because the most important such product, addition, is strictly increasing on positive quantities \(a + b \geq a, a + b \geq b\). The reason
for this is the deep connection between joins and addition, which allows us to use \cref{defn_:24} for distinguishing finitude and infinity. This connection will only be sketched here, a deeper treatment will have to wait for the introduction of tols in
subsequent work. But the essential point is in how the joins of subsets of natural numbers interact with addition of their cardinalities.

Consider the preorder \(\mathbf{P}(ℕ)\) of subsets
of \(ℕ\): by \cref{xmp:9} no non-trivial element of \(\mathbf{P}(ℕ)\) is small or dense. On the other hand, in \(ℕ_{∞}\) only \(∞\) is not small by \cref{xmp:75}, so if we project
each finite subset of \(ℕ\) to its cardinality and each infinite subset to \(∞\), then the smallness structure of \(\mathbf{P}(ℕ)\) is destroyed. This is because this projection concentrates
\(\mathbf{P}(ℕ)\) into a line \(ℕ_{∞}\).

In particular, if \(ι \leq_{ℕ_{∞}} κ\) in the induced order \(\leq_{ℕ_{∞}}\) on \(\mathbf{P}(ℕ)\) given by \(ι \leq_{ℕ_{∞}} κ := \mathopen{∣}ι\mathclose{∣} \leq \mathopen{∣}κ\mathclose{∣}\), then there is
some \(ι'\) of the same cardinality that is a subset of \(κ\). So the cardinality \(\mathopen{∣}-\mathclose{∣}\) reflects \(\mathopen{∣}ι\mathclose{∣} ∨ \mathopen{∣}κ\mathclose{∣}\) if \(ι \leq κ\) but otherwise it essentially brings \(ι\) in a position where
its overlap with \(κ\) is the largest and \emph{then} takes the join, which is the least expansive way to take a join in \(\mathbf{P}(ℕ)\). But what if we take the join in \emph{most} expansive way before
projecting to \(ℕ_{∞}\), i.e. make \(ι\) and \(κ\) disjoint before taking the join? The natural numbers are actually the smallest set for which we can coherently do this, since they have an
isomorphism to their sum:

\[γ{:} ℕ \overset{\sim}{→} ℕ + ℕ = ℕ_{1} + ℕ_{2}{:} n ↦
\begin{cases}
  (n/2)_{1} \text{ if } n\ mod\ 2 = 0\\
  ((n{-}1)/2)_{2}\text{ else.}
\end{cases}\]
So if we have two subsets \(ι, κ \subseteq ℕ\), we can apply \(γ^{-1}\) to \(ι + κ \subseteq ℕ + ℕ\), which gives us two subsets \(ι' := γ^{-1}(ι + ∅)\) and \(κ' := γ^{-1}(∅ + κ)\) such that \(ι' \cap κ' = ∅\) and
thus \(ι' \cup κ'\) is maximally large. So we can define \(ι \oplus κ := γ^{-1}(ι + κ) = ι' \cup κ'\), which provides a semicocartesian laxly unital monoidal operation on \(\mathbf{P}(ℕ)\) (although we postpone
proofs of semicocartesianism, lax unitality and weak associativity to future work, where we provide them in a more general setting) such that \(\mathopen{∣}ι \oplus κ\mathclose{∣} = \mathopen{∣}ι\mathclose{∣} + \mathopen{∣}κ\mathclose{∣}\). So addition is the disjoint
union of subsets in the most literal sense.

Now, for \(n ∈ ℕ_{∞}\) consider the sub-preorder \([0, n]\) of \(ℕ_{∞}\) with the monoidal product \(+_{n}\) given by \((a + b) ∧ n\) (note that this is still associative). Since \([0, n]\)
is a total order, the only non-small element of \([0, n]\) is \(n\), but when is
\(a ∈ [0, n]\) \(+_{n}\)-small? It is easy to see that if \(n\) is finite, then \(a\) is \(+_{n}\)-small if and only if \(a = 0\) but if \(n = ∞\) then \(a\) is \(+_{n}\)-small if
and only if \(a < ∞\).

So the difference
in size between finitude and infinity only becomes significant if we assume the capability of duplication, whereupon every finite quantity becomes small compared to infinity, which cannot be
reached through a finite number of duplications. In particular, if we identify all infinite subsets of \(ℕ\) (e.g. use the \(ℕ_{∞}\)-valued scale from \cref{xmp_:10}), then \(ι\) is \(\oplus\)-small
if and only if it is finite.

\section{Tablets \& Scales}
\label{sec:tablets-spans-scales}

Let us put our new paradigm into praxis by providing the means for different subobject preorders to interact. For now, we do this abstractly on the level of preorders, in
\cref{sec:applications} we will lift our constructions to the level of categories.

\subsection{Tablets}
\label{sec:tablets}

We start by providing some tablets:

\begin{Definition}
  \label{defn_:14}\index{small under a tablet}\index{reduced under a tablet}\index{tablets}
  Let \(τ{:} T → P\) be a morphism of upward-bounded preorders and \(a\) an element of \(P\).
  We call \(τ\) the \emph{tablets} we use.

  \(a\) is \emph{small under \(τ\)} if for each join \(a ∨ τ(b) = \top_{P}\) it holds that \(b = \top_{T}\).

  \(a\) is \emph{reduced under \(τ\)} if for each element \(b ∈ P\) that is small under \(τ\) and each \(c ∈ T\) it holds that

  \begin{equation}
    \label{eq:2}
    \text{if } a \leq b ∨ τ(c) \text{ then } a \leq τ(c).
  \end{equation}
\end{Definition}
\begin{Remark}
  \label{rmα_:3}
  The use of inequalities for reducedness in \cref{defn_:14} allows us to bound the fuzziness between the images of elements of \(T\). Note that the formula \(a ∨ τ(b) = \top_{P}\) is already equivalent
  to \(a ∨ τ(b) \geq \top_{P}\) due to the definition of a top, so for smallness the question of whether to use an inequality doesn't arise.
\end{Remark}
\begin{Example}
  \label{xmp_:20}
  \cref{defn_:14} reduces to \cref{defn_:58} and \cref{defn_:7} when \(τ = \mathbf{id}_{P}\), i.e. each subobject of \(P\) is a tablet.
\end{Example}

We can understand these tablets as the forms we use to infer subobject-theoretic questions in our base structure when the collection of all its subobjects is too fine. This mirrors how it is done with topological spaces:

\begin{Example}
  \label{xmp_:26}
\index{Cauchy completeness!generalized}
\index{tablets!topological}
   Building on \cref{xmp_:27}, consider the inclusion \(ι_{\mathbf{C},X}{:}\mathbf{C}(X)\hookrightarrow\mathbf{P}(X)\) of the closed subsets of a topological space \(X\) into all subsets. A subset
  \(a \subseteq X\) is small under \(ι_{\mathbf{C},X}\) if and only if \(X\) itself is the only closed subset \(c\) such that \(a \cup c = X\), so if its closure is small, i.e. nowhere dense. A subset
  \(b \subseteq X\) is reduced under \(ι_{\mathbf{C},X}\) if and only if whenever \(b \leq a \cup c\) with \(a\) nowhere dense and \(c\) closed, then already \(b \leq c\), so if
  the closure \(\overline{b}\) is regular. So our tablets allow us to bound the sets in \(\mathbf{P}(X)\) by those in \(\mathbf{C}(X)\).

  Dually, if we take \(ι_{\mathbf{O}, X}{:} \mathbf{O}(X) \hookrightarrow \mathbf{P}(X)\) as the tablets, then a subset \(a\) of \(X\) is dense under \(ι_{\mathbf{O}, X}\) if and only if its meet with every open subset is non-empty,
  so if it is dense in the topological sense. \(b\) is sated if and only if whenever \(b\) is locally dense in that \(b \geq a \cap c\) with \(a\) open and \(c\) dense, then \(b\) just contains
  \(a\).

  In particular, let \(X\) be a metric space with its underlying topology and the tablets \(τ\) again given by the opens of \(X\). Then a subset \(ι{:} A \hookrightarrow X\) is \(τ\)-dense if and only if
  for each \(x ∈ X, r ∈ ℝ_{> 0}\) the intersection \(ι \cap U_{x, r}\) with the open ball around \(x\) with radius \(r\) is nonempty, i.e. if for each \(x\) there exists some Cauchy sequence
  converging to \(x\). Thus, \(ι\) is sated if and only if for each union of Cauchy sequences \(δ\) that converge to all of \(X\), if \(ι\) contains the restriction of \(δ\) to some open
  \(U\), then \(ι\) contains all of \(U\). So our notion generalizes a topological relative version of Cauchy completeness.

  We can also zoom into a point \(x\) by reducing our collection of tablets to the opens of \(X\) containing \(x\). This allows us to say that \emph{\(ι\) is \(τ\)-dense at \(x\)} if it
  contains some sequence converging to \(x\). But actual Cauchy completeness is not straight-forwardly described as satedness in this way, which makes sense since it is not
  homeomorphism-invariant.
\end{Example}

There are some subtleties in the definition though, namely the inequality and the restriction to \(T\) for reducedness. To understand these, let us consider another example where we
see how they work out:

\begin{Example}[Density and satedness for monoid actions]
  \label{xmp:monoid-action-density}
\index{monoid action!density and satedness}
\index{radial density}
  Let \(M\) be a monoid and \(X\) a set with a left action \(λ{:} M \hookrightarrow \mathbf{End}(X)\). Write \(\mathbf{SOb}_λ(X)\) for the preorder of sub-\(λ\)-sets of \(X\), ordered by inclusion; it is a bounded lattice with meets given by
  intersection and joins by union (since each union of \(λ\)-sets is a \(λ\)-set).

  Consider the tablets given by the inclusion \(τ_{\mathbf{SOb}_λ} {:} \mathbf{SOb}_λ(X) \hookrightarrow \mathbf{P}(X)\) of sub-λ-sets. Applying the definitions of density and satedness from
  \cref{rmk:dual-sigma-mu} we obtain that:

  \begin{enumerate}
  \item A subset \(D \subseteq X\) is dense under \(τ_{\mathbf{SOb}_λ}\) if and only if every sub-\(λ\)-set \(S \subseteq X\) such that \(D ∩ S = ∅\) is equal to
    \(\bot_{τ} = \bigcap_{ι ∈ \mathbf{SOb}_{λ}(X)} ι = ∅\) since the empty set is closed under the action of \(λ\). Equivalently, \(D\) meets every orbit.

    If \(k\) is a ring and \(X\) a \(k\)-module, then we can further restrict the tables \(τ'\) to be sub-modules of \(X\). Then \(\bot_{τ'}\) is no longer empty however, since each
    submodule of \(X\) contains a neutral element. In particular let \(k\) be a field and \(V\) a \(k\)-vector space equipped with the action of \(k\) by multiplication. The orbits are \(\bot_{τ} = \{0\}\) and the lines
  \(\mathcal{O}_v = \{ λv \mid λ ∈ k\}\) for \(v \neq 0\). Thus a subset \(D \subseteq V\) is dense if and only if \(D \cap \mathcal{O}_v \neq ∅ \text{ for every } v \neq 0\). One simple way this can happen is when
  \(0\) is in \(D\), otherwise \(D\) has to intersect all orbits individually.

    For instance if we take \(k = ℝ\) and \(V = ℝ^n\), the nonzero orbits are the lines through the origin, which correspond bijectively to points of the real projective space \(ℝ\mathbf{P}^{n-1}\) via the direction map
  \[
    q {:} ℝ^n \setminus \{0\} → R\mathbf{P}^{n-1},\quad v \mapsto [v],
  \]
  where \([v]\) denotes the equivalence class of \(v\) under scalar multiplication. If we instead take \(k\) as \(ℝ_{\geq 0}\) (but leave \(V\) the same), then \(q\) goes to \(\mathbb{S}^{n-1}\). Hence \(D\) is dense precisely when it contains the origin or its image under \(q\) is all of \(\mathbb{S}^{n-1}\); in geometric language, \(D\) is \emph{radially dense}: it is either at the center
  or has at least one point in every direction from the origin.

  \item A subset \(A \subseteq X\) is sated under \(τ_{\mathbf{SOb}_λ}\) if and only if whenever \(A \supseteq D ∩ B\) with \(D\) dense and \(B\) an \(λ\)-subset of \(X\), then already
    \(A \supseteq B\).
    If now \(a\) is some element of \(A\) and \(O\) an orbit that contains \(a\), then \((\{a\} \cup \bigcup_{\mathcal{O}'} \mathcal{O}') ∖ \bot_{τ}\) is dense, where
    \(\mathcal{O}'\) ranges over the orbits of \(X\) that do not contain \(a\) and the complement is set-theoretic. Moreover, if \(O\) is an orbit that contains \(a\) and whose intersection with any
    orbit that does not contain \(a\) factors through \(\bot_{τ}\), then \(O \cap (\{a\} \cup \bigcup_{\mathcal{O}'} \mathcal{O}') = \{a\} \subseteq A\), so that \(O \subseteq A\). In particular, if \(λ\) is a group or the action is one of a
    field on a vector space, then all orbits fulfill this condition , so that \(A\) has to contain every orbit it intersects.

  \item\label{item:55} \(S \subseteq X\) is small under \(τ_{\mathbf{SOb}_λ}\) if the smallest \(λ\)-set \(T\) whose union with \(S\) is \(X\) is the largest \(λ\)-subset, which is \(X\) itself, so if
    the complement of \(S\) is in no proper \(λ\)-subset of \(X\), meaning that the closure of \(¬ S\) under the action of \(λ\) contains \(S\), so if each element of \(S ∖ \bot_{τ}\) is a multiple of
    one outside it. If \(M\) is a group then its orbits are disjoint partitions, so that \(S\) is small if and only if it does not contain an orbit. So in \(ℝ^{n}\) as an \(ℝ\)-module \(S\)
    is small if it does not contain a (linear) line. But in a general monoid the property is looser: for instance in \(ℝ^{n}\) as a \(ℤ\)-module no central \(n\)-disk of radius \(> 0\) is small, since its
    elements cannot be generated by elements outside it.

  \item\label{item:56} \(R \subseteq X\) is reduced if for each small \(S\) and \(λ\)-set \(Q\) if \(R \subseteq S \cup Q\), then \(R \subseteq Q\). Now, if \(R \subseteq S \cup Q\), then \(R \subseteq \overline{S} \cup Q\), where \(\overline{S}\) is
    the closure of \(S\) under the action of \(λ\), so \(R\) is reduced if any part of \(R\) that is in a small \(λ\)-subset is also in any \(λ\)-subset that contains the rest of \(R\).
  \end{enumerate}
\end{Example}

Between these two examples we can get at how \cref{defn_:14} works: in \cref{xmp_:26} we infer about arbitrary subsets of a topological space through its open and closed subsets.
In \cref{xmp:monoid-action-density} we similarly infer about arbitrary subsets of a module, but with some added difficulties and particularities: for one, the intersection of all
tablets is not empty, so if, for instance, we did not use the meet of tablets \(\bot_{M}\) but demanded that a dense subset have a non-empty intersection with each non-empty tablet,
then a dense subset would have to intersect with \(\bot_{M}\) and so over a field contain \(0\), which would mess up the satedness calculation. Similarly, if, in \cref{eq:2} we
assumed that \(c\) would be in \(P\), so if we applied that equation to all subsets in \cref{xmp:monoid-action-density}, then the condition for satedness would be way too stringent
and trivialize. \emph{But} if we do restrict \(c\) to \(T\) but assume equality in \cref{eq:2} instead of inequality, then \cref{xmp_:26} would no longer work because we could not
infer anything about subsets more general than intersections of opens and dense subsets. It is only when we take inequalities and restrict \(c\) to \(T\) that we get the
``obviously'' right results for both examples.

\begin{Remark}
  \label{rmα_:28}
  In \cref{defn_:14} we use the same set of tablets for both the definition of smallness and reducedness. Most of our examples work fine like this, but it is possible to use a
  second set of tablets for reducedness instead of those used for smallness, i.e. use a second morphism \(τ'{:} T' → P\) to quantify \(\)
  second relativization so generally remain with the \(P\)-quantified case for simplicity, but occasionally the twice-relativized notion is necessary, e.g. in \cref{xmp_:22}.

 in \cref{eq:2} is ``absolute'', i.e. quantified over \(P\). Quantifying it over \(T\)
  does not generally yield good results, but we \emph{could} introduce a second set of tablets \(τ'{:} T' → P\) to quantify over \(c\). This further degree of generality can
  sometimes be both necessary and sufficient to achieve nice results, see in particular the results for satedness in \cref{xmp_:22}.
\end{Remark} 

But the calculus of tablets generally allows us to be creative in its usage:

\begin{Example}
  \label{xmp:21}
\index{abstract polytope!smallness and density}
\index{universal vertex}
  In \cref{coR_:3} we found that if a preorder \(P\) is atomistic, then its notion of density is trivial, and similar for smallness. However, if \(P\) is bounded in both directions
  we can obtain an interesting notion of density by considering the embedding \(ι_{cat}{:} P_{cat} \hookrightarrow P\) of the bounded preorder on co-atoms \(P_{cat}\) into \(P\). Density with
  respect to these tablets is more about the position in the preorder than dimension: an element is \(ι_{cat}\)-dense if and only if its intersection with every co-atom is
  non-empty. Dually, an element is \(ι_{at}\)-small if and only if it is in a proper face with every other atom. For points in a polytope this is a known property and a vertex that
  fulfills it is called a \emph{universal vertex} or \emph{dominating vertex}.

  If we apply this to the objects of the sites in \cref{obs_:1} an interesting split emerges:

  \begin{enumerate}
  \item \label{item:28} In globes and simplices (above dimension one), every (maximal) face touches every other face and thus all are \(cat\)-dense, but no other faces are
    \(cat\)-dense. On the other hand, every non-empty proper face shares some non-empty non-maximal face with any vertex, so all are small, while every maximal face shares no
    proper face with the vertex opposite to it, and thus is not small.

  \item\label{item:29} In cubes over dimension one, every proper face has an opposite face which it does not touch and thus is not small or dense. So no non-trivial faces or points
    are \(cat\)-dense or \(at\)-small.

  \item\label{item:30} In abstract polytopes, any case can arise. Moreover, while the abstract duality of the two makes it so that in many cases, such as the above, the two develop
    analogously, this is not according to any straightforward formal duality. For instance, consider an \(n\)-ary pyramid, consisting of a tip \(τ\), an opposite face \(ω\), \(n\)
    rays \(ρ_{i}\) between them and the \(n\) induced faces \(φ_{i}\). In this, there is a clear duality: there is one \(at\)-small vertex \(τ\) and one \(cat\)-dense maximal face
    \(ω\) and for \(n > 3\) no other faces are \(cat\)-dense or \(at\)-small. But if we now take any ray \(ρ_{i}\) and put a small digon into the middle of it in such in a way that
    it does not touch anything except \(ρ_{i}\) and its two adjacent faces, then \(τ\) remains small but (again for \(n > 3\)) there is no dense face anymore.
  \end{enumerate}

  This points again at the interesting geometry of abstract polytopes, which will be featured more in future work.
\end{Example}

But we postpone further examples a little. Let us first note that this extension is rather conservative:

\begin{Proposition}
  \label{■pn_:4}
  Let \(L\) be an upward-bounded preorder and \(T' \xrightarrow{τ'} T \xrightarrow{τ} L\) be morphisms of preorders of \(L\) such that \(τ\) and \(τ'\) reflect tops . If \(b ∈ L\) is \(τ\)-small, then it is \((τ'⨾ τ)\)-small. Moreover, if
  \(τ'\) is surjective, then \((τ'⨾ τ)\)-smallness and \(τ\)-smallness are equivalent. Thus, the smallness of \(b\) only depends on the underlying set of \(T\). In particular, we can assume without loss of
  generality that \(τ, τ'\) are order-reflecting embeddings (and thus fully faithful).
\end{Proposition}
\begin{proof}
  \label{prf_:6}
  If the only \(a ∈ T\) such that \(τ(a) ∨ b = \top_{L}\) is \(\top_{T}\), then if \((τ'⨾ τ)(a') ∨ b = \top_{L}\) it follows from \(τ\)-smallness that \(τ'(a') = \top_{T}\) and since \(τ'\)
  reflects \(\top_{T}\), \(a' = \top_{T'}\). Moreover, since the join is taken in \(L\), the order-theoretic structure of \(T\) does not matter.
\end{proof}

In particular \cref{■pn_:4} tells us that, the rougher our tablets, the harder it is to use them to complement some other \(ι\) without them covering the entire space, so the more things
will look not small, which makes sense. More generally:

\begin{Proposition}[Lax comparison for smallness]
  \label{prop:lax-smallness}
  Let \(L\) be an upward-bounded preorder, and let \(τ {:} T → L\) and \(τ' {:} T' → L\) be monotone maps. Suppose there exists a monotone map \(f {:} T' → T\) such that:
  \begin{enumerate}
    \item\label{item:78} \(τ' \le f⨾τ\), i.e. \(τ'(t') \le τ(f(t'))\) for all \(t' ∈ T'\);
    \item\label{item:111} if \(τ(f(t')) = \top_L\) then \(t' = \top_{T'}\).
  \end{enumerate}
  Then every \(τ\)-small element of \(L\) is also \(τ'\)-small.
\end{Proposition}
\begin{proof}
  Assume \(b ∈ L\) is \(τ\)-small, and let \(t' ∈ T'\) satisfy \(b ∨ τ'(t') = \top_L\).
  From \(τ'(t') \le τ(f(t'))\) we obtain \(b ∨ τ(f(t')) = \top_L\).
  By \(τ\)-smallness of \(b\), we must have \(τ(f(t')) = \top_L\).
  But \cref{item:111} then forces \(t' = \top_{T'}\), which is exactly what is required for \(b\) to be \(τ'\)-small.
\end{proof}
\begin{Remark}
  In the special case where \(f\) is surjective and \(τ' = f⨾τ\) so the inequality in \cref{item:78} becomes an equality \cref{item:111} is trivial, we recover the original statement of
  \cref{■pn_:4}.
  When \(τ\) and \(τ'\) are order-embeddings, \cref{item:111} simply says that if \(τ\) contains the top of \(L\), then \(τ'\) does as well.
\end{Remark}

This also allows us to recover the usual order-theoretic notion of density:

\begin{Example}
  \label{xmp_:25}
  We found in \cref{xmp:17} that a subset of a graph is small if and only if contains no edges or isolated vertices. Given a preorder \(P\), let
  \(ι_{\mathbf{non\text{-}small_{G}}}{:} \mathbf{non\text{-}small_{G}}(P) \hookrightarrow \mathbf{SOb}(P)\), be the embedding of all full sub-preorders of \(P\) that are not small in the
  underlying graph of \(P\) (without identities in \(P\)). Then \(\mathbf{non\text{-}small_{G}}(P)\) contains all isolated vertices and proper intervals in \(P\), i.e. the full
  sub-preorders \([a, b]\) for \(a < b\) consisting of all elements \(c\) in \(P\) such that \(a \leq c \leq b\). Thus, a sub-preorder of \(P\) is
  \(ι_{\mathbf{non\text{-}small_{G}}}\)-dense if and only if it meets isolated vertex and every proper interval in \(P\) non-trivially, so if it is dense in the order-theoretic
  sense.
\end{Example}

% In \cref{sec:open-clos-dual} we will lift our use of tablets to the categorical level, which will generalize how topological spaces interact with their category \(\mathbf{Top}\).

The final point we want to make is that, as with scales, the power of tablets exceeds our use case. In reality they are ubiquitous in mathematics and what has to be underlined is both their
elementarity and their heteromorphy, which can be displayed by an example we hinted at in the introduction: cones and hyperplanes are not common
subobjects of \(ℝ^{n}\) in any natural category that does not also contain many more subobjects. What connects them is that they can be embedded into the same object, \(ℝ^{n}\), which is
incarnated in different categories. \emph{So we fix the object and vary the categories} when we use hyperplanes as tablets to study cones. And while this is a somewhat abstract way to state
what we are doing, the result are conic sections, one of the most venerable subjects of mathematics!\footnote{Note that all the geometric information of \(ℝ^{n}\) is contained in
  \(\mathbf{P}(ℝ^{n})\) and its automorphisms, so when we analyze conic sections as the subobjects of a cone \(χ{:} C \hookrightarrow ℝ^{n}\) generated by the tablets that are hypersurfaces we can apply
  standard geometric arguments. Of course for these arguments we might need to introduce further kinds of subobjects, e.g. lines, but note how these have to be treated on a different levelling
  from hypersurfaces to not mess up the definition of conic sections. Note also that we are using only a tiny part of \(\mathbf{P}(ℝ^{n})\), we could restrict it a lot, e.g. to semialgebraic
  subsets. To obtain algebraic representations we would also have to fix the coordinate system induced by the coprojections \(ℝ \hookrightarrow ℝ^{n}\).}

\subsection{Scales}
\label{sec:scales}

We can relativize this system in the other direction as well, which allows us to encompass measure-theoretic examples, as we will do in \cref{sec:categ-meas-theory}:

\begin{Definition}
  \label{defn_:1}
\index{scale}
  A \emph{scale} on an upward-bounded preorder \(P\) is a map \(P \overset{σ}{→} S\) into some preorder \(S\). An element \(a\) of \(P\)  is \emph{small over \(σ\)} if for each \(b ∈ P\) and join
  \(a ∨ b= \top_{P}\) it holds that \(σ(b) = \top_{σ}\) where \(\top_{σ}\) is the top of the image of \(σ\). Dually for density.

  \(c\) is \emph{reduced over \(σ\)} if for each \(a\) that is small over \(σ\) and \(b ∈ P\) such that
  \(a ∨ b \geq c\)
 it holds that
  \(σ(b) \geq σ(c)\). Dually for satedness.
\end{Definition}
\begin{Remark}
  \label{rmα_:6}
  Unlike with tablets and \(P\) itself we do not necessarily assume that scales are upward-bounded. The difference is small, but conceptually a scale need not come from the
  subobjects of some object.
\end{Remark}
\begin{Example}
  \label{xmp_:21}
  Similarly to \cref{xmp_:20}, \cref{defn_:1} reduces to the absolute definition of the fourfold calculus when \(σ = \mathbf{id}_{P}\).
\end{Example}
\begin{Example}
  \label{xmp:28}
\index{measure!as scale}
  Let \(σ{:} \mathbf{M}(X) → ℝ_{\geq0, ∞}\) be a measure on the \(σ\)-algebra of a measure space \(X\). A set \(S\) in \(M\) is small over \(σ\) if and only if for each \(T\) such that
  \(S \cup T = X\) it holds that \(σ(T) = σ(X)\). There is a minimal such choice, which is the Boolean negation \(¬ X\), so \(S\) is small when \(σ(¬ S) = σ(X)\). By binary
  additivity, \(σ(S) + σ(¬ S) = σ(X)\), so when \(σ(X)\) is finite we can form \(σ(S) = σ(X) - σ(¬ S) = 0\) since \(ℝ_{\geq 0}\) is cancellative, so \(S\) is not \(σ\)-small if and
  only if \(σ(S) = 0\), so if it is negligible. But if \(σ(X) = ∞\), then

  \begin{itemize}
  \item each finite \(S\) is small by additivity,
    
  \item it does not hold that each \(S\) with \(σ(S) = ∞\) is not small: \(σ(¬ S)\) can still be infinite. So \(S\) is only small if and only if it is co-finite.
  \end{itemize}

  so our notion is not \emph{exactly} the same as negligibility, but we can derive measure-theoretic negligibility from \(σ\)-smallness by defining it
  through restriction to finite-measure subspaces. What is interesting however is the our notion of density \emph{does} reduce to having a negligible complement:
  \(D\) is dense over \(σ\) if for each \(T\) such that \(D \cap T = ∅\) it holds that \(σ(T) = 0\), so if its complement is negligible. Moreover, \(A\) is sated if for each such dense subset
  \(D\) it holds that if \(A\) contains \(D \cap T\) with \(D\) dense, then \(σ(A) \geq σ(T)\). Since \(σ(T) = σ(T \cap D) + σ(T \cap ¬ D) \leq σ(A) + σ(¬ D) = σ(A)\), this holds trivially in our case.

  Now, \(R\) is reduced over \(σ\) if for each small \(S\) and arbitrary \(T\) if \(R \subseteq S \cup T\) then \(σ(R) \leq σ(T)\). It follows from additivity that \(σ(S \cup T) \leq σ(S) + σ(T)\), so
  if \(σ(X)\) is finite then each non-negligible element of \(\mathbf{M}(X)\) is reduced. But also, if \(σ(R) = 0\), then \(σ(R)\) is minimal, so that \(R\) is \emph{also} reduced.
  So actually \emph{every} subset of \(X\) is reduced.

  If \(σ(X)\) is \(∞\) then we have to think some more. If \(R\) is small, then \(R \subseteq R \cup ∅\), so \(σ(R) \leq σ(∅) = 0\), so \(R\) is negligible. So to be non-negligible, \(R\) has to
  be non-small i.e. co-finite. But then if \(R \subseteq S \cup T\), then \(∞ = σ(R) = σ(T)\) since \(∞ = \top\). So \(\mathbf{M}(X)\) is the union of small and reduced elements, with the
  intersection of the two being negligible ones.

  Note in particular that we can do this for the linear
  extension \(\mathopen{∣}-\mathclose{∣}: \mathbf{P}(ℕ) → ℕ_{∞}\) of the unit measure, i.e. cardinality for finite sets and \(∞\) for infinite ones. Then a subset \(S\) of \(ℕ\) is \(σ\)-small if and only if each \(T\) such that
  \(S \cup T = ℕ\) has infinite cardinality, so if the complement of \(S\) is infinite, so if \(S\) is co-infinite.
\end{Example} 

So in general our framework allows us to measure smallness by introducing two optional parameters, tablets and a scale:

\begin{Definition}
  \label{defn_:21}
  \index{small under a tablet over a scale}\index{reduced under a tablet over a scale}
  Let \(T \overset{τ}{→} P \overset{σ}{→} S\) be a triple of maps of preorders where we assume that \(T\) and \(P\) are upward-bounded, and let \(a\) be
  an element of \(P\). Then \(a\) is \emph{small under \(τ\) over \(σ\)} if for each join \(a ∨ τ(b) = \top_{P}\) with \(b\) in \(T\) it holds that \((τ⨾ σ)(b) = \top_{σ}\) where
  \(\top_{σ}\) is the top of the image of \(σ\). Dually for density.

  \(c\) is \emph{reduced under \(τ\) over \(σ\)} if for each \(a\) that is small under \(τ\) over \(σ\) and \(b ∈ T\) such that \(a ∨ τ(b) \geq c\) it holds that
  
  \[(τ⨾ σ)(b) \geq σ(c).\]

  Dually for satedness.
\end{Definition}
\begin{Remark}
  \label{rmα_:38}
  As in \cref{rmα_:28} we can quantity \(b\) in \cref{defn_:21} for reducedness and satedness with a second set of tablets.
  % Again we don't generally assume this increased level of generality but
  % again it can be useful, particularly in more practical examples, such as \cref{xmp_:22}, not just where we need two different scales but where defining a scale for
  % smallness/density itself would be counter-productive (i.e. when there exists a canonical scale for the first ta).
\end{Remark} 

The first generalizes topological inferences, the second measure theoretic ones.

\begin{Notation}
  \label{ntn:1}
  As a convenient shorthand, we just write the parameters next to each other separated by \(|\)-signs. So we write ``\((τ \mid σ)\)-small'' for ``small under \(τ\) over \(σ\)''. Since all
  parameters are optional we leave those out that we don't need.
\end{Notation}

We have a statement similar to \cref{■pn_:4} but on the scale side:

\begin{Proposition}[Lax comparison for scales]
  \label{prop:lax-scale}
  Let \(P\) be an upward-bounded preorder, let \(τ {:} T → P\) be a monotone map of tablets and \(σ {:} P → S\), \(σ' {:} P → S'\) monotone maps of scales. Suppose there is a monotone
  map \(g {:} S → S'\) such that:
  \begin{enumerate}
    \item \label{item:57} \(σ' \ge σ⨾ g\), i.e. for all \(p ∈ P\),
      \[
        g(σ(p)) ≤_{S'} σ'(p),
      \]
    \item \label{item:58} the top elements of the images of \(σ\) and \(σ'\) correspond:
      \[
        g(\top_{σ}) = \top_{σ'},
      \]
      where \(\top_{σ}\) and \(\top_{σ'}\) denote the top elements of the images \(\mathbf{Im}(σ)\) and \(\mathbf{Im}(σ')\), respectively.
  \end{enumerate}
  Then any element of \(P\) that is \((τ \mid σ)\)-small is also \((τ \mid σ')\)-small.
\end{Proposition}
\begin{proof}
  Let \(a ∈ P\) be \((τ \mid σ)\)-small. By \cref{defn_:21}, this means:
  \[
    (\forall b ∈ T)\; \bigl(a ∨ τ(b) = \top_P\bigr) \Rightarrow σ\bigl(τ(b)\bigr) = \top_{σ}.
  \]
  Now suppose \(b ∈ T\) satisfies \(a ∨ τ(b) = \top_P\). Then by \((τ \mid σ)\)-smallness,
  \[
    σ\bigl(τ(b)\bigr) = \top_{σ}.
  \]
  Applying \(g\) and using \cref{item:58} gives
  \[
    g\bigl(σ(τ(b))\bigr) = g(\top_{σ}) = \top_{σ'}.
  \]
  Using \cref{item:57}, we also have
  \[
    g\bigl(σ(τ(b))\bigr) \le σ'(τ(b)).
  \]
  Since \(\top_{σ'}\) is the maximum element in the image of \(σ'\), this inequality forces
  \[
    σ'(τ(b)) = \top_{σ'}.
  \]
  This is exactly the condition that \(a\) is \((τ \mid σ')\)-small.
\end{proof}
\begin{Corollary}
  \label{coR_:4}
  Given \(σ{:} P → S\) and an  order-preserving injection \(f{:} S → S'\), smallness over \(σ\) and \(σ⨾ f\) are equivalent.
\end{Corollary}

This is the exact scale-side analogue of \cref{■pn_:4} for tablets: the ``rougher'' the scale (in the sense of collapsing more values into the top of its image), the more elements of \(P\)
appear small with respect to it. In particular, the only contribution of \(σ\) to the measurement is the filter in the preimage of its maximal image element, so one might question whether its
definition might not be simplified. But here, as with the tablets there are reasons for doing things this way, even beyond those we outlined in \cref{sec:subobject-embeddings}: for one, consider the dualization of \cref{defn_:21}, concerning dense elements: it is identical except for using downward-bounded preorders.

\begin{Remark}[Dualization for density and satedness]
  \label{rmk:dual-sigma-mu}
  By order-theoretic duality, one obtains analogous notions for downward-bounded preorders, which we call \emph{density} and \emph{satedness} under a tablet and over a scale.
  Explicitly, let \(T \xrightarrow{τ} P \xrightarrow{σ} S\) be morphisms of preorders, and assume that \(P\) and \(S\) are downward-bounded (with bottoms \(\bot_P\) and \(\bot_S\))
  and that the image of \(σ\) has a bottom element \(\bot_σ\).

  An element \(d \in P\) is \emph{dense under \(τ\) over \(σ\)} if for every \(b \in T\) with \(d ∧ τ(b) = \bot_P,\) it holds that
  \((τ⨾σ)(b) = \bot_σ\). An element \(e \in P\) is \emph{sated under \(τ\) over \(σ\)} if for every \(d \in P\) that is dense under \(τ\) over \(σ\) and every \(b \in T\) such that \(d ∧ τ(b) \le e\)
  we have \((τ⨾σ)(b) \le σ(e)\).
\end{Remark}

Thus, if the preorders in question have both an upward and a downward bound, \emph{the same data} can be used to define them, in which case both the upper filter and the lower ideal of
\(σ\) matter. Moreover, our treatment allows us to consider the smallness relations in different preorders using the same scale and/or tablet preorder by forming under/overcategories and
provides a fluid transition to measure theory, which will be used in \cref{sec:applications}.

In \cref{xmp:28} we saw already that \cref{prp:5} generalizes in a relativized way. Let us take a moment to describe this:

\begin{Proposition}
  \label{■pn_:13}
  Let \(T \xhookrightarrow{τ} P \xhookrightarrow{σ} S\) be tablets and a scale with \(T\) being a bounded preorder (and \(P\) upward-bounded). Then \(a\) is \((τ, σ)\)-small and
  \((τ, σ)\)-reduced if and only if \(σ(a) \leq \bot_{τ⨾ σ}\).
\end{Proposition}
\begin{proof}
  \label{prf_:5}
  The proof of \cref{■pn_:5} mostly carries over:

  Because \(T\) is bounded, \(\bot_{τ⨾ σ} := σ(τ(\bot))\) exists.
  If \(σ(a) \leq \bot_{τ⨾ σ}\), then \(a\) is \((τ, σ)\)-small: if \(a ∨ τ(b) = \top_{P}\), then

  \[\top_{σ} = σ(\top_{P}) = σ(a ∨ τ(b)) \leq σ(a) ∨ σ(τ(b)) \leq  σ(τ(\bot_{T})) ∨ σ(τ(b)) = σ(τ(b)),\]

  and thus \(\top_{σ} = σ(τ(b))\).
  
  If \(σ(a) \leq \bot_{τ⨾ σ}\), then \(a\) is \((τ, σ)\)-reduced:\(\bot\) is reduced: if \(τ(b) ∨ c \geq a\) with \(b\) small, then \(σ(τ(b)) \geq σ(τ(\bot_{T})) = \bot_{τ⨾ σ}\) by the universal property
  of \(\bot_{T}\).
\end{proof}

Both tablets and scales are general notions and sometimes smallness of one of them can be generated by the other:

\begin{Proposition}
  \label{■pn_:12}
  Let \(ι{:} A \hookrightarrow P\) be a proper full sub-preorder in \(\mathbf{UBP}\) such that for each \(a, b ∈ ι\) there exists a \(c\) in \(ι\) such that \(a \leq c, b \leq c\). Then there exists a scale \(σ{:} P → S\)
  such that \(d ∈ P\) is small if and only if there exists an \(a ∈ ι\) such that \(d \leq a\).
\end{Proposition}
\begin{proof}
  \label{prf_:3}
  Take \(S\) to be the two-element total order and define \(σ\) by

  \[σ(d) := \begin{cases}
      \bot \text{ if and only if there exists some \(a ∈ ι\) such that} d \leq a, \\
      \top \text{ else}
    \end{cases}.\]

  This is a preorder morphism by definition. Since \(P\) has a top element and \(ι\) is proper, the fiber \(σ^{-1}(\bot)\) over \(\top\) is inhabited, and by assumption \(σ^{-1}(\bot)\) is
  closed under binary joins, so each element in it is \(σ\)-small.
\end{proof}

So we might think that we can construct a scale from each family of tablets with the same smallness relations, like this:

\begin{Example}
  \label{xmp_:18}
  In \cref{xmp_:26} we saw that a subset of a topological space \(X\) is small with respect to the tablets consisting of closed subsets of \(X\) if and only if it is nowhere dense.
  We can define a scale \(σ{:} \mathbf{P}(X) → S\) on \(\mathbf{P}(X)\) where \(S\) consists of the nowhere dense subsets of \(X\) and a freely added top \(\top\), and scale is the
  identity on nowhere dense subsets and maps every other subset to \(\top\).
\end{Example}

But this is not true:

\begin{Counterexample}
  \label{cxmp_:1}
  Let \(P\) be the upward-bounded preorder on elements
  \(\top, a, b, c, d, e\) with relations \(a \leq d, b \leq e, c \leq d, c \leq e\) and, of course, \(\top\) above everything. Then \(a ∨ c = d\) and \(b ∨ c = e\). Let moreover \(τ\) be the embedding of
  \(\{c, \top\}\). Then both \(a\) and \(b\) are \(τ\)-small but the join of \(a\) and \(b\) is \(\top\), so that they are not small under any scale that trivializes \(d\) and \(e\) by mapping them to
  \(\top{σ} := σ(\top_P)\) but preserves \(a\) and \(b\) by mapping them below \(\top_{σ}\). The underlying reason for this is that, unlike in the absolute case, where the tablets are
  \(id_{P}\) where smallness is stable under joins, this is not always the case for non-trivial tablets. So in particular \(τ\)-small elements are can be non-small against each other.
\end{Counterexample}

In general, \cref{■pn_:12} tells us that for \emph{any} ideal in \(P\) we can construct \emph{some} scale. But usually, as in the following examples, the point rather to start with
something that it seems reasonable to use as a scale, and specialize the fourfold calculus to it.

\subsection{Examples}  
\label{sec:examples-1}

This system allows us to mix topological and measure-theoretic smallness:

\begin{Example}
  \label{xmp:2}
\index{density!topological-measure-theoretic}
\index{fat Cantor set}
  A (not necessarily open) subset \(ι\) of \(X\) is dense with respect to the morphism triple
  \(\mathbf{O}(X) \hookrightarrow \mathbf{M}(X) → ℝ_{∞, \geq 0}\), where \(\mathbf{M}(X)\) are the measurable subsets of \(X\), if and only if for each open subset \(κ\) of \(X\) it holds that if
  \(ι ∧ κ = ∅\) in \(\mathbf{M}(X)\), then \(κ\) has measure \(0\). So \(ι\) is dense if and only if the only open subsets that it does not intersect with are measure-theoretically negligible.

  If we take, for concreteness, \(X = ℝ^{n}\) with its usual measure and topology, then this is a slightly disappointing example since no non-trivial open subset of \(ℝ^{n}\) has measure \(0\).
  But what if we switch out \(\mathbf{O}(X)\) with \(\mathbf{C}(X)\)? Then \(ι\) is dense if and only if each closed subset \(κ\) of \(ℝ^{n}\) that \(ι\) does not intersect with has measure
  \(0\). Now, if \(κ\) is regular, so if it is the closure of a non-trivial open, then its measure is surely not \(0\), so it has to intersect with \(ι\). But there are small (and so in
  particular non-regular) closed subsets of \(ℝ^{n}\) that do not have measure \(0\), such as a \emph{fat Cantor set} (or \emph{Smith-Volterra-Cantor set}), which is closed, nowhere dense,
  yet has positive Lebesgue measure.

  Thus, the density condition with respect to closed tablets and the Lebesgue scale is sensitive to these ``thick'' nowhere dense sets while also excluding subsets that are too small such as
  atoms. In any case, it is a different and unexpected notion. Now, what does completeness mean? Well, if \(ι \supseteq κ \cap λ\), with \(κ\) closed-dense and \(λ\) closed, then \cref{rmk:dual-sigma-mu} tells
  us that \(σ(ι) \geq σ(λ)\), so completeness says that if \(ι\) is dense in a closed \(λ\), then \(ι\) is at least as big as \(λ\) when measured.

  Dually, \(ι\) is closed-\(σ\)-small if and only if each closed \(κ\) such that \(ι \cup κ = X\) has the same measure as \(X\) and similar for opens, and \(ι\) is closed-reduced if for each
  closed \(κ\) and small \(λ\) it holds that if \(ι\) is contained in \(κ \cup λ\) then \(σ(ι) \leq σ(κ)\) and similar for opens.
\end{Example}

It also allows us to introduce weight and rounding:

\begin{Example}
  \label{xmp_:10}
  Let \(S\) be a set, \(w{:} S → ℝ_{\geq 0, ∞}\) an arbitrary weighting function and \(\overline{w}{:} \mathbf{P}(S) → ℝ_{\geq 0, ∞}\) be defined by linear extension, i.e. for some subset
  \(ι{:} A \hookrightarrow S\) of \(S\) and well-ordering \(\overset{>}{A}\) of \(A\) the limit \(\mathbf{lim}_{a ∈ \overset{>}{A}} ∑_{b ∈ (a)} w(ι(b))\), with that being \(∞\) if that series
  diverges. Note that since the reals are Cauchy-complete and \(w\) is positively valued, each such \(\overline{w}\) on each \(ι\) either goes to infinity for every well-ordering
  of \(A\) or is absolutely convergent on some well-ordering, which implies that it is so on every well-ordering. In other words, this definition is independent of the chosen
  order. Since \(\overline{w}\) is given by linear extension, it is binarily additive, so smallness and density are special cases of \cref{xmp:28}.
\end{Example}
\begin{Example}
  \label{xmp_:12}
  In the situation of \cref{xmp:28} let \(\lfloor - \rfloor_{c}{:} ℝ_{\geq0, ∞} → ℕ_{\geq0, ∞}\) be given on \(r ∈ ℝ_{\geq 0, ∞}\) by the largest \(n ∈ ℕ_{\geq0, ∞}\) such that
  \(r \geq nc\). Note that these definitions also work for \(r = ∞\) and \(c = ∞\). Then \(ι\) is \((σ⨾ \lfloor - \rfloor_{c})\)-small if and only if
  \(\lfloor σ(X) \rfloor_{c} = \lfloor σ(¬ ι) \rfloor_{c}\). So since \(σ(X) = σ(¬ ι) + σ(ι)\) we have that \(ι\) is
  \((σ⨾ \lfloor - \rfloor_{c})\)-small if and only if

  \[\lfloor σ(¬ ι) + σ(ι) \rfloor_{c} = \lfloor σ(¬ ι)\rfloor_{c}\].

  So basically \(ι\) is small if and only if it fits within a rounding error.

  Dually, \(ι\) is dense if its intersection with each \(\geq c\)-large subset is non-nil, and sated if for every dense \(κ\) and \(λ\), if \(ι\) contains \(κ\) in \(λ\), then the
  measure of \(ι\) is as large or larger as the measure of \(λ\), up to rounding downwards.
\end{Example}

But it is generally of considerable flexibility. For instance we can say:

\begin{Example}
  \label{xmp:106}
  Let \(ι{:} \mathbf{P}_{inf}(ℕ) \hookrightarrow \mathbf{P}(ℕ)\) be the embedding of the sub-preorder of \(\mathbf{P}(ℕ)\) whose non-trivial elements are the infinite subsets of \(ℕ\). Then \(S \subseteq ℕ\) is \(ι\)-dense if
  and only if its intersection with each infinite subset of \(ℕ\) is infinite, i.e. if its complement is finite, so if \(S\) is co-finite. \(S\) is \(ι\)-sated if and only if whenever
  \(S \supseteq Q \cap T\) with \(Q\) infinite and \(T\) co-finite, then \(S \supseteq Q\). This holds trivially for finite sets. If \(S\) is infinite but not \(ℕ\), choose \(x ∉ S\). Set \(Q = S ∪ \{x\}\)
  (infinite) and \(T = ℕ \setminus \{x\}\) (co-finite). Then \(Q ∩ T = S\) but \(Q ≠ S\), violating satedness. For \(ℕ\), satedness holds again.
\end{Example}

Related to this is the following example:

\begin{Example}
  \label{xmp:6}
  Let \(ι{:} \mathbf{P}_{\cancel{∂}}(ℝ^{n}) \hookrightarrow \mathbf{P}(ℝ^{n})\) be the embedding of unbounded subsets of \(ℝ^{n}\). Then \(S \subseteq ℝ^{n}\) is \(ι\)-dense if and only if its intersection with every unbounded
  subset of \(ℝ^{n}\) is non-empty, so if its complement is bounded. We can again prove as in \cref{xmp:106} that the only subsets of \(ℝ^{n}\) that are \(ι\)-sated are bounded or
  \(ℝ^{n}\) itself.
\end{Example}
\begin{Example}
  \label{xmp_:15}
  In \cref{dfn:25} we defined geodesically convex sub-2-categories in such a way that each geodesically convex \(ι{:} \mathcal{A} \hookrightarrow \mathcal{B}\) is uniquely determined by the objects it contains. Thus,
  there exists a preorder morphism \(τ_{conv}{:} \mathbf{Conv}(\mathcal{B}) \hookrightarrow \mathbf{Ob}(\mathcal{B})\) from the convex sub-2-categories of \(\mathcal{B}\) forward ordered by inclusion
  (which is non-evil if \(\mathcal{B}\) has no non-trivial isomorphisms like the examples we looked at in \cref{sec:other-examples}). so that a subset of the objects of \(\mathcal{B}\) is
  \(τ_{conv}\)-small if and only if the smallest geodesically convex sub-2-category of \(\mathcal{B}\) that contains all objects outside \(ι\) is \(\mathcal{B}\) itself.
\end{Example}

Interestingly, our system also allows us to ``stack'' notions of smallness/density:

\begin{Example}
  \label{xmp:30}
  If the tablets \(τ{:} \mathbf{SOb}_{\mathbf{Mon}}(M) \hookrightarrow \mathbf{P}(M)\) are given by the embedding of the sub-monoids of a monoid \(M\), then a subset \(ι\) of \(M\) is
  \(τ\)-small if and only if its complement \(¬ ι\) is not contained in a proper submonoid. So for
  instance in \(ℕ\) the even numbers are small while the odd numbers are not. \(ι\) is \(τ\)-reduced if whenever \(ι\) is contained in the union of a submonoid and a small subset, then
  \(ι\) is in the submonoid.

  Note that \(\bot_{τ}\) is \(\{e\}\), not \(\emptyset\) (the bottom of \(\mathbf{P}(M)\)), so that \(ι\) is \(τ\)-dense if for each non-trivial submonoid \(κ\) the intersection of \(ι\) with \(κ\) contains an
  element outside \(\{e\}\). It suffices to test this on principally generated submonoids, so \(ι\) is \(τ\)-dense if its intersection with each orbit \(⟨a⟩\) contains a non-zero
  power of \(a\). \(ι\) is \(τ\)-sated if for each \(τ\)-dense \(κ\) and submonoid \(λ\) it holds that \(ι\) contains \(κ \cap λ\) if and only if \(ι\) contains \(λ\). Given now \(a ∈ M\) such
  that \(ι\) contains \(a^{n}\) let \(κ\) consist of \(a^{n}\) and all \(b\) such that \(a\) is not in the orbit of \(b\). Then \(κ\) is \(τ\)-dense and \(ι\) contains \(a^{n}\); so that \(ι\) is
  again \(τ\)-sated if and only if it contains each radical of its elements and thus if it is reduced in the monoid-theoretic sense.
\end{Example}
% \begin{Example}
%   \label{xmp:26}
%   In \cref{xmp:30} we found that a subset \(D\) of a set with a monoid \(M\) is monoid-dense if and only if it contains an element of every non-trivial orbit of \(M\). Let now
%   \(ι_{\mathbf{M\text{-}dense}}{:} \mathbf{M\text{-}dense}(M) \hookrightarrow \mathbf{P}(M)\) be the inclusion of monoid-dense subsets. Then a subset \(E\) of \(M\) is \(ι_{\mathbf{M\text{-}dense}}\)-dense
%   if and only if its intersection with every monoid-dense subset of \(M\) is non-empty, so if it contains \(\bot_{M}\) and some non-trivial orbit in \(M\).
% \end{Example}

% Let us also do the dual case:

% \begin{Example}
%   \label{xmp:22}
%   Let the embedding \(ι_{SMon}\) be given as in \cref{xmp:subsets-smon-sated}. Then a subset \(A\) of \(M\) is small if and only if there is no non-trivial submonoid \(N \hookrightarrow M\) such that \(A \cup N\) does
%   not contain the complement of any non-trivial submonoid. \(A\) is reduced if and only if for each small \(B\), if \(A = B ∨ C\), then \(A = C\), so basically if it is determined by its intersections with submonoid-complements.
% \end{Example}

Let us also add some examples of a completely different variety:

\begin{Example}
  \label{xmp:1}
\index{cost!as scale}
\index{smallness!applied to machine parts}
  Let \(M\) be some kind of machine and \(\mathbf{SOb}(M)\) the preorder of parts of \(M\) where \(a \leq b\) if and only if \(b\) is not functional without \(a\). Then \(a\) is small
  if and only if it factors through each maximal part of \(M\) by \cref{■pn_:8} and the finitude of real-world machines.
  Now let \(τ{:} \mathbf{T}(M) \hookrightarrow \mathbf{SOb}(M)\) be the embedding of those parts that can be bought stand-alone. Then a part \(ι\) is \(τ\)-dense if and only if it factors through each part of
  \(M\) that can be bought, so if we get it delivered each time we buy a replacement. Notice that the embedding \(τ\) is connected to the material we exposed in
  \cref{sec:monoidal-smallness}: \(τ\) embeds those parts that can be duplicated, so for which addition makes sense (and arises, for instance if \(b\) is broken and we buy a replacement
  \(b'\), but \(a \leq b\) is fine, so we obtain an \(a\)-spare \(a'\)).

  But because \(\mathbf{SOb}(M)\) is elementary, i.e. consists of parts and not families of parts, its smallness relations are somewhat rough. So let us endow
  \(\mathbf{P}(\mathbf{SOb}(M))\) with the relation \(F \leq G\) if and only if for each \(a ∈ \mathbf{SOb}(M)\) such that each element of \(G\) is \(\leq a\) it holds that each
  element of \(F\) is \(\leq a\). So in particular if \(a = \bigvee B\) then \(a = B\) in this order up to univalence, so as a category this preorder is actually in part a weak quotient and
  equivalent to the ideal completion.

  Now, smallness under the powerset-induced function
  \(\mathbf{P}(τ){:} \mathbf{P}(\mathbf{T}(M)) \hookrightarrow \mathbf{P}(\mathbf{SOb}(M))\) gives us a more useful size measure. Then a family of parts \(F\) of \(M\) is
  not small if and only if we can buy a family \(G\) that contains all parts of \(M\) missing in \(F\) but not all parts of \(M\) in itself, so that we can assemble \(M\) from \(F\) and \(G\), given the
  necessary skill and machinery.

  But of course there is a \emph{cost} associated with duplication. So let \(σ{:} \mathbf{T}(M) → ℝ_{\geq 0, ∞}\) be given by the price of each part and
  \(\overline{σ}{:} \mathbf{SOb}(M) → ℝ_{\geq 0, ∞}\) by \(\overline{σ}(a) := \bigwedge_{a' ∈ \mathbf{T}(M){:} a \leq τ(a')} σ(a')\) so basically the least cost we have to pay for \(a\). Finally, let \(\hat{σ}\) be
  given by  \(\hat{σ}(F) := \bigwedge_{F \simeq G} ∑_{ι ∈ G} \overline{σ}(ι)\), so the least cost we have to pay for the parts of \(F\).  Then \(F\) is \((\mathbf{P}(τ) \mid \hat{σ})\)-small if and only if
  each family of parts of \(M\) that we can buy costs at least much as \(M\) itself — a frequent occurrence when one tries to repair something. If we combine this scale with the rounding we
  introduced in \cref{xmp_:12}, we can further refine smallness so that a family is small if and only if the cheapest way to complete it is as expensive as \(M\) up to a rounding error.
\end{Example}
\begin{Example}
  \label{xmp_:1}
\index{infinitesimal}
\index{nonstandard analysis!infinitesimality}
\index{density function!for gases}
  As noted in the introduction, our notion of density is motivated by the topological one, which is binary. Here we want to sketch out the connection to the treatments of density
  through continuous functions, as arise in chemistry. So let \(X\) be spacetime and \(γ\) some gas. When should we say that \(γ\) is dense in \(X\)? This question is tricky
  because \(γ\) has, strictly speaking, only finitely many atoms while \(X\) is continuous; and often glanced over. Arguably the most conceptually thorough way to handle such
  questions is through Robinson's nonstandard analysis,\cite{robinsonNonstandardAnalysis1974} or even more particularly through Nelson's internal set
  theory.\cite{nelsonInternalSetTheory} In these, the ``standard'' natural numbers \(ℕ\) are embedded in a \emph{nonstandard} extension \(\overline{ℕ}\) which fulfills the same
  algebraic axioms but contains additional numbers \(N > ℕ\), which are larger than any standard natural number (this is called \emph{hyperfinite}). So a nonstandard number is a
  number that is larger than any finite number but acts in essentially the same way. As Nelson worked out, this gives a rigorous idealization of what it means that ``a number
  beyond the standard range'', as is often assumed in the natural sciences. Similarly the ``standard'' real numbers \(ℝ\) are embedded into a nonstandard extension \(\overline{ℝ}\)
  that also contains the nonstandard natural numbers \(\overline{ℕ}\), so that there exists for each nonstandard number \(N\) an inverse \(\frac{1}{N}\) and thus an
  \emph{infinitesimal number} \(δ\) such that \(0 < δ < ℝ_{> 0}\).  We can define a scale \(\mathbf{st}_{∞}{:} \overline{ℝ} → ℝ_{\pm ∞}\), given by assigning to each
  extended real \(r\) its standard part, and a scale \(\mathbf{infl}{:} \overline{ℝ} → \mathbf{h}(0)_{\pm ∞}\) to the \emph{halo} \(\mathbf{h}(0)\) of such infinitesimal numbers
  around \(0\), given by the identity on \(\mathbf{h}(0)\) and \(\pm ∞\) otherwise. These two jointly
  describe non-standard finitude and infinitesimality through smallness.

  It is also well-known (see for instance \cite{alamApplicationsNonstandardAnalysis2021}) that in such a setting each Hausdorff space admits a nonstandard extension, so if we
  assume that our spacetime is Hausdorff, which seems reasonable, then there exists an extension \(ξ_{X}{:} X \hookrightarrow \overline{X}\) and a partial function
  \(\mathbf{st}{:} \overline{X} → X\) such that \(ξ_{X}⨾ \mathbf{st} = \mathbf{id}_{X}\). So \(\overline{X}\) houses \(X\) but also for each point \(x\) of \(X\) a \emph{monad} or
  \emph{halo} \(\mathbf{h}(x) := \mathbf{st}^{-1}(x)\) filled with nonstandard elements around \(X\).

  A very simple definition for ``γ is dense in \(X\)'' is then that each open subset of \(X\) contains some atom in \(γ\). If we assume that the spatial extension of each atom is
  0,\footnote{More generally if we assume that the diameter of each atom is infinitesimal.}, then this is impossible if \(γ\) has only a standard amount of atoms (since \(X\) is
  continuous), but quite possible if \(γ\) contains a nonstandard set \(G\) of them. It makes sense then to treat \(γ\) as a function \(γ{:} G \hookrightarrow \overline{X}\) which assigns to each
  atom of \(γ\) its position. But this just says that \(γ\) is dense with respect to the collection of tablets given by the composite
  \(\mathbf{O}(X) \hookrightarrow \mathbf{P}(X) \hookrightarrow \mathbf{P}(\overline{X})\).

  The question is now whether we can derive from such a nonstandard gas a function which assigns to each point \(x\) in \(X\) a density of \(γ\) at \(x\), which would complete the
  bridge from our notion to density functions. An easy approach in line with our notion of density would be to first fix a
  \(δ ∈ \mathbf{h}(0)\) that we use to scale infinitesimal
  lengths to normal ones. We can derive a density function \(δ_{γ}'(x) := \frac{1}{\mathbf{min}\{r ∈ ℝ_{∞} | ∃ g ∈ G{:} γ(g) ∈ ●^{δr}_{x}\}}\) where \(●_{x}^{δr}\) is the ball around
  \(x\) of radius \(δ r\), so by finding the smallest ball around each \(x\) that does not contain any elements of \(γ\) and thus the representators of the smallness-ideal. This
  gives us the right result at \(x\) in that if there is an atom of \(γ\) precisely at \(x\) then the density function maps to \(∞\) and the function decreases with the perimeter.

  But in chemistry a density function is dependent on both the local number of atoms and
  their distance, and this introduces a parameter that cannot cleanly be handled just by the current framework. We could try to sum over all atoms in \(\mathbf{h}(x)\), perhaps
  scaled by their mass and inversely scaled by their distance to \(x\), but this sum is ill-defined in non-standard analysis. What we can do is fix a nonstandard \(δ\) and sum over
  all atoms in the \(δ\)-ball around \(x\). But the result of this sum would, without further assumptions, depend on \(δ\), and not necessarily just up to a nonstandard part. So
  sadly it seems like non-standard analysis cannot be used to derive a density function in a way that accounts for both the mass of atoms and their distance from \(x\) without some
  arbitrariness, even in an idealized way.
\end{Example}
\begin{Example}
  \label{xmp_:22}
  \index{digitization!as tablets}
  \index{density function!from lattices}
  Digitization, such as pixellation, is ubiquitous in the modern world, and we often test things up to pixels, e.g. when we film with a digital
  camera. Abstracting a little we can describe this by taking the inclusion of the (order-theoretic) subobject lattice
  \(\mathbf{P}(λ){:} \mathbf{P}(Λ) \hookrightarrow \mathbf{P}(M)\) as tablets, where \(λ\) is the embedding of a (module-theoretic) lattice, i.e. additive sub-monoid
  \(λ{:} Λ \hookrightarrow M\) into some module \(M\) over some semiring or ring. So we use \(λ\) to digitize \(M\). For \(V = R^{n}\) with \(R\) a ring we always have a canonical such
  sub-monoid given by the embedding \(⟨1⟩^{n}{:} ℤ^{n} \hookrightarrow R^{n}\) that is the \(n\)-fold power of the initial ring-morphism into \(R\) (i.e. the submonoid additively generated by
  \(1 ∈ R\)), but we can also use other lattices. Then a subset of \(V\) is \(λ\)-dense if and only if it contains all elements of \(λ\). Often we are only interested in the density of a
  subobject \(ι{:} A \hookrightarrow V\) in some neighborhood \(U\), such as the convex closure of \(ι\), which we can evaluate by restricting \(λ\) to \(U\). Then \(ι\) is
  \(λ|_{U}\)-dense if and only if it contains all elements of \(λ\) in \(U\).

  Now, many module-theoretic lattices are structured into parameter-dependent families which we can use to define density functions. For instance, \(⟨1⟩\) generalizes to the
  \(R\)-parameterized family \(⟨-⟩^{n}\) of \(n\)-cubic lattices of \(R\). So if \(R\) is also complete and totally ordered, i.e. \(ℝ\), then we can define a density function
  \(δ{:} \mathbf{P}(ℝ^{n}) → ℝ_{\geq 0, ∞}\) with \(δ(ι)\) given by

  \[\mathbf{inf} \{r ∈ ℝ | ι \text{ is \(⟨r⟩^{n}\)-dense}\}\]

  if any such \(r\) exists, and \(∞\) otherwise. This is monotone and thus at least a morphism of preorders, but according to this definition, no \(ι\) could be dense which did not contain \(0\), since \(0\) is
  invariant under scaling. We can remedy this defect by introducing a slight variant of our definition of density: we say that \emph{\(ι\) is \(λ\)-dense up to
    translation} if there exists some translation, i.e. affine isomorphism \(θ\) such that \(ι\) is \((λ⨾ θ)\)-dense. So using lattices we do actually get a decently
  well-behaved density-function. We could also replace the points in \(λ\) with \(n\)-dimensional disks, cube or other shapes of some diameter, which would give us two density
  parameters that are, essentially, pixel size and pixel density — the denser the pixels, the fewer subobjects will look dense compared to them, and the larger the pixels, the
  more subobjects will look dense compared to them, since this increases the area that they can intersect with.

  Now, \(ι\) is \(λ\)-sated if and only if for each subset \(κ\) of \(R^{n}\) such that \(ι\) contains \(κ \cap λ\) it holds that \(ι\) also contains \(κ\). This is too general,
  however, to be of much interest: \(ι\) would in particular have to contain \((ι \cap λ) \cup \left(R^{n} ∖ (λ ∖ ι)\right)\). But if we restrict \(κ\) to a second set of tablets
  \(τ\) as we described in \cref{rmα_:28}, then we can exercise much finer control over satedness. In particular, let \(τ\) be given by the embedding of all convex subsets of
  \(R^{n}\). Then \(ι\) is sated with regard to \(λ\) and \(τ\) if and only if for each family \(F\) of points in \(λ\) it contains the convex closure \(\overline{F}\) of \(F\), so if it is convex relative to
  \(λ\): only on its intersection with\(λ\) need it be convex.

  But if we have a set of generators \(B\) of \(Λ\), then we can restrict the tablets still more and take \(τ'\) to be given by the embedding of all convex subsets formed by points \(q_{i}\)
  that are each adjacent to a point \(p\) in \(λ\), where \(p\) and \(q\) are adjacent if and only if there exists a \(b ∈ B\) such that \(bp = q\) (so if they are connected by an
  edge in their Ceyley graph). So for \(λ\) being the embedding of \(ℤ^{n}\) into \(ℝ^{n}\), \(τ'\) is given by all subsimplices with integral vertices that have a tip \(p\) and
  \(p\)-adjacent points \(q_{i}\), so that the triangle for each triple \(p, q_{i}, q_{j}\) with \(i ≠ j\) is rectangular in \(p\). So basically, \(ι\) is \((τ, τ')\)-dense if and only if
  it contains the triangulation of its intersection with \(λ\). For instance this shape is \((λ, τ')\)-sated with \(λ\) the unit cubic lattice and \(τ'\) the minimal rectangular subsimplices of \(λ\):

  \begin{figure}[H]
    \centering
    \begin{tikzpicture}
      \fill (0, 0) -- (0, 1) -- (1, 1) -- (1, 0); \fill (0, 0) -- (0, -1) -- (-1, -1) -- (-1, 0);
    \end{tikzpicture}
    \caption{\label{fi♭_:23}}
  \end{figure}

  where all squares are unit squares whose vertices are on \(λ\), and it would remain so whatever we add around this ``core'' part as long as it does not intersect with any other
  vertices in \(λ\). But this part of the unit square is not:

  
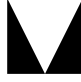
\begin{figure}[H]
    \centering
    \begin{tikzpicture}
      \fill[black] (0, 0) -- (0, 1) -- (0.5, 0) -- cycle; \fill[black] (1, 0) -- (1, 1) -- (0.5, 0) -- cycle;
    \end{tikzpicture}
    % \caption{\label{fig:9}}
  \end{figure}

  but it is \(\frac{λ}{4}\)-sated, i.e. sated if we quarter the grid size:

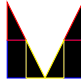
\begin{figure}[H]
  \centering
  \begin{tikzpicture}
    \fill[black] (0, 0) -- (0, 1) -- (0.5, 0) -- cycle; \fill[black] (1, 0) -- (1, 1) -- (0.5, 0) -- cycle; \draw[blue] (0, 0) rectangle (0.25, 0.5) -- cycle; \draw[red] (0, 0.5) -- (0, 1)
    -- (0.25, 0.5) -- cycle; \draw[yellow] (0.25, 0) -- (0.25, 0.5) -- (0.5, 0) -- cycle; \draw[blue] (0.75, 0) rectangle (1, 0.5) -- cycle; \draw[red] (1, 0.5) -- (1, 1) -- (0.75, 0.5) --
    cycle; \draw[yellow] (0.75, 0) -- (0.75, 0.5) -- (0.5, 0) -- cycle;
  \end{tikzpicture}
  % \caption{\label{fig:9}}
\end{figure}

 By using a different discrete lattice, in particular a hexagonal one, we can
change the shape of these simplices, in particular to equilateral ones. So satedness says that a subobject \(ι\) actually contains its approximation
\(\overline{ι \cap λ}\), obtained by taking the convex closure of each of its minimal \(λ\)-sub-simplices, i.e. its ``support''
\(ι \cap λ\). In particular the minimal sub-simplices of \(ι \cap λ\) form a subdivision of \(\overline{ι \cap λ}\).\footnote{We will define subdivisions formally in future work, here we
  just mean a decomposition of \(ι\) into topological closed a irreducible subsets whose intersection factors through through their boundary.} 

So an element is sated if and only if it contains the simplices spanned by the vertices of its lattice squares. If each point of \(λ\) stands for a sensor, and an image is created by
flipping the color of a pixel at \((i, j)\) if and only if \((i, j)\) is in \(ι\), then the resulting image essentially corresponds to \(\overline{ι \cap λ}\), so \(ι\) being sated
states that it the image we can form of it through \(λ\) corresponds to \(ι\) at least on the parts of it we can see. Note also that for
\emph{any} discrete lattice \(λ{:} Λ → ℝ^{n}\), i.e. a lattice in \(ℝ^{n}\) consisting of nowhere dense vertices, any union of the subsimplices spanned by its vertices is locally
convex and piecewise linear and thus locally contractible. So if \(ι\) is \(λ\)-sated for any discrete \(λ\), \(ι\) is locally contractible up to some stuff outside its pixellated
support \(ι \cap λ\).

If we now add a scale \(σ{:} \mathbf{P}(ℝ^{n}) → S\) we can relax this condition up to some kind of neglibility. Note that, since we do not assume general additivity, we do not
need to restrict to the measurable subsets of \(ℝ^{n}\) but can instead take \(σ(S)\) as the volume measure of the closure of \(S\), which is Borel measurable since it is the
closure of a subset. Since each sub-simplex of \(ℝ^{n}\) is in particular closed, this reduces to the volume measure \(υ\) on the tablets we use to measure satedness, i.e.
minimal sub-simplices of \(λ\). Then \(ι\) is \((λ, σ)\)-sated if and only if each minimal sub-simplex on the vertices in \(ι \cap λ\) that \(ι\) does not fully contain is of volume
\(0\), i.e. lower-dimensional. But we can also take as \(σ\) the postcomposition \(υ_{h}\) of \(υ\) (on closures) with some rounding scale \(ρ{:} ℝ_{\geq 0, ∞} → ℕ_{∞}\) as in \cref{xmp_:12}
that maps everything under a certain threshold \(h\) to \(0\). Then \(ι\) is \((λ, υ_{h})\)-sated if and only if the sub-simplices on its \(λ\)-vertices that it does not contain
have volume \(< h\). We can also keep this \(σ\) fixed but allow the tablets to vary, and so say that \(ι{:} A \hookrightarrow ℝ^{N}\) is \emph{\(h\)-lattice-sated} if there exists some discrete
sub-lattice \(λ\) of \(ℝ^{n}\) such that \(ι\) is \((λ, υ_{h})\)-sated, so if the volume of the sub-simplices on \(λ\)-vertices in \(ι\) that are not themselves in \(ι\) is
\(< h\). Again it makes sense to relativize this to translations of discrete lattices. We can call this \emph{equivariantly \(h\)-lattice-satedness}, where the equivariance is
understood as translation invariance.
\end{Example}
\begin{Example}
  \label{xmp_:23}
\index{Minkowski's theorem!via tablets}
On the other hand, we can we take as the tablets the embedding \(τ{:} \{λ, \{0\}\} \hookrightarrow \mathbf{P}\supseteq \{0\}(ℝ^{n})\) where \(λ{:} Λ \hookrightarrow ℝ^{n}\) is a lattice embedding (i.e. an embedding of an
additive subgroup), \(\mathbf{P}_{\supseteq \{0\}}(ℝ^{n})\) is the set of subsets of \(ℝ^{n}\) that contain \(0\). Then \(ι{:} A \hookrightarrow ℝ^{n}\) is \(τ\)-dense if and only if it contains a point in \(λ\) apart
  from \(0\). This kind of condition characterizes \emph{Minkowski's Theorem}, which states that each radially symmetric convex set around \(0\) in \(ℝ^{n}\) whose volume is larger than \(2^{n}\)
  times the determinant of \(λ\) is \(τ\)-dense.

  More generally, if we take as the tablets \(τ{:} T \hookrightarrow P\) for a bounded preorder only one tablet \(λ\) and a bottom tablet in the preimage of \(\bot_{P}\), then any element of \(P\)
  whose intersection with \(τ(λ)\) is not \(\bot_{P}\) is \(τ\)-dense.
\end{Example}
   
% Let us also quickly revisit our convexity example:

% \begin{Example}
%   \label{xmp_:3}
%   Let \(τ{:} \mathbf{SOb}_{gcc}(\mathcal{B}) \hookrightarrow \mathbf{SOb}_{\mathopen{∣}ff\mathclose{∣}}(\mathcal{B})\) be the embedding of geodesically convex sub-2-categories into the preorder of all on all levels
%   full sub-2-categories of \(\mathcal{B}\), ordered by inclusion. We can identify both the elements of \(\mathbf{SOb}_{gcc}(\mathcal{B})\) and those of
%   \(\mathbf{SOb}_{\mathopen{∣}ff\mathclose{∣}(\mathcal{B})}\) with families of objects of \(\mathcal{B}\) but in different ways: for two objects \(X, Y\) of an element \(ι\) of
%   \(\mathbf{SOb}_{gcc}(\mathcal{B})\), the sub-hom-category \(\mathbf{hom}_{ι}(X, Y)\) of \(ι\) is only equivalent to the full sub-2-category \(\hat{ι}\) on the same objects if and only if
%   no morphisms between objects in \(ι\) go through any objects outside of \(ι\). So \(κ ∈ \mathbf{SOb}_{\mathopen{∣}ff\mathclose{∣}}\)

% a sub-2-category of \(\mathcal{B}\) is \(τ\)-small if and only if the only convexly closed sub-2-category
% \end{Example}

\section{The Structure of \(\mathbf{UBP}\)}
\label{sec:smalln-under-morph}

So far we have collected a stock of examples, but to understand how they are connected we need to consider how smallness interacts with morphisms. For this we first need a
deeper understanding of our surrounding 2-category. Let us start by formally introducing it with its relatives:

\begin{Definition}
  \label{dfn:3}\index{\(\mathbf{PrO}\)}\index{\(\mathbf{UBP}\)}\index{\(\mathbf{UBP}_{w}\)}
  The \emph{2-category of preorders} \(\mathbf{PrO}\) is the 2-category whose objects are preorders, whose morphisms are preorder morphisms and whose 2-morphisms are objectwise inequalities
  between morphisms.

  The 2-category \(\mathbf{UBP}\) of \emph{upward-bounded preorders and preorder morphisms} is the full sub-2-category of \(\mathbf{PrO}\) on upward-bounded preorders.

  The
  wide sub-2-category \(\mathbf{UBP}_{w}\) of \emph{upward-bounded preorders and top-preserving preorder morphisms} is the wide sub-2-category of \(\mathbf{UBP}\) whose objects are upward-bounded preorders and whose morphisms are
  top-preserving preorder morphisms.
\end{Definition}

It might seem weird not to demand preservation of tops in \(\mathbf{UBP}\), but note that the local property of having a maximal element assembles into a global 2-categorical one even
without strict preservation. To see this, let us start by studying the
natural transformations from the constant functor on the terminal object \(\mathbf{K}(*){:} \mathbf{Aut}(\mathbf{PrO})\) to the identity functor \(\mathbf{id}{:} \mathbf{Aut}(\mathbf{PrO})\). A
natural transformation \(η{:} \mathbf{K}(*) \Rightarrow \mathbf{Aut}(\mathcal{C})\) on some category \(\mathcal{C}\) is a family of points \(η_{X}{:} * → X\) that are preserved by all morphisms in
\(\mathcal{C}\), something that any pointed category like that of pointed sets and point-preserving maps or that of monoids and identity-preserving maps has. \(\mathbf{PrO}\) has no such distinguished
points, so unsurprisingly \(\mathbf{Aut}(\mathbf{PrO})(\mathbf{K}(*), \mathbf{id})\) is empty: for each family of points in all preorders we can easily find a map that avoids them.

If we replace \(\mathbf{PrO}\) with \(\mathbf{UBP}\), the same holds, with one almost-exception: the family \(\top_{P}{:} * \hookrightarrow P\) is not preserved by all morphisms, but for each \(f{:} P → Q\)
the image \(\top_{f} := f(\top_{P}) = \top_{P}⨾ f\) is the source of a unique 1-cell \(\top_{f} \leq \top_{Q}\) since \(\top_{Q}\) is maximal. Moreover, if \(η{:} \mathbf{K}(*) → \mathbf{id}\) is some other
\(2\)-cell, and there exists some \(η_{P} < \top_{P}\), then \(η_{P}\) is not laxly preserved by the constant automorphism of \(P\) on \(\top_{P}\). Thus, the family \((\top_{P}, \top_{f})\) is the unique
lax point of \(\mathbf{id}\). Moreover, since \(\top_{P}\) is maximal, \(\top_{(-)}\) is a right adjoint to the natural transformation \(*_{(-)}\) that represents \(*\) as a terminal object. Note that
an abstract way to state that a preorder is upward-bounded is to say that \(*_{P}\) has a right adjoint. So \(\mathbf{id}\) is itself upward-bounded in a non-trivial way even though it only has one
element.

Practically speaking, this means in particular that the image of a morphism \(f\) has a unique maximal element \(\top_{f}\), which is natural with respect to composition in that for composable
\(f,g\) it holds that \(\top_{f⨾ g} \leq \top_{g}\). This allows us to represent morphisms \(f\) by elements \(\top_{f}\) to some extent, which we will use a lot in \cref{sec:covar-repr-theor}.

\begin{Remark}
  \label{rmk:17}
  \index{ideal embedding!is principal}
  In particular, each ideal embedding \(ι{:} a \hookrightarrow P\) in \(\mathbf{UBP}\) is principal since  \(\top_{ι}\) represents \(ι\) as an ideal.
\end{Remark}
\begin{Remark}
  \label{rmα_:35}
  So visually, \(\mathbf{UBP}\) looks a bit like a saw: \(\top_{P}\) is the tip of a tooth, \(f{:} P → Q\) a tooth-diagonal downward and \(\top_{f} \leq \top_{Q}\) the blade going back up to the tip.
\end{Remark}

Note that \(\mathbf{UBP}\), unlike \(\mathbf{PrO}\), does not have an initial object since the empty preorder is not upward-bounded by \cref{defn_:29}. But since \(\top_{P}\) is for each preorder
\(P\) maximal, \(*\) satisfies a laxening of the universal property of an initial object. Let us give this a name:

\begin{Definition}
  \label{defn_:9}
\index{contravariantly terminal}
\index{contravariantly initial}
  An object \(O\) of a 2-category \(\mathcal{B}\) is \emph{contravariantly initial} if for each object \(X\) of \(\mathcal{B}\) the category \(\mathcal{B}(O, X)\) has a terminal object. Dually, an object
  \(T{:} \mathcal{B}\) is \emph{contravariantly terminal} if for each object \(X\) of \(\mathcal{B}\) the category \(\mathcal{B}(X, T)\) has an initial object.
\end{Definition}
\begin{Remark}
  \label{rmα_:21}
  This definition has a global component and a local one, and in the name we track the variance between the two.
\end{Remark}

So \(*\) is contravariantly initial in \(\mathbf{UBP}\). This makes sense if we view \(\mathbf{UBP}\) in analogy with pointed categories, where the point becomes the initial object.
Here, we only have lax point preservation, so the initiality is lax as well.

 Note that contravariantly terminal objects are not generally unique in 2-categories:

\begin{Counterexample}
  \label{cxmp:2}
\index{contravariantly terminal object!non-uniqueness}
  Let \(\mathcal{B}\) be the 2-category on two elements \(a, b\), generated by two morphisms \(f{:} a → b, g{:} b → a\) and two 2-morphisms
  \(α{:} \mathbf{id}_{a} \Rightarrow f⨾ g, β{:} \mathbf{id}_{b} \Rightarrow g⨾ f\). Then each of \(a\) and \(b\) are contravariantly terminal, but they are not equivalent. Moreover, if we introduce an inverse to
  \(β\), then \(a\) is (weakly) initial but \(b\) is still contravariantly initial.
\end{Counterexample}

So there can be more than one contravariantly initial object in a 2-category and in particular there can be an initial contravariantly initial object (which then has to be an actual initial
object) and a terminal initial object, a ``largest of the first'' so to speak.

\begin{Example}
  \label{xmp_:4}
  Actually, every object \(P\) of \(\mathbf{UBP}\) is contravariantly initial, with the terminal object in \(\mathbf{UBP}(P, Q)\) given by \(\top_{Q}\) for each \(Q{:} \mathbf{UBP}\), but \(*\)
  is uniquely described by being terminally contravariantly initial, since it is terminal. Moreover, \(*\) is the unique contravariantly
  initial object in the monic part \(\overset{\hookrightarrow}{\mathbf{UBP}}\) of \(\mathbf{UBP}\).
\end{Example}

In \cref{sec:univ-char-mathbf} we will show that the monic part of \(\overset{\hookrightarrow}{\mathbf{UBP}}\) with \(\mathbf{SOb}\) is initial among the contravariantly terminal coefficient objects for a certain
kind of coefficient system.

We can also isolate two kinds of morphisms \(f{:} P → Q\) in \(\mathbf{UBP}\) that are extremal with respect to \(\top_{f}\):

\begin{enumerate}
\item \label{item:142} the \emph{top-preserving morphisms}, where \(\top_{f} = \top_{Q}\),

\item and the \emph{ideal embeddings}, which are principal by \cref{rmk:17} and thus determined by \(\top_{f}\).
\end{enumerate}

These two actually form a factorization system:

\begin{Proposition}[Ideal-top factorization system on \(\mathbf{UBP}\)]
  \label{■pn_:ofs}
  \index{orthogonal factorization system!on \(\mathbf{UBP}\)}
  \index{ideal-top factorization system@ideal-top factorization system}
  \index{top-preserving map!in factorization system}
  \index{ideal embedding!in factorization system}
  In the \(2\)-category \(\mathbf{UBP}\), the classes
  \[
    \mathcal{E} \;=\; \{\text{top-preserving maps}\}, \qquad \mathcal{M} \;=\; \{\text{ideal embeddings}\}
  \]
  constitute an orthogonal factorization system. Explicitly, every morphism \(f {:} P → Q\) factors uniquely up to unique isomorphism as
  \[
    P \;\xrightarrow{\;f'\;}\; (\top_{f}) \;\xhookrightarrow{\;\iota_{\top_{f}}\;}\; Q,
  \]
  where \((\top_{f}) = \{q ∈ Q \mid q \leq \top_{f}\}\) is the principal ideal of \(\top_{f}\) in \(Q\), the map \(f'\) is the corestriction \(f'(p) = f(p)\), and \(\iota_{\top_{f}}\) is the inclusion.
\end{Proposition}
\begin{proof}
  We verify the defining properties of an orthogonal factorization system: existence of factorizations and the diagonal fill-in.

  \emph{Factorization.}
  Given \(f {:} P → Q\), set \(\top_{f} := f(\top_{P})\). Since \(f\) is monotone, \(f(p) \leq f(\top_{P}) = \top_{f}\) for every \(p ∈ P\), so \(f\) corestricts to \(f' {:} P → (\top_{f})\) with \(f'(p) = f(p)\). This map is top-preserving:
  \[
    f'(\top_{P}) \;=\; f(\top_{P}) \;=\; \top_{f} \;=\; \top_{(\top_{f})}.
  \]
  The inclusion \(\iota_{\top_{f}} {:} (\top_{f}) \hookrightarrow Q\) is an ideal embedding whose image is the principal ideal \((\top_{f})\), and \(f = f' ⨾ \iota_{\top_{f}}\).

  \emph{Diagonal fill-in.}
  Suppose given a strictly commuting square
  \[
    \begin{tikzcd}[column sep=large]
      P \ar[r, "h"] \ar[d, "e"'] & R \ar[d, "m"] \\
      Q \ar[r, "k"'] & S
    \end{tikzcd}
    \qquad e ⨾ k \;=\; h ⨾ m,
  \]
  where \(e {:} P → Q\) is top-preserving and \(m {:} R \hookrightarrow S\) is an ideal embedding with image \((\top_{m})\). We construct a unique diagonal \(d {:} Q → R\).

  For each \(q ∈ Q\), monotonicity of \(k\) gives \(k(q) \leq k(\top_{Q})\). Since \(e\) is top-preserving, \(\top_{Q} = e(\top_{P})\), whence
  \[
    k(\top_{Q}) \;=\; k(e(\top_{P})) \;=\; m(h(\top_{P})) \;\leq\; m(\top_{R}) \;=\; \top_{m},
  \]
  so \(k(q) \leq \top_{m}\), i.e.\ \(k(q) ∈ (\top_{m}) = \mathbf{Im}(m)\). Because \(m\) is an order-embedding, it admits an order-theoretic inverse \(m^{-1} {:} (\top_{m}) \xrightarrow{\;\cong\;} R\). We may therefore define
  \[
    d(q) \;:=\; m^{-1}(k(q)),
  \]
  which is monotone as the composite \(Q \xrightarrow{k} (\top_{m}) \xrightarrow{m^{-1}} R\). The two triangles commute:
  \begin{itemize}
  \item \(d ⨾ m = k\) since for all \(q ∈ Q\), \(m(d(q)) = m(m^{-1}(k(q))) = k(q)\).
  \item \(e ⨾ d = h\) since for all \(p ∈ P\), \(d(e(p)) = m^{-1}(k(e(p))) = m^{-1}(m(h(p))) = h(p)\).
  \end{itemize}
  Uniqueness is immediate: if \(d' {:} Q → R\) also satisfies \(d' ⨾ m = k\), then \(m(d'(q)) = m(d(q))\) for all \(q\), and injectivity of \(m\) forces \(d' = d\).

  \emph{Intersection.}
  If \(f {:} P → Q\) is both top-preserving and an ideal embedding, then \(f\) is an order-embedding with image \((\top_{f}) = (\top_{Q}) = Q\), hence an order-isomorphism. Thus
  \(\mathcal{E} \cap \mathcal{M}\) is the core groupoid \(\mathbf{Core}(\mathbf{UBP})\) of isomorphisms in \(\mathbf{UBP}\).
\end{proof}

\begin{Remark}
  \label{rmα_:ofs}
  The diagonal fill-in admits a natural oplax \(2\)-categorical strengthening: if the square commutes only up to a \(2\)-cell \(e ⨾ k \leq h ⨾ m\), the same construction yields \(d\) with \(d ⨾ m = k\) strictly and a \(2\)-cell \(e ⨾ d \leq h\). This oplax lifting is unique, since \(\mathbf{UBP}\) is preorder-enriched and there is at most one \(2\)-cell between any pair of parallel morphisms.
\end{Remark}

Both of these will be important later on.

\begin{Remark}
  \label{rmα_:22}
  Building on \cref{rmα_:35} we can understand the ``saw-toothedness'' of \(\mathbf{UBP}\) as a correspondence between horizontal and vertical distance for ideals in that for each \(a, b ∈ P\) the
  interval between \(a\) and \(b\) is in one-to-one correspondence with the principal ideals that \((a)\) factors through that factor through \((b)\).
\end{Remark}
\begin{Remark}
  \label{rmα_:34}
  Of course dual statements hold for the category \(\mathbf{DBP}\) of downward-bounded preorders, with filters instead of ideals.
\end{Remark}

\subsection{Smallness \& Density under Morphisms}
\label{sec:structure-mathbfubp}

Now to preservation and reflection:

\begin{Definition}
  \label{defn:elemwise-small-pres-reflect}\index{smallness-preserving morphism}\index{smallness-reflecting morphism}
  Let \(f {:} P → Q\) be a morphism in \(\mathbf{UBP}\):
  \begin{enumerate}
  \item \(f\) \emph{preserves smallness at \(a\)} if \(a\) small in \(P\) implies \(f(a)\) small in \(Q\).

  \item \(f\) \emph{reflects smallness at \(a\)} if \(f(a)\) small in \(Q\) implies \(a\) small in \(P\).
  \end{enumerate}
\end{Definition}

Now, the proofs of this section are very simple in principle. Consider for instance smallness-preservation. For this we need some way of relating joins of \(f(a)\) in \(Q\) to those of \(a\) in \(P\):

\begin{Definition}[Lifting of \(a\)-joins]
  \label{def:a-cover-lifting}
\index{top-join-lifting}
\index{join-lifting}
  Let \(f{:} P → Q\) in \(\mathbf{UBP}\) and \(a\in P\).
  We say that \(f\) is \emph{\(\top\)-join-lifting on \(a\)} if for every \(c ∈ Q\) with
  \[
    f(a)∨ c = \top_{Q}
  \]
  there exists a \(b ∈ P\) such that \(f(b) \leq c\) and \(a ∨ b = \top_P\).
\end{Definition}

But then the proof is simple:

\begin{Proposition}[Elementwise preservation of smallness]
  \label{prp:8}
  Given \(f {:} P → Q\) in \(\mathbf{UBP}\) and \(a ∈ P\), assume that \(f\) preserves joins of \(a\) to \(\top_{P}\) and lifts \(\top\)-joins on \(a\). Then \(f\) preserves smallness at \(a\).
\end{Proposition}
\begin{proof}
  Let \(c ∈ Q\) satisfy \(f(a)∨c = \top_{Q}\). By the \(a\)-join-lifting property there exists a \(b ∈ P\) with \(f(b)\le c\) and \(a∨b=\top_P\). If \(a\) is small, then \(b = \top_{P}\), so that
  \(f(a) \leq f(a) ∨ f(b) = f(a) ∨ \top_{f} = \top_{f} \leq c\). So \(f(a) ∨ c = c = \top_{Q}\).
\end{proof}

Let us also treat the reflecting case: assuming we have \(f{:} P → S\) and \(a ∈ P\) with small \(f(a)\). Why could \(a\) \emph{not} be small in \(P\)? Three reasons come to mind:

\begin{enumerate}
\item\label{item:40} \(\top_{f}\) is not reduced in \(S\).
\item\label{item:3} there is a \(b ∈ P\) with \(a ∨ b\) such that \(f(a) ∨ f(b)\) either doesn't exist or is smaller than \(\top_{f}\).
\item\label{item:36} each \(f(b)\) does join with \(f(a)\) to \(\top_{f}\) and thus \(f(b) = \top_{f}\) \emph{but} \(b\) is something smaller than \(\top_{P}\) that is still mapped to \(\top_{f}\).
\end{enumerate}
\begin{Example}
  \label{xmp:105}
  Let \(f^{*}{:} \mathbf{C}([0, 1]) → \mathbf{C}([0, 2])\) from \cref{xmp:7} be induced through pullback along the continuous function \(f{:} [0, 2] \twoheadrightarrow [0, 1]\) that is the identity on
  \([0, 1] \subseteq [0, 2]\) and maps \([1, 2]\) to \(1\). Then \(f^{*}\) is not smallness-preserving because the fiber over \(1\) is not small:
  \([0, 1] ∨ f^{*}(\{1\}) = [0, 1] ∨ [1, 2] = [0, 2]\), so \(f^{*}(\{1\})\) is not small even though \(\{1\}\) is.
\end{Example}
\begin{Proposition}
  \label{■pn_:2}
  Let \(f{:} P → Q\) be a top-join preserving morphism of preorders that is injective on \(\top_{f}\). If \(\top_{f}\) is reduced then \(f\) is smallness-reflecting.
\end{Proposition}
\begin{proof}
  \label{prf:3}
  Let \(a \in P\) be an element such that \(f(a)\) is small in \(Q\).
  Suppose \(b \in P\) is such that \(a ∨ b = \top_{P}\).
  Since \(f\) preserves top-joins, we have:
  \[
    f(a) ∨ f(b) \;=\; f(\top_{P}) \;=\; \top_{f}.
  \]
  By the hypothesis that \(f(a)\) is small in \(Q\) and \(\top_{f}\) is reduced in \(Q\), it follows from \cref{defn_:7} that:
  \[
    f(b) \;=\; \top_{f}.
  \]
  Using the assumption that \(f\) is injective on \(\top_{f}\), this implies \(b = \top_{P}\).
  Thus, the only element \(b\) in \(P\) that joins with \(a\) to give \(\top_{P}\) is \(\top_{P}\) itself.
  Therefore, \(a\) is small in \(P\).
\end{proof}

\begin{Corollary}
  \label{crl:11}
  Let \(ι\) be an ideal embedding. If \(\top_{ι}\) is reduced then \(ι\) reflects smallness.
\end{Corollary}

What complicates general proofs a little is that the use of tablets and scales decomposes the functions of smallness-measurement in a way that has to be kept track of. In principle this
is more fine-grained and will allow the reader to apply the resulting calculus in the topological and measure-theoretic contexts we discuss in \cref{sec:applications}. But we
need to extend our considerations to a commutative diagram of upward-bounded preorders and monotone maps:

\begin{figure}[H]
  \centering
  \begin{tikzcd}
  T \ar[r,"τ"] \ar[d,"h"] & P \ar[r,"σ"] \ar[d,"g"] & S \ar[d,"k"] \\
  T'  \ar[r,"τ'"]            & P'  \ar[r,"σ'"]            & S'
\end{tikzcd}
  \caption{\label{fi♭_:9}}
\end{figure}

\begin{Definition}[Relative preservation of top-joins]
  \label{def:rel-preserve-topjoin}
  With \(f = (h,g,k)\) as in \cref{fi♭_:9} and \(a ∈ P\), we say that \((h,g,k)\) \emph{preserves \((τ,a)\)-top-joins} or \emph{preserves \(τ\)-top-joins on \(a\)} if for every
  \(b ∈ T\) such that \(a ∨ τ(b) = \top_{P}\) one has
  \[
    g\!\bigl(a ∨ τ(b)\bigr) \;=\; g(a) ∨ g(τ(b)) = g(a) ∨ τ'(h(b)) \;=\; \top_{g}.
  \]

  On the other hand we say that \(f\) \emph{reflects \(τ'\)-top-joins on \(a\)} if for each \(b ∈ T\) with \(g(a) ∨ (τ⨾ g)(b) = \top_{g}\) it holds that \(a ∨ τ(b) = \top_{P}\).
\end{Definition}

\begin{Proposition}[Elementwise reflection of relative smallness]
  \label{prop:rel-reflect-small}
  Assume:
  \begin{enumerate}
  \item\label{item:71}  \((h,g,k)\) preserves \(τ\)-top-joins on \(a\);
  \item\label{item:72}  the top \(\top_{g}\) is reduced under \(τ'\) over \(σ'\),
  \item\label{item:66}  \(k\) reflects \(\top_{σ⨾ k}\),
  \end{enumerate}
  Under these hypotheses \((h,g,k)\) reflects relative smallness at \(a{:} P\).
\end{Proposition}
\begin{proof}
  Suppose \(a' := g(a)\) is \((τ' \mid σ')\)-small. To show \(a\) is \((τ \mid σ)\)-small, take any \(b ∈ T\) such that
  \[
    a ∨ τ(b) = \top_{P}.
  \]
  Let now \(b'\) be given by \(h(b)\). Applying \(g\) and using preservation of \(τ\)-top-joins on \(a\)-joining families:

  \begin{equation}
    \label{qn_:1}
    \top_{g}
    \;=\;
    g(a ∨ τ(b))
    \;=\;
    g(a) ∨ (τ⨾ g)(b)
    \;=\;
    a' ∨ τ'(b').
  \end{equation}

  So \(b'\) joins with \(a'\) to contain \(\top_{g}\) under \(τ'\) in \(P'\), and since \(a'\) is \((τ' \mid σ')\)-small in \(P'\) and \(\top_{g}\) is \((τ' \mid σ')\)-reduced, it follows from
  \cref{defn_:21} that
\begin{equation}
    \label{qn_:3}
    (τ⨾ σ⨾ k)(b) = (τ'⨾ σ')(b') \geq σ'(\top_{g}) = (g⨾ σ')(\top_{P}) = (σ⨾ k)(\top_{P}) = \top_{σ⨾ k},
  \end{equation}

  and since \(τ(b) \leq \top_{P}\) and thus \(τ'(b') \leq g(\top_{P})\), the reverse inequality also holds and thus \((τ'⨾ σ')(b') = \top_{σ⨾ k}\).

  But since \(k\) reflects \(\top_{σ⨾ k}\) and \((τ⨾ σ⨾ k) (b) = \top_{σ⨾ k}\), it follows that \((τ⨾ σ)(b) = \top_{σ}\). Since this holds for every \(b ∈ T\), \(a\) is small.
\end{proof}
\begin{Remark}
  \label{rmα_:14}
  It seems like not much (or nothing) can be gained from laxifying the squares of \cref{fi♭_:9}: for instance to apply \((τ' \mid σ')\)-reducedness in the proof of \cref{prop:rel-reflect-small}
  we would have to assume that the left square is lax so that the last equality in \cref{qn_:1} became \(g(a) ∨ (τ⨾ g)(b) \leq a' ∨ τ'(b')\). But then the first equality of \cref{qn_:3} would
  become an inequality in the wrong direction \((τ⨾ σ⨾ k)(b) \leq (τ'⨾ σ')(b')\), so that the proof would fall apart.
\end{Remark}

Let us now turn to preservation of smallness at \(a\):

\begin{Definition}[Relative lifting of top-joins]
  \label{defn_:8}
  With \((h,g,k)\) as in \cref{fi♭_:9} and \(a {:} P\), \(g\) and \(h\) are \emph{\(\top\)-join-lifting on \(a\)} if for each \(b' ∈ T'\) such that \(g(a) ∨ τ'(b') = \top_{P'}\) there exists a
  \(b ∈ T\) such that \(h(b) \leq b'\) and \(a ∨ τ(b) = \top_{P}\).
\end{Definition}
\begin{Proposition}[Elementwise preservation of relative smallness]
  \label{prop:rel-preserve-small}
  Given a diagram of preorders as in \cref{fi♭_:9} and  \(a\) in \(P\). Suppose:
  \begin{enumerate}
  \item \((h,g,k)\) preserves \(τ\)-top-joins on \(a\)-joining families;
  \item \(g\) and \(h\) are \(\top\)-join-lifting on \(a\);
  \item\label{item:77} \(k\) maps \(\top_{σ}\) to \(\top_{σ'}\).
  \end{enumerate}
  Then \((h,g,k)\) preserves relative smallness at \(a\).
\end{Proposition}
\begin{proof}
  Assume \(a\) is \((τ \mid σ)\)-small.  Let \(b'\) be such that
  \[
    g(a) ∨ τ'(b') = \top_{P'}.
  \]
  Using the \((τ,a)\)-join-lifting property, there exists a \(b ∈ T\) with
  \(h(b) \leq b'\) and
  \[
    a ∨ τ(b) = \top_{P}.
  \]
  By smallness of \(a\) in the unprimed context we have
  \[
    (τ⨾σ)(b) = \top_{σ}.
  \]
  Since \(k\) maps \(\top_{σ}\) to \(\top_{σ'}\) and \(h(b) \leq b'\), \((τ'⨾ σ')(b') \geq (τ⨾ σ⨾ k)(b) = \top_{σ'}\), hence \((τ'⨾ σ')(b') = \top_{σ'}\), so \(g(a)\) is \((τ' \mid σ')\)-small, hence \((h,g,k)\) preserves relative smallness at \(a\).
\end{proof}

Now we universalize over \(a\):

\begin{Definition}\label{def:global-rel-pres-reflect}
  In the above diagram, we say that the triple \((h,g,k)\):
  \begin{itemize}
  \item \emph{reflects \((τ'\mid σ')\)-smallness} if it reflects \((τ'\mid σ')\)-smallness at every \(a{:}P\);
  \item \emph{preserves \((τ\mid σ)\)-smallness} if it preserves \((τ\mid σ)\)-smallness at every \(a{:}P\).
  \end{itemize}
\end{Definition}

Now we can describe the global behavior of smallness with respect to morphisms in one proposition:

\begin{Proposition}[Global relative reflection/preservation of smallness]
  \label{prop:global-rel-small}\index{global reflection/preservation of smallness}
  In the situation of \cref{fi♭_:9}:

  \begin{enumerate}
  \item\label{item:global-reflect}
    Assume that:
    \begin{enumerate}
    \item\label{item:67}  \((h,g,k)\) preserves \(τ\)-top-joins,
    \item\label{item:69} \(\top_{g}\) is \((τ' \mid σ')\)-reduced in \(P'\) in the sense of \cref{defn_:21};
    \item\label{item:70} \(k\) reflects \(\top_{σ⨾ k}\).
    \end{enumerate}
    Then \((h,g,k)\) globally reflects \((τ'\mid σ')\)-smallness.

  \item\label{item:global-preserve}
    Assume that for every \(a{:}P\):
    \begin{enumerate}
    \item\label{item:12} \((h,g,k)\) preserves \(τ\)-top-joins;
    \item\label{item:11} \((h,g,k)\) are \(\top\)-join-lifting;
    \item\label{item:13} \(k\) maps \(\top_{σ}\) to \(\top_{σ'}\).
    \end{enumerate}
    Then \((h,g,k)\) globally preserves \((τ\mid σ)\)-smallness.
  \end{enumerate}
\end{Proposition}
\begin{proof}
  The preservation and reflection statements follow directly from \cref{prop:rel-reflect-small} and  \cref{prop:rel-preserve-small}
\end{proof}
\begin{Remark}\label{rmk:global-rel-small-intuition}
  Conceptually:
  \begin{itemize}
  \item The \emph{reflection} part says that if the comparison triple \((h,g,k)\) preserves joins and reflects \(\top_{σ⨾ k}\), then smallness-reflection follows from the
    reducedness of \(\top_{g}\).
  \item The \emph{preservation} part says that, if the upstairs context \((τ,σ)\) sees \(a\) as negligible and \((h,g,k)\) can systematically lift downstairs joining situations back upstairs without changing
    the measurements, then the downstairs image \(g(a)\) remains small for \((τ',σ')\).
  \end{itemize}
\end{Remark}

Let us also note that reducedness and smallness-reflection are equivalent for ideals:

\begin{Corollary}
  \label{crl:1}
  Let \(ι{:} P → Q\) be an ideal embedding. If \(ι\) reflects smallness, then \(\top_{ι}\) is reduced.
\end{Corollary}
\begin{proof}
  \label{prf:16}
  Let \(\top_{ι} \leq a ∨ b\)  in \(Q\) be such that \(ι^{*}(b)\) is small in \(P\). Since ideals are downward-closed, \(ι\) both preserves and reflects joins in its image. Thus,
  \(ι^{-1}(a) ∨ ι^{-1}(b) = \top_{ι}\). Since \(ι\) reflects smallness, \(ι^{-1}(b)\) is small in \(P\), so that \(ι^{-1}(a) = \top_{P}\). So \(ι(ι^{*}(a)) = \top_{ι}\), so \(\top_{ι} \leq a\).
  Since this holds for each \(a, b{:} Q\), \(\top_{ι}\) is reduced in \(Q\).
  \end{proof}

Let us also note the dual case:

\begin{Remark}[Dualization for density and satedness]
  \label{rmk:dual-global-rel}
  Given a commutative diagram as above but in the  category of downward-bounded preorders, the following hold:

  \begin{enumerate}
    \item (Global reflection of \((τ'\mid σ')\)-density)
      Assume that:
      \begin{enumerate}
      \item \((h,g,k)\) preserves \(τ\)-bottom-meets;
      \item \(\bot_{g}\) is sated in \(P'\) relative to \((τ' \mid σ')\);
      \item \(k\) reflects \(\bot_{σ⨾ k}\).
      \end{enumerate}
      Then \((h,g,k)\) globally reflects \((τ'\mid σ')\)-density.

    \item (Global preservation of \((τ\mid σ)\)-density)
      Assume that:
      \begin{enumerate}
      \item \((h,g,k)\) lifts \(τ\)-bottom-meets;
      \item \((h,g,k)\) preserves \(τ\)-bottom-meets;
        \item \(k\) maps \(\bot_{σ}\) to \(\bot_{σ'}\).
      \end{enumerate}
      Then \((h,g,k)\) globally preserves \((τ\mid σ)\)-density.
  \end{enumerate}
  The proofs are obtained from those of \cref{prop:global-rel-small} by systematically reversing all orders, replacing joins by meets, \(\top\) by \(\bot\), smallness by density, and reducedness by satedness.
\end{Remark}

\begin{Corollary}
  \label{coR_:9}
  Let \(ι\) be a downward-closed embedding \(ι{:} A \hookrightarrow P\) with \(P\) a downward-bounded preorder. Then \(ι\) is density-reflecting.
\end{Corollary}
\begin{proof}
  \label{prf:15}
  Since \(ι\) is downward-closed it in particular contains \(\bot{:} P\) and thus by the dual of \cref{■pn_:2} density-reflection holds if and only if \(ι\) is injective on \(\bot\). But
  this holds because \(ι\) is an embedding.
\end{proof}

But also:

\begin{Corollary}
  \label{coR_:density-preserved}
\index{density!preservation by surjections}
  Let \(f{:} P → Q\) be a monotone map between downward-bounded preorders. Assume that:

  \begin{enumerate}
  \item\label{item:121} \(f\) reflects bottoms;
  \item\label{item:122} \(f\) is surjective on elements.
  \end{enumerate}

Then if \(d\in P\) is dense, then \(f(d)\) is dense in \(Q\).
\end{Corollary}
\begin{proof}
  Let \(c\in Q\) with \(c∧f(d)=\bot_Q\). Since \(f\) is surjective, choose \(b\in P\) with \(f(b)=c\). Because \(f\) is monotone we have
  \[
    f(b∧d) \;\le\; f(b)∧f(d) \;=\; c∧f(d) \;=\; \bot_Q.
  \]
  Hence \(f(b∧d) = \bot_Q = f(\bot_P)\). Thus, since \(f\) reflects bottoms, \(b ∧ d = \bot_{P}\).
  Since \(d\) is dense in \(P\), this yields \(b=\bot_P\), hence \(c=f(b)=\bot_Q\), which is exactly the density of \(f(d)\) in \(Q\).
\end{proof}

% This also allows us to describe elements based on their size under pullbacks:

% \begin{Definition}
%   \label{dfn:16}
%   Given an element \(a ∈ P'\) and a scale \(τ{:} P' → S'\), \(a\) is \emph{vanishing} in \(P\) under \(τ\) if and only if \(a\) is pullback-small in \((b')\) for each principal ideal \(b\).
% \end{Definition}
% \begin{Example}
%   \label{xmp:100}
%   Let \(σ{:} P' → ℝ_{\geq 0, ∞}\) be a measure. Then \(a ∈ P'\) is vanishing under \(σ\) if and only if for each \(ι{:} P \hookrightarrow P'\) it holds that either \(a\) is small in the \(σ\)-pullback size on \(P\) or
%   \(\mathbf{Im}(ι⨾ σ) = 0\). This follows immediately from the fact that \(0\) is the minimum of \(ℝ_{\geq 0, ∞}\). So we don't need measure theory to use its negligible subsets.
% \end{Example}

Finally we have to describe how smallness acts with regard to the 2-categorical structure in \(\mathbf{UBP}\). This is easy though:

\begin{Proposition}
  \label{■pn_:16}
  Given \(\leq{:} f \Rightarrow g{:} P → S\) then \(∀ a ∈ P\):

  \begin{enumerate}
  \item \label{item:82} if \(g\) is smallness-preserving at \(a\), then \(f\) is smallness-preserving at \(a\),

  \item\label{item:83} if \(f\) is smallness-reflecting at \(a\), then \(g\) is smallness-reflecting at \(a\).
  \end{enumerate}
\end{Proposition}
\begin{proof}
  \label{prf_:11}
  For \cref{item:82} first assume that \(a\) is small. Since \(g\) is smallness-preserving at \(a\), \(g(a)\) is small. But since the small elements of \(S\) form an ideal and \(f(a) \leq g(a)\), so is
  \(f(a)\). If \(a\) is not small, then the statement is void.

  For \cref{item:83}, assume that \(g(a)\) is small, else the statement is void. Since \(f(a) \leq g(a)\), \(f(a)\) is small, so \(a\) is small.
\end{proof}

Let us also clarify the relation between pseudocomplements, tablets and scales, which we'll use in \cref{prop:completeness-reflection-dense-envelope}:

\begin{Definition}
  \label{dfn:1}
\index{pseudocomplement}
Given an element \(a\) of a bounded preorder \(P\), an \emph{open pseudocomplement} \(\overset{\circ}{¬} a\) of \(a\) is a largest \(b\) such that \(a ∧ b = \bot_{P}\).

More
generally, given an embedding \(τ{:} T \hookrightarrow P\), a \emph{\((τ \mid σ)\)-relative open pseudocomplement} of \(a\) is a largest \(b ∈ T\) such that \(σ(a ∧ τ(b)) = \bot_{σ}\).

Dually, \(a\) has a \emph{closed
    pseudocomplement} \(\overline{¬}a\) if there exists a smallest \(b\) such that \(a ∨ b = \top_{P}\) and similar for a \emph{\((τ \mid σ)\)-relative closed pseudocomplement}.
\end{Definition}
\begin{Example}
  \label{xmp_:14}
  For the tablets \(τ{:} \mathbf{O}(X) \hookrightarrow \mathbf{P}(X)\) given by the opens of a topological space, a \(τ\)-open pseudocomplement \(\overset{\circ}{¬}ι\) is a largest open that does not intersect with a given subspace \(ι\), so the
  interior of the set-theoretic complement of \(ι\). Dually for closeds.
\end{Example}
\begin{Remark}
  \label{rmα_:15}
  We will generally refrain from studying closed pseudocomplements in this paper, as we will do that more in an upcoming paper anyway. So when we say ``pseudocomplement'' without a qualifier then we mean open
  pseudocomplements.
\end{Remark}

\begin{Proposition}[Pseudocomplements and density]
  \label{prop:pseudocomplement-density}
\index{density!pseudocomplement characterization}
Let \(T \xrightarrow{τ} P \xrightarrow{σ} S\) be morphisms of downward-bounded preorders with bottom \(\bot_P\), and let \(a\in P\) have a \(τ\)-relative pseudocomplement
  \(¬a\) . Then the following are equivalent:
  \begin{enumerate}
  \item \(a\) is \((τ \mid σ)\)-dense in the sense of \cref{rmk:dual-sigma-mu};
  \item \((τ⨾ σ)(¬a) = \bot_σ\).
  \end{enumerate}
\end{Proposition}
\begin{proof}
  Assume \(a\) is dense. Since \(a∧ τ(¬a)=\bot_P\), the definition of \((τ \mid σ)\)-density forces \((τ⨾ σ)(¬a) = \bot_{σ}\).
  Conversely, if \((τ⨾ σ)(¬a) = \bot_{σ}\) then for any \(c \in T\) with \(a∧ τ(c) = \bot_P\) we have \(c\le ¬a\), hence \((τ⨾ σ)(c) = \bot_σ\). Thus \(a\) is \((τ \mid σ)\)-dense.
\end{proof}
\begin{Remark}
  In a Heyting algebra, where the pseudocomplement is defined by the Heyting implication \(¬a' := a'→\bot\), the condition \(¬a'=\bot\) is exactly the usual definition of a dense element.
\end{Remark}

\subsection{Pullback and Subobject Tablets and Scales}
\label{sec:pullb-tabl-scal}

Let us note that tablets and scales can be transported along pullback:

\begin{Definition}
  \label{dfn:5}
\index{pullback tablets}
  Let \(f{:} Q → P, T \xrightarrow{τ} P \xrightarrow{σ} S\) be preorder morphisms between upward-bounded preorders. We define the \emph{\(T\)-pullback tablets and scales} on \(Q\) as the sequence
  \(T_{f} \xrightarrow{τ_{f}} Q \xrightarrow{σ_{f}} S\) with \(τ_{f} := f^{*}(τ), σ_{f} := f⨾ σ\).

  We say that \emph{the pullback smallness} under \(f\) \emph{reflects \((τ \mid σ)\)-smallness} if
  \[
    \forall a∈Q,\quad f(a) \text{ is \((τ\mid σ)\)-small in \(P\)} \;\Rightarrow\; a \text{ \((τ_f \mid σ_f)\)-small in \(Q\)}
  \].

  We say that \emph{the pullback smallness} under \(f\) \emph{preserves \((τ \mid σ)\)-smallness} if
  \[
    \forall a∈Q,\quad a \text{ is \((τ_f \mid σ_f)\)-small in \(Q\)} \;\Rightarrow\; f(a) \text{ \((τ \mid σ)\)-small in \(P\)}.
  \]

  If \(f\) is an ideal embedding we call this the \emph{subobject tablets and scale}.
\end{Definition}
\begin{Remark}
  \label{rmα_:26}
  If the tablets \(τ{:} \mathbf{O}(X) \hookrightarrow \mathbf{P}(X)\) are the opens of a topological space and \(ι{:} U \hookrightarrow X\) is a subset of \(X\), then the pullback tablets \((ι_{*})^{*}(τ)\), where
  \(ι_{*}{:} \mathbf{P}(U) \hookrightarrow \mathbf{P}(X)\) is the morphism induced by postcomposition, are the opens contained in \(ι\). If \(ι\) is open in \(X\) then this is the same as the
  subspace topology on \(U\) but if \(ι\) is not closed in \(X\), then the opens of the subspace topology of \(U\) are the \emph{intersections} of opens in \(X\) with \(U\) and
  thus \emph{not} the same as the opens contained in \(ι\), so the pullback tablets of do not describe the subspace topology of non-opens. In fact, there is not even a canonical morphism
  from the opens of the subspace topology of \(U\) to the opens of \(X\) if \(ι\) is not open. Dually for \(\mathbf{C}(X)\).
\end{Remark}

So pullback-smallness describes e.g. restrictions of nowhere dense subsets to subspaces.

\begin{Proposition}
  \label{prp:7-re}
  In a situation as in \cref{dfn:5} assume that \(f\) preserves joins to \(\top_{Q}\). Then:

  \begin{enumerate}
  \item \label{item:156} if \(\top_{f}\) is reduced, then the pullback smallness under \(f\) reflects \((τ \mid σ)\)-smallness.

  \item \label{item:151} If \(f\) is a full epimorphism, \(f\) preserves \(f^{*}(τ)\)-top-joins and \(f\) reflects \(τ\)-top-joins, then the pullback smallness under \(f\) preserves \((τ \mid σ)\)-smallness.
  \end{enumerate}
\end{Proposition}
\begin{proof}
  \cref{item:156} follows from \cref{crl:11}. For \cref{item:151} assume \(f\) is a full epimorphism and preserves joins with \(\top\). To show global preservation of smallness, we use
  \cref{prop:global-rel-small}(\cref{item:global-preserve}) in our setting, with \(h\) the pullback morphism, \(g=f\), \(k=\mathbf{id}_S\). The three conditions required there reduce to:

\begin{enumerate}
  \item \(f\) preserves top-joins on \(τ\)-joining families: this is our hypothesis that \(f\) is top-join-preserving.

  \item \((h,g,k)\) lift \(\top\)-joins: for any \(f(a) ∨ τ(b) = \top_P\) there exists a \(c ∈ f^{*}(T)\) with \(τ^{*}(f) (c) = b\). Then
    by \(τ\)-top-join-preservation
    \[f(a ∨ f^{*}(τ)(c)) = f(a) ∨ f(f^{*}(τ)(c)) = f(a) ∨ τ(b) = \top_P = \top_{f}.\]
    Because \(f\) reflects \(τ\)-top-joins, we must have \(a ∨ f^{*}(τ)(c) = \top_{P}\).

  \item \(k = \mathbf{id}_S\) maps \(\top_{σ_f}\) to \(\top_{σ}\), which is trivial
\end{enumerate}

Thus all hypotheses of \cref{prop:global-rel-small}(\cref{item:global-preserve}) are met, so \(f\) globally preserves \((τ \mid σ)\)-smallness. This proves \cref{item:151}.
\end{proof}

Note that it is not the case that each ideal embedding \(ι{:} A \hookrightarrow P\) preserves smallness. Rather, there has to be at least some
homogeneity to ensure that joins always act the same way:

\begin{Counterexample}
  \label{cxmp:1}
  Let \(P\) be the order given by elements \(v_{1}, v_{2}, v_{3}, l_{1|2}, l_{2|3}, l_{3|1}, v_{1} \cup v_{2}, \top, \bot\) such that \(v_{2} ∨ v_{3} = l_{2|3}\) and \( v_{3} ∨ v_{1} = l_{3 | 1}\) but
  \(v_{1} ∨ v_{2} = v_{1} \cup v_{2} < l_{1|2}\). Geometrically, this is a triangle simplex except that the two vertices \(v_{1}\) and \(v_{2}\) join to their set-theoretic union \(v_{1} \cup v_{2}\) while the
  other vertices join to their containing line. Let \(ι{:} (l_{1|2}) \hookrightarrow P\) be the principal ideal embedding of \(l_{1|2}\). Since \(v_{1} ∨ v_{2} = v_{1} \cup v_{2} < l_{1|2}\), \(v_{1}\) is small in
  \((l_{1|2})\), but \(v_{1} ∨ l_{2|3} = \top\), so \(v_{1}\) is not small in \(\mathbf{SOb}(\top)\). So not every ideal embedding preserves smallness.
\end{Counterexample}

Modularity suffices however to ensure smallness is preserved:

\begin{Lemma}[Preservation of smallness in modular lattices.]
  \label{lmm:4}
\index{small element!preservation in modular lattices}
\index{modular lattice!smallness preservation}
  Let \(L\) be a modular lattice and \(ι{:} (a) \hookrightarrow L\) a principal ideal embedding. Then \(ι\) preserves smallness.
\end{Lemma}
\begin{proof}
  \label{prf:8}
  Since principal ideal embeddings are in particular join-preserving, we can use \cref{prp:8}. So we just have to check that for each \(b \leq a\) and \(c\) such that \(b ∨ c = \top\) there is
  a \(d \leq c\) such that \(b ∨ d = a\). But it follows from modularity that
  \[b ∨ (c ∧ a) = (b ∨ c) ∧ a = \top ∧ a = a.\]
\end{proof}

\begin{Remark}
  \label{rmk:18}
  Geometrically the failure of general smallness-preservation can be understood through the tablets we introduced in \cref{sec:tablets}: the image of an element \(a\) of \((ι)\) under
  \(ι\) is still small according to the tablets from \((ι)\), but \(P\) can contain wholly new tablets that are fine enough to make \(a\) seem big.
  Modularity then says that the subobjects of \(Y\) restrict to \(X\) cleanly enough to guarantee fine tablets in each ideal.
\end{Remark}

\subsection{Reducedness \& Satedness under Morphisms}
\label{sec:reduc-under-morph}

Let us also record general conditions for reducedness under morphisms:

\begin{Proposition}[Reducedness under morphisms in \(\mathbf{UBP}\)]
  \label{prop:reducedness-under-morphisms}
  Consider a commuting diagram in \(\mathbf{UBP}\):

  \begin{figure}[H]
    \centering
      \begin{tikzcd}
    T \ar[r,"τ"] \ar[d,"h"] & P \ar[r,"σ"] \ar[d,"g"] & S \ar[d,"k"] \\
    T'  \ar[r,"τ'"]            & P'  \ar[r,"σ'"]            & S'
  \end{tikzcd}
    \caption{\label{fi♭_:11}}
  \end{figure}

  \begin{enumerate}
  \item \textbf{(Preservation)}
    Assume that:
    \begin{enumerate}
    \item\label{item:refl-small} \((h,g,k)\) reflects \((τ'\mid σ')\)-smallness (as in \cref{prop:global-rel-small} \cref{item:global-reflect});
    \item\label{item:68} \(g\) and \(h\) are ideal embeddings.
    \end{enumerate}
    Then for any \(a ∈ P\) that is \((τ\mid σ)\)-reduced, \(g(a)\) is \((τ'\mid σ')\)-reduced in \(P'\).

  \item \textbf{(Reflection)}
    Assume that:
    \begin{enumerate}
    \item\label{item:pres-small} \((h,g,k)\) preserves \((τ'\mid σ')\)-smallness;
      \item\label{item:79} \(g\) preserves binary joins;
    \item\label{item:pres-inj} \(k\) is order-reflecting.
    \end{enumerate}
    Then for any \(a ∈ P\), if \(g(a)\) is \((τ'\mid σ')\)-reduced in \(P'\), then \(a\) is \((τ\mid σ)\)-reduced in \(P\).
  \end{enumerate}
\end{Proposition}
\begin{proof}
  We treat the two parts separately.

  \begin{enumerate}
  \item \emph{Preservation.}
    Let \(a\) be \((τ\mid σ)\)-reduced.
    Suppose \(g(a) \leq b' ∨ τ'(c')\) in \(P'\) with  \(b'\) being \((τ'\mid σ')\)-small.
    Since \(g\) and \(h\) are ideal embeddings and thus order-reflecting and bijective on its image,
    \[a = g^{*}(a) \leq g^{*}(b' ∨ c') = g^{*}(b') ∨ g^{*}(c') = g^{*}(b') ∨ g^{*}(c') = g^{*}(b') ∨ τ(h^{*}(c')).\]
thus there exist \(b ∈ P, c ∈ T\) with \(g(b) = b',\; h(c) = c'\) and \(a \leq b ∨ τ(c)\), so that \(σ(a) \leq σ(b ∨ τ(c))\).
    By \cref{item:refl-small} the condition that \(b'\) is \((τ'\mid σ')\)-small implies \(b\) is \((τ\mid σ)\)-small.
    Since \(a\) is \((τ\mid σ)\)-reduced, \(σ(a) = σ(τ(c))\), hence
    \[σ'(g(a)) = k(σ(τ(c))) = σ'(τ'(c')).\]
    Thus \(σ'(g(a)) \leq σ'(τ'(c'))\), showing that \(g(a)\) is \((τ'\mid σ')\)-reduced.

  \item \emph{Reflection.}
    Let \(g(a)\) be \((τ'\mid σ')\)-reduced in \(P'\).
    Suppose \(a \leq b ∨ τ(c)\) in \(P\) with \(b\) being \((τ\mid σ)\)-small. Then we have \(g(a) \leq g(b) ∨ g(τ(c))\).

    By \cref{item:pres-small} the fact that \(b\) is \((τ\mid σ)\)-small implies \(g(b)\) is \((τ'\mid σ')\)-small.
    Because \(g(a)\) is \((τ'\mid σ')\)-reduced, we obtain

    \[k(σ(a)) = σ'(g(a)) \leq σ'(g(τ(c))) = k(σ(τ(c))).\]
    Order-reflectivity of \(k\) then forces \(σ(a) \leq σ(τ(c))\), so \(a\) is \((τ\mid σ)\)-reduced.
  \end{enumerate}
\end{proof}

We also include the dual case:

\begin{Remark}[Dualization for satedness under morphisms]
  \label{rmk:dual-completeness-under-morphisms}
  Consider the diagram in \cref{fi♭_:11} but in the category of downward-bounded preorders.

  \begin{enumerate}
  \item \textbf{(Preservation of satedness)}
    Assume that:
    \begin{enumerate}
    \item\label{item:dual-refl-small} \((h,g,k)\) reflects \((τ'\mid σ')\)-density;
    \item\label{item:dual-68} \(g\) and \(h\) are filter embeddings.
    \end{enumerate}
    Then for any \(a ∈ P\) that is \((τ\mid σ)\)-sated, \(g(a)\) is \((τ'\mid σ')\)-sated in \(P'\).

  \item \textbf{(Reflection of satedness)}
    Assume that:
    \begin{enumerate}
    \item\label{item:pres-dense} \((h,g,k)\) preserves \((τ'\mid σ')\)-density;
    \item \(g\) preserves binary meets;
    \item\label{item:dual-pres-inj} \(k\) is order-reflecting.
    \end{enumerate}
    Then for any \(a ∈ P\), if \(g(a)\) is \((τ'\mid σ')\)-sated in \(P'\), then \(a\) is \((τ\mid σ)\)-sated in \(P\).
  \end{enumerate}
\end{Remark}

Moreover, we have the following effect:

\begin{Proposition}
  \label{■pn_:11}
  Let \(a\) be an element of a downward-bounded preorder \(P\) and \(b \leq a\) such that \(b\) is sated in the ideal \((a)\) of \(a\). Then for each dense \(c\) and arbitrary \(d\) in
  \(P\) it holds that:

  \[\text{ if } b \geq c ∧ d, \text{ then } b \geq a ∧ d\].
\end{Proposition}
\begin{proof}
  \label{prf_:13}
  Since \(b \leq a\) and meets commute and are idempotent,
  \[b \geq c ∧ d = a ∧ c ∧ d = (a ∧ c) ∧ (a ∧ d)\].

  These are all in \((a)\), so since \(b\) is sated in \((a)\) and \((a ∧ c)\) is dense in \((a)\),
  \[b \geq (a ∧ c) ∧ (a ∧ d) = a ∧ d\].
\end{proof}

Conceptually, a subobject \(\kappa \in \mathbf{SOb}(X)\) is sated if it is "saturated'' with respect to dense subobjects: whenever \(\kappa\) contains the meet of some dense subobject \(\mu\) and another subobject \(\lambda\), the dense part can be replaced by the ambient part \(ι {:} A \hookrightarrow X\) without changing the intersection — i.e.
\[
  \kappa = \mu ∧ \lambda \quad\Longrightarrow\quad \kappa = \lambda ∧ ι .
\]
Thus a sated subobject is one that "absorbs'' dense subobjects: it cannot be obtained by meeting with a dense part without actually being the whole of the other factor.  In short,
satedness is the order-theoretic expression of being so solid that all dense subobjects in it are locally filled.

\begin{Corollary}
  \label{coℝ_:1}
  In the situation of \cref{■pn_:11}, if \(P\) is also upward-bounded, then it follows that if \(b\) contains any dense subobject of \(P\), then \(b = a\). Thus, no subobject
  containing a dense subobject can be sated in anything but itself.
\end{Corollary}
\begin{proof}
  \label{prf_:1}
  Apply \cref{■pn_:11} for \(d = \top\).
\end{proof}

Curiously, it does not seem to follow that in the situation of \cref{■pn_:11} it holds that if \(a\) is smallted in \(P\) then \(b\) is sated in \(P\) without further assumptions.

\begin{Counterexample}
  \label{cxmp:8}
  We present a counterexample in a meet-semilattice (with bottom) that is not a Heyting algebra.  The example shows that the conditions

  \begin{enumerate}
  \item \(a\) is sated in \(P\),
  \item \(b\le a\) and \(b\) is sated in the ideal \((a)\),
  \end{enumerate}

  do not force \(b\) to be sated in \(P\).

Let \(P\) be the partially ordered set consisting of the eight elements
\[
\{0,\; a,\; b,\; c,\; d,\; e,\; f,\; 1\}
\]
with the order generated by the following covering relations (reflexive closure is understood):

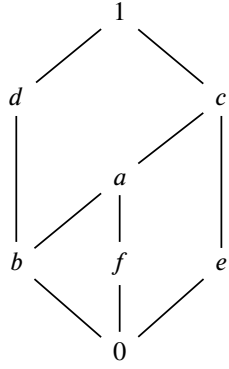
\[
\begin{tikzcd}
& 1 \ar[dl,-] \ar[dr,-] & \\
d \ar[dd,-] & {} & c \ar[dl,-] \ar[dd,-] \\
& a \ar[dl,-] \ar[d,-] & \\
b \ar[dr,-] & f \ar[d,-] & e \ar[dl,-] \\
& 0  &
\end{tikzcd}
\]

More explicitly:
\[
0 < b, f < a < c < 1,\quad 0 < e < c,\quad b < d < 1,
\]
Only \(c\) and \(1\) are dense:
\begin{itemize}
\item \(c\) meets every nonzero element nontrivially;
\item \(1\) is trivially dense;
\item \(a\) is not dense because \(a∧ e = 0\) with \(e\neq0\);
\item \(d\) is not dense because \(d∧ f = 0\) with \(f\neq0\);
\item \(b,e,f\) are not dense for similar reasons.
\end{itemize}

\(a\) is sated:
\begin{itemize}
\item  \(a = 1∧ y\text{ if and only if }a = y\) is trivially true.
\item For \(c\), the equation \(a = c∧ y\) forces \(y\ge a\).
  Among the elements \(\ge a\) we have \(a,c,1\).  Only \(y=a\) gives \(c∧ a = a\);
  \(y=c\) gives \(c∧ c = c\neq a\); \(y=1\) gives \(c\neq a\).
  Hence the condition that \(a = c∧ y \Rightarrow a = y\) holds vacuously except for the valid
  instance \(y=a\).
\end{itemize}

In \((a)\) the only dense element is \(a\). Thus, every element is sated in \((a)\), so in particular \(b\) is.

Take the dense element \(c\) and the element \(d\).  We have
\[
b = c∧ d \qquad\text{(by construction)},
\]
but \(d > b\) (so \(d \neq b\)).  A sated element would force \(b = d\); hence \(b\) is not
sated in \(P\).

So our counterexample evades \cref{prop:completeness-reflection-dense-envelope} because \(P\) is not distributive.
\end{Counterexample}

We have one more result, which is inspired by and will be applied to topology: satedness stabilizes with pseudocomplements.

\begin{Proposition}[Reflection of satedness via dense envelopes]
  \label{prop:completeness-reflection-dense-envelope}
  \index{sated element!reflection} Consider a commuting diagram in \(\mathbf{DBP}\):

\begin{figure}[H]
  \centering
  \begin{tikzcd}
    T \ar[r,"τ"] \ar[d,"h"] & P \ar[r,"σ"] \ar[d,"g"] & S \ar[d,"k"] \\
    T' \ar[r,"τ'"] & P' \ar[r,"σ'"] & S'
  \end{tikzcd}
  % \label{fig:3}
\end{figure}

Assume the following:

\begin{enumerate}
\item \label{item:155} \(g\) is an ideal embedding;
\item\label{item:16} \(P'\) is distributive;
\item \label{item:1} There exists a \((τ', σ')\)-pseudocomplement for each element of \(P'\) in \(T'\) and pseudocomplements distribute over joins;
\item\label{item:8} \(k\) is order- and bottom-reflecting;

  \end{enumerate}
  Then \((h, g, k)\) is satedness-reflecting.
\end{Proposition}
\begin{proof}
Let \(a ∈ P\) and assume that \(g(a)\) is \((τ' \mid σ')\)-sated in \(P'\); we show that \(a\) is \((τ \mid σ)\)-sated in \(P\). Unfolding \cref{rmk:dual-sigma-mu}, this amounts to the following: whenever \(d ∈ P\) is \((τ \mid σ)\)-dense and \(t ∈ T\) satisfies \(d ∧ τ(t) \leq a\), then \((τ⨾ σ)(t) \leq σ(a)\). So fix such a pair \(d, t\).

The strategy is to transplant the situation upstairs to \(P'\): we enlarge \(g(d)\) by the pseudocomplement of \(\top_{g}\), obtaining an element
\[d' \;:=\; g(d) ∨ τ'\bigl(¬\top_{g}\bigr) \;\in\; P'\]
that is still \((τ' \mid σ')\)-dense and whose meet with \(τ'(h(t))\) remains below \(g(a)\); the satedness of \(g(a)\) then delivers the desired inequality, which we pull back along \(k\).

\emph{Step 1: \(d' ∧ τ'(h(t)) \leq g(a)\).} Since the left square commutes, \(τ'(h(t)) = g(τ(t))\), so distributivity of \(P'\) (\cref{item:16}) gives
\[d' ∧ τ'(h(t)) \;=\; \bigl(g(d) ∨ τ'\bigl(¬\top_{g}\bigr)\bigr) ∧ g\bigl(τ(t)\bigr) \;=\; \bigl(g(d) ∧ g\bigl(τ(t)\bigr)\bigr) ∨ \bigl(τ'\bigl(¬\top_{g}\bigr) ∧ g\bigl(τ(t)\bigr)\bigr).\]
The ideal embedding \(g\) preserves binary meets, so the first summand equals \(g(d ∧ τ(t)) \leq g(a)\). For the second summand, note that \(τ(t) \leq \top_{P}\) implies \(g(τ(t)) \leq \top_{g}\), whence by \cref{dfn:1}
\[τ'\bigl(¬\top_{g}\bigr) ∧ g\bigl(τ(t)\bigr) \;\leq\; τ'\bigl(¬\top_{g}\bigr) ∧ \top_{g} \;=\; \bot_{P'}.\]
Thus the second summand vanishes and \(d' ∧ τ'(h(t)) = g(d ∧ τ(t)) \leq g(a)\).

\emph{Step 2: \(d'\) is \((τ' \mid σ')\)-dense.} By \cref{prop:pseudocomplement-density} it suffices to verify \((τ'⨾ σ')(¬d') = \bot_{σ'}\). Pseudocomplements of elements of \(P'\) exist in \(T'\) and distribute over joins by \cref{item:1}, so
\[¬d' \;=\; ¬\bigl(g(d) ∨ τ'\bigl(¬\top_{g}\bigr)\bigr) \;=\; ¬g(d) ∧ ¬τ'\bigl(¬\top_{g}\bigr).\]
The density of \(d\) translates, again via \cref{prop:pseudocomplement-density}, into \((τ⨾ σ)(¬d) = \bot_{σ}\). Since \(g(d) \leq \top_{g}\), every element of \(T'\) that meets \(\top_{g}\)
trivially also meets \(g(d)\) trivially, whence \(¬\top_{g} \leq ¬g(d)\). Moreover, for each \(e ∈ P'\) it holds that if \(ε ∧ g(d)\) is not \(\bot_{P'}\), then \(g^{*}(ε)\) is not
\(\bot_{P}\) since \(g\) is an ideal embedding, so that \(¬\top_{g} = ¬g(d)\)
. Evaluating the scale on the complement therefore yields
\[(τ'⨾ σ')(¬d') \;=\; σ'\Bigl(τ'\bigl(¬g(d)\bigr) ∧ τ'\bigl(¬τ'\bigl(¬\top_{g}\bigr)\bigr)\Bigr) = σ'\Bigl(τ'\bigl(¬\top_{g}\bigr) ∧ τ'\bigl(¬τ'\bigl(¬\top_{g}\bigr)\bigr)\Bigr) \;=\; σ'(\bot_{P'}) \;=\; \bot_{σ'},\]
where the penultimate identity holds because \(τ'(¬\top_{g}) ∧ τ'(¬τ'(¬\top_{g})) = \bot_{P'}\) by the defining property of the pseudocomplement \(¬τ'(¬\top_{g})\). Hence \(d'\) is \((τ' \mid σ')\)-dense.

\emph{Step 3: apply the satedness of \(g(a)\).} Combining Steps 1 and 2, the element \(d'\) is \((τ' \mid σ')\)-dense and satisfies \(d' ∧ τ'(h(t)) \leq g(a)\). Since \(g(a)\) is \((τ' \mid σ')\)-sated, \cref{rmk:dual-sigma-mu} yields
\[(τ'⨾ σ')\bigl(h(t)\bigr) \;\leq\; σ'\bigl(g(a)\bigr).\]
By commutativity of the diagram, \(h⨾ τ'⨾ σ' = τ⨾ g⨾ σ'\) and \(g⨾ σ' = σ⨾ k\), so
\[(τ⨾ σ⨾ k)(t) \;=\; (τ'⨾ σ')\bigl(h(t)\bigr) \;\leq\; σ'\bigl(g(a)\bigr) \;=\; (σ⨾ k)(a) \;=\; k\bigl(σ(a)\bigr).\]
Since \(k\) is order-reflecting by \cref{item:8}, this forces \((τ⨾ σ)(t) \leq σ(a)\).

As \(d\) and \(t\) were arbitrary, \(a\) is \((τ \mid σ)\)-sated in \(P\); hence \((h, g, k)\) is satedness-reflecting.
\end{proof}

\begin{Corollary}
  \label{coR_:7}
\index{sated element!reflection by monomorphisms}
  If \(P\) is a Heyting algebra, then the ideal embedding \((a)\) of an element \(a ∈ P\) is satedness-reflecting.
\end{Corollary}
\begin{proof}
  \label{prf_:15}
  \cref{item:8} is trivial if we don't use a non-trivial scale, \cref{item:155} is true by assumption and \cref{item:16} and \cref{item:1} are true because \(P\) is a Heyting algebra.
\end{proof}

So like smallness depends on the size of the tablets, so does satedness depend on the number of dense subobjects. With smallness, it is modularity that ensures larger tablets restrict
cleanly so smallness is preserved; with satedness, Heyting complements ensure that each locally dense subobject can be made globally dense, so that satedness is reflected.

\section{The Inter-Categorical Development of the Calculus}
\label{sec:inter-categ-size}

What gives our analysis its true power lies in how it interacts with categorical structure through the \(\mathbf{SOb}\)-functor, so let us take stock of that.
For this, we have to avoid some size issues:

\begin{Definition}
  \label{defn_:31}
  (See, e.g. Section 0.6 of \cite{adamekLocallyPresentableAccessible1994})A category \(\mathcal{C}\) is \emph{well-powered} if for each object \(X{:} \mathcal{C}\) the subobject preorder \(\mathbf{SOb}_{\mathcal{C}}(X)\) is essentially small, i.e. a set (up to equivalence).
\end{Definition}
\begin{Example}
  \label{xmp:103}
  \(\mathbf{Set}\) is well-powered but not small.
\end{Example}

Given now a well-powered category \(\mathcal{C}\) we have a functor \(\mathbf{SOb}_{\mathcal{C}}{:} \overset{\hookrightarrow}{\mathcal{C}} → \mathbf{UBP}\) on the subcategory \(\overset{\hookrightarrow}{\mathcal{C}}\) of monomorphisms in \(\mathcal{C}\) given by mapping \(X{:} \mathcal{C}\) to the preorder \(\mathbf{SOb}_{\mathcal{C}}(X)\) of
monomorphisms into \(X\).  It would be tempting to think that \(\mathbf{SOb}_{\mathcal{C}}\) is faithful on \(\overset{\hookrightarrow}{\mathcal{C}}\), but this is actually not always true:

\begin{Counterexample}
  \label{cxmp:3}
  If \(\mathcal{C}\) is the delooping of a group, then all monics into
  its point \(*\) are isomorphisms and thus isomorphic to the identity in \(\mathcal{C}_{/*}\). So \(\mathbf{SOb}_{\mathcal{C}}(*)\) is the one-element preorder, and thus the representation
  \(\mathbf{SOb}_{\mathcal{C}}{:} \mathcal{C} → \mathbf{UBP}\) is trivial.
\end{Counterexample}

This could be avoided by restricting to monic categories with no non-trivial isomorphisms, but that would be excessive: the problem is not that there are isomorphisms, but that they are not sufficiently
supported by non-isomorphic monics:

\begin{Definition}
  \label{dfn:2}\index{subobject-separated category}
  A category \(\mathcal{C}\) is \emph{subobject-separated} if for each pair of morphisms \(f, g{:} X \rightrightarrows Y\) there exists a subobject \(ι{:} A \hookrightarrow X\) such that for each \(σ{:} \mathbf{Aut}(A)\):

  \begin{equation}
    \label{qn_:2}
    \text{if } ι⨾ f = σ⨾ ι⨾ g \text{ then } f = g.
  \end{equation}
\end{Definition}
\begin{Example}
  \label{xmp:101}
  Suppose \(\mathcal{C}\) has a separating family of subterminal objects \(\mathcal{S}\). Remember that an object \(S\) is subterminal if and only if for every \(Y\), the hom-category \(\mathcal{C}(Y,S)\) has at most one element
  and a family of morphisms \(\mathcal{S}\) is separating if and only if for any two distinct parallel morphisms \(f,g {:} X \rightrightarrows Y\) there exists \(S \in \mathcal{S}\) and a morphism \(h {:} S \to X\) with \(h⨾ f \neq h⨾ g\).

  Given such \(f, g\), pick an \(h {:} S \to X\) such that \(h⨾ f \neq h⨾ g\). Because \(S\) is subterminal, any morphism out of \(S\) is trivially monic.
  Therefore \(h\) is a monomorphism. Is it proper? If \(h\) is an isomorphism, then \(X \cong S\) is subterminal, so both \(h\) and \(σ\) in \cref{qn_:2}  necessarily equal
  \(\mathbf{id}_{X}\), so \(\mathbf{id}_{X}\) is a suitable \(ι\). If \(h\) is a proper monomorphism, then, \(σ\) in \cref{qn_:2} still has to be an identity, so that \(ι\) subobject-separates \(f\) and \(g\) by the
  separativeness of \(\mathcal{S}\).

  This applies in particular to all localic topoi (see C1.4.7. of \cite{johnstoneSketchesElephantTopos2002}).
\end{Example}

\begin{Proposition}[\(\mathbf{SOb}_{\mathcal{C}}\) is faithful when \(\mathcal{C}\) is subobject-separated]
  \label{■pn_:1}
  Let \(\mathcal{C}\) be a monic well-powered category and let \(\overset{\hookrightarrow}{\mathcal{C}}\) be its wide subcategory of monomorphisms.
  If \(\mathcal{C}\) is subobject-separated (in the sense of \cref{dfn:2}), then the subobject functor
  \[
    \mathbf{SOb}_{\mathcal{C}} {:} \overset{\hookrightarrow}{\mathcal{C}} \longrightarrow \mathbf{UBP}
  \]
  is faithful.
\end{Proposition}
\begin{proof}
  Let \(f,g {:} X \hookrightrightarrows Y\) be two monomorphisms and assume \(\mathbf{SOb}_{\mathcal{C}}(f) = \mathbf{SOb}_{\mathcal{C}}(g)\).
  Concretely, this means that for every subobject \(ι {:} A \hookrightarrow X\), the two monomorphisms \(ι⨾ f\) and \(ι⨾ g\) become equal in the subobject preorder of \(Y\); i.e. there exists an isomorphism \(σ {:} A \to A\) (depending on \(ι\)) such that
  \[
    ι⨾ f \;=\; σ ⨾ ι⨾ g .
   \]

   Now apply the hypothesis that \(\mathcal{C}\) is subobject-separated: then for the pair \((f,g)\) there exists a subobject \(ι {:} A \hookrightarrow X\) with the property that if there exists such an isomorphism \(σ {:} A \to A\) for \(ι\), then \(f = g\).  Hence \(\mathbf{SOb}_{\mathcal{C}}\) is faithful.
\end{proof}

\begin{Remark}
  \label{rmk:14}
  Note that this means that \(\mathbf{SOb}\) induces an equivalence relation on \( \overset{\hookrightarrow}{\mathbf{Cat}}\), given by \(\mathcal{C} \sim D\) if and only if the image of \(\mathbf{SOb}_{\mathcal{C}}\) is isomorphic to that of
  \(\mathbf{SOb}_{\mathcal{D}}\), which is like Morita-equivalence for \(\mathbf{UBP}\). We can call \(\mathcal{C}\) \emph{subobject-equivalent} to \(\mathcal{D}\) in this case.
\end{Remark}
\begin{Example}
  \label{xmp:32}
  The delooping of every group is subobject-equivalent to the terminal category.
\end{Example}

\begin{Definition}
  \label{defn_:32}
\index{monomorphism-preserving functor}
  A functor \(F{:} \mathcal{C} \to \mathcal{D}\) is \emph{monomorphism-preserving} if for each monomorphism \(ι{:} X \hookrightarrow Y\) in \(\mathcal{C}\), \(F(ι)\) is a monomorphism.
\end{Definition}
\begin{Example}
  \label{xmp:109}
  Every pullback-preserving functor is monomorphism-preserving because every monomorphism is a pullback along itself. In particular, every right adjoint \(R\) is monomorphism-preserving since it preserves limits.
\end{Example}

Given now a monomorphism-preserving functor \(F{:} \mathcal{C} → \mathcal{D}\), we can to each monomorphism \(ι{:} X → Y\) in \(\mathcal{C}\) associate a monomorphism \(F(ι)\) in \(\mathcal{D}\), and thus obtain a natural transformation
\(F_{*}{:} \mathbf{SOb}_{\mathcal{C}} \Rightarrow \mathbf{SOb}_{\mathcal{D}}\), whose component at \(X\) is the map
\(F_{*, X}{:} \mathbf{SOb}_{\mathcal{C}}(X) \to \mathbf{SOb}_{\mathcal{D}}(F(X))\) sending each monomorphism \(ι\) to \(F(ι)\).
So \(F(ι)\), viewed as an element of \(\mathbf{SOb}_{\mathcal{C}}(F(X))\), is precisely the top \(\top_{F(ι)_{*}}\) of the ideal embedding \(F(ι)_{*}{:} \mathbf{SOb}(F(A)) \hookrightarrow \mathbf{SOb}_{\mathcal{C}}(F(X))\).
This same construction — extracting the element designated by the image of a monomorphism — will reappear throughout this text in several different categorical settings, and we
shall uniformly denote it by \(F_{*, X}\) in each case. In particular they will be important in \cref{sec:covar-repr-theor}.

Note that this means that both parts of the top-preserving/ideal-embedding factorization system serve complementary roles for displaying a shadow of \(\overset{\hookrightarrow}{\mathbf{UBP}}\):
the top-preserving morphisms are the shadows of functors, the ideal embeddings are shadows of morphisms between objects. And, as we will learn later, the 2-cells, i.e.
inequalities, between morphisms in \(\mathbf{UBP}\) display, in some sense, natural transformations.

\subsection{The Duality Breaking Point}
\label{sec:dual-break-point}

This is where the duality between smallness and density finally breaks: first of all, the subobject preorder \(\mathbf{SOb}_{\mathcal{C}}(X)\) of every object \(X\) of every category \(\mathcal{C}\) is \emph{upward}-bounded but
not necessarily \emph{downward}-bounded, so in a general category this notion of density makes no sense. It is only when \(\mathcal{C}\) has a strict initial object that \(\mathbf{SOb}_{\mathcal{C}}(X)\) is
also downward-bounded. But this loss of generality is only the tip of the iceberg.

\subsubsection{Density \& Satedness}
\label{sec:density}

For each \(ι{:} X \hookrightarrow Y\) the morphism \(ι_{*}{:} \mathbf{SOb}_{\mathcal{C}}(X) \hookrightarrow \mathbf{SOb}_{\mathcal{C}}(Y)\) is an \emph{ideal embedding}, which is principal by \cref{rmk:17}, so it preserves
\(\bot\) if it exists (but not \(\top\)), and it preserves and reflects both meets and joins. Thus it fulfills the requirements of \cref{coR_:9}, so that \(ι_{*}\) is density-reflecting. And this is why density
deserves its name: because it is reflected by \emph{all} subobject inclusions in \emph{all} categories with strict initial objects.

Moreover, if a functor \(F{:} \mathcal{C} → \mathcal{D}\) preserves pullbacks and the initial object \(∅\), then it preserves monomorphisms and the induced natural transformation \(F_{*}\) preserves meets to \(\bot\). So if
\(F\) also reflects the initial object, then \(F_{*}\) reflects density by \cref{prop:rel-reflect-small}.

\paragraph{Adjoined Initial Objects}

Since every subobject preorder has an upward-bound but not a downward-bound, density seems less general than smallness, but this is only somewhat true. If some category \(\mathcal{C}\) does
not have a strict initial object, we can form a category \(\mathcal{C}_{∅}\) by freely adjoining an object \(∅\) with one morphism to each object in \(\mathcal{C}\), with the obvious composition and
embedding \(\mathcal{c}_{∅}{:} \mathcal{C} \hookrightarrow \mathcal{C}_{∅}\). So \(\mathcal{C}\) inherits a slightly different notion of density from \(\mathcal{C}_{∅}\):

\begin{Definition}
  \label{defn_:16}
  An object \(a\) of a preorder \(P\) is \emph{dense} if for each \(b ∈ P\) there exists some \(c\) with \(c \leq a, c \leq b\).
\end{Definition}

This doesn't assume a strict initial object exists and is still a sensical notion. Of course, if \(\mathcal{C}\) \emph{already has} a strict initial
object and we apply the same notion, it will not work because we should be using \cref{defn_:58} instead. However, if we remove that strict initial object from \(\mathcal{C}\), then the
resulting category \(\mathcal{C} ∖ ∅\) is a full subcategory of \(\mathcal{C}\) such that \(ι{:} A \hookrightarrow X\) in \(\mathcal{C} ∖ ∅\) is dense according to \cref{defn_:16} if and only if \(ι\) is dense in \(\mathcal{C}\)
according to \cref{defn_:58}. So the difference is essentially in whether we treat bottoms as part of the signature, which we should when we are dealing with categories with a
strict initial object. Note that there are several notions in order theory that have to be similarly modified when bottoms are part of the signature, such as that of atoms.

\paragraph{The Yoneda Embedding}

Let us also study how density behaves with respect to the Yoneda embedding:

\begin{Example}
  \label{xmp_:6}
  Let \(\widehat{(-)}{:} \mathcal{S} \hookrightarrow \widehat{\mathcal{S}}\) be the Yoneda embedding of a site with a subcanonical coverage and a strict initial object \(∅\) such
  that \(∅_{\widehat{\mathcal{S}}} \cong \widehat{∅_{\mathcal{S}}}\)
  (i.e. the empty sheaf is representable) and consider the induced map on subobject preorders \(\top_{\widehat{(-)}, X}{:} \mathbf{SOb}_{\mathcal{S}}(X) → \mathbf{SOb}_{\widehat{\mathcal{S}}}(\widehat{X}){:} ι ↦ \hat{ι}\).
  Since \(\mathcal{S}\) is subcanonical, \(\widehat{(-)}\) preserves finite limits and by assumption it preserves the initial object \(∅\). Since \(\top_{\widehat{(-)}, X}\)  is also fully faithful (by the
  Yoneda lemma and subcanonicity), it reflects \(∅\). Thus, \(\widehat{(-)}\) reflects density.
\end{Example}

Not just that, but for preorders the Yoneda embedding also \emph{preserves} density. For this, we need the following lemma, which was probably recorded somewhere else before:

\begin{Lemma}
  \label{le↘_:3}
  Let \(\mathcal{S}\) be a thin site with a subcanonical coverage. Then every nonzero sheaf \(F\) on \(\mathcal{S}\) receives a monomorphism from some representable.
\end{Lemma}
\begin{proof}
  \label{prf_:12}
  Consider presheaves \(F {:} \mathcal{S}^{op} → \mathbf{Set}\). Recall that a natural transformation \(\widehat{x}{:} \widehat{X} ⇒ F\) from a representable corresponds, by Yoneda, to a choice of element \(x ∈ F(X)\). Its component at \(Y\) is
  \[
  \overline{x}_Y {:} \mathcal{S}(Y,X) → F(Y),\quad f \mapsto F(f)(x).
\]
In a presheaf category, \(\widehat{x}\) is monic if and only if every component \(\widehat{x}_Y\) is injective, and since \(\mathcal{S}\) is subcanonical, finite limits in \(\widehat{\mathcal{S}}\) reduce to those in the surrounding presheaf
topos \(\mathbf{PSh}(\mathcal{S})\). So we have to find a componentwise monic into \(F\). If \(F\) is non-trivial, then there exists some  \(x ∈ F(X)\) for some object \(X\). But since for each object \(Y\), the hom-set \(\mathcal{S}(Y,X)\) has size \(0\) or \(1\) and any function with domain of size ≤ 1 is trivially
injective. Thus, \(\widehat{x}\) is componentwise injective. So any element of any sheaf in \(\widehat{\mathcal{S}}\) is monic.
\end{proof}
\begin{Proposition}
  \label{prp:6}
  Given a thin site \(\mathcal{S}\) with a subcanonical coverage and an initial object, the natural transformation \(\top_{\widehat{(-)}}{:} \mathbf{SOb}_{\mathcal{S}} \hookrightarrow \widehat{(-)}⨾ \mathbf{SOb}_{\widehat{\mathcal{S}}}\) induced by the Yoneda embedding \(\widehat{(-)}{:} \mathcal{S} \hookrightarrow \mathbf{Set}^{\mathcal{S}^{op}}\) is density-preserving.
\end{Proposition}
\begin{proof}
  \label{prf:9}
  We need to show that if \(ι\) is dense in \(\mathbf{SOb}_{\mathcal{S}}(X)\), then \(\widehat{ι}\) is dense in \(\mathbf{SOb}_{\widehat{\mathcal{S}}}(\widehat{X})\). Let
  \(κ{:} K \hookrightarrow \widehat{X}\) be any subsheaf such that
  \[
    κ ∧ \widehat{ι} = \bot
  \]
  in \(\mathbf{SOb}_{\widehat{\mathcal{S}}}(\widehat{X})\). We must show \(κ = \bot\). Since \(\mathcal{S}\) is subcanonical, the sheafification functor preserves meets, so we can take the meet of
  \(κ ∧ \widehat{ι}\) as sub-presheaves, which is given by objectwise pullback. If \(κ\) is not initial, then by \cref{le↘_:3} there is some \(X{:} \mathcal{S}\) such that \(K(X)\) contains an element
  \(y\), so that \(κ ∧ \widehat{ι} \geq \widehat{y}⨾ κ\), where \(\widehat{y}\) is the morphism \(\widehat{X} → K\) induced by the Yoneda lemma. But since each cover of \(\mathcal{S}\) is
  jointly epimorphic and thus in particular non-empty, the empty presheaf on \(\mathcal{S}\) is a sheaf and thus initial. So in particular \(\widehat{y}⨾ κ > \bot\).
\end{proof}
\begin{Remark}
  \label{rmk:1}
\index{density-equivalence}
  So the embedding of a thin subcanonical site with an initial object into its sheaf topos is both density-preserving and density-reflecting, so a \emph{density-equivalence}.
\end{Remark}

\paragraph{Coproducts}

Another aspect where the duality between smallness and density breaks is with respect to coproducts, whose subobject structure in good cases is controlled by the summands:

\begin{Definition}[Disjoint coproduct, monic coprojections]
Let \(\mathcal{C}\) be a category with an initial object \(∅\). A coproduct \(X + Y\) with coprojections
\[
τ_X {:} X → X+Y,\qquad τ_Y {:} Y → X+Y
\]
is called a \emph{disjoint coproduct with monic coprojections} if \(τ_{X}\) and \(τ_{Y}\) are monic and the pullback of \(τ_X\) and \(τ_Y\) is initial.

Equivalently, the images of \(τ_X\) and \(τ_Y\) are subobjects of \(X+Y\) with intersection equal to the initial subobject.
\end{Definition}
\begin{Proposition}[Density in a disjoint coproduct]
  \label{prop:density-coproduct}
\index{density!in coproducts}
\index{disjoint coproduct!density}
  Let \(\mathcal{C}\) be a category with an initial object \(∅\) and disjoint coproducts with monic coprojections. Then the following equivalent statements hold:
  \begin{enumerate}
  \item\label{item:146} If for some monomorphism \(ι {:} A \hookrightarrow X\) the composite
    \[
      ι⨾τ_X {:} A \longrightarrow X+Y
    \]
    is dense in \(\mathbf{SOb}(X+Y)\), then \(Y \simeq ∅\).

  \item\label{item:147} The coprojection \(τ_X\) itself is dense in \(\mathbf{SOb}(X+Y)\) if and only if \(Y \simeq ∅\).
  \end{enumerate}
\end{Proposition}
\begin{proof}
  For \cref{item:146} assume \(ι⨾τ_X\) is dense.  Consider the subobject \(τ_Y {:} Y \hookrightarrow X+Y\).
  Because the coproduct is disjoint, the pullback of \(τ_X\) and \(τ_Y\) is initial.  Since \(ι⨾τ_X\) factors through \(τ_X\), the pullback of \(τ_Y\) and \(ι⨾τ_X\) is also initial; in symbols,
  \[
    τ_Y ∧ (ι⨾τ_X) = \bot .
  \]
  By definition of density, this forces \(τ_Y = \bot\), i.e. the monomorphism \(τ_Y\) is isomorphic to the initial subobject.  Hence \(Y \simeq ∅\).

  The first direction of \cref{item:147} follows from \cref{item:146} by taking \(ι = \mathbf{id}_X\).  For the converse, note that if \(Y \simeq ∅\),
  then the coproduct \(X+∅\) is isomorphic to \(X\) via \(τ_X\).
\end{proof}

\paragraph{Satedness}

This regularity of density does not extend to satedness however without further assumptions: by \cref{prop:reducedness-under-morphisms} the reflectiveness of density is useful
for morphisms that induce a surjection on subobject preorders, but this is never the case for non-trivial embeddings. Of course, it is the case for quotients, which we will turn to
in a moment. Moreover, if \(\mathcal{C}\) is a Heyting category, then each \(\mathbf{SOb}_{\mathcal{C}}(X)\) fulfills the requirements of \cref{coR_:7}. Thus, satedness is reflected along embeddings in a
Heyting category, which explains why every open subset of a sated open subset of a topological space is sated in every open neighborhood. We also have the effect of
\cref{■pn_:11}: that, for any \(ι{:} A \hookrightarrow X\) in any category with an initial object, if \(κ = κ'⨾ ι\) factors through \(ι\), \(κ'{:} B \hookrightarrow A\) is sated in \(\mathbf{SOb}(A)\) and \(κ \geq λ ∧ μ\) with \(μ\)
dense, then \(κ \geq λ ∧ ι\), so we can switch out any dense subobject with \(ι\) in such meet inequalities.

\subsubsection{Smallness \& Reducedness}
\label{sec:smallness-1}

Smallness on the other hand is a lot more whimsical. For instance, it is not the case that each \(ι_{*}\) preserves smallness for \(ι{:} X \hookrightarrow Y\) in some category \(\mathcal{C}\). We have however the
following consequence of \cref{lmm:4}:

\begin{Corollary}
  \label{crl:2}
  Let \(\mathcal{C}\) be a category such that each subobject preorder of objects in \(\mathcal{C}\) is a modular lattice. Then each monomorphism in \(\mathcal{C}\) is smallness-preserving.
\end{Corollary}

Smallness is also respectively more densely interwoven with reducedness: by \cref{crl:1} a subobject \(ι{:} A \hookrightarrow X\) is reduced if and
only if it reflects smallness.

Another deep difference between smallness and density regards the Yoneda embedding: we saw in \cref{sec:density} that the density of sheaves on a thin subcanonical site is completely
controlled by that in \(\mathcal{C}\), while we saw in \cref{■pn_:64} that \emph{every} non-trivial representable subobject is small in a representable presheaf. So Yoneda just makes everything small. Sheafification then makes the smallness relation non-trivial again by identifying unions of representable subobjects with the entirety of the sheaf. So what is
small in a sheaf topos depends solely on the coverage.

Finally, another consequence of the ideality of subobject preorder embeddings is that the smallness stratification from \cref{sec:smalln-strat} is \emph{way} more useful than its dense counterpart: in
categorical language, ``\(ι\) is \(n\)-small'' just means ``\(ι\) is \((n-1)\)-small in a small subobject'' while ``\(ι\) is \(n\)-dense'' means ``\(ι\) is \((n-1)\)-dense in a dense principal filter'' \emph{which does not
correspond to anything category-theoretically}, at least on a basic level. Thus, the smallness stratification can be found \emph{everywhere} while the density stratification is much more obtuse.

\subsubsection{Comparison}
\label{sec:comparison}

So while the order-theoretic definitions are perfectly dual, the categorical realisations are not: density imposes much stricter constraints on how an object can be decomposed.
Since density is also preserved by a wider class of functors (those preserving pullbacks and preserving and reflecting the initial object) and reflected by all monomorphisms while smallness
is preserved by a more restricted class and reflected only by those monomorphisms that are reduced, the structural properties of density are more tightly interwoven with the way objects
interact. So density becomes a global notion, which is highly sensitive to disturbances, while smallness becomes local — but local along reducedness, which itself is
defined through smallness, leading to a something of a chicken-and-egg situation when trying to infer global smallness from local data.

Reducedness and satedness are interesting — since reducedness is essentially equivalent to (covariant) smallness-reflection, it means that size-wise the image of a morphism
looks relatively the same as the containing object: at least in coherent categories smallness is then both preserved and reflected along reduced monomorphisms. Satedness doesn't work
that way at all. This is because for a proper principal ideal embedding \(ι_{a}{:} (a) \hookrightarrow P\) it trivially holds that \(ι(\top_{(a)}) < \top_{P}\), so reducedness is all about how joins to
\(a\) look compared to \(\top_{P}\); but since bottoms are preserved by \(f_{*}\) for each morphism \(f\) in each category, the criterion for completion is both in \(P\) and in \((a)\) based on how meets to \(\bot\) act, So the question is: ``How much is under \(a\)?'' and the answer has to be ``enough to make up the difference between \(\top\) and a dense subobject.''

\subsection{Extension to Non-Monic Categories}
\label{sec:extens-arbitr-categ}

\subsubsection{Extension along Pullbacks}
\label{sec:extens-along-pullb}

If a category \(\mathcal{C}\) has pullbacks along monics, then there also exists a functor \(\mathbf{SOb}_{\mathcal{C}}^{*}{:} \mathcal{C}^{op} → \mathbf{UBP}\), which maps a morphism \(f{:} Y → X\) to the pullback map
\(f^{*}{:} \mathbf{SOb}_{\mathcal{C}}(X) → \mathbf{SOb}_{\mathcal{C}}(Y)\). This is the arguably the most canonical hyperdoctrine, since it is defined on each category with pullbacks. We mean
``hyperdoctrine'' here in a general sense as a contravariant 2-functor on \(\mathcal{C}^{op}\), without necessarily assuming adjoints.

Moreover, since pullbacks always preserve identities, \(f^{*}\) preserves tops. As before, \(\mathbf{SOb}^{*}_{\mathcal{C}}\) is not faithful for instance when \(\mathcal{C}\) is the delooping of a groupoid. Note however
that \(\mathbf{SOb}^{*}\) is still faithful on (perhaps surprisingly) many categories:

\begin{Example}
  \label{xmp:13}
  In the internal language of any topos \(\mathcal{T}\), for any \(f{:} X → Y, ι{:} A \hookrightarrow Y\), \(f^{*}(ι) = \{x ∈ X| ∃ a ∈ A{:} f(x) ∈ ι\}\), so that it holds that \(x ∈ f^{*}(ι) \Leftrightarrow f(x) ∈ ι\), so if for all \(ι\) in \(Y\)
  \(f^{*}(ι) = g^{*}(ι)\), then \(∀ x ∈ X{:} f(x) = g(x)\). Since elements in the internal language of a topos correspond to \(κ{:} B \hookrightarrow X\) such that \(B\) is subterminal, it follows that \(f = g\) if \(\mathcal{T}\) is
  separated by its subterminal objects. So again \(\mathbf{SOb}^{*}\) is faithful on all topoi that are separated by subterminal objects.
\end{Example}
 
On the other hand, for any monic \(ι{:} A \hookrightarrow Y, κ{:} Y \hookrightarrow X\) in any category, the square

\begin{figure}[H]
  \centering
  \begin{tikzcd}
  Y \ar[r, hook, "κ"] & X \\
  A \ar[u, hook, "ι"] \ar[r, "\mathbf{id}_{A}"] & A \ar[u, hook, "ι⨾ κ"]
\end{tikzcd}
  % \caption{\label{fi♭_:14}}
\end{figure}

is a pullback square, so that the restriction \(\mathbf{SOb}^{*}_{\mathcal{C}}|_{\overset{\hookrightarrow}{\mathcal{C}}}\) of \(\mathbf{SOb}^{*}_{\mathcal{C}}\) to the monic part \(\overset{\hookrightarrow}{\mathcal{C}}\) of \(\mathcal{C}\) maps \(ι^{*}\) to the right adjoint of
\(ι_{*}\). Thus, \(\mathbf{SOb}^{*}_{\mathcal{C}}\) lifts with the covariant subobject functor from the last section to a functor
\(\mathcal{C}^{op}|_{\overset{\hookrightarrow}{\mathcal{C}}} \hookrightarrow \mathbf{UBP}_{l\text{-}adj}\) into the category on upward-bounded preorders whose morphisms \((l, r){:} P → Q\) are pairs of left-right adjoints (going in the left direction) in
\(\mathbf{UBP}\) , or equivalently the subcategory of \(\mathbf{UBP}\) on those morphisms that are left adjoints. Thus, we can compare smallness and density along all morphisms both covariantly (through postcomposition) and contravariantly (through pullback).

\paragraph{Density and Smallness of Coproducts}
\label{sec:density-smallness-coproducts}

In the presence of pullbacks, we can sharpen the behavior of density and smallness in coproducts. The following standard result characterizes the subobject structure of a coproduct in an extensive category.

\begin{Proposition}
  \label{prp:1-restated}
\index{extensive category!subobjects}
  Let \(\mathcal{C}\) be an extensive category.
  For any coproduct \(X + Y\) with coprojections
  \[
    τ_X {:} X → X+Y,\qquad τ_Y {:} Y → X+Y,
  \]
  the canonical map
  \[
    (τ_X^*, τ_Y^*) {:} \mathbf{SOb}_{\mathcal{C}}(X+Y) \longrightarrow \mathbf{SOb}_{\mathcal{C}}(X) × \mathbf{SOb}_{\mathcal{C}}(Y)
  \]
  is an isomorphism of preorders.
\end{Proposition}
\begin{proof}
  Let \(m {:} A \hookrightarrow X+Y\) be a subobject. By extensivity, the pullbacks \(m_X {:} A_X \hookrightarrow X\) and \(m_Y {:} A_Y \hookrightarrow Y\) along the coprojections form a coproduct \(A ≅ A_X + A_Y\). This establishes a bijection between subobjects of \(X+Y\) and pairs of subobjects of \(X\) and \(Y\).
  Conversely, given a pair \((a {:} A \hookrightarrow X, b {:} B \hookrightarrow Y)\), the universal property of the coproduct induces a monomorphism \(a + b {:} A + B \hookrightarrow X + Y\) (the coproduct of monomorphisms in an extensive category is monic). These constructions are clearly order-preserving and mutually inverse, yielding the isomorphism.
\end{proof}
\begin{Remark}
  \label{rmk:10}
  Thus, in an extensive category, the subobject functor converts coproducts into products, behaving much like a powerset. This preservation of limits allows us to analyze smallness and density component-wise.
\end{Remark}

We can now describe the behavior of smallness under coproducts:

\begin{Corollary}
  \label{coR_:5}
\index{small element!in coproducts}
  Let \(\mathcal{C}\) be extensive. For a monomorphism \(ι {:} A \hookrightarrow X+Y\) with associated decomposition \(A ≅ A_X + A_Y\) under \cref{prp:1-restated}, the following are equivalent:
  \begin{enumerate}
  \item \(ι\) is small in \(\mathbf{SOb}(X+Y)\);
  \item Both pullbacks \(ι_X {:} A_X \hookrightarrow X\) and \(ι_Y {:} A_Y \hookrightarrow Y\) are small in \(\mathbf{SOb}_{\mathcal{C}}(X)\) and \(\mathbf{SOb}_{\mathcal{C}}(Y)\) respectively.
  \end{enumerate}
\end{Corollary}
\begin{proof}
  Under the isomorphism \(\mathbf{SOb}(X+Y) \cong \mathbf{SOb}_{\mathcal{C}}(X) × \mathbf{SOb}_{\mathcal{C}}(Y)\), joins are computed componentwise. Thus, \(ι ∨ κ = \top_{X+Y}\) holds if and only if \(ι_X ∨ κ_X = \top_X\) and \(ι_Y ∨ κ_Y = \top_Y\).
  \(ι\) is small precisely when the only pair \((κ_X, κ_Y)\) satisfying this condition is \((\top_X, \top_Y)\). This is equivalent to the requirement that there is no proper \(κ_X\) in \(X\) joining with \(ι_X\) to \(\top_X\) (and similarly for \(Y\)); otherwise, taking \((κ_X, \top_Y)\) would contradict the smallness of \(ι\). Hence, both components must be small.
\end{proof}

Consequently, in an extensive category, the smallness of a subobject of a coproduct is determined entirely by its restrictions to the summands. The behavior of density is then the abstractly same:

\begin{Corollary}
  \label{crl:6}
\index{density!in coproducts}
  Let \(\mathcal{C}\) be an extensive category. Then each subobject \(ι{:} A \hookrightarrow X + Y\) of a coproduct \(X + Y\) is dense if and only if both pullbacks \(τ_{X}^{*}(ι)\) and \(τ_{Y}^{*}(ι)\) are dense in \(\mathbf{SOb}_{\mathcal{C}}(X)\) and \(\mathbf{SOb}_{\mathcal{C}}(Y)\) respectively.
\end{Corollary}
\begin{proof}
  By \cref{prp:1-restated}, we identify \(\mathbf{SOb}(X+Y)\) with the product \(\mathbf{SOb}_{\mathcal{C}}(X) × \mathbf{SOb}_{\mathcal{C}}(Y)\), where the bottom element is \((\bot_X, \bot_Y)\) and meets are computed pointwise. An element \((d_X, d_Y)\) in a product of downward-bounded preorders is dense if and only if for any \((a, b)\), the condition \((a ∧ d_X, b ∧ d_Y) = (\bot_X, \bot_Y)\) implies \((a, b) = (\bot_X, \bot_Y)\).

  This condition holds precisely when \(d_X\) is dense in \(\mathbf{SOb}_{\mathcal{C}}(X)\) and \(d_Y\) is dense in \(\mathbf{SOb}_{\mathcal{C}}(Y)\): testing with \(b = \bot_Y\) shows that \(a ∧ d_X = \bot_X\) forces \(a = \bot_X\), and symmetrically for \(b\). The claim follows by transporting this characterization back along the isomorphism.
\end{proof}

but different with respect to component inclusions:

\begin{Corollary}
  \label{coR_:6}
\index{small element!in coproducts}
  If \(ι{:} A \hookrightarrow X\) is small in \(X\) and \(\mathcal{C}\) is extensive, then \(ι⨾ τ_{X}{:} A \hookrightarrow X + Y\) is small for each \(Y\).
\end{Corollary}

\subsubsection{Extension to Regular Categories}
\label{sec:extens-regul-categ}

Now, since each subobject meet is a limit, it is preserved by pullback, so that \(f^{*}\) preserves meets. Thus, if \(\mathcal{C}\) has an initial object \(∅\) and \(f^{*}\)
reflects it, then it fulfills the requirements of \cref{prop:rel-reflect-small} and thus \emph{also} reflects density. In particular, if \(f\) is left orthogonal in a factorization system \(\mathfrak{f}\) whose right orthogonal
morphisms includes all initial morphisms and whose left-orthogonal cells are stable under pullback, then for each \(ι{:} A \hookrightarrow X\) such that \(f^{*}(ι) = ι_{∅} {:} ∅ \hookrightarrow Y\), the left square in the diagram

\begin{figure}[H]
\index{regular category}
  \centering
  \begin{tikzcd}
    A  \ar[r, "\mathbf{id}_{A}"] & A \ar[r, hook, "ι"] & X \\
    ∅ \ar[r, "\mathbf{id}_{∅}"] \ar[u, "ι_{∅}"] & ∅ \ar[u, "ι^{*}(f)"] \ar[r, hook, "f^{*}(ι)"] & Y \ar[u, "f"]
  \end{tikzcd}
  % \caption{\label{fi♭_:18}}
\end{figure}

is a factorization square and thus there exists a \(k{:} A → ∅\). Thus:

\begin{Corollary}
  \label{crl:5}
\index{density!reflection by regular epimorphisms}
  In a category \(\mathcal{C}\) with a strict initial object, let \(\mathfrak{f}\) be a factorization system whose right orthogonal part includes all initial morphisms. Then for each morphism \(f\) in
  the left part of \(\mathfrak{f}\)  the induced pullback morphism \(f^{*}\) is density-reflecting.
\end{Corollary}

 There is an obvious candidate for such a factorization, which is the regular-epi/mono-factorization that exists in every
regular category \cite{granIntroductionRegularCategories2021}. Remember that a regular category is, roughly speaking, a setting where we can take images of maps in a well-behaved, pullback-stable way:

\begin{Definition}[Regular category]
  A category \(\mathcal{C}\) is \emph{regular} if:

  \begin{enumerate}
\index{regular category}
  \item \(\mathcal{C}\) has all finite limits.

  \item Every morphism \(f {:} X → Y\) admits a factorization
    \[
      X \xrightarrow{e_f} I_f \xrightarrow{\mathbf{Im}(f)} Y
    \]
    where
    \begin{itemize}
    \item \(e_f\) is a \emph{regular epimorphism} (i.e. a coequalizer of some parallel pair),
    \item the \emph{image} \(\mathbf{Im}(f)\) of \(f\) is a monomorphism,
    \end{itemize}
    and this factorization is \emph{stable under pullback}: for any morphism \(g {:} Z → Y\), if we form the pullback
    \[
      \begin{tikzcd}
        X' \ar[r] \ar[d] & X \ar[d,"f"] \\
        Z \ar[r,"g"'] & Y
      \end{tikzcd}
    \]
    then the induced map \(X' → Z\) again factors as a regular epi followed by a mono, and this factorization is the pullback of \(e_f\) and \(\mathbf{Im}(f)\).
  \end{enumerate}
\end{Definition}

 So given a regular category \(\mathcal{C}\), for each \(f{:} X → Y\) the morphism \(\mathbf{SOb}^{*}_{\mathcal{C}}(f){:} \mathbf{SOb}_{\mathcal{C}}^{*}(Y) → \mathbf{SOb}_{\mathcal{C}}^{*}(X)\) has a left adjoint given by mapping \(ι\) for
\(ι{:} A \hookrightarrow X\) to the element \(\mathbf{Im}(ι⨾ f)\) of \(\mathbf{SOb}_{\mathcal{C}}(Y)\).

\paragraph{The Density Coverage}
Let now \(f{:} X → Y\) be a morphism in a regular category \(\mathcal{C}\) and let \(ι{:} A \hookrightarrow Y, κ{:} B \hookrightarrow Y\) be such that \(ι ∧ κ = ∅\). If the initial object of \(\mathcal{C}\) is strict, then
\((ι ∧ κ)^{*}(f)\) is \(\mathbf{id}_{∅}\) (up to equivalence). Since pullbacks commute with each other,
\[f^{*}(ι) ∧ f^{*}(κ) = f^{*}(ι ∧ κ) = f^{*}(∅).\]
Since the source of \((ι ∧ κ)^{*}(f)\) is the same as that of \(f^{*}(∅)\), meaning \(∅\), we have \(f^{*}(ι) ∧ f^{*}(κ) = \mathbf{id}_{∅}\).

So if \(ι\) is dense and
\(f^{*}(ι) ∧ λ = ∅\) for some \(λ{:} C \hookrightarrow X\) then \(ι^{*}(λ⨾ f)\) factors through \(∅\) and thus is initial if \(∅\) is strict. Since \(ι\) is dense it follows that \(\mathbf{Im}(λ⨾ f) = ∅\), so \(λ \leq f^{*}(\mathbf{Im}(f(λ))) = f^{*}(∅) = ∅\). Thus:

\begin{Corollary}
  \label{crl:4}
\index{density!preservation by pullbacks}
  Given a morphism \(f{:} X → Y\) in a regular category \(\mathcal{C}\) with a strict initial object, the induced morphism \(f^{*}{:} \mathbf{SOb}_{\mathcal{C}}^{*}(Y) → \mathbf{SOb}_{\mathcal{C}}^{*}(X)\) is
  density-preserving. In particular, if \(g\) has a dense image \(m\), then since the regular-epi/mono factorization \(g = e⨾ m\) is functorial, the pullback \(f^{*}(g)\) of \(g\) has a dense
  image \(\mathbf{Im}(f^{*}(g)) = f^{*}(m)\).

  Thus, and by \cref{crl:5}, the contravariant subobject morphism \(e^{*}\) of a regular epimorphism \(e\) is a density-equivalence.
\end{Corollary}
\begin{Remark}
  \label{rmα_:11}
  When \(ι{:} A \hookrightarrow X\) is a monomorphism in a regular category with a strict initial object then the density-reflection along \(ι\) is a special case of \cref{crl:4} but it is important to note that although \(f^{*}\) is density-preserving, \(f_{*}\) is not always density-reflecting.
\end{Remark}

Thus, we can make density into a coverage:

\begin{Definition}
  \label{dfn:6}\index{density coverage}
  Given a regular category \(\mathcal{C}\) with  a strict initial object, the \emph{density coverage} is given by the morphisms in \(\mathcal{C}\) whose image is dense.
\end{Definition}

Note that the density coverage is decidedly \emph{not} subcanonical (since each dense morphism in \(\mathcal{C}\) is a cover for this coverage). So by itself, sheafification of this site is
more about inferring stuff from dense subobjects:

\begin{Example}
  \label{xmp_:7}
\index{density coverage}
\index{rational map}
Let \(k\) be a field, \(\mathbf{Var}_{k}\) the category of varieties over \(k\) and for each \(V{:} \mathbf{Var}_{k}\) let \(\mathbf{O}(V)\) be given by the preorder of open
subvarieties of \(V\). Fix some variety \(Y\) over \(k\) and let \(\mathbf{Var}_{k, V}(-, Y){:} \mathbf{O}(V)^{op} → \mathbf{Set}\) be the composition of the forgetful functor
\(\mathbf{s}{:} \mathbf{O}(V) → \mathbf{Var}_{k}\) with the Yoneda functor
  of \(Y\), which assigns to each \(V\) the set of algebraic maps from \(V\) to \(Y\). The sheafification \(\widehat{\mathbf{Var}_{k}(-, Y)}\) of \(\mathbf{Var}_{k}(-, Y)\) with
  respect to the density coverage on \(\mathbf{O}(V)\) freely
  adds for each open dense \(d{:} V' \hookrightarrow V\) and algebraic map \(f{:} V' → Y\) an element \(\hat{f}\) to \(\widehat{\mathbf{Var}_{k}(V, Y)}\) such that if \(e{:} V'' \hookrightarrow V'\) is dense in \(V'\), then
  \(\widehat{e⨾ f} = \hat{f}\). It follows that \(\hat{f}\) represents a rational map from \(V\) to \(Y\).
\end{Example}

However, the density coverage can also be made subcanonical by intersecting it with the canonical coverage.

\paragraph{Density is Preserved by Regular Epimorphisms}

But density is \emph{also} covariantly preserved by regular epimorphisms. For this, fix a regular epi \(ε{:}X\twoheadrightarrow Y\). For each
subobject \(ι{:}A\hookrightarrow X\) in \(\mathbf{SOb}_{\mathcal{C}}(X)\), define
\[
  \mathbf{SOb}_{\mathcal{C}}(ε)(ι) \;:=\; \mathbf{Im}(ι⨾ε) \in \mathbf{SOb}_{\mathcal{C}}(Y).
\]
Because \(\mathcal{C}\) is regular, images exist and are stable under pullback; \(\mathbf{SOb}_{\mathcal{C}}(ε)\) is monotone and preserves finite joins, and since \(\mathcal{C}\) is strict initial, it also
preserves \(\bot\). We need:

\begin{Lemma}
  \label{lem:SOb-epi}
  Let \(\mathcal{C}\) be a regular category and \(ε {:} X \twoheadrightarrow Y\) a regular epimorphism. Then the induced monotone map
  \[
    \mathbf{SOb}_{\mathcal{C}}(ε) {:} \mathbf{SOb}_{\mathcal{C}}(X) \to \mathbf{SOb}_{\mathcal{C}}(Y),\qquad
    ι \mapsto \,\mathbf{Im}(ι⨾ε)\,
  \]
  is surjective on elements.
\end{Lemma}
\begin{proof}
  To see this, take any subobject \(κ{:} B \hookrightarrow Y\) and form its pullback along \(ε\), yielding a subobject \(ι{:} A \hookrightarrow X\) and a pullback square where \(ε'{:} A \to B\) is a regular epimorphism. Since \(ε'⨾ κ = ι⨾ ε\), the image of \(ι⨾ ε\) is isomorphic to \(κ\). Thus, \(\mathbf{SOb}(ε)(ι) = κ\), proving surjectivity.
\end{proof}

\begin{Proposition}[Density is covariantly preserved by regular epimorphisms]
  \label{prop:density-regular-epi}
\index{density!preservation by regular epimorphisms}
  Let \(\mathcal{C}\) be a regular category with strict initial object, and \(ε{:}X\twoheadrightarrow Y\) a regular epimorphism. If \(ι{:}A\hookrightarrow X\) is dense in \(\mathbf{SOb}_{\mathcal{C}}(X)\), then its image \(\mathbf{Im}(ι⨾ε)\hookrightarrow Y\) is dense in \(\mathbf{SOb}_{\mathcal{C}}(Y)\).
\end{Proposition}
\begin{proof}
  Let \(ι\in\mathbf{SOb}_{\mathcal{C}}(X)\) be dense. Then \(\mathbf{SOb}_{\mathcal{C}}(ε)(ι) = \mathbf{Im}(ι⨾ε)\) is dense in \(\mathbf{SOb}_{\mathcal{C}}(Y)\) by \cref{coR_:density-preserved}, specialized to \(f=\mathbf{SOb}_{*}(ε)\),
  using the regular-epi properties noted above and the fact that regular epimorphisms in a strict initial regular category reflect bottoms, as we saw at the section beginning.
\end{proof}
\begin{Remark}
  \label{rmα_:12}
  So density is stable under composition with regular epimorphisms and thus \(e^{*}\) is a density-equivalence and \(e_{*}\) preserves density. The only thing \(e\) does not do is induce a density-reflecting \(e_{*}\).
\end{Remark}

\begin{Corollary}
  \label{crl:8}
  Let \(e{:} X → Y\) be a regular epimorphism in a regular category. Then the pullback morphism \(e^{*}(\mathbf{SOb}_{\mathcal{C}}(Y)) → \mathbf{SOb}_{\mathcal{C}}(X)\) reflects satedness.
\end{Corollary}
\begin{proof}
  \label{prf:10}
  By \cref{rmk:dual-completeness-under-morphisms} we have to show that \(e^{*}\) preserves meets, preserves density and is injective. Meet-preservation is clear because limits
  commute with each other. Density-preservation follows from \cref{crl:4}. Finally, order-reflection is proved below.
\end{proof}

This lemma may well have been stated somewhere else already in one form or another:

\begin{Lemma}
  \label{lem:pb-reg-epi-reflect}
\index{regular epimorphism!subobject reflection}
  Let \(\mathcal{C}\) be a regular category and \(e {:} X \twoheadrightarrow Y\) a regular epimorphism. Then the pullback functor
  \[
    e^{*} {:} \mathbf{SOb}_{\mathcal{C}}(Y) \longrightarrow \mathbf{SOb}_{\mathcal{C}}(X)
  \]
  is order-reflecting: for subobjects \(ι, ι'\) of \(Y\), if \(e^{*}(ι) \leq e^{*}(ι')\) in \(\mathbf{SOb}_{\mathcal{C}}(X)\), then \(ι \leq ι'\) in \(\mathbf{SOb}_{\mathcal{C}}(Y)\).
\end{Lemma}
\begin{proof}
  Let \(ι {:} A \hookrightarrow Y\) and \(ι' {:} A' \hookrightarrow Y\) be subobjects of \(Y\), and form the pullbacks of \(e\) along each:
  \[
    \begin{tikzcd}
      P \ar[r, "e^{*}(ι)"] \ar[d, "ι^{*}(e)"'] & X \ar[d, "e"] \\
      A \ar[r, "ι"'] & Y
    \end{tikzcd}
    \qquad
    \begin{tikzcd}
      P' \ar[r, "e^{*}(ι')"] \ar[d, "ι'^{*}(e)"'] & X \ar[d, "e"] \\
      A' \ar[r, "ι'"'] & Y
    \end{tikzcd}
  \]
  Since pullbacks of monomorphisms are monomorphisms, \(e^{*}(ι)\) and \(e^{*}(ι')\) are monic; and since regular epimorphisms are stable under pullback in a regular category,
  \(ι^{*}(e) {:} P \to A\) and \(ι'^{*}(e) {:} P' \to A'\) are regular epimorphisms. The pullback squares give
  \[
    ι^{*}(e) ⨾ ι \;=\; e^{*}(ι) ⨾ e, \qquad ι'^{*}(e) ⨾ ι' \;=\; e^{*}(ι') ⨾ e.
  \]

  Suppose \(e^{*}(ι) \leq e^{*}(ι')\), so there exists \(h {:} P \to P'\) with \(h ⨾ e^{*}(ι') = e^{*}(ι)\). Then:
  \[
    ι^{*}(e) ⨾ ι \;=\; e^{*}(ι) ⨾ e \;=\; (h ⨾ e^{*}(ι')) ⨾ e \;=\; h ⨾ (e^{*}(ι') ⨾ e) \;=\; h ⨾ (ι'^{*}(e) ⨾ ι') \;=\; (h ⨾ ι'^{*}(e)) ⨾ ι'. \tag{\(\star\)}
  \]

  Since \(ι^{*}(e)\) is a regular epimorphism, it is the coequalizer of some pair \(r, s {:} R \rightrightarrows P\). By definition, \(r ⨾ ι^{*}(e) = s ⨾ ι^{*}(e)\), hence
  \(r ⨾ ι^{*}(e) ⨾ ι = s ⨾ ι^{*}(e) ⨾ ι\). Applying \((\star)\) to both sides:
  \[
    r ⨾ (h ⨾ ι'^{*}(e)) ⨾ ι' \;=\; s ⨾ (h ⨾ ι'^{*}(e)) ⨾ ι'.
  \]
  Since \(ι'\) is monic, we may cancel it on the right to obtain \(r ⨾ (h ⨾ ι'^{*}(e)) = s ⨾ (h ⨾ ι'^{*}(e))\). By the universal property of the coequalizer, there exists a unique
  \(g {:} A \to A'\) with \(ι^{*}(e) ⨾ g = h ⨾ ι'^{*}(e)\). Composing both sides with \(ι'\) and using \((\star)\):
  \[
    ι^{*}(e) ⨾ (g ⨾ ι') \;=\; (h ⨾ ι'^{*}(e)) ⨾ ι' \;=\; ι^{*}(e) ⨾ ι.
  \]
  Since \(ι^{*}(e)\) is an epimorphism, we conclude \(g ⨾ ι' = ι\), which is precisely \(ι \leq ι'\).

  The equality-reflection statement follows by applying the argument in both directions.
\end{proof}

\paragraph{Regular Quotients}

Regular categories are also pleasant because their quotient maps descend along the subobject functor:

\begin{Definition}
  \label{dfn:9}
\index{regular quotient}
A \emph{regular quotient} in a regular category \(\mathcal{C}\) is the coequalizer \(e{:} X \twoheadrightarrow Y\) of its kernel pair \((ι, κ){:} R \rightrightarrows X\).
\end{Definition}

\begin{Proposition}[Exactness of the kernel pair square]
  \label{prop:exact-square}
  Let \(\mathcal{C}\) be a regular category, and let
  \begin{figure}[H]
    \centering
      \begin{tikzcd}
    R \ar[r,"ι"] \ar[d,"κ"'] & X \ar[d,two heads,"e"] \\
    X \ar[r,two heads,"e"']          & Y
  \end{tikzcd}
    % \caption{\label{fi♭_:20}}
  \end{figure}
  be the commutative square formed by a regular epimorphism \(e\) and its kernel pair \((ι, κ)\). Then this square is both a pullback and a pushout.
\end{Proposition}
\begin{proof}
  This is a standard fact in regular categories: every regular epimorphism is the coequalizer of its kernel pair. By definition, a kernel pair is a pullback of \(e\) along itself, so the square is a pullback. Because \(e\) is the coequalizer of \(ι\) and \(κ\), the square is also a pushout.
\end{proof}

\begin{Proposition}[Subobjects of a regular quotient]
  \label{prop:subobject-quotient}
  If \(\mathcal{C}\) is a regular category, then the subobject functor \(\mathbf{SOb}_{\mathcal{C}}\) preserves regular quotients. Specifically, in the situation of \cref{prop:exact-square}, the induced diagram
  \[
    \mathbf{SOb}_{\mathcal{C}}(R) \;\;\xrightrightarrows[\kappa_{*}]{ι_{*}}\;\; \mathbf{SOb}_{\mathcal{C}}(X) \;\xtwoheadrightarrow{e_{*}}\; \mathbf{SOb}_{\mathcal{C}}(Y)
  \]
  is a coequalizer diagram in \(\mathbf{UBP}\).
\end{Proposition}
\begin{proof}
  Since \(ι ⨾ e = κ ⨾ e\), the map \(e_{*}\) coequalizes \(ι_{*}\) and \(κ_{*}\): for every subobject \(r ∈ \mathbf{SOb}_{\mathcal{C}}(R)\),
  \[
    e_{*}\bigl(ι_{*}(r)\bigr)
    = \mathbf{Im}\bigl(\mathbf{Im}(r ⨾ ι) ⨾ e\bigr)
    = \mathbf{Im}(r ⨾ ι ⨾ e)
    = \mathbf{Im}(r ⨾ κ ⨾ e)
    = e_{*}\bigl(κ_{*}(r)\bigr),
  \]
  since images are unchanged by composition with epimorphisms. Conversely, let \(h{:} \mathbf{SOb}_{\mathcal{C}}(X) → P\) be a monotone map with \(ι_{*} ⨾ h = κ_{*} ⨾ h\). By \cref{lem:SOb-epi}, \(e_{*}\) is surjective on elements, so a monotone map
  \(k{:} \mathbf{SOb}_{\mathcal{C}}(Y) → P\) with \(e_{*} ⨾ k = h\), if it exists, is forced on elements: \(k(y) := h(x)\) for any \(x\) with \(e_{*}(x) = y\). All the work lies in verifying that this prescription is independent of the chosen preimage and monotone in \(y\).

  \emph{Independence of the chosen preimage.} Suppose \(x{:} A \hookrightarrow X\) and \(x'{:} A' \hookrightarrow X\) satisfy \(e_{*}(x) = e_{*}(x') = y{:} B \hookrightarrow Y\). By uniqueness of image factorizations, \(x ⨾ e\) and \(x' ⨾ e\) factor through their common image \(y\) as
  \[
    A \xrightarrow{s} B \xhookrightarrow{y} Y,
    \qquad
    A' \xrightarrow{s'} B \xhookrightarrow{y} Y,
  \]
  where \(s\) and \(s'\) are regular epimorphisms. Form the pullback \(P := A \times_{B} A'\) with projections \(p_{1}{:} P → A\) and \(p_{2}{:} P → A'\); by stability of regular epimorphisms under pullback, both \(p_{1}\) and \(p_{2}\) are regular epimorphisms. The pullback identity \(p_{1} ⨾ s = p_{2} ⨾ s'\) yields
  \[
    p_{1} ⨾ x ⨾ e
    = p_{1} ⨾ s ⨾ y
    = p_{2} ⨾ s' ⨾ y
    = p_{2} ⨾ x' ⨾ e,
  \]
  so the universal property of the kernel pair \(R = X \times_{Y} X\) provides a map \(u{:} P → R\) with \(u ⨾ ι = p_{1} ⨾ x\) and \(u ⨾ κ = p_{2} ⨾ x'\). Let \(r{:} \bar{P} \hookrightarrow R\) be the monomorphism in the image factorization of \(u\), with epimorphic part \(P \twoheadrightarrow \bar{P}\). Then
  \[
    ι_{*}(r) = \mathbf{Im}(r ⨾ ι) = \mathbf{Im}(u ⨾ ι) = \mathbf{Im}(p_{1} ⨾ x) = x,
  \]
  where the second equality holds because \(u\) differs from \(r\) by an epimorphism, and the last because \(p_{1} ⨾ x\) is already a regular-epi/mono factorization. By the same token, \(κ_{*}(r) = x'\). The hypothesis \(ι_{*} ⨾ h = κ_{*} ⨾ h\) now gives \(h(x) = h(x')\), as required.

  \emph{Monotonicity.} Let \(y \leq y'\) in \(\mathbf{SOb}_{\mathcal{C}}(Y)\), and choose preimages \(x, x'\) with \(e_{*}(x) = y\) and \(e_{*}(x') = y'\). Put \(z := e^{*}(y)\) and \(z' := e^{*}(y')\), the pullbacks along \(e\). The composite \(z ⨾ e\) is the pullback of the regular epimorphism \(e\) along \(y\), hence a regular epimorphism followed by \(y\); therefore \(e_{*}(z) = y\), and similarly \(e_{*}(z') = y'\). By independence of the chosen preimage, \(h(x) = h(z)\) and \(h(x') = h(z')\). Finally, \(y \leq y'\) implies \(z = e^{*}(y) \leq e^{*}(y') = z'\), so monotonicity of \(h\) yields
  \[
    h(x) = h(z) \leq h(z') = h(x').
  \]

  Hence \(k\) is well defined and monotone, and by construction \(h = e_{*} ⨾ k\). Uniqueness of the factorization is immediate from the surjectivity of \(e_{*}\): any two candidates agree on every element. Thus \(\mathbf{SOb}_{\mathcal{C}}(Y)\) carries the universal property of the coequalizer of \(ι_{*}\) and \(κ_{*}\) in \(\mathbf{UBP}\).
\end{proof}

So at this point we have a fairly good picture of how density behaves in a decently well-behaved category: reflected by all morphisms, preserved by regular epimorphisms and given
componentwise in (disjoint, monic) coproducts. We have far fewer results for smallness — that is both its weakness and strength. But it is also to be expected because joins in
subobject preorders are more category-theoretically involved than meets: meets align completely with categorical pullbacks, while joins in subobject preorders do not correspond to
\emph{any} universal constructs in the ambient category \(\mathcal{C}\), they have to be added by definition.

\subsubsection{Extension to Coherent Categories}
\label{sec:extens-coher-categ}

Finally let us take a look at coherent categories. Remember that a coherent category is one in which we have well-behaved unions of subobjects:

\begin{Definition}[Coherent category]
\index{coherent category}
  A category \(\mathcal{C}\) is \emph{coherent} if:

  \begin{enumerate}
  \item \(\mathcal{C}\) is regular;

  \item for every object \(X\), the subobject poset \(\mathbf{SOb}_{\mathcal{C}}(X)\) has finite joins and for every arrow \(f {:} Y → X\), the pullback map
    \[
      f^{*} {:} \mathbf{SOb}_{\mathcal{C}}(X) → \mathbf{SOb}_{\mathcal{C}}(Y)
    \]
    preserves finite joins.
  \end{enumerate}
\end{Definition}
\begin{Remark}
  \label{rmα_:30}
  Since \(f^{*}\) commutes with other pullbacks, it also preserves finite meets, which exist since \(\mathcal{C}\) is regular, so that \(f^{*}\) is a lattice-morphism.
\end{Remark}

So in a coherent category \(\mathcal{C}\), each subobject preorder has per definition all joins and they are preserved by pullbacks. So how does this change the behavior of
smallness under morphisms? Well, for one, it fulfills the preconditions of \cref{lmm:4}, thus:

\begin{Corollary}
  \label{coR_:11}
\index{small element!preservation by monomorphisms}
  Smallness is covariantly preserved and reducedness reflected under monomorphisms in a coherent category.
\end{Corollary}

But preservation of joins under pullbacks also means that for each morphism \(f\) the induced contravariant subobject preorder morphism \(f^{*}\) fulfills \cref{item:67}. Moreover:

\begin{Lemma}
  \label{lmm:1}
  Let \(\mathcal{C}\) be a regular category and \(f {:} X → Y\) a morphism. Then the following are equivalent:
  \begin{enumerate}
\index{coherent functor}
  \item \(f\) is a regular epimorphism;
  \item the pullback functor on subobjects
    \[
      f^* {:} \mathbf{SOb}_{\mathcal{C}}(Y) → \mathbf{SOb}_{\mathcal{C}}(X)
    \]
    \emph{reflects tops}, i.e. if \(ι {:} A \hookrightarrow Y\) satisfies \(f^*(ι) = \top_X\) in \(\mathbf{SOb}_{\mathcal{C}}(X)\), then \(ι = \top_Y\).
  \end{enumerate}
\end{Lemma}
\begin{proof}
  \(\Rightarrow\): If \(f^*(ι) = \top_X\), then \(f\) factors through the subobject \(ι {:} A \hookrightarrow Y\). Since \(f\) is a regular epimorphism and \(ι\) is a monomorphism, the orthogonality property of the (regular epi, mono) factorization system supplies a morphism \(d{:}Y\to A\) with \(d⨾ι=\mathbf{id}_Y\), so \(ι\) is a split epimorphism and hence an isomorphism, so \(ι = \top_Y\).

  \(\Leftarrow\): Factor \(f\) into a regular epimorphism followed by a monomorphism, \(f = e⨾ m\). By the universal property of pullbacks, pulling back the subobject \(m\) along \(f\) yields \(f^*(m) = \top_X\). Because \(f^*\) reflects tops by hypothesis, we must have \(m = \top_Y\), meaning \(m\) is an isomorphism. Thus, \(f\) is a regular epimorphism.
\end{proof}

Thus, a regular epimorphism in a coherent category fulfills all requirements of \cref{item:global-reflect} and thus

\begin{Corollary}
  \label{crl:7}
\index{small element!reflection by pullbacks}
  For each regular epimorphism \(e{:} X → Y\) in a coherent category \(\mathcal{C}\) the pullback morphism \(e^{*}{:} \mathbf{SOb}_{\mathcal{C}}(Y) → \mathbf{SOb}_{\mathcal{C}}(X)\) is smallness-reflecting.
\end{Corollary}

Note that some care has to be taken in parsing this statement: taking the pullback reverses direction, then the reflection reverses it again. So the statement says that if the preimage of
\(ι{:} A \hookrightarrow Y\) is small, then \(ι\) is small, so taking the preimages along monomorphisms is size-increasing, which makes sense since the universal property of the pullback makes everything as large as possible.

\paragraph{Coherent functors}

We also have a nice result about many coherent functors:

\begin{Definition}[Coherent functor]
  Let \(\mathcal{C},\mathcal{D}\) be coherent categories. A functor \(F{:}\mathcal{C}→\mathcal{D}\) is \emph{coherent} if:

  \begin{enumerate}
  \item It preserves finite limits.
  \item for every object \(X{:}\mathcal{C}\), the induced map on subobjects
    \[
      \mathbf{SOb}_{\mathcal{C}}(X) \xrightarrow{\,\mathbf{SOb}(F)_X\,} \mathbf{SOb}_{\mathcal{D}}(F(X)),\quad ι \mapsto F(ι)
    \]
    preserves finite joins (including the initial subobject).
  \end{enumerate}

  A functor \(F{:} \mathcal{C} → \mathcal{D}\) \emph{preserves pure monics} if for each non-isomorphic monomorphism \(ι{:} A \hookrightarrow X\) in \(\mathcal{C}\) the image \(F(ι)\) is a non-isomorphic monomorphism. \(F\) is \emph{purely
    monomorphism-preserving} if it preserves monomorphisms and pure monics.
\end{Definition}

\begin{Lemma}
  \label{lem:pure-monic-fiber}
  Let \(F{:}\mathcal{C}→\mathcal{D}\) be a functor that
  \begin{enumerate}
  \item preserves monomorphisms;
  \item preserves pure monics.
  \end{enumerate}
  Fix \(X{:}\mathcal{C}\). Then the fiber
  \[
    (\mathbf{SOb}(F)_X)^{-1}(\top_{F(X)}) \;\subseteq\; \mathbf{SOb}_{\mathcal{C}}(X)
  \]
  consists only of \(\top_X\); i.e. the only subobject \(ι{:}A\hookrightarrow X\) with \(F(ι) = \top_{F(X)}\) is \(ι = \top_X = \mathbf{id}_X\).
\end{Lemma}
\begin{proof}
  Let \(ι{:}A\hookrightarrow X\) be a subobject such that \(F(ι) = \top_{F(X)}\). This means that \(F(ι){:}F(A)\hookrightarrow F(X)\) is an isomorphism. But since \(F\) preserves pure monics, \(ι\) must itself be an
  isomorphism in \(\mathcal{C}\) and thus \(ι=\top_X\) in \(\mathbf{SOb}_{\mathcal{C}}(X)\) (up to univalence).
\end{proof}

\begin{Proposition}
  \label{■pn_:3}
\index{small element!reflection by coherent functors}
  Let \(F{:} \mathcal{C} → \mathcal{D}\) be a coherent functor that is purely monomorphism-preserving. Then for each \(X{:} \mathcal{C}\) the induced morphism \(F_{*, X}\) is smallness-reflecting.
\end{Proposition}
\begin{proof}
  \label{prf_:9}
  By \cref{lem:pure-monic-fiber}, \(F_{*, X}\) reflects \(\top_{F(X)}\) and by coherence it preserves joins, so smallness-reflection follows from \cref{prop:rel-reflect-small}.
\end{proof}

\subsection{Summary}
\label{sec:summary}

So we have four different notions: smallness, density, reducedness and satedness, which all develop differently in either variance, and we have traced them from general
categories to coherent ones. The four notions of smallness, density, reducedness, and satedness unfold in a richly asymmetric pattern across categorical settings:

In a bare category with an initial object, density is \emph{reflected by every monomorphism} \cref{coR_:9}, while smallness is not generally preserved. A subobject is reduced
precisely when its corresponding ideal embedding reflects smallness (\cref{crl:11} and \cref{crl:1}).

Adding pullbacks yields a contravariant subobject functor; if then a category \(\mathcal{C}\) has a strict initial object, then density is stable under pullback, so that dense subobjects form a
coverage. Moreover, if \(\mathcal{C}\) is extensive, then coproducts are mapped to products under the \(\mathbf{SOb}^{*}\)-functor,
which implies a behavior of smallness and density with respect to coproducts that is again abstractly the same but concretely different (\cref{prop:density-coproduct},
\cref{coR_:5}, \cref{crl:6}, \cref{coR_:6}).

In general, regular categories sharpen the picture: regular epimorphisms \emph{preserve density} (images of dense subobjects are dense, \cref{prop:density-regular-epi}) and their pullbacks are
density-equivalences (\cref{crl:5} and \cref{crl:4}) and \emph{reflect satedness} (\cref{crl:8}).

Coherent categories bring smallness into better control: in a modular subobject lattice, monomorphisms \emph{preserve smallness}
(\cref{lmm:4,crl:2}) and \emph{reflect reducedness} (by \cref{prop:reducedness-under-morphisms} together with \cref{crl:2}). Regular epimorphisms now also \emph{reflect smallness} via
pullback, thanks to the join-preservation (\cref{crl:7}). Moreover, coherent functors reflect smallness pointwise (\cref{■pn_:3}).

Heyting categories also allow a pseudocomplement characterization of density (\cref{prop:pseudocomplement-density}), which shows that \emph{satedness is reflected by monomorphisms}
(\cref{coR_:7}). Moreover, we will see in upcoming work that co-Heyting categories decompose every object into a small and a reduced
part.

% Combining this with \cref{crl:8} and the regular epi/mono-factorization in a regular category we get that satedness is preserved by pullbacks — explaining
% why, for example, satedness is a local property in topology and giving us another coverage.

% Coproducts with disjoint monic coprojections behave like products of subobject lattices: density, smallness, reducedness, and satedness are all \emph{componentwise} in the summands (\cref{prp:1-restated,crl:6,crl:10}).

Throughout, density emerges as the more rigid, globally behaved notion-reflected by all monos, preserved by pullbacks and regular epis, and forming a coverage-while smallness is
more local, requiring modularity or co-Heyting structure to really come into itself.

Let us summarize these results in a comprehensive table:

\renewcommand{\arraystretch}{1.3}
\begin{table}[ H]
  \centering
  \index{smallness and density table}
  \begin{tabular}{p{4cm}cccc}
    \toprule
    \backslashbox{\textbf{Context}}{\textbf{Type}} & \textbf{Smallness} & \textbf{Density} & \textbf{Reducedness} & \textbf{Satedness} \\
    \midrule

    \quad reduced subobject & \(\mathbf{refl}_{*}\) & - & - & - \\
    \midrule

    \textbf{Initial object} \\
    \quad monomorphism & - & \(\mathbf{refl}_{*}\)  & - & - \\
    \midrule

    \textbf{Pullbacks and strict \(∅\)} \\
    \quad complete functor that reflects \(∅\) & - &  \(\mathbf{refl}\) along \(F_{*}\) & - & - \\
    \midrule

    % \quad disjoint coproduct \(X+Y\) (monic coprojections) & componentwise & componentwise & componentwise & componentwise \\
    % \midrule

    \textbf{Regular with strict \(∅\)} \\
    \quad arbitrary morphism & - & \(\mathbf{prv}^{*}\)  & - & - \\
    \quad regular epi  & - & \(\mathbf{prv}_{*} + \mathbf{refl}^{*}\) & - & \(\mathbf{refl}^{*}\) \\
    \midrule

    \textbf{Coherent} \\
    \quad monomorphism  & \(\mathbf{prv}_{*}\)  & - & \(\mathbf{refl}_{*}\) & - \\
    \quad regular epi  & \(\mathbf{refl}^{*}\) & - & - & - \\
    \quad purely-monic-preserving coherent functor & \(\mathbf{refl}\) along \(F_{*}\) & - & - & - \\
    \midrule

    \textbf{Heyting} \\
    \quad monomorphism  & - & - & - & \(\mathbf{refl}_{*}\) \\
    \midrule

    \textbf{Co-Heyting} \\
    \quad Arbitrary object & small+red-dcomp. & - & small+red-dcomp.  & - \\

    % \textbf{(Co)products} \\
    % \quad product \(X×Y\) (projections) & - & if projections have disjoint images then meet is \(0\) & - & - \\
    \bottomrule
  \end{tabular}
  \caption{Preservation and reflection of size notions in various categorical contexts.}\label{tab:size-notions-summary}
\end{table}
\renewcommand{\arraystretch}{1}

\textbf{Key:}
\begin{itemize}
  \item ``prv'' = the property is preserved.
  \item ``refl'' = the property is reflected.
  \item \((-)_{*}\) preservation/reflection happens along the induced covariant morphism in \(\mathbf{UBP}\), i.e. through postcomposition.
  \item \((-)^{*}\) preservation/reflection happens along the induced contravariant morphism in \(\mathbf{UBP}\), i.e. through pullback.
    \item \(∅\) is the initial object.
  \item ``-'' = no general result stated in the text.
\end{itemize}

\textbf{Notes:}
\begin{enumerate}
  \item All density statements assume a strict initial object so that \(\mathbf{SOb}_{\mathcal{C}}(X)\) is downward-bounded.
  \item The decomposition in co-Heyting categories is mentioned in the text as an upcoming result.
  \item Behavior on disjoint coproducts with monic coprojections is induced by the isomorphism \(\mathbf{SOb}_{\mathcal{C}}(X+Y)\cong\mathbf{SOb}_{\mathcal{C}}(X)×\mathbf{SOb}_{\mathcal{C}}(Y)\).
  \end{enumerate}

\section{Applications}
\label{sec:applications}

We provide two applications, which jointly illustrate the power of the ``tablets and scales'' approach. For this we first need a means to restrict the subobjects under
consideration however. This can be motivated by an irony in the nature of regular categories: they provide us with the incredibly convenient regular-epi/mono-factorization, but it
is actually the notion of a regular monomorphism that most closely describes an embedding in a strong sense. What is in-between the two — monomorphisms and regular monomorphisms —
is missing local structure. We saw an example of this in \(\mathbf{PrO}\) already, where the regular monomorphisms are fully faithful functors i.e. order-reflecting, but the
monomorphisms are merely injective functions. Another important example is \(\mathbf{Top}\), where a monomorphism \(ι{:} A \hookrightarrow X\) is an injective continuous function while a regular
monomorphism is a subspace embedding — in-between them are all the opens \(A\) might have but its image might not.

So we need to abstract what we are doing from the \(\mathbf{SOb}\)-functor to others:

\begin{Definition}
  \label{dfn:11}\index{submonic hyperdoctrine}
  A \emph{submonic hyperdoctrine} \(\mathcal{h}{:} \mathcal{H} \hookrightrightarrows \mathbf{SOb}^{*}_{\mathcal{C}}{:} \mathcal{C}^{op} → \mathbf{UBP}\) on a category \(\mathcal{C}\) is a monic (weak but necessarily strict up to homotopy)
  natural transformation of 2-functors into \(\mathbf{UBP}\) such that for each \(X{:} \mathcal{C}\) the top \(\top_{\mathbf{SOb}_{\mathcal{C}}(X)} = \mathbf{id}_{X}\) is in the image of \(η_{X}\).

  A submonic hyperdoctrine \(\mathcal{h}\) is \emph{categorical} if for each \(ι{:} A \hookrightarrow X\) in \(\mathcal{H}(X)\) the composite \(\mathcal{h}_{A}⨾ ι_{*}\) factors through \(\mathcal{h}(X)\), 
\end{Definition}

So \(\mathcal{h}\) is a pullback-stable collection of monomorphisms that contains all identities. \(\mathcal{h}\) is  categorical if for each \(κ{:} B \hookrightarrow A\) in \(\mathcal{H}(A)\) and
\(ι{:} A \hookrightarrow X\) in \(\mathcal{H}(X)\) the composite \(κ⨾ ι\) is in \(\mathcal{H}(X)\), i.e. if these monomorphisms are stable under composition.
Thus, each categorical submonic hyperdoctrine is equivalent to a wide monic pullback-stable subcategory of \(\mathcal{C}\).

\begin{Example}
  \label{xmp:20}
  Since regular monomorphisms are stable under pullback (since for a regular \(ι\) the pullback \(f^{*}(ι)\) equalizes the composite of \(f\) and the cokernel pair of \(ι\)),
  the functor \(ι_{reg}{:} \mathbf{SOb}_{reg}^{*} \hookrightarrow \mathbf{SOb}^{*}{:} \mathcal{C}^{op} → \mathbf{UBP}\) is a submonic hyperdoctrine.
\end{Example}

Using submonic hyperdoctrines we can establish calculi of \emph{well-behaved} subobjects.

\subsection{Open-Closed Dualities}
\label{sec:open-clos-dual}

In this section we lift the calculus of tablets and scales to a categorical level and connect it to standard topology and measure theory. As a motivation, let us consider a
preliminary definition that connects, among other things, topological and semialgebraic sets:

\begin{Definition}[Heyting spaces, preliminary]
  \label{defn_:13}\index{Heyting space}
  A \emph{Heyting space} \(X\) is an embedding of bounded lattices \(ω{:} \mathbf{O}(X) \hookrightarrow \mathbf{B}(X)\) where \(\mathbf{O}(X)\) is a bounded Heyting algebra of \emph{opens} and
  \(\mathbf{B}(X)\) is a bounded Boolean algebra.
\end{Definition}

\begin{Example}
  \label{xmp:102}
  Every topological space \(X\) is a concrete Heyting space.
\end{Example}
\begin{Example}
  \label{xmp_:11}
  For each real closed field \(R\) and finite-dimensional \(R^{n}\), the embedding \(ω_{salg}{:} \mathbf{OSSAlg}(R^{n}) \hookrightarrow \mathbf{B}(R^{n})\) of the open real semialgebraic subsets of
  \(R^{n}\) from \cref{xmp:14} into the powerset lattice of \(R^{n}\) is a concrete Heyting space.
\end{Example}
\begin{proof}
  \label{prf_:17}
  We saw in \cref{xmp:14} already that the category of semialgebraic subsets of \(R^n\) has a distinguished
  collection \(\mathbf{O}(V)\) of \emph{open semialgebraic} subsets of \(V\), closed under finite unions and finite intersections; Thus, almost all requirements of
  \cref{dfn:12} are already known. Note that \(U \subseteq R^{n}\) is semialgebraic open if and only if it is semialgebraic and open. For this assume otherwise, i.e. that \(ι\) is
  semialgebraic and open but not semialgebraic open. Since \(U\) is semialgebraic there exists a family of polynomial strict inequalities \(p_{i} > q_{i}\) and \(p'_{j} = q'_{j}\)
  such that \(U\) is the union of intersections of the sets defined by these relations. Since \(U\) is not open semialgebraic, there exist some \(j\) such that the equality \(p'_{j} = q'_{j}\) cannot
  be removed from the defining relations of \(U\), i.e. there exists a point \(x ∈ R^{n}\) such that \(p'_{j}(x) = q'_{j}(x)\) but the same does not hold for any \(p_{i}\) and
  \(q_{i}\). But then \(U\) does not contain any open disk around \(x\) and is thus not open. The implication in the other direction is trivial.

  We must still show that \(\mathbf{O}(V)\) is in fact a Heyting algebra; equivalently, that it is closed under \emph{Heyting implication}. But
  for this we can just use the Heyting implication in the normal opens
  \[
    U \Rightarrow W \;:=\; \bigl( (V \setminus U) \cup W \bigr)°.
  \]
  First, we verify that this is indeed an element of \(\mathbf{O}(V)\):

  \begin{itemize}
  \item \((V \setminus U) \cup W\) is a semialgebraic subset of \(V\), being a Boolean combination of semialgebraic sets.

  \item A classical result about semialgebraic sets is that they satisfy the Tarski-Seidenberg projection property, from which it follows that the closure, interior, and convex
    hull of a semialgebraic set are again semialgebraic (see the beginning of \cite{driesTameTopologyOminimal1998} for details). Since \(U\) is open, by definition of the topological
    interior, and semialgebraic, it is open semialgebraic, so that \(U° ∈ \mathbf{O}(V)\).
  \end{itemize}

  The argument that this has the required universal property proceeds as with the Heyting implication in a locale.
\end{proof}

A basic idea behind \cref{dfn:12} is to reduce a structure with an unbounded signature, topological spaces, with one with a finite signature, that of Heyting and Boolean algebras.
The objects under consideration, e.g. topological spaces, can still have unbounded joins and meets, but the theory only requires what is constructible under finite operations of
unions, meets and two types of complements. Moreover, the theory implies its dual: the Boolean negation maps joins to meets and vice versa. Moreover, the universal property of a
Heyting implication \(\Rightarrow\) is given by \(a \leq (b \Rightarrow c)\) if and only if \(a ∧ b \leq c\), which dualizes to \(¬ a \geq ¬ (b \Rightarrow c)\) if and only if \(¬ a ∨ ¬ b \geq ¬ c\), i.e. \(¬ (b \Rightarrow c)\)
is a co-Heyting complement and thus the sub-lattice consisting of the Boolean complements of elements in \(ω\) is a co-Heyting sub-algebra \(χ{:} \mathbf{C}(X) \hookrightarrow \mathbf{B}(X)\) of
\emph{closeds}. So \(ω\) and \(χ\) are co-defining and we could just as well have started with \(\mathbf{C}(X)\). 

But there is a deeper way to define a Heyting space, which is based on a similarity between topological spaces and the category of topological spaces \(\mathbf{Top}\) itself: an
equivalent way to state the definition of a Heyting space \(X\) is that it consists of a category \(\mathbf{B}(X)\) and for each object \(Y{:} \mathbf{B}(X)\) a triple

\begin{equation}
  \label{qn_:4}
  \mathbf{O}(Y) \xhookrightarrow{ω} \mathbf{B}(Y) \xhookleftarrow{χ} \mathbf{C}(Y)
\end{equation}

and for each \(ι{:} Z → Y\) restriction morphisms \(ι^{*}_{\mathbf{B}}{:} \mathbf{B}(Y) → \mathbf{B}(Z)\), \(ι^{*}_{\mathbf{O}}{:} \mathbf{O}(Y) → \mathbf{O}(Z)\) and
\(ι^{*}_{\mathbf{C}}{:} \mathbf{C}(Y) → \mathbf{C}(Z)\). Now, it just so happens that \(\mathbf{B}(X)\) is a preorder with a top element, but apart from those particularities the
same structure also exists on \(\mathbf{Top}\) itself: for each topological space \(X\) there exist preorder embeddings as in \cref{qn_:4}, where \(\mathbf{B}(Y)\) is the powerset
lattice \(\mathbf{P}(Y)\) of the set underlying \(Y\) and the very definition of a continuous function is such that it reflects opens and closeds and thus induces for each
\(f{:} X → Y\) morphisms \(ι^{*}_{\mathbf{O}}, ι^{*}_{\mathbf{P}}, ι_{\mathbf{C}}^{*}\) as above.

Here, \(ω\) and \(χ\) are lattice embeddings, but this is not necessarily the case:

\begin{Example}
  \label{xmp:25}
  Let \(X\) be a scheme, \(\mathbf{O}(X)\) and \(\mathbf{C}(X)\) the open and closed immersions into \(X\) and \(\mathbf{R}(X)\) the preorder of all immersions into \(X\). As per
  \cite{Section261001IM} an open immersion \(ι_{\mathbf{O}}\) of schemes is given by an open subspace
  \(\mathopen{∣}ι_{\mathbf{O}}\mathclose{∣}{:} \mathopen{∣}A\mathclose{∣} \hookrightarrow \mathopen{∣}X\mathclose{∣}\) of the underlying topological space
  \(\mathopen{∣}X\mathclose{∣}\) of \(X\) with the structure sheaf induced by restriction. A closed immersion is locally in each affine neighborhood given by an embedding of an
  affine neighborhood, so dual to a ring quotient. An immersion is a closed immersion into an open immersion. By the Zariski topology, for each open immersion \(ι_{\mathbf{O}}\)
  the complement \(\mathopen{∣}X\mathclose{∣} ∖ \mathopen{∣}ι_{\mathbf{O}}\mathclose{∣}\) in the underlying topological space of \(X\) is the underlying subset of some reduced closed
  subscheme \(ι_{\mathbf{C}}\). Then \(ι_{\mathbf{O}}\) is the pseudocomplement
  of \(ι_{\mathbf{C}}\): since the two don't share any points and every non-trivial scheme has a point (since every non-trivial ring has a maximal ideal by the axiom of choice), no
  non-trivial subscheme factors through their intersection. On the other hand, every other subscheme \(κ\) that does not factor through \(ι_{\mathbf{O}}\) shares some point \(p\)
  with \(\mathopen{∣}ι_{\mathbf{C}}\mathclose{∣}\) and thus there exists a subscheme of \(X\) under \(ι\) and \(κ\) (the maximal subscheme corresponding to \(p\)), so that
  \(ι_{\mathbf{O}}\) is in fact maximal with respect to being disjoint with \(ι_{\mathbf{C}}\).
\end{Example}

So what we have in general is not a direct bijection but a more mediated kind of correspondence:

\begin{Definition}
  \label{dfn:10}\index{open-closed duality}
  An \emph{open-closed duality} \(\mathcal{D}\) on a category \(\mathcal{C}\) is a triple \(\mathbf{o} \xhookrightarrow{ω} \mathbf{d} \xhookleftarrow{χ} \mathbf{c}\) of submonic hyperdoctrines such
  that for each \(X{:} \mathcal{C}\) :

  \begin{enumerate}
  \item \label{item:10} for each element \(c ∈ \mathbf{C}(X)\)  of the underlying preorder of \(\mathbf{c}(X){:} \mathbf{C}(X) \hookrightarrow \mathbf{SOb}(X)\) there exists an element
    \(o ∈ \mathbf{O}(X)\) such that \(ω(o)\) is the pseudocomplement of \(χ(c)\) in the source \(\mathbf{D}(X)\) of \(\mathbf{d}(X)\),

    \item\label{item:25}  and for each element \(o ∈ \mathbf{O}(X)\) there exists an element
      \(c ∈ \mathbf{C}(X)\) for which \(ω(o)\) is the pseudocomplement of \(χ(c)\).

     \item\label{item:37} \(\mathbf{d}\) is categorical.
  \end{enumerate}

  An open-closed duality \(\mathcal{D}\) is \emph{regular} if \(\mathbf{d}\) is \(ι_{reg}\).

  \(\mathcal{D}\) is \emph{reduced} if also for each \(c\) in each \(\mathbf{C}(X)\) there exists an \(o\) such that \(χ(c)\) is the pseudocomplement of \(ω(o)\).

  \(\mathcal{D}\) is \emph{Heyting} if:

  \begin{enumerate}
  \item \label{item:75} for each \(X\) it holds that \(\mathbf{D}(X)\) is a Boolean lattice, \(\mathbf{O}(X)\) a Heyting algebra and \(\mathbf{C}(X)\) a co-Heyting algebra and
    \(ω\) and \(χ\) are morphisms of bounded lattices,

  \item\label{item:76} and for \(f{:} X → Y\) in \(\mathcal{C}\) the morphisms \(\mathbf{O}(f){:} \mathbf{O}(Y) → \mathbf{O}(X)\) and \(\mathbf{C}(f){:} \mathbf{C}(Y) → \mathbf{C}(X)\) are
    lattice morphisms and \(\mathbf{D}(f){:} \mathbf{D}(Y) → \mathbf{D}(X)\) is a Boolean algebra morphism with a right adjoint.
  \end{enumerate}
\end{Definition}
\begin{Remark}
  \label{rmα_:37}
  We take the direction of pseudocomplementation in \cref{dfn:10} such that it aligns with the case of subschemes of a scheme.
\end{Remark}
\begin{Definition}
  \label{dfn:open-closed-space}\index{open-closed space}
  An \emph{open-closed space} \(X\) is an open-closed duality whose middle submonic preorder \(\mathbf{D}(X)\) is also its underlying category. Thus the data are preorder
  embeddings
  \[
    \mathbf{O}(X) \xhookrightarrow{ω} \mathbf{D}(X) \xhookleftarrow{χ} \mathbf{C}(X)
  \]
  satisfying the same pseudocomplement condition as in \cref{dfn:10}, but with \(\mathbf{D}\) itself playing the role of the ambient category of subobjects.

  A \emph{morphism of open-closed spaces} \(f{:} X → Y\) is the dual of a morphism of diagrams

  \[f^{*}_{\mathbf{O}}{:} \mathbf{O}(Y) → \mathbf{O}(X), f^{*}_{\mathbf{D}}{:} \mathbf{D}(Y) → \mathbf{D}(X), f^{*}_{\mathbf{C}}{:} \mathbf{C}(Y) → \mathbf{C}(X)\]

  from \(Y\) to \(X\) which preserves tops in \(\mathbf{O}, \mathbf{D}\) and \(\mathbf{C}\).
\end{Definition}
\begin{Example}
  \label{xmp_:2}
  Every object \(X {:} \mathcal{C}\) in the base category of an open-closed duality \(\mathbf{o} \xhookrightarrow{ω} \mathbf{d} \xhookleftarrow{χ} \mathbf{c}\) gives rise to an open-closed space
  \[
    \mathbf{O}(X) \xhookrightarrow{ω_X} \mathbf{D}(X) \xhookleftarrow{χ_X} \mathbf{C}(X),
  \]
  whose underlying category is the middle preorder \(\mathbf{D}(X)\). Thus, open-closed spaces are the fibrewise form of an open-closed duality.
  Properties such as regularity and Heytingness are inherited fibrewise whenever they are present in the ambient open-closed duality; reducedness is fulfilled by definition.
\end{Example}
\begin{Remark}
  \label{rmα_:17}
  We \emph{don't} assume that \(\mathbf{o}\) and \(\mathbf{c}\) are categorical, i.e. that opens and closeds are closed under composition because we can find significant examples
  in topoi on semialgebraic sets where subobjects can be locally tame (and thus internally open) but not globally so. We do assume \(\mathbf{d}\) is categorical though, so opens
  and closeds are at least part of some common category of subobjects.
\end{Remark} 

Of course the paradigmatic example is the topological one:

\begin{Example}
  \label{xmp:24}
  Let \(\mathbf{O}, \mathbf{C}, \mathbf{D}{:} \mathbf{Top}^{op} → \mathbf{UBP}\) be given by mapping each topological space \(X\) to its lattices of opens and closeds and subsets. Then
  \(\mathbf{O}\), \(\mathbf{C}\) and \(\mathbf{D}\) with their embeddings into \(\mathbf{SOb}_{\mathbf{Top}}\) form the components of a regular Heyting open-closed duality.
\end{Example}

But an advantage of this more restricted notion is that it can be applied in other settings:

\begin{Example}
  \label{xmp_:19}
  Open and closed immersions of schemes are also stable under pullbacks. For open immersions this follows from the pullback-stability of the opens of a topological space, for
  closed ones this is Exercise II.3.11.a of \cite{Hartshorne_engl}, see \cite{babaiAnswerClosedImmersions2014} for a proof. Now let \(ι{:} A \hookrightarrow X\) be a general immersion of
  schemes. Since \(ι\) is a composite of a closed and an open immersion, the pullback of \(ι\) along any morphism \(f{:} Y → X\) is also an immersion.
  But note that the set-theoretic image of a scheme need not be a scheme, so the morphism \(f^{*}{:} \mathbf{R}(X) → \mathbf{R}(Y)\) induced by a morphism of schemes through
  pulling back regular immersions does not generally have a right adjoint.

  Concretely, let \(f\) be the projection \(π_{1}{:} ℝ^{3} → ℝ^{2}\) and \(H\) be the hyperbola in \(ℝ^{3}\) given by the  \(xy = 1\). Then the
  set-theoretic image \(f(H)\) is \(\{(u, v)\mid u ≠ 0\} \cup \{0, 0\}\). This is not a sub-scheme of \(ℝ^{2}\) since it is not closed within any open (since each open that contains \(f(H)\)
  also contains some disk around \(0\)). Now, if the
  preimage morphism \(f^{*}\) had a right adjoint \(f_{*}\), then there would have to exist an \(f_{*}(H)\) such that \((f^{*}⨾ f_{*})(A) \leq H\) and
  \(f_{*}⨾ f^{*}⨾ f_{*}(H) = f_{*}\). If \(f_{*}(H)\) did not contain a point \(p\) in \(f(H)\), then \(p\) would not be in the image of the preimage of \(f_{*}(H)\), so
  \(f_{*}(H)\) has to contain \(f(H)\). There is a smallest sub-scheme of \(ℝ^{2}\) that contains \(f(H)\), which is its closure \(\overline{f(H)}\) of \(f(H)\) in \(ℝ^{2}\), which is
  \(ℝ^{2}\). But \(H\) is not larger than \(f^{*}(ℝ^{2}) = ℝ^{3}\), so \(f^{*}\) cannot have a right adjoint.

  There is a smallest \emph{closed} subscheme containing the image for each subscheme under a morphism, which \cite{Section29601R5} uses as an image, but the failure of this to be
  adjoint to the pullback functor on all subschemes is one of the primary reasons for the prominence of constructible closures in algebraic geometry. We will use the existence
  of such a right adjoint as a defining condition in \cref{dfn:12}.
\end{Example}

These are not particularly well-behaved preorders because the closed part is both too encompassing and too strict: it is too encompassing because many closed subschemes can
correspond to one open subscheme and too strict because there are too few closed subobjects for many topological constructions. But we can avoid this by proceeding along the lines of
\cite{delfsSemialgebraicTopologyReal1981}. But with one crucial difference: the authors introduce, somewhat \emph{ad-hoc}, the notion of a \emph{restricted topology}
as a set with a collection of subsets which is closed under finite unions and intersections. We instead use Heyting algebras as in \cref{defn_:13} but with some refinements:

\begin{Definition}
  \label{dfn:12}\index{abstract Heyting space}\index{concrete Heyting space}
  An \emph{abstract Heyting space} is a Heyting open-closed space
  \[
    \mathbf{O}(X) \xhookrightarrow{ω_X} \mathbf{D}(X) \xhookleftarrow{χ_X} \mathbf{C}(X)
  \]
  in the sense of \cref{dfn:open-closed-space}: thus \(\mathbf{D}(X)\) is a bounded Boolean lattice, \(\mathbf{O}(X)\) is a bounded Heyting algebra, \(\mathbf{C}(X)\) is a bounded co-Heyting algebra, and
  the open and closed parts are concrete duals under the complementation inside \(\mathbf{D}(X)\).

  A \emph{morphism of abstract Heyting spaces} \(f{:} X → Y\) is a morphism of open-closed spaces such that

  \begin{enumerate}
  \item \label{item:15} each component map of \(f\) is a bounded lattice morphism,

  \item\label{item:80} \(f^{*}_{\mathbf{D}}{:} \mathbf{D}(Y) → \mathbf{D}(X)\) preserves Boolean negations.
  \end{enumerate}

  A \emph{concrete Heyting space} is a set \(\mathopen{∣}X\mathclose{∣}\) with an abstract Heyting space \(X\) for which the middle Boolean lattice \(\mathbf{D}(X)\) is the subset
  preorder \(\mathbf{P}(X) := \mathbf{P}(\mathopen{∣}X\mathclose{∣})\) of \(\mathopen{∣}X\mathclose{∣}\), with \(ω_X\) and \(χ_X\) given by subset embeddings. Equivalently, it is a set \(\mathopen{∣}X\mathclose{∣}\)
  equipped with a lattice embedding \(ι_{\mathbf{O}}{:} \mathbf{O}(X) \hookrightarrow \mathbf{P}(X)\), where \(\mathbf{O}(X)\) is a Heyting algebra; so that the \emph{closed subsets} are the set-theoretic
  complements of opens, with embedding \(ι_{\mathbf{C}}{:} \mathbf{C}(X) \hookrightarrow \mathbf{P}(X)\).

  A \emph{morphism of concrete Heyting spaces} \(f {:} X → Y\) is a function \(\mathopen{∣}X\mathclose{∣} → \mathopen{∣}Y\mathclose{∣}\) such that the pullback along \(f\) induces a
  morphism of abstract Heyting algebras.
\end{Definition}
\begin{Remark}
  \label{rmk:12}
  Note that we do \emph{not} assume that \(\mathbf{O}(f)\) and \(\mathbf{C}(f)\) preserve the Heyting and co-Heyting implications, which is in line with the topological case, but we \emph{do} assume that \(\mathbf{D}(f)\) preserves
  complements, as it does in the case of pullbacks of sets (and thus regular subobjects of a topological space).
\end{Remark}
\begin{Remark}
  \label{rmk:11}
  In pointless topology, originally developed in \cite{johnstoneStoneSpaces}, the assumption is frequently made that what is important about a topological space is its locale of
  opens. Here we challenge this assumption: not just is it that the closeds matter as well, but in particular also that these two are jointly embedded into some larger algebra,
  which allows us to extend the notions we derive from opens and closeds to other elements.
\end{Remark}
\begin{Remark}
  \label{rmα_:36}
  So according to \cref{dfn:12} a concrete Heyting space is an abstract Heyting space whose middle is maximal. We could abstract this setup a little further and say that a
  \emph{concretification} of an abstract Heyting space is a morphism of Heyting spaces that is the identity on opens and closeds and an embedding of the middle Boolean algebra.
\end{Remark}
\begin{Example}
  \label{xmp:27}
  It follows almost immediately from \cref{dfn:12} that both the category of concrete Heyting spaces \(\mathbf{CHS}\) and that of abstract Heyting spaces \(\mathbf{AHS}\) house a Heyting
  open-closed duality.
  % \cref{xmp:25} is \emph{not} however, and not just because its subobject lattices are not Heyting/Boolean/Co-Heyting but also because the 
\end{Example}

One nice consequence of Heyting spaces is that satedness becomes stable under restriction to subobjects:

\begin{Definition}
  \label{defn_:10}
  Given an abstract Heyting space \(X\) and \(ι{:} \mathbf{O}(X)\), the \emph{induced Heyting space} \(X|_{ι}\) on \(ι\) is the Heyting space:

  \begin{itemize}
  \item such that \(\mathbf{D}(X)|_{ι}\) is given by the ideal \((ι)\) in \(\mathbf{D}(X)\),

  \item \(ω|_{ι} := ω_{X|_ι}\) is given by the embedding into \((ι)\) of the image of \(ω_{X}\) under the morphism \((- ∧ ι)\) given by taking meets with \(ι\).

  \item \(χ|_{ι}\) is  also given like that and that opens and closeds are complementary follows from the fact that if \(χ_{X}(C) = ¬ ω_{X}(O)\) then for each
    \(κ{:} \mathbf{D}(X)\) it holds that \(κ ∧ χ_{X}(C) = κ ∧ ¬ω_{X}(O)\) since \(\mathbf{D}(X)\) is Boolean.
  \end{itemize}

  This space comes with a morphism of Heyting spaces \(ι_{*}{:} X|_{ι} \hookrightarrow X\) where

  \begin{itemize}
  \item \((ι_{*})^{*}_{\mathbf{D}}\), \((ι_{*})^{*}_{\mathbf{O}}\) and \((ι_{*})^{*}_{\mathbf{C}}\) are given by \((- ∧ ι)\),

    \item \((ι_{*})_{*, \mathbf{D}}\) (which we'll just denote \(ι_{*}\)) is given by the ideal embedding \((ι) \hookrightarrow \mathbf{D}(X)\).
  \end{itemize}
\end{Definition}
\begin{Proposition}[Satedness restricts to Heyting subspaces]
  \label{■pn_:10}\index{satedness restricts to heyting subspaces}
  Given a concrete Heyting space \(X\), \(κ{:} \mathbf{O}(X)\) and \(ι \leq ω_{X}(κ)\), if \(ι\) is \(ω_{X}\)-sated in \(X\), then \(ι\) is \(ω|_{ι}\)-sated in \(X|_{\iota}\).
\end{Proposition}
\begin{proof}
  \label{prf_:18}
  It follows from \cref{defn_:10} and the closedness opens under meets that \(ω|_{ι}⨾ ι_{*}\) factors through \(ω_{X}\), so that there exists a diagram

  \begin{figure}[H]
    \centering
    \begin{tikzcd}
      \mathbf{O}(X|_{ι}) \ar[r, "ω|_{ι}"] \ar[d, "ι_{*}|_{\mathbf{O}(X|_{ι})}"] & \mathbf{B}(X) \ar[d, "ι_{*}"] \\
      \mathbf{O}(X) \ar[r, "ω_{X}"] & X
    \end{tikzcd}
    % \caption{\label{fi♭_:13}}
  \end{figure}
  that fulfills the preconditions of \cref{prop:completeness-reflection-dense-envelope} for a trivial scale.
\end{proof}

\subsubsection{Constructible Closures and Adjoint Dualities}
\label{sec:constr-clos-heyt}

Every concrete Heyting space has an underlying abstract Heyting space obtained by replacing \(\mathbf{P}(X)\) with the Boolean subalgebra generated by its opens and closeds. This is a
generalization of the notion of a constructible subset:

\begin{Definition}
  \label{defn_:27}\index{constructible closure}
  Given an open-closed space \(X\) such that \(\mathbf{D}(X)\) is a lattice with (open) pseudocomplements, the \emph{constructible closure} in \(\underline{X}\) is the open-closed space with the same opens
  and closeds such that the middle \(\underline{\mathbf{D}}(X)\) of \(\underline{X}\) is the joint closure of \(ω\) and \(χ\) in \(\mathbf{D}(X)\) under joins, meets and pseudocomplements.
\end{Definition}
\begin{Remark}
  \label{rmα_:5}
  If \(\mathbf{D}(X)\) does not have joins, meets or pseudocomplements, then it does still make sense to close \(ω\) and \(χ\) under whatever \(\mathbf{D}(X)\) does have and refer
  to that as the constructible closure.
\end{Remark}
\begin{Remark}
  \label{rmα_:39}
  In infinite-dimensional topology, the definition of constructibility is often restricted to those being generated by \emph{retrocompact} opens, i.e. opens that are compact on
  each compact open. We use the more naïve definition of constructibility here, since the intuitive meaning of retrocompactness in our setting is not yet clear.
\end{Remark}

So each concrete Heyting space determines an abstract Heyting space by its constructible closure.

\begin{Example}
  \label{xmp_:13}
  The constructible closure \(\underline{ω}_{salg}\) of \(ω_{salg}\) from \cref{xmp_:11} is the embedding \(\mathbf{OSSAlg}(R^{n}) \hookrightarrow \mathbf{SAlg}(R^{n})\) of
  \(\mathbf{OSSAlg}(R^{n})\) into the preorder of all semialgebraic subsets of \(R^{n}\).
\end{Example}

Now, for concrete Heyting spaces we can often adopt standard topological arguments almost verbatim:

\begin{Proposition}
  \label{■pn_:10-restated}
  Let \(f{:} X → Y\) be a morphism in \(\mathbf{CHS}\).

  \begin{enumerate}
  \item \label{item:104} \(f\) is a monomorphism in \(\mathbf{CHS}\) if and only if the underlying function \(\mathopen{∣}X\mathclose{∣} → \mathopen{∣}Y\mathclose{∣}\) is injective.

  \item \label{item:105} \(f\) is a regular monomorphism in \(\mathbf{CHS}\) if and only if the Heyting structure on \(X\) is the pullback structure induced from \(Y\), i.e.
    \[
      \mathbf{O}(X) \;\cong\; \{\, f^{*}(U) \mid U ∈ \mathbf{O}(Y)\,\}
    \]
    via the lattice embedding \(ι_{\mathbf{O},X}\), and similarly for closed sets.
  \end{enumerate}
\end{Proposition}
\begin{proof}
  We establish the two claims in turn.

  \textbf{Monomorphisms.}
  If \(f\) is monic, distinct points \(x_{1} \neq x_{2}\) in \(X\) cannot map to the same point in \(Y\); otherwise, the two maps from the terminal object \(1\) (a single point) picking \(x_{1}\) and \(x_{2}\) would compose with \(f\) to give the same morphism, violating monicity. Thus, \(f\) is injective.
  Conversely, if \(\mathopen{∣}f\mathclose{∣}\) is injective and \(f⨾g = f⨾h\), then for every \(z\) we have \(\mathopen{∣}f\mathclose{∣}(g(z)) = \mathopen{∣}f\mathclose{∣}(h(z))\), which forces \(g(z) = h(z)\). Since the underlying functions agree, their inverse images of opens must also agree, so \(g = h\) as morphisms of \(\mathbf{CHS}\).

  \textbf{Regular monomorphisms.}
  Recall that regular monomorphisms in concrete categories are equalizers.
  Suppose \(f\) is regular. As equalizers in \(\mathbf{CHS}\) are computed pointwise (set-theoretically), \(f\) represents an inclusion of a subset \(X \subseteq Y\).
  The equalizer carries the coarsest Heyting structure making the inclusion continuous; consequently, every open of \(X\) is the pullback of an
  open from \(Y\), establishing the isomorphism.

  For the converse, assume \(\mathbf{O}(X) \cong f^{*}(\mathbf{O}(Y))\).
  Let \(Z\) be the codiscrete space on the disjoint union \(Y + \{*\}_{1} + \{*\}_{2}\).
  Define two morphisms \(g, h {:} Y \to Z\) which are equal on \(X\) but map the complement of \(X\) to the two distinct points \(\{*\}_{1}\) and \(\{*\}_{2}\) respectively.
  (Continuity is guaranteed because \(Z\) is codiscrete.)
  One verifies immediately that \(X\) is the set of points where \(g\) and \(h\) agree and that \(f\) is the equalizer of \(g\) and \(h\).
\end{proof}

But the constructible closure is the greatest extent we can safely and finitistically reason about. In particular, they are what we have interior and closure operators on:

\begin{Definition}
  \label{defn_:26}
  \index{adjoint open-closed duality}
  An \emph{adjoint open-closed duality} is an open-closed duality such that

  \begin{enumerate}
  \item \label{item:125} \(ω\) has an \emph{interior} right adjoint \((-)^{\circ}{:} \mathbf{D} → \mathbf{O}\),

  \item\label{item:126} \(χ\) has a \emph{closure} left adjoint \(\overline{(-)}{:} \mathbf{D} → \mathbf{C}\).
  \end{enumerate}
\end{Definition}
\begin{Remark}
  \label{rmα_:2}
  In the Heyting case it follows from duality that if either adjoint to the embeddings of an open-closed duality exists then the other exists. For instance, if \((-)^{\circ}\) is given then for
  \(r ∈ \mathbf{D}(X)\) of some \(X\) the closure \(\overline{ℝ}\) is just the complement of \((¬ r)^{\circ}\). For non-Heyting open-closed dualities it makes sense to consider variants of \cref{defn_:26} for
  which only one adjoint is given.
\end{Remark}

\begin{Proposition}
  \label{prop:constructible-closure-adjoint}
  Let \(X\) be a Heyting space with constructible closure \(\underline{X}\) with middle \(\underline{\mathbf{D}}(X)\). Then the restricted embeddings form an adjoint open-closed duality.
\end{Proposition}
\begin{proof}
  Since \(\underline{\mathbf{D}}(X)\) is the Boolean subalgebra of \(\mathbf{D}(X)\) generated by the image of \(\omega\), every element \(k \in \underline{\mathbf{D}}(X)\) decomposes as a finite join of locally closed elements \(L = \omega(o) \wedge \neg_{\mathbf{D}} \omega(v)\) with \(o, v \in \mathbf{O}(X)\). (Here \(\neg_{\mathbf{D}}\) denotes Boolean complementation in \(\mathbf{D}(X)\), so that \(\neg_{\mathbf{D}} \omega(v) = \chi(c_{v})\) for the closed element \(c_{v}\) complementary to the open \(v\).)

  \emph{Closure on locally closed elements.}
  For \(L = \omega(o) \wedge \neg_{\mathbf{D}} \omega(v)\), define
  \[
    \overline{L} \;:=\; \neg_{\mathbf{D}}\, \omega(o \Rightarrow v),
  \]
  where \(\Rightarrow\) is the Heyting implication in \(\mathbf{O}(X)\). This is a closed element, and we claim it is the least closed element above \(L\).

  First, \(L \leq \overline{L}\): the identity \(o \wedge (o \Rightarrow v) \leq v\) gives \(\omega\bigl(o \wedge (o \Rightarrow v)\bigr) \leq \omega(v)\), hence
  \[
    L \wedge \omega(o \Rightarrow v) \;=\; \omega(o) \wedge \omega(o \Rightarrow v) \wedge \neg_{\mathbf{D}} \omega(v) \;\leq\; \omega(v) \wedge \neg_{\mathbf{D}} \omega(v) \;=\; \bot,
  \]
  so \(L \leq \neg_{\mathbf{D}} \omega(o \Rightarrow v) = \overline{L}\).

  For minimality, write any closed element as \(\chi(c) = \neg_{\mathbf{D}} \omega(w)\) for some \(w \in \mathbf{O}(X)\). Then the following are equivalent:
  \begin{align*}
    L \leq \chi(c)
    &\;\Longleftrightarrow\; L \wedge \omega(w) = \bot \\
      &\;\Longleftrightarrow\; \omega(o \wedge w) \wedge \neg_{\mathbf{D}} \omega(v) = \bot \\
      &\;\Longleftrightarrow\; \omega(o \wedge w) \leq \omega(v) \\
      &\;\Longleftrightarrow\; o \wedge w \leq v \\
      &\;\Longleftrightarrow\; w \leq o \Rightarrow v \\
      &\;\Longleftrightarrow\; \neg_{\mathbf{D}}\, \omega(o \Rightarrow v) \leq \neg_{\mathbf{D}}\, \omega(w) \\
      &\;\Longleftrightarrow\; \overline{L} \leq \chi(c).
  \end{align*}
  Thus \(L \leq \chi(c)\) holds if and only if \(\overline{L} \leq \chi(c)\). This universal property characterizes \(\overline{L}\) as the least closed element containing \(L\), so the definition is independent of the chosen presentation of \(L\) as \(\omega(o) \wedge \neg_{\mathbf{D}} \omega(v)\).

  \emph{Extension over finite joins.}
  For a general element \(k = \bigvee_{i} L_{i}\), set \(\overline{k} := \bigvee_{i} \overline{L_{i}}\). This is the least closed element above \(k\): for any \(c \in \mathbf{C}(X)\),
  \[
    k \leq \chi(c)
    \;\;\Longleftrightarrow\;\;
    L_{i} \leq \chi(c) \text{ for every } i
    \;\;\Longleftrightarrow\;\;
    \overline{L_{i}} \leq c \text{ for every } i
    \;\;\Longleftrightarrow\;\;
    \overline{k} \leq c,
  \]
  where the middle equivalence uses the locally closed case. Hence \(\overline{k}\) depends only on \(k\) and not on the decomposition, and the assignment \(k \mapsto \overline{k}\) is a left adjoint to \(\chi {:} \mathbf{C}(X) \hookrightarrow \underline{\mathbf{D}}(X)\).

  \emph{Interior by duality.}
  Define \(k^{\circ} := \neg_{\mathbf{D}}\, \overline{\neg_{\mathbf{D}} k}\). This is an open element (the Boolean complement of a closed). For every \(o \in \mathbf{O}(X)\), writing \(c_{o}\) for the closed complement of \(o\) so that \(\chi(c_{o}) = \neg_{\mathbf{D}} \omega(o)\), we have
  \[
    \omega(o) \leq k
    \;\;\Longleftrightarrow\;\;
    \neg_{\mathbf{D}} k \leq \chi(c_{o})
    \;\;\Longleftrightarrow\;\;
    \overline{\neg_{\mathbf{D}} k} \leq c_{o}
    \;\;\Longleftrightarrow\;\;
    \omega(o) \leq \neg_{\mathbf{D}}\, \overline{\neg_{\mathbf{D}} k} = k^{\circ},
  \]
  where the second equivalence applies the left-adjoint property of \(\overline{(-)}\). Thus \(k^{\circ}\) is the right adjoint to \(\omega\).
\end{proof}

\subsubsection{Semialgebraic Heyting Spaces}
\label{sec:semi-heyt-spac}

We can now extend the definition of semialgebraic spaces in \cite{delfsSemialgebraicTopologyReal1981}. For this we first add rings in the usual way:

\begin{Definition}
  \label{dfn:13}\index{R-ringed Heyting space}
  Given a ring \(R\), an \emph{\(R\)-ringed abstract Heyting space} \(X\) is an abstract Heyting space \([X]\) equipped with a \emph{structure sheaf} \(\mathcal{S}_{X}{:} \mathbf{O}(X)^{op} → R-\mathbf{Alg}\) into the category of \(R\)-algebras
  \(R-\mathbf{Alg}\) that is a sheaf with respect to the finite union coverage.

  A \emph{morphism of \(R\)-ringed abstract Heyting spaces} \(f{:} X → Y\) is a tuple \(([f], f_{\mathcal{S}})\) where \([f]\) is a
   morphism of abstract Heyting spaces \([X] → [Y]\) and \(f_{\mathcal{S}}{:} \mathcal{S}(Y) → [f]_{*}\left(\mathcal{S}(X)\right)\) is a natural transformation from  \(\mathcal{S}(Y)\) to the pushforward of \(\mathcal{S}(X)\) as a
   sheaf of \(R\)-algebras.

   Similar for concrete Heyting spaces.
\end{Definition}
\begin{Remark}
  \label{rmα_:27}
  Note that we do not assume that our space is \emph{locally ringed}, i.e. that the stalks of the structure sheaf are local rings. This is in line with \cite{delfsSemialgebraicTopologyReal1981}.
\end{Remark}

\begin{Proposition}
  \label{prp:4}
  Let \(ι{:} V \hookrightarrow R^{n}\) be a semialgebraic subset of \(R^{n}\). Then \(V\) is an \(R\)-ringed concrete Heyting space.
\end{Proposition}
\begin{proof}
  \label{prf:13}
  We fix a real closed field \(R\), an integer \(n ≥ 0\), and a semialgebraic subset \(V \subseteq R^n\). We saw in \cref{xmp_:11} that \(V\) under its usual semialgebraic topology carries the structure of a
  concrete Heyting space in the sense of \cref{dfn:12,dfn:13}. We can endow this with the structure of the usual semialgebraic function sheaf \(\mathcal{S}_V{:} \mathbf{O}(V)^{op} → R\text{-}\mathbf{Alg}\)
  of \(R\)-algebras, which assigns to each open semialgebraic \(U \subseteq V\) the \(R\)-algebra of semialgebraic maps \(U → R\), and restriction maps along inclusions. Delfs-Knebusch show that
  this is a sheaf for the finite-union coverage and that the resulting ringed space structure matches the ``usual'' semialgebraic geometry.
\end{proof}

So we call a semialgebraic subset of some \(R^{n}\) an \emph{affine semivariety}. We can thus define:

\begin{Example}\label{xmp:112}
\index{semialgebraic space}
Let \(R\) be a real closed field. We define the category \(\mathbf{CSAS}_{R}\) of concrete semialgebraic spaces over \(R\) as the category whose objects \(S\) are \(R\)-ringed
  concrete Heyting spaces such that there exists an open cover of \(\mathopen{∣}S\mathclose{∣}\) by affine semivarieties. Then it follows from the definition of concrete Heyting spaces that
  \(\mathbf{CSAS}_{R}\) features a Heyting open-closed duality. Passing fibrewise to the Boolean algebra generated by opens and closeds gives the corresponding abstract constructible
  closures.
\end{Example}

\(\mathbf{CSAS}_{R}\) is by no means unique or even minimal with respect to this property. We for instance can do similar constructions for semianalytic spaces. But the main point is that
the assumption of a Heyting open-closed duality allows us to transfer all the basic notions from topology: density, regularity, boundaries, closure and interior, but also the logical
structure inherent in a (co)-Heyting algebra, while the use of a semialgebraic structure sheaf keeps all notions tame, like Grothendieck wanted.

\subsection{Categorical Measure Theory}
\label{sec:categ-meas-theory}

\(\mathbf{UBP}\) also allows us to define a notion of measure theory directly on categories using the \(\mathbf{SOb}\)-functor. For this we want to highlight the following excerpt from Fremlin's \emph{Measure
  Theory}, \cite{fremlinMeasureTheory12011} found on the nlab: \cite{nlab:categories_of_measure_theory}

\begin{quote}
  \label{qo_:1}
  One of the striking features of measure theory, compared with other comparably abstract branches of pure mathematics, is the relative unimportance of any notion of 'morphism'.
  The theory of groups, for instance, is dominated by the concept of 'homomorphism', and general topology gives a similar place to 'continuous function'. In my view, the nearest
  equivalent in measure theory is the idea of 'inverse-measure-preserving function'.
\end{quote}

Inverse-measure-preservation is precisely what is implied by a weak natural transformation of hyperdoctrines \(σ{:} \mathcal{H} \Rightarrow \mathbf{K}(P)\) to the constant functor \(\mathbf{K}(P)\) on some
preorder \(P\). So we can define:

\begin{Definition}
  \label{dfn:14}\index{scale on a hyperdoctrine}
  Given a preorder \(P\) and a submonic hyperdoctrine \(\mathcal{H}\), a \emph{weak \(P\)-scale } on \(\mathcal{H}\) is a weak natural transformation \(σ{:} \mathcal{H} \Rightarrow \mathbf{K}(P)\) to the constant functor \(\mathbf{K}(P)\) on \(P\).
\end{Definition}

However, let us consider \cref{qo_:1} more carefully: there might be a reason for this relative morphism-scarcity that goes beyond the incidental, and one natural candidate is that
the preservation of measures is not straightforward. Consider a very simple kind of measure space: \(ℕ\) with the cardinality measure \(\mathopen{∣}-\mathclose{∣}{:} \mathbf{P}(ℕ) → ℕ_{∞}\) on its subobject
lattice. This is an \(ℕ_{∞}\)-scale on the subobject lattice of \(N\) as a set, so on \(ℕ\) as a discrete open-closed space and in particular a concrete Heyting space.
Let \(ι{:} X \hookrightarrow ℕ\) be some finite subset of \(ℕ\) and consider the induced measure \(ι_{*}⨾ \mathopen{∣}-\mathclose{∣}{:} \mathbf{P}(X) → ℕ_{∞}\). It would seem very reasonable that
\(ι\) would be inverse-measure-preserving, but it is not: \(\mathopen{∣}ι_{*}(ι^{*}(ℕ))\mathclose{∣} = \mathopen{∣}ι_{*}(X)\mathclose{∣} < \mathopen{∣}ℕ\mathclose{∣}\) by
assumption. And if we try to further extend \(\mathopen{∣}-\mathclose{∣}\) to the category \(\mathbf{CHS}\) of concrete
Heyting spaces, we encounter yet more problems, as even though cardinality is well-defined for each concrete Heyting space in \(\mathbf{CHS}\), its behavior on regular epimorphisms is
measure-increasing.

So while inverse-measure-preservation makes sense for fully local measures, such a probability measures, which always map the entirety of a space to \(1\), this is not a
property of all measures, and if the global structure of a space itself is measured, a bifurcation in variance becomes apparent, roughly between monomorphisms and epimorphisms, but more
often between their regular variants. Thus, it makes sense to extend the definition somewhat:

\begin{Definition}
  \label{defn_:28}
\index{regular scale}
\index{prescale}
  Given a preorder \(P\) and a submonic hyperdoctrine \(\mathcal{H}{:} \mathcal{C}^{op} → \mathbf{PrO}\), a \emph{\(P\)-prescale} is a natural transformation \(σ{:} \mathbf{Core}_{\mathcal{C}}⨾ \mathcal{H} \Rightarrow \mathbf{K}(P)\), where
  \(\mathbf{Core}_{\mathcal{C}}\) is the embedding of the core groupoid on the objects of \(\mathcal{C}\) into \(\mathcal{C}\) (so that naturality reduces to isomorphism-invariance).

  \(σ\) is a \emph{lax \(P\)-scale} if it descends
  to a lax natural transformation of hyperdoctrines \(\mathcal{H} \Rightarrow \mathbf{K}(P)\) and dually an \emph{oplax \(P\)-scale}.

  \(σ\) is a \emph{regular \(P\)-scale} if it is a lax scale on the monic part of \(\mathcal{C}\) and an
  oplax scale on the regular epimorphisms in \(\mathcal{C}\).
\end{Definition}
\begin{Example}
  \label{xmp_:8}
  Cardinality is a regular scale on each concrete category whose monomorphisms are injections and whose regular epimorphisms are surjections, e.g. in \(\mathbf{Top}\), valued in
  the category of cardinals.
\end{Example}
\begin{Example}
  \label{xmp_:24}
  Let \(\mathbf{M}(X)\) be the \(σ\)-algebra of measurable subsets of a measure space and \(μ{:} X → ℝ_{\geq 0}\) be its measure. Then the family \(μ|_{M}{:}(M) → ℝ_{\geq 0}\) with
  \(M{:} \mathbf{M}(X)\) a measurable subset of \(X\) and \((M)\) its principal ideal in \(\mathbf{M}(X)\) is a scale on \(\mathbf{M}(X)\).
\end{Example}

Depending on the situation we can also demand other things:

\begin{enumerate}
\item\label{item:93} if \(\mathcal{C}\) and \(P\) have an initial object, then it makes sense to demand that \(σ\) preserves bottoms,

\item\label{item:95} as defined, \(σ\) is both a local measure, which assigns to each subobject \(ι{:} A \hookrightarrow X\) in \(\mathcal{C}\) a measure of how large \(ι\) is, and a global one, as it assigns a size to \(X\) itself.
  Depending on the application this can be desired or not. If we want a purely local measure, it makes sense to demand that \(σ(\mathbf{id}_{X}) = \top_{P}\) for each \(X{:} \mathcal{C}\), so that each space is maximal
  in itself.

\item\label{item:101} We can also add whatever kind of additivity conditions we want if we simply assume that \(P\) has an internal monoid structure.
\end{enumerate}

More particularly, we can combine the usage of such a scale with the open-closed dualities we described in \cref{sec:open-clos-dual} to generalize Borel measures:

\begin{Definition}
  \label{dfn:15}\index{Borel-Heyting measure}
  Given an ordered monoid \(M\) and an open-closed duality \(\mathcal{D} = (\mathbf{o}, \mathbf{d}, \mathbf{c})\) on a category \(\mathcal{C}\), an \emph{open Borel-Heyting measure} \(μ\) on \(\mathcal{D}\) is
  an \(M\)-scale on the underlying functor \(\mathbf{O}\) of \(\mathbf{o}\) such that if
  \(ι ∨_{\mathbf{O}} κ\) and \(ι ∧_{\mathbf{O}} κ\) exist and \(ι ∧_{\mathbf{O}} κ = \bot_{\mathbf{O}}\), then \(μ(ι ∨_{\mathbf{O}} κ) = μ(ι) + μ(κ)\).

  A \emph{closed Borel-Heyting measure} is the same but defined on \(\mathbf{C}\).
\end{Definition}

This goes to the core of the intersection of topology and measure theory in Borel measures: we use ``simple'' objects like opens or closeds to infer the size of more complex ones. A key point
is now that we can do this for all categories with a sufficiently rich subobject structure, e.g. Heyting spaces, and often it does not matter whether we define and measure on the
opens or closeds. For this first note a standard subtlety of measure theory on lattices: because the open (or closed) subobjects do not generally form a Boolean algebra, disjoint
additivity is not sufficient to uniquely extend a measure to the generated Boolean algebra, as we cannot decompose overlapping opens into disjoint ones. We must instead require the measure
to be \emph{modular} (also known as a \emph{valuation}).

\begin{Definition}
  \label{dfn:modular-measure}\index{modular measure}
  An open (resp. closed) Borel-Heyting measure \(μ\) on \(\mathcal{D} = (\mathbf{o}, \mathbf{d}, \mathbf{c})\) with values in a commutative monoid \(M\) is \emph{modular} if for all subobjects \(ι, κ\),
  \[
    μ(ι ∨ κ) + μ(ι ∧ κ) = μ(ι) + μ(κ).
  \]

  \(μ\) is \emph{standard} if \(\mathcal{D}\) is defined on a category \(\mathcal{C}\) with an initial object \(∅\), \(\mathbf{O}\) resp. \(\mathbf{C}\) contains each initial morphism \(0_{X}{:} ∅ → X\) and \(μ\)
  maps each \(0_{X}\) to the neutral element \(0 ∈ M\).
\end{Definition}

Assuming this, let us first remember (or learn about) the \emph{Smiley-Horn-Tarski Theorem}:

\begin{Theorem}[Smiley-Horn-Tarski Theorem for Left-Cancellative Monoids]
  \label{thm:2}
\index{Smiley-Horn-Tarski theorem}
\cite{gierzContinuousLatticesDomains}(IV-9.3., IV-9.4.) A modular standard measure \(μ{:} L → ℝ_{\geq 0, ∞}\) from a sub-lattice \(L\) of a Boolean algebra \(B\) to a left-cancellative monoid
  extends to a finitely additive measure on the Boolean closure \(\overline{L}\) of \(L\) in \(B\). Moreover, if \(μ\) is finitely valued, then this extension is unique.
\end{Theorem}
\begin{Proposition}[Extension of Borel-Heyting Measures]
  \label{prop:borel-heyting-extension}
\index{Borel-Heyting measure!extension}
Let \(\mathcal{D} = (\mathbf{o}, \mathbf{d}, \mathbf{c})\) be an abstract Heyting open-closed duality.
  Then any standard modular open Borel-Heyting measure \(μ{:} \mathbf{O} \Rightarrow \mathbf{K}(ℝ_{\geq 0, ∞})\) extends to a finitely additive measure \(\hat{μ}\) on the constructible closure
  \(\underline{\mathbf{D}}\) of \(\mathbf{D}\).
  The same holds for any standard modular closed Borel-Heyting measure \(ν{:} \mathbf{C} \Rightarrow \mathbf{K}(ℝ_{\geq 0, ∞})\). Moreover, if \(μ\) is finitely valued, then \(\hat{μ}\) is unique.
  Consequently, there is a natural bijection between finite open and closed modular Borel-Heyting measures.
\end{Proposition}
\begin{proof}
  \label{prf:borel-heyting-extension}
  Fix an object \(X {:} \mathcal{C}\). By the definition of a Heyting open-closed duality, \(\mathbf{O}(X)\) is a distributive sublattice of the Boolean algebra \(\mathbf{D}(X)\), and
  \(\mathbf{C}(X)\) consists precisely of the Boolean complements of elements of \(\mathbf{O}(X)\).

  Let \(\underline{\mathbf{D}}(X)\) be the constructible closure. By the Smiley-Horn-Tarski extension theorem \(μ_{X}\) uniquely
  extends to an additive measure \(\hat{μ}_{X} {:} \underline{\mathbf{D}}(X) → ℝ_{\geq 0, ∞}\), which is unique if \(μ\) is finite. Because the closed subobjects \(\mathbf{C}(X)\) are the Boolean complements of the opens, the Boolean algebra generated by \(\mathbf{C}(X)\) is precisely
  \(\underline{\mathbf{D}}(X)\). Restricting \(\hat{μ}_{X}\) to \(\mathbf{C}(X)\) yields a function \(ν_{X} {:} \mathbf{C}(X) → ℝ_{\geq 0, ∞}\) which is trivially modular, hence a closed
  Borel-Heyting measure.

  For any open \(U ∈ \mathbf{O}(X)\) with closed complement \(¬U ∈ \mathbf{C}(X)\), finite additivity on \(\underline{\mathbf{D}}(X)\) implies:
  \[
    μ_{X}(U) + ν_{X}(¬U) = \hat{μ}_{X}(U) + \hat{μ}_{X}(¬U) = \hat{μ}_{X}(\top_X) = μ_{X}(\top_X).
  \]
  Conversely, starting with a finite modular closed Borel-Heyting measure \(ν\), it uniquely extends to \(\underline{\mathbf{D}}(X)\), and restricting this extension to \(\mathbf{O}(X)\)
  recovers an open measure. These two extension-restriction procedures are mutually inverse.

  Finally, we must check that the extension is natural (i.e. it acts as a hyperdoctrine scale). For any morphism of Heyting spaces \(f {:} Y → X\), the pullback \(f^{*}\) on
  \(\mathbf{D}\) is a Boolean homomorphism that restricts to the given pullbacks on \(\mathbf{O}\) and \(\mathbf{C}\). Because \(f^{*}\) preserves the Boolean operations and
  \(\hat{μ}_{X}\) is the unique additive extension of \(μ_{X}\), it follows that \(f^{*} ⨾ \hat{μ}_{Y} = \hat{μ}_{X}\) on \(\underline{\mathbf{D}}(X)\). Thus, naturality holds for the
  extensions, and the bijection is complete.
\end{proof}

So as long as our measure is modular (which is standard for measures on lattices that are not Boolean algebras), we can cleanly translate between measuring the open parts of a space,
measuring the closed parts, and measuring the constructible parts. This is analogous to the extension of a Borel measure to the Borel subsets of a topological space, except necessarily
everything is kept finite and thus more restricted, which is a
mixed blessing: everything is calculable, but we have access to fewer sets. In particular, we cannot define \(σ\)-additive measures without assuming countable joins in
\(\mathbf{O}\). But this runs counter to our paradigmatic example of semialgebraic sets because the union of countably many semialgebraic sets is not semialgebraic — in particular,
no polynomial has infinitely many zeroes. To remedy this we could either:

\begin{enumerate}
\item \label{item:106} use semianalytic sets and introduce countable joins to \(\mathbf{O}\),

\item\label{item:107} or, what is perhaps even better, transition into the nonstandard-realm and use Loeb measures — in this setting the notion of a lattice becomes much stronger, since a
  nonstandard lattice contains hyperfinite joins and meets. This way, we can dispense of the requirement of \(σ\)-additivity altogether and do categorical measure theory in what is internally a purely
  finite setting.\footnote{See \cite{cutlandNonstandardMeasureTheory} for an introduction to nonstandard measure theory.}
\end{enumerate}

For now we leave these as suggestions however.

\section{The Representability Theorems}
\label{sec:covar-repr-theor}

Throughout this text we have been using the \(\mathbf{SOb}_{\mathcal{C}}\)-functor. But what is it? Where does it come from? Here we will derive a universal property of \(\mathbf{SOb}_{\mathcal{C}}\) and \(\mathbf{UBP}\) that
shows that \(\mathbf{UBP}\) is, in some sense, the largest context in which something like the reflection of \(\mathbf{Cat}\) in itself exists for monics. First of all, we will extend
\(\mathbf{SOb}_{\mathcal{C}}\) to a 2-functor \(\mathbf{SOb}\) on the 2-category \(\overset{\hookrightarrow}{\mathbf{Cat}}\) of monic categories.

\subsection{Preliminary Definitions}
\label{sec:prel-defin}

Let us recapitulate the definitions of the basic objects
involved in this section.

\begin{Remark}
  \label{rmα_:16}
  As mentioned in \cref{sec:size-assumptions} we take everything strict (up to contractible homotopy). This might seem evil but a weak framework would further complicate the
  already strained exposition for little gain: the material is orthogonal to questions of weakness, the only significant adaptation would be to have to consider the naturality
  2-cells of weak natural transformations in 2-categories other than \(\mathbf{UBP}\) (which is locally thin up to homotopy) in the latter part of \cref{sec:univ-char-mathbf}.
\end{Remark}

\begin{Definition}[Strict 2-category]
  \label{def:2-category}
\index{2-category}
  A \emph{(strict) 2-category} \(\mathcal{C}\) consists of:
  \begin{enumerate}
  \item A class of \emph{objects} (or \emph{0-cells}) \(X, Y, Z, \dots\)
  \item For each pair of objects \(X, Y\), a category \(\mathcal{C}(X,Y)\) whose
    \begin{itemize}
    \item objects are called \emph{1-cells} (or \emph{morphisms}) \(f, g {:} X \to Y\),
    \item morphisms are called \emph{2-cells} (or \emph{2-morphisms}) \(α {:} f \Rightarrow g {:} X \to Y\).
    \end{itemize}
  \item A \emph{horizontal composition} functor
    \[
      ⨾_{X,Y,Z} {:} \mathcal{C}(X,Y) \times \mathcal{C}(Y,Z) \to \mathcal{C}(X,Z)
    \]
    for each triple of objects, written simply as \(f⨾ g\) on 1-cells and \(β⨾ α\) on 2-cells.
  \item For each object \(X\), a distinguished 1-cell \(\mathbf{id}_X {:} X \to X\) serving as the identity for horizontal composition.
  \end{enumerate}
  Subject to the \emph{strict} associativity and unit axioms:
  \begin{align*}
    f ⨾ (g ⨾ h) &= (f ⨾ g) ⨾ h, & f ⨾ \mathbf{id}_Y &= f = \mathbf{id}_X ⨾ f
  \end{align*}
  for all composable 1-cells (and similarly for 2-cells). Thus horizontal composition is strictly associative and unital.

  Vertical composition of 2-cells (i.e. composition within each hom-category) is written with a semicolon: given \(α {:} f \Rightarrow g\) and \(β {:} g \Rightarrow h\), the composite is \(α; β {:} f \Rightarrow h\). The \emph{interchange law}
  \[
    (α; β)⨾ (α'; β') = (α⨾ α'); (β⨾ β')
  \]
  holds automatically from the functoriality of \(⨾\).
\end{Definition}
\begin{Definition}[Strict 2-functor]
  \label{def:2-functor}
\index{2-functor}
  Let \(\mathcal{C}, \mathcal{D}\) be strict 2-categories. A \emph{(strict) 2-functor} \(F {:} \mathcal{C} \to \mathcal{D}\) consists of:
  \begin{enumerate}
  \item A function on objects: \(X \mapsto F(X)\).
  \item For each pair \(X, Y\), a functor \(F_{X,Y} {:} \mathcal{C}(X,Y) \to \mathcal{D}(F(X), F(Y))\) on hom-categories, mapping 1-cells \(f \mapsto F(f)\) and 2-cells \(α \mapsto F(α)\).
  \end{enumerate}
  Subject to \emph{strict} preservation of identities and horizontal composition:
  \begin{align*}
    F(\mathbf{id}_X) &= \mathbf{id}_{F(X)}, & F(f ⨾ g) &= F(f) ⨾ F(g), & F(α; β) &= F(α); F(β).
  \end{align*}
\end{Definition}
\begin{Definition}[Strict 2-natural transformation]
  \label{def:2-nat-trans}
\index{2-natural transformation}
  Let \(F, G {:} \mathcal{C} \to \mathcal{D}\) be strict 2-functors. A \emph{(strict) 2-natural transformation} \(θ {:} F \Rightarrow G\) consists of:
  \begin{enumerate}
  \item For each object \(X {:} \mathcal{C}\), a 1-cell \(θ_X {:} F(X) \to G(X)\) in \(\mathcal{D}\) (the \emph{component} at \(X\)),
  \end{enumerate}
  such that for every 1-cell \(f {:} X \to Y\) in \(\mathcal{C}\) the following square of 1-cells in \(\mathcal{D}\) \emph{commutes strictly}:
  \[
    \begin{tikzcd}
      F(X) \ar[r, "F(f)"] \ar[d, "θ_X"'] & F(Y) \ar[d, "θ_Y"] \\
      G(X) \ar[r, "G(f)"'] & G(Y)
    \end{tikzcd}
  \]
  i.e. \(θ_X ⨾ G(f) = F(f) ⨾ θ_Y\), and for every 2-cell \(α {:} f \Rightarrow g {:} X \to Y\):
  \[
    F(α) ⨾ θ_Y = θ_X ⨾ G(α) \qquad \text{(equality of 2-cells).}
  \]
  If only the 1-cell condition holds and the 2-cell condition is replaced by a morphism \(θ(f){:} θ_{X}⨾ G(f) → F(f)⨾ θ_{Y}\) satisfying coherence conditions that make it behave
  sensibly with respect to composition and identities,\footnote{The coherence conditions are the same as for a pseudonatural transformation, see e.g. p.5 of
    \cite{leinsterBasicBicategories1998} for details.} the transformation is called \emph{lax}; if the
  inequality is reversed, \emph{oplax}.
  % To us, the most important lax case is when 
\end{Definition}
\begin{Definition}[Strict 2-adjunction]
  \label{def:2-adjunction}
\index{2-adjunction}
  Let \(F {:} \mathcal{C} \to \mathcal{D}\) and \(U {:} \mathcal{D} \to \mathcal{C}\) be strict 2-functors. A \emph{(strict) 2-adjunction} \(F \dashv U\) consists of strict 2-natural transformations
  \[
    η {:} \mathbf{id}_{\mathcal{C}} \Rightarrow F ⨾ U \qquad \text{(the \emph{unit})}, \qquad \varepsilon {:} U ⨾ F \Rightarrow \mathbf{id}_{\mathcal{D}} \qquad \text{(the \emph{counit})},
  \]
  satisfying the \emph{triangle identities} strictly:
  \[
    F \xrightarrow{η ⨾ F} F ⨾ U ⨾ F \xrightarrow{F ⨾ \varepsilon} F \;=\; \mathbf{id}_F, \qquad U \xrightarrow{U ⨾ η} U ⨾ F ⨾ U \xrightarrow{\varepsilon ⨾ U} U \;=\; \mathbf{id}_U.
  \]
  Equivalently, for each \(X {:} \mathcal{C}\) and \(Y {:} \mathcal{D}\) there is an isomorphism of hom-categories
  \[
    \mathcal{D}(F(X), Y) \cong \mathcal{C}(X, U(Y))
  \]
  that is strictly natural in \(X\) and \(Y\) and strictly 2-natural in the 2-cells.
\end{Definition}

\subsection{Initiality of \(\mathbf{SOb}\)}
\label{sec:initiality-mathbfsob}

Now let us begin. For each (necessarily monomorphism-preserving) functor \(F{:} \mathcal{C} → \mathcal{D}\) between monic categories \(\mathcal{C}\) and \(\mathcal{D}\)
we can use the natural transformations \(F_{*}{:} \mathbf{SOb}_{\mathcal{C}} \Rightarrow F⨾ \mathbf{SOb}_{\mathcal{D}}\) that we introduced in  \cref{sec:density}, which maps \(ι{:} A \hookrightarrow X\) to \(F(ι){:} F(A) \hookrightarrow F(X)\).
To describe what \(\mathbf{SOb}\) does on 2-cells let \(η{:} F \Rightarrow G\) be a natural transformation.
Again we have to assume \(η\) is monic (which is trivial on \(\overset{\hookrightarrow}{\mathbf{Cat}}\))
to say anything about how it acts on subobjects. Then we can say that the monomorphisms \(η_{X}\) induce for each \(X\) an ideal embedding
\(η_{*, X}{:} \mathbf{SOb}(F(X)) \hookrightarrow \mathbf{SOb}(G(X)){:} F(ι) ↦ G(ι)\), which is necessarily principal by \cref{rmk:17}. Since \(η\) is natural between \(F\) and \(G\), the diagram

\begin{figure}[H]
  \centering
  \begin{tikzcd}
    G(A) \ar[r, hook, "G(ι)"] & G(X) \\
    F(A) \ar[u, hook, "η_{A}"] \ar[r, hook, "F(ι)"'] & F(X) \ar[u, hook, "η_{X}"']
  \end{tikzcd}
  % \caption{\label{fi♭_:1}}
\end{figure}

commutes, so for each \(ι{:} A \hookrightarrow X\) it holds that \(F(ι)_{*}⨾ η_{*, X} \leq G(ι)_{*}\). Thus, \(η\) induces for each \(X\) a triangle

\begin{figure}[H]
  \centering
  \begin{tikzcd}
  & \mathbf{SOb}(G(X)) & \\
  & \geq & \\
  \mathbf{SOb}(X) \ar[uur, "G_{*, X}"] \ar[rr, "F_{*, X}"] && \mathbf{SOb}\left(F(X)\right) \ar[uul, "η_{*, X}"']
\end{tikzcd}
  % \caption{\label{fi♭_:2}}
\end{figure}

so that \(G_{*, X} \geq F_{*, X}⨾ η_{*, X}\). For every morphism \(ι{:} X \hookrightarrow Y\) in \(\mathcal{C}\) the image \(ι_{*}\) is an ideal embedding, so this extends to a modification

\[G_{*} \geq F_{*}; (η⨾ \mathbf{SOb}).\]

We can capture this in a diagram:

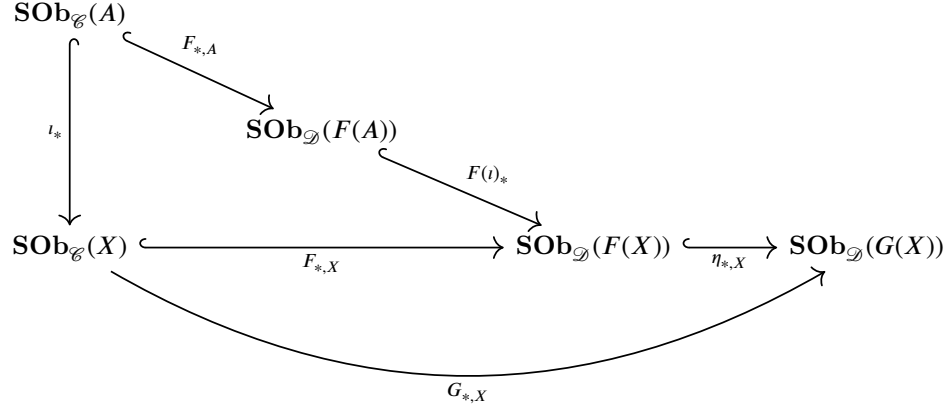
\begin{figure}[H]
  \centering
  \begin{tikzcd}[column sep=large, row sep=large]
  \mathbf{SOb}_{\mathcal{C}}(A)
    \ar[dd, hook, "\iota_{*}"']
    \ar[dr, hook, "F_{*,A}", pos=0.35]
    % \ar[drr, bend left=30, "G_{*,A}"{name=G}]
    & & & \\
  & \mathbf{SOb}_{\mathcal{D}}(F(A))
  \ar[dr, hook, "F(ι)_{*}"{name=eta}]
    & & \\
  \mathbf{SOb}_{\mathcal{C}}(X)
    \ar[rr, hook, "F_{*,X}"'{name=FX}]
    \ar[rrr, bend right=30, "G_{*,X}"'{name=GX}]
    & & \mathbf{SOb}_{\mathcal{D}}(F(X))
    \ar[r, hook, "η_{*,X}"']
    & \mathbf{SOb}_{\mathcal{D}}(G(X))
  % \ar[from=G, to=eta, Rightarrow, shorten >=4pt, shorten <=4pt, "\geq" description]
\end{tikzcd}

\caption{\label{fi♭_:24} For each subobject \(\iota\colon A\hookrightarrow X\) in \(\mathcal{C}\), naturality of \(η\) gives \(F(\iota)⨾η_X=η_A⨾ G(\iota)\), so the image ideal \(G(ι)\) of \(ι\) under \(G\) contains
  the image ideal \(F(ι)\) under \(F\) postcomposed with \(η\). The inequality \(G_*\ge F_*;(η⨾\mathbf{SOb})\) is exactly the componentwise inclusion of principal ideals.}
\end{figure}

Thus,
\[\mathbf{SOb}: X ↦ (X, \mathbf{SOb}(X)), F: ↦ (F, F_{*}), η ↦ (η, \mathbf{SOb}_{η})\]
 extends to a 2-functor \(\mathbf{\overset{\hookrightarrow}{Cat}} → {\mathbf{\overset{\hookrightarrow}{Cat}}_{◆ \mathbf{UBP}}}\):

\begin{enumerate}
\item from the (2-large) 2-category \(\overset{\hookrightarrow}{\mathbf{Cat}}\) whose objects are well-powered purely monic categories, whose morphisms are (necessarily
  monomorphism-preserving) functors and whose 2-morphisms are (necessarily monic) natural transformations,

\item\label{item:17} to the (3 -large) 2-category \({\mathbf{\overset{\hookrightarrow}{Cat}}_{◆ \mathbf{UBP}}}\)

  \begin{enumerate}
  \item \label{item:87} whose objects are functors \(\mathcal{C} → \mathbf{UBP}\) into \(\mathbf{UBP}\),

  \item\label{item:88} morphisms are triangles

    
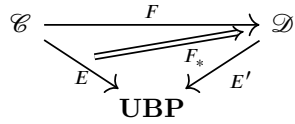
\begin{figure}[H]
      \centering
      \begin{tikzcd}
        \mathcal{C} \ar[rr, "F"] \ar[dr, "E" below, ""{name=S}] && \mathcal{D} \ar[dl, "E'"] \\
        & \mathbf{UBP} \arrow[Rightarrow, from=S, to=1-3, "F_{*}"', shorten >=6pt] &
      \end{tikzcd}
      % \caption{\label{fi♭_:3}}
    \end{figure}

    where \(F_{*}\) is componentwise top-preserving.

  \item\label{item:89} 2-morphisms are given by (necessarily monic) natural transformations \(η{:} F \Rightarrow G{:} \mathcal{C} → \mathcal{D}\) over \(\mathbf{UBP}\) and the induced triangles
    \begin{figure}[H]
      \centering
      \begin{tikzcd}
        F⨾ E' \ar[rr, Rightarrow, "η⨾ E'"] && G⨾ E' \\
        & E \ar[ul, Rightarrow, "F_{*}"] \ar[ur, Rightarrow, "G_{*}"{below, yshift=-2pt, xshift=3pt}, ""{name=G}] \ar[triple, from=G, to= 1-1, "\leq", shorten >=6pt, shorten
        <=6pt]
      \end{tikzcd}
      % \caption{\label{fi♭_:4}}
    \end{figure}
  \end{enumerate}
\end{enumerate}

To show well-definedness, the only non-trivial part is the defining inequality for 2-cells:

     \[F_{*}; (η⨾ \mathbf{SOb}_{\mathcal{C}'}) ≤ F'_{*}.\]

     We did this in \cref{fi♭_:24}. In symbols:

     \[
       (F_{*}; (η⨾ \mathbf{SOb}_{\mathcal{C}'}))_X(\iota) = (η ⨾ \mathbf{SOb}_{\mathcal{C}'})_X(\top_{F(\iota)}) = F(\iota) ⨾ η_{X} = η_A ⨾ F'(ι) ≤ F'(ι) = F'_{*, X}(ι).
     \]

\begin{Remark}
  \label{rmα_:18}
  So \(\mathbf{\overset{\hookrightarrow}{Cat}}_{◆ \mathbf{UBP}}\) is a particular variant of a slice-2-category with morphisms and 2-morphisms given by two lax kinds of triangles. We pronounce this construction
  as ``diamond-over''.
\end{Remark}

\(\mathbf{SOb}_{(-)}\) maps a well-powered category \(\mathcal{C}\) to the functor \(\mathbf{SOb}_{\mathcal{C}}\), a functor \(F\) to \((F, F_{*})\) and a natural transformation \(η{:} F \Rightarrow G\) to \((η_{*}, \leq)\) as described above. Moreover, since for each monomorphism
\(ι{:} X \hookrightarrow Y\) the induced map \(ι_{*}{:} \mathbf{SOb}(X) \hookrightarrow \mathbf{SOb}(Y)\) is an ideal-embedding, the functor \(\mathbf{SOb}_{\mathcal{C}}{:} \mathcal{C} → \mathbf{UBP}\) factors through the category
\(\mathopen{↓}\mathbf{UBP}\mathclose{↓}\) of upward-bounded-preorders and ideal embeddings.

\begin{Remark}
  \label{rmk:9}
  We can describe the image of \(\mathbf{SOb}_{(-)}\) more precisely: since each \(ι ∈ \mathbf{SOb}_{\mathcal{C}}(X)\) corresponds to a monomorphism \(ι{:} A \hookrightarrow X\) and every principal ideal of
  \(\mathbf{SOb}_{\mathcal{C}}(X)\) corresponds to an element \(\mathbf{SOb}_{\mathcal{C}}(X)\) by definition, every principal ideal embedding \(κ{:} (a) \hookrightarrow \mathbf{SOb}_{\mathcal{C}}(X)\) with \(a\) an element of
  \(\mathbf{SOb}_{\mathcal{C}}(X)\) is in the image of \(\mathbf{SOb}_{\mathcal{C}}\). Thus, \(\mathbf{Im}(\mathbf{SOb}_{\mathcal{C}})\) is closed under precomposition and thus an ideal\footnote{Note that we mean ideal in the category-theoretic sense, i.e. closed under precomposition but without binary upward bounds.} in
  \(\mathopen{↓}\mathbf{UBP}\mathclose{↓}\), and, conversely, every 1-categorical ideal in \(\mathopen{↓}\mathbf{UBP}\mathclose{↓}\) is a subcategory of
  \(\mathopen{↓}\mathbf{UBP}\mathclose{↓}\) that maps to itself.

  It is important that the ideal that is the image of \(\mathbf{SOb}_{\mathcal{C}}\) is not generally a principal ideal as a subcategory of \(\mathbf{UBP}\), i.e. not upward-bounded: each object in \(\mathcal{C}\) is composed
  of its parts, but \(\mathcal{C}\) itself is generally only a context in which these objects interact, they don't assemble into one whole.

  % It will also become important later on that this implies that \(\overset{\hookrightarrow}{\mathbf{UBP}}\) itself is pre-mirrored in \(\overset{\hookrightarrow}{\mathbf{Cat}}\) in a way. More precisely, each
  % preorder \(P\) is a monic category and the functor \(\mathbf{SOb}_{P}{:} P\) maps \(P\) to the principal ideal \(\mathopen{∣}(P)\mathclose{∣}_{1}\) of \(P\) in
  % \(\mathopen{↓}\mathbf{UBP}\mathclose{↓}\), consisting of all principal ideal mappings into \(P\) but without any of the non-invertible 2-cells between them.
\end{Remark}

Thus, \(\mathbf{SOb}_{(-)}{:} \overset{\hookrightarrow}{\mathbf{Cat}} \hookrightarrow {\mathbf{\overset{\hookrightarrow}{Cat}}_{◆ \mathbf{UBP}}}\) factors through the full sub-2-category
\(\mathbf{\overset{\hookrightarrow}{Cat}}_{◆ \mathbf{UBP}|_{↓}}\) of \(\mathbf{\overset{\hookrightarrow}{Cat}}_{◆ \mathbf{UBP}}\) on those functors to \(\mathbf{UBP}\) that factor through \(\mathopen{↓}\mathbf{UBP}\mathclose{↓}\). Moreover, these \(F_{*}\) are unique in their position:

\begin{Lemma}
  \label{lmm:3}
  Given a functor of monic categories \(F{:} \mathcal{C} \hookrightarrow \mathcal{D}\), \(F_{*}\) is the only top-preserving natural transformation \(ε{:} \mathbf{SOb}_{\mathcal{C}} \Rightarrow  F⨾ \mathbf{SOb}_{\mathcal{D}}{:} \mathcal{C} → \overset{\hookrightarrow}{\mathbf{UBP}}\).
\end{Lemma}
\begin{proof}
  \label{prf:6}
  \begin{figure}[H]
    \centering
      \begin{tikzcd}
        \mathbf{SOb}_{\mathcal{C}}(A) \ar[r, hook, "ε_A"] \ar[d, hook, "ι_{*}"'] & \mathbf{SOb}_{\mathcal{D}}(F(A)) \ar[d, hook, "F(ι)_{*}"] \\
    \mathbf{SOb}_{\mathcal{C}}(X) \ar[r, hook, "ε_X"'] & \mathbf{SOb}_{\mathcal{D}}(F(X))
  \end{tikzcd}
    % \caption{\label{fi♭_:5}}
  \end{figure}

  Evaluate at \(\mathbf{id}_A\). The left-down leg gives \(\mathbf{id}_A⨾ ι = ι\); the top-right leg gives \(ε_A(\mathbf{id}_A) = \mathbf{id}_{F(A)}\) (by top-preservation) and then \(\mathbf{id}_{F(A)}⨾ F(ι) = F(ι)\).
  Thus, naturality forces \(ε_X(ι) = F(ι)\). Since every element of \(\mathbf{SOb}_{\mathcal{C}}(X)\) is some monomorphism in \(\mathcal{C}\) we have \(ε_X = F_{*, X}\) for all \(X\); hence \(ε = F_{*}\).
\end{proof}

Given now some monic category \(\mathcal{C}\), functor \(G{:} \mathcal{C} → \mathbf{UBP}\) and \(X{:} \mathcal{C}\) we can define a map \(G_{★, X}{:} \mathbf{SOb}(X) \hookrightarrow G(X)\) which maps each
\(ι{:} \mathbf{SOb}(X)\) to the top \(\top_{G(ι)}\) of the order-morphism \(G(ι){:} G(A) \hookrightarrow G(X)\), which is an embedding since \(ι\) is monic. If \(ι \leq κ\) for some \(κ{:} B \hookrightarrow X\), then there exists a \(λ{:} A \hookrightarrow B\) in
\(\mathcal{C}\) such that \(λ⨾ κ = ι\), so that \(G(λ)⨾ G(κ) = G(ι)\) and thus \(G_{★, X}(ι) \leq G_{★, X}(κ)\). So \(G_{★}\) is the component of a reflector at \(G\). This is how we can tell that
\(\mathbf{SOb}\) is actually an embedding of a reflective sub-2-category.

  Note that this is the basically same construction as the \(F_{*}\) from \cref{sec:density}: when \(F{:} \mathcal{C} → \mathbf{UBP}\) (instead of \(F{:} \mathcal{C} → \mathcal{D}\) between general categories), the map
  \(F_{*, X}\) sends \(ι\) to \(\top_{F(ι)} \in F(X)\), which is the element of \(F(X)\) that corresponds to the ideal embedding \(ι\), whereas in the general case it sends \(ι\) to
  \(F(ι) \in \mathbf{SOb}(F(X))\), so in the second case we take the subobject-functor from a general category while in the first we don't need to since we're already in
  \(\mathbf{UBP}\) anyway.

\begin{Lemma}
  \label{lmm:2}
  For each functor \(F{:} \mathcal{C} → \mathbf{UBP}\) and morphism \(f {:} X \hookrightarrow Y\) in \(\mathcal{C}\) it holds that \(f_{*}⨾ F_{★, Y}  = F_{★, X}⨾ F(f)\). Thus \(F_{★}\) is a natural transformation
  \[
    F_{★} {:} \mathbf{SOb}_{\mathcal{C}} ⇒ F.
  \]
\end{Lemma}
\begin{proof}
  \label{prf:lmm2}
  Consider the diagram:

  \begin{figure}[H]
    \centering
    \begin{tikzcd}
      \mathbf{SOb}_{\mathcal{C}}(X) \ar[r, "f_{*}"] \ar[d, "F_{★, X}"'] & \mathbf{SOb}_{\mathcal{C}}(Y) \ar[d, "F_{★, Y}"] \\
      F(X) \ar[r, "F(f)"'] & F(Y)
    \end{tikzcd}
    % \caption{\label{fi♭_:17}}
  \end{figure}
  Given \(ι {:} A \hookrightarrow X\),  we have:
  \[f_{*}⨾ F_{★, Y}(ι⨾ f) = \top_{F(ι⨾ f)} = \top_{F(ι)⨾ F(f)} = F(f)(\top_{F(ι)}) =  F_{★, X}⨾ F(f).\]

  Since this holds for all \(ι ∈ \mathbf{SOb}(X)\) for all \(X{:} \mathcal{C}\) we have that \(F_{★}\) is natural.
\end{proof}

So \(\mathbf{SOb}_{\mathcal{C}}\) is the \emph{universal representation of \(\mathcal{C}\) in \(\mathbf{UBP}\)}: every other representation, i.e. monomorphism-preserving functor \(F\) from \(\mathcal{C}\) to \(\mathbf{UBP}\), comes with a
unique natural transformation \(F_{★}{:} \mathbf{SOb}_{\mathcal{C}} → \mathbf{UBP}\) that maps out precisely the structure of the image of \(\mathcal{C}\) under \(F\). The reason it factors through \(\mathopen{↓}\mathbf{UBP}\mathclose{↓}\) is that this is the densest possible stacking of all
monomorphisms in \(\mathcal{C}\) without equalizations, while the \(F_{★}\) for some functor \(F\) maps that structure into some larger preorder.

Note that the cells induced by functors and natural transformations in \(\overset{\hookrightarrow}{\mathbf{Cat}}\) according to \cref{item:88} do \emph{not} factor through
\(\mathopen{↓}\mathbf{UBP}\mathclose{↓}\). Rather, each \(F_{*}\) and \(F_{★}\) is a natural transformation that preserves tops (since functors preserve identities), so they factor through the wide
sub-2-category \(\mathbf{UBP}_{w}\) of \(\mathbf{UBP}\) on top-preserving maps.

So to summarize:

\begin{Theorem}
  \label{thm:1}
\index{subobject functor!representability}
\index{subobject functor!adjoint}
  The 2-functor \(\mathbf{SOb}_{(-)}{:} \overset{\hookrightarrow}{\mathbf{Cat}} → {\mathbf{\overset{\hookrightarrow}{Cat}}_{◆ \mathbf{UBP}}}\) is left adjoint to the 2-functor
  \(\mathbf{s}{:} {\mathbf{\overset{\hookrightarrow}{Cat}}_{◆ \mathbf{UBP}}} → \overset{\hookrightarrow}{\mathbf{Cat}}{:} (F{:} \mathcal{C} → \mathbf{UBP}) ↦ \mathcal{C}\).
\end{Theorem}
\begin{proof}
  We establish the adjunction by explicitly constructing the unit and counit. First, we verify that \(\mathbf{SOb}_{(-)}\) is a well-defined 2-functor. For the unit of the
  adjunction \(\mathbf{SOb}_{(-)} ⊣ \mathbf{s}\), since \(\mathbf{s}(\mathbf{SOb}_{\mathcal{C}}) = \mathcal{C}\), we may simply take the identity functor:
  \[
    \mathbf{id}_{\mathcal{C}}{:} \mathcal{C} → \mathbf{s}(\mathbf{SOb}_{\mathcal{C}}).
  \]

  For the counit, given an object \(E{:} \mathcal{C} → \mathbf{UBP}\) in \(\mathbf{\overset{\hookrightarrow}{Cat}}_{◆ \mathbf{UBP}}\), we need a 1-cell
  \(\mathbf{SOb}_{\mathcal{C}} → E\) over \(\mathbf{id}_{\mathcal{C}}\). This 1-cell \((\mathbf{id}_{\mathcal{C}}, E_{★})\) consists of \(\mathbf{id}_{\mathcal{C}}\) with the natural transformation
  \(E_{★}{:} \mathbf{SOb}_{\mathcal{C}} ⇒ E\), defined componentwise for each \(X ∈ \mathcal{C}\) by
  \[
    E_{★, X}{:} \mathbf{SOb}_{\mathcal{C}}(X) → E(X), \qquad ι ↦ \top_{E(ι)}:
  \]

\begin{figure}
  \centering
  \begin{tikzcd}[row sep=large]
    \mathcal{C} \ar[d, "\mathbf{SOb}_{\mathcal{C}}"{name=S}, bend left=50] \ar[d, "\mathbf{E}"'{name=E}, bend right=50] \\
    \mathbf{UBP} \ar[from=S, to=E, Rightarrow, shift left=6pt, shorten <= 6pt, xshift=3pt]
  \end{tikzcd}
\end{figure}

This is order-preserving since \(E(ι)\) is a monotone map and naturality follows from \cref{lmm:2}.

It remains to verify the triangle identities. The first identity requires
\((\mathbf{id}_{\mathcal{C}})_{*}⨾ (\mathbf{SOb}_{\mathcal{C}})_{★} = \mathbf{id}_{\mathbf{SOb}_{\mathcal{C}}}\). Indeed, for any \(X{:} \mathcal{C}\) this evaluates to

\[\mathbf{id}_{\mathbf{SOb}_{\mathcal{C}}}(X)⨾ \mathbf{id}_{\mathbf{SOb}_{\mathcal{C}}}(X) = \mathbf{id}_{\mathbf{SOb}_{\mathcal{C}}}(X).\]

The second identity requires \(\mathbf{s}(E_{★}) ⨾ \mathbf{id}_{\mathbf{s}(E)} = \mathbf{id}_{\mathbf{s}(E)}\). Since the underlying functor of the counit \(E_{★}\) is
\(\mathbf{id}_{\mathcal{C}}\) and \(ε_{\mathcal{C}} = \mathbf{id}_{\mathcal{C}}\), this identity holds trivially.

Having verified the triangle identities, we conclude that \(\mathbf{SOb}_{(-)} ⊣ \mathbf{s}\).
\end{proof}
\begin{Remark}
  \label{rmα_:8}
  So the basic distinguishing feature of \(\mathbf{UBP}\) that we use in this proof and the entire section is that each preorder in it has a top, which ensures that each
  monic \(ι: A \hookrightarrow X\) in a category \(\mathcal{C}\) is mapped to an ideal embedding \(ι_{*}\) under \(\mathbf{SOb}_{\mathcal{C}}\), gives us shadows of morphisms in categories and functors and natural
  transformations between categories ensures that many problematic functors from \(\mathcal{C}\) into the category of all preorders \(\mathbf{PrO}\) don't factor through \(\mathbf{UBP}\).
\end{Remark}
\begin{Remark}
  \label{rmk:6}
  If we don't assume that the natural transformations of  the 1-cells in \(\overset{\hookrightarrow}{\mathbf{Cat}}_{◆ \mathbf{UPB}}\) are top-preserving, then we still get a kind of laxification
  of an adjunction, where, given a category \(\mathcal{C}\) and functor \(F: \mathcal{C} → \mathbf{UBP}\), the morphisms from \((\mathcal{C}, \mathbf{SOb}_{\mathcal{C}})\) to \((\mathcal{C}, F)\) are bounded by \(F_{★}\) but no
  longer (weakly) unique.
\end{Remark}
\begin{Remark}
  \label{rmk:4}
  There is a variant of this diamond-over category for the contravariant subobject-functor \(\mathbf{SOb}^{*}\). Take some finite-limit-preserving functor \(F{:} \mathcal{C} → \mathcal{D}\). This induces for each
  \(X{:} \mathcal{C}\) a morphism \(F_{*}{:} \mathbf{SOb}(X) → \mathbf{SOb}(F(X))\), which preserves \(\top\) and meets if they exist in \(\mathcal{C}\). Given now a natural transformation
  \(η{:} F \Rightarrow G{:} \mathcal{C} → \mathcal{D}\), we get a natural transformation \(η^{*}_{X}{:} G⨾ \mathbf{SOb}_{\mathcal{D}} → F(X)⨾ \mathbf{SOb}_{\mathcal{D}}\) and for each \(ι{:} A \hookrightarrow X\) a diagram

\begin{figure}[H]
  \centering
  \begin{tikzcd}[column sep=large, row sep=large]
    & G(A) \ar[r, hook, "G(ι)"] & G(X) \\
    & η^{*}(G(ι)) \ar[u, "G(ι)^{*}(η_{X})"'] \ar[r, hook, "η_{X}^{*}(G(ι))"] & F(X) \ar[u, "η_{X}"'] \\
    F(A) \ar[uur, "η_{A}"] \ar[ur, hook, "{(η_{A}, F(ι))}"'{yshift=14pt, xshift=20pt}] \ar[urr, hook, "F(ι)"'] &&
  \end{tikzcd}
\caption{\label{fi♭_:7}}
\end{figure}
where \(\leq\) is the morphism to \(η^{*}(G(ι))\) induced by the universal property of the pullback, which is a monomorphism since \(F(ι)\) and \(η_{X}^{*}(G(ι))\) are.
Thus, we have for each such \(ι\) a triangle

\begin{figure}[H]
  \centering
  \begin{tikzcd}
    & \mathbf{SOb}(G(X)) \ar[ddr, "\mathbf{SOb}(η^{*}_{X})"] & \\
    & \rotatebox[origin=c]{90}{\(\leq\)} & \\
    \mathbf{SOb}(X) \ar[uur, "G_{*, X}"] \ar[rr, "F_{*, X}"'] && \mathbf{SOb}(F(X))
  \end{tikzcd}
  \caption{\label{fig:1}}
\end{figure}
so the inequality goes precisely in the other direction as before. Moreover, if \(η\) is Cartesian, then by definition each square as in \cref{fi♭_:7} is a pullback square, so that this inequality is an
equivalence. Thus, we can lift \(\mathbf{SOb}^{*}_{\mathcal{C}}\) to a 2-functor

\[\mathbf{SOb}^{*}_{(-)}{:} \mathbf{Cat}_{pb} → (\mathbf{Cat}_{pb})_{◆^{op} \mathbf{UBP}}\]

from the 2-category \(\mathbf{Cat}_{pb}\) of well-powered categories with pullbacks along monomorphisms, pullback-preserving functors and natural transformations to the 2-category \((\mathbf{Cat}_{pb})_{◆^{op} \mathbf{UBP}}\) which is defined as in \cref{def:diamond-over} but with the 2-cell modifications as in \cref{fig:1} instead of as in \cref{item:88}.

Note however that the contravariant \(\mathbf{SOb}^{*}\)-functor does not have the same universal property as the covariant one in the contravariant diamond category over
\(\mathbf{UBP}\): since for some \(ι{:} A \hookrightarrow X\) the morphism \(ι^{*}{:} \mathbf{SOb}(X) → \mathbf{SOb}(A)\) preserves tops, the entire calculus we're developing cannot be
applied and we're left without a way to compare \(\mathbf{SOb}_{\mathcal{C}}\) to some other functor \(F{:} \mathcal{C} → \mathbf{UBP}\).
\end{Remark}
\begin{Remark}
  \label{rmk:13}
  Note that there exists a 2-functor \(\mathbf{Cat}_{\hookrightarrow-prv} → \overset{\hookrightarrow}{\mathbf{Cat}}\) from the 2-category of small categories, monomorphism-preserving functors and natural
  transformations between them, which we can postcompose with \(\mathbf{SOb}\). This is what we are implicitly using when we extend smallness from monic categories to all
  categories. Since every monic category is a category, \(\mathbf{Cat}_{\hookrightarrow-prv} → \overset{\hookrightarrow}{\mathbf{Cat}}\) has a section. Moreover, for each category
  \(\mathcal{C}\) there exists a functor \( \overset{\hookrightarrow}{\mathcal{C}} \hookrightarrow \mathcal{C}\) embedding the monic part of a category. All in all this sure looks like a coreflective embedding — but it is very far from
  that. The universal property of an adjunction would force that for each category \(\mathcal{D}\) and functor \(F{:} \overset{\hookrightarrow}{\mathcal{C}} → \mathcal{D}\) there would exist a unique functor
  \(\mathcal{C} → \mathcal{D}\) over \(F\), we have nothing to determine how to map the non-monomorphisms in \(\mathcal{C}\). If we restrict to regular categories we have a decomposition
  \(f = e⨾ m\) for each \(f\) with \(m\) monic and \(e\) regular epi, and \(F\) fixes objects and monomorphisms, so in this case we would only have to determine how to map \(e\).
  Now, \(e\) is a coequalizer of its kernel pair \(k_{1}, k_{2}\) which is monic and \(F\) thus mapped by \(F\), so might try to map \(e\) to the coequalizer
  \(\mathbf{CEq}(F(k_{1}), F(k_{2}))\), but there is no reason to even assume that \((F(k_{1}), F(k_{2}))\) is a kernel pair in \(\mathcal{D}\), it can be a monic pair without a
  co-equalizer. And even if we assume \(\mathbf{CEq}(F(k_{1}), F(k_{2}))\) exists, there is no reason to assume that its codomain is \(\mathbf{s}(F(m))\), and one can construct
  counterexamples.
\end{Remark}

\subsection{\texorpdfstring{Universal Characterization of \(\overset{\hookrightarrow}{\mathbf{UBP}}\)}{Universal Characterization of UBP}}
\label{sec:univ-char-mathbf}

But we are not done yet: we have universally described \(\mathbf{SOb}\) as an adjoint to the 2-functor assigning to each representation \(F{:} \mathcal{C} → \mathbf{UBP}\) its source. But the
hardest part is still before us: to provide a universal description of \(\mathbf{UBP}\) as a 2-category. This is a slightly dizzying task somewhat akin to staring into two reflecting
mirrors, due to one crucial question: we have described the category \(\mathbf{\overset{\hookrightarrow}{Cat}}_{◆ \mathbf{UBP}}\), which is a kind of like a slice category in \(\mathbf{Cat}\), \emph{but
  how can we do that if \(\mathbf{UBP}\) is not a category?} It is, after all, a 2-category itself, and a sub-2-category of \(\mathbf{Cat}\). So we're treating something like a subset like an element.
Why can we do that?

The answer is that the hierachy of \(n\)-categories is closed under taking powers analogous to how \(\mathbf{Set}\) is under taking power sets, just with the categorical level
increasing each time. This is a triviality: each set is a special category, each category a special 2-category and so on. But this just means that there exist fully faithful embeddings
\(\mathbf{Set} \hookrightarrow \mathbf{Cat} \hookrightarrow \mathbf{2\text{-}Cat}...\), so that \(\mathbf{Cat}\) is both an object and a sub-2-category of \(\mathbf{2\text{-}Cat}\). So we can indeed define
\(\mathbf{\overset{\hookrightarrow}{Cat}}_{◆ \mathbf{UBP}}\) as a slice-category-like particular kind of comma-2-category from the image \(\mathbf{Im}(\mathbf{\overset{\hookrightarrow}{Cat}})\) of
\(\mathbf{\overset{\hookrightarrow}{Cat}}\) in \(\mathbf{2\text{-}Cat}\) to the constant functor on \(\mathbf{UBP}\). We might ask whether we could extend this structure further from \(\mathbf{Im}(\overset{\hookrightarrow}{\mathbf{Cat}})\) but it seems that we are touching on something universal: that
not just \(\mathbf{Cat}\) but every \(\mathbf{n\text{-}Cat}\) is reflected in \(\mathbf{UBP}\), in increasingly many ways as the possibilities for sensible subobject notions increase — we got a taste of this
in \cref{sec:category-theoretic}.

And this is also what we have to use to derive a universal
characterization of \(\mathbf{UBP}\): let us assume we have a similar structure \(D\) going from \(\mathbf{\overset{\hookrightarrow}{Cat}}\) to some other 2-category \(\mathcal{B}\). If we have a subobject-2-functor
\(\mathbf{SOb}_{\mathcal{B}}{:} \mathcal{B} → \mathbf{UBP}\), then we can take that functor right back to \(\mathbf{UBP}\), and then proceed along the lines of the proof of \cref{thm:1} to get a natural transformation
from \(\mathbf{SOb}_{(-)} → D⨾ \mathbf{SOb}_{\mathcal{B}}\). So to derive a universal characterization of \(\mathbf{UBP}\) we have to essentially re-do \cref{thm:1}, but on a 2-categorical level.
We're going to need some patience though: while most of the following pieces are simple to comprehend on their own, those pieces do sum up!

The trouble starts with the subobject structures in 2-categories: in \cref{sec:category-theoretic} we found several, so which one should we use? Presumably the most general. But to express
this we need to upgrade the notion of a monomorphism to 2-categories. The standard way to do this is to assume either faithfulness or fully faithfulness:

\begin{Definition}
  \label{defn_:2}\index{faithful 2-cell}\index{fully faithful 2-cell}
  \cite{FaithfulMorphismNLab} Let \(\mathcal{B}\) be a 2-category. A 1-cell \(ι {:} A → X\) in \(\mathcal{B}\) is called \emph{faithful} if for every object \(C\) of \(\mathcal{B}\), the induced functor of
  hom-categories
  \[
    ι_{*, \mathcal{C}} \;:=\; (- ⨾ ι) \;{:}\; \mathcal{B}(C,A) → \mathcal{B}(C,X)
  \]
  is faithful.

  Similarly, \(ι\) is \emph{fully faithful} if each \(ι_{*, \mathcal{C}}\) is fully faithful.
\end{Definition}

However, mere faithful functors are not good
subobject-notions to us: we already saw in \cref{xmp:5} that this results in pathologies such as every set being a subcategory of \(*\). This would be tolerable, but what really shows that
an additional condition is necessary is that the composition of such cells does not work right, as we will see in \cref{rmα_:25}. What is missing is isomorphism-reflectivity. Let
us recall the
notion of a pseudomonic morphism from \cite{carboniModulatedBicategories1994}:

\begin{Definition}
  \label{defn_:3}\index{pseudomonic morphism}\index{pseudomonic 2-category}
  \cite{PseudomonicMorphismNLab} A morphism \(ι: X → Y\) in a 2-category \(\mathcal{B}\) is \emph{pseudomonic} if \(ι\) is faithful and for each pair \(f, g\) and 2-isomorphism
  \(σ: f⨾ ι \overset{\simeq}{\Rightarrow} g⨾ ι\) there exists a 2-isomorphism \(\overline{σ}: f \Rightarrow g\) such that \(σ = \overline{σ}⨾ ι\).

  \(\mathcal{B}\) itself is \emph{pseudomonic} if all its morphisms
  are pseudomonics.

  \(\mathcal{B}\) is \emph{2-pseudomonic} if it is pseudomonic and all its 2-cells are monic in their hom-categories.

  A morphism \(f{:} X → Y\)  is \emph{monomorphism-preserving} if for every \(A{:} B\) the functor \(f_{*}{:} \mathcal{B}(A, X) → \mathcal{B}(A, Y)\) is monomorphism-preserving.

  A 2-functor \(F{:} \mathcal{B} → \mathcal{B}'\) is \emph{2-pseudomonic} if

  \begin{enumerate}
  \item \label{item:53} for each \(X, Y{:} \mathcal{B}\) the induced functor \(F_{X, Y}{:} \mathcal{B}(X, Y) → \mathcal{B}'(F(X), F(Y))\) is pseudomonic,
  \item\label{item:54} \(F\) preserves pseudomonics and reflects equivalences.
  \end{enumerate}
\end{Definition}
\begin{Example}
  \label{xmp:10}
  Every monic category is a 2-pseudomonic 2-category. This follows because 2-cell preservation and isomorphism-reflection are trivially true for morphisms in a category.
\end{Example}

We do need one further definition, which is slightly technical: note that in \(\mathbf{Cat}\) a natural transformation \(μ\) consisting of monomorphisms is monic (since for each
pair of natural transformations \(φ, φ'\) that we can precompose \(μ\) with, componentwise evaluation ensures that \(μ\) distinguishes \(φ\) and \(φ'\)). Thus,
monic natural transformations are stable under precomposition with functors but not postcomposition. This is what allows \cref{xmp:4} and its variants to work, so we need to
abstract it:

\begin{Definition}
  \label{defn_:11}
  A 2-category is \emph{monically left-ideal} if for each monic 2-cell \(φ{:} f \Rightarrow g{:} X → Y\)  and \(h{:} A → X\) it holds that the left whiskering \(h⨾ φ\) is monic.
\end{Definition}
\begin{Example}
  \label{xmp_:16}
  Each 2-pseudomonic 2-category is monically left-ideal, since all its 2-cells are monic.
\end{Example}

On these we can nicely define subobject preorders, generalizing \cref{xmp:5}:

\begin{Definition}
  \label{defn_:6}
  \index{pseudomonic subobject preorder}
  Given an object \(X\) of a monically left-ideal 2-category \(\mathcal{B}\), the \emph{(2-covariant) pseudomonic subobject preorder} \(\mathbf{SOb}_{f, \mathcal{B}}(X)\) is the preorder:

  \begin{enumerate}
  \item\label{item:97} whose objects are pseudomonics \(ι{:} A \hookrightarrow X\) into \(X\) in \(\mathcal{B}\),
  \item\label{item:98} such that \(ι \leq κ{:} B \hookrightarrow X\) if and only if there exists a morphism \(μ{:} A → B\) and a monic 2-cell \(φ{:} ι \Rightarrow μ⨾ κ\).
  \end{enumerate}
\end{Definition}
\begin{Lemma}
  \label{le↘_:4}
  In a monically left-ideal 2-category \(\mathcal{B}\), \(\mathbf{SOb}_{f, \mathcal{B}}(X)\) is in fact a preorder.
\end{Lemma}
\begin{proof}
  We must show that the relation \(\le\) on the class of pseudomonic 1-cells \(ι {:} A → X\) into a fixed object \(X {:} \mathcal{B}\) is reflexive and transitive.

  Let \(ι {:} A \hookrightarrow X\) be a pseudomonic. The identity 2-cell is monic in any 2-category, so this witnesses
  \(ι \le ι\). For transitivity, suppose we have pseudomonics \(ι {:} A \hookrightarrow X\), \(κ {:} B \hookrightarrow X\), and \(λ {:} C \hookrightarrow X\) with \(ι \le κ\) and \(κ \le λ\). By definition, there
  exists an \(f{:} A → B\) and a monic 2-cell \(α {:} ι ⇒ f ⨾ κ\) in \(\mathcal{B}(A,X)\) and a monomorphism-preserving \(g{:} B → C\) and a monic 2-cell
  \(β {:} κ ⇒ g ⨾ λ\) in \(\mathcal{B}(B,X)\). Define \(h := f ⨾ g {:} A → C\) and set
  \[
    \gamma \;:=\; α ; (f ⨾ β) {:} ι \Rightarrow h ⨾ λ .
  \]

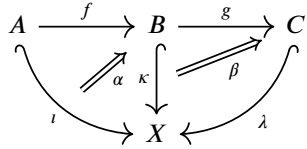
  \begin{figure}[H]
    \centering
    \begin{tikzcd}[column sep=large, row sep=large]
  A \ar[r, "f"]
    \ar[dr, hook', bend right=35, "ι"', ""{name=I}]
  & B \ar[d, hook, "κ"', ""{name=K}]
  \ar[r, "g"]
  & C \ar[dl, hook, bend left=35, "λ"]
  \\
  & X
  \arrow[Rightarrow, from=I, to=1-2, "α"', shorten <=3pt, shorten >=5pt]
  \arrow[Rightarrow, from=K, to=1-3, "β"', shorten <=3pt, shorten >=5pt]
\end{tikzcd}
    % \caption{\label{fig:8}}
  \end{figure}

  Since \(\mathcal{B}\) is monically left-ideal, whiskering preserves monics. Moreover, vertical composition of monics is monic and \(\gamma\) is monic, so \(ι \le λ\). Thus \(\le\) is reflexive and transitive.
\end{proof}

\begin{Remark}
  \label{rmα_:23}
  The slice category of an object in a category is constructed as a comma category, i.e. a special kind of 2-limit. In 2-categories, this definition trifurcates into a lax, an
  oplax and a weak variant. Accordingly, \cref{defn_:6} has a 2-contravariant variant \(\mathbf{SOb}_{f^{op}, \mathcal{B}}(X)\) with inverse 2-cells, and a weak variant \(\mathbf{SOb}_{\mathopen{∣}f\mathclose{∣}, \mathcal{B}}(X)\), where the 2-cells have to be
  2-isomorphisms. But since we don't use these other functors in this subsection, we will write \(\mathbf{SOb}_{\mathcal{B}}(X)\) for \(\mathbf{SOb}_{f, \mathcal{B}}(X)\) to simplify notation.
\end{Remark}

\begin{Remark}
  \label{rmα_:25}
  For \(ι{:} X \hookrightarrow Y\) faithful in a 2-category \(\mathcal{B}\) whose morphisms are faithful but not isomorphism-reflecting, we \emph{do} get for each
  \(κ{:} A \hookrightarrow X, κ'{:} B \hookrightarrow X\) and pair \((f, α){:} (A \hookrightarrow B, κ \hookrightarrow f⨾ κ')\) displaying \(κ \leq κ'\) a pair
  \((f, α⨾ ι){:} (A \hookrightarrow B, κ⨾ ι \hookrightarrow f⨾ κ⨾ ι)\), so that \(κ⨾ ι \leq κ' ⨾ ι\). Thus, we get an order-preserving morphism
  \(ι_{*}{:} \mathbf{SOb}_{\mathcal{B}}(X) \hookrightarrow \mathbf{SOb}_{\mathcal{B}}(Y)\), which is essentially injective since \(ι\) is faithful. But, given
  \(f{:} A \hookrightarrow B\) and \(β{:} f⨾ κ⨾ ι \hookrightarrow κ'⨾ ι\) there is no reason to assume that \(β\) is in the image of \(ι_{*}{:} \mathcal{B}(A, X) → \mathcal{B}(A, Y)\), which means that
  \(ι_{*}\) is \emph{not} generally a monomorphism because \(β\) might be an isomorphism. So we need to at least assume that \(ι\) is isomorphism-reflecting i.e. each induced
  functor \(ι_{*}\) is isomorphism-reflecting. So we can understand isomorphism-reflectivity as a weakening of injectivity. This is why it excludes pathologies as in \cref{xmp:5},
  which go against our intuition on what a subcategory should be and why we have to use pseudomonics, not just faithful morphisms.
\end{Remark}

\begin{Proposition}
  \label{■pn_:6}
  Let \(\mathcal{B}\) be a 2-pseudomonic 2-category. Then the assignment \(X ↦ \mathbf{SOb}_{\mathcal{B}}(X)\) from \cref{defn_:6} defines a 2-functor \(\mathbf{SOb}_{\mathcal{B}}{:} \mathcal{B} → \overset{\hookrightarrow}{\mathbf{UBP}}\).
\end{Proposition}
\begin{proof}
  For a 1-cell \(f {:} X → Y\) in \(\mathcal{B}\), define the action of \(\mathbf{SOb}_{\mathcal{B}}\) on \(f\) by postcomposition:
  \[
    \mathbf{SOb}_{\mathcal{B}}(f) {:} \mathbf{SOb}_{\mathcal{B}}(X) \longrightarrow \mathbf{SOb}_{\mathcal{B}}(Y), \qquad ι \mapsto ι⨾ f.
  \]
  Since \(\mathcal{B}\) is 2-pseudomonic, the composite \(ι⨾ f\) is pseudomonic, hence an object of \(\mathbf{SOb}_{\mathcal{B}}(Y)\). To verify monotonicity, let
  \(ι \leq \kappa\) in \(\mathbf{SOb}_{\mathcal{B}}(X)\). Then there exist a pseudomonic \(g {:} A \hookrightarrow B\) and a monic 2-cell
  \(α {:} ι \Rightarrow g⨾ \kappa\). Whiskering with \(f\) yields a monic 2-cell \(α⨾ f {:} ι⨾ f \Rightarrow g⨾ (\kappa⨾ f)\), which witnesses \(ι⨾ f \leq \kappa⨾ f\).

  For a 2-cell \(θ {:} f \Rightarrow g {:} X → Y\), define
  \(\mathbf{SOb}_{\mathcal{B}}(θ) {:} \mathbf{SOb}_{\mathcal{B}}(f) \Rightarrow \mathbf{SOb}_{\mathcal{B}}(g)\) pointwise. For each \(ι {:} A \hookrightarrow X\), whiskering yields
  \(ι⨾ θ {:} ι⨾ f \Rightarrow ι⨾ g\). Because \(\mathcal{B}\) is 2-pseudomonic, \(ι⨾ θ\) is a monic 2-cell, corresponding to the relation
  \(ι⨾ f \leq ι⨾ g\) in \(\mathbf{SOb}_{\mathcal{B}}(Y)\). The interchange law in \(\mathcal{B}\) ensures that this assignment is natural.

  The 2-functor axioms follow directly from the unitality and associativity of composition in \(\mathcal{B}\). Specifically, \(\mathbf{SOb}_{\mathcal{B}}(\mathbf{id}_X)\) is the identity on
  \(\mathbf{SOb}_{\mathcal{B}}(X)\) since \(ι⨾ \mathbf{id}_X = ι\), and \(\mathbf{SOb}_{\mathcal{B}}\) preserves composition of 1-cells since
  \(ι⨾ (f⨾ g) = (ι⨾ f)⨾ g\). Preservation of vertical and horizontal composition of 2-cells follows analogously from the functoriality of whiskering.
\end{proof}

\begin{Remark}
  \label{rmα_:9}
  Let us take a look at the action of \(\mathbf{SOb}\) on \(\overset{\hookrightarrow}{\mathbf{UBP}}\) itself, which will become crucial in a moment: it sends a preorder \(P\) to the preorder
  \(\mathbf{SOb}_{\overset{\hookrightarrow}{\mathbf{UBP}}}(P)\) we discussed in \cref{xmp:3} and acts on morphisms by postcomposition. Thus there exist morphisms

  \begin{enumerate}
  \item\label{item:21} \(\{-\}{:} P → \mathbf{SOb}_{\overset{\hookrightarrow}{\mathbf{UBP}}}(P)\) which maps \(a ∈ P\) to the one-element sub-preorder \(\{a\}\),
  \item \label{item:9} \((-){:} P → \mathbf{SOb}_{\overset{\hookrightarrow}{\mathbf{UBP}}}(P)\) whose image is the ideal \((a)\),
  \end{enumerate}

  \begin{figure}[H]
    \centering
    \begin{tikzcd}
  P && \mathbf{SOb}_{\overset{\hookrightarrow}{\mathbf{UBP}}}(P)
	\arrow[""{name=0, anchor=center, inner sep=0}, shift left=3, from=1-1, to=1-3, hook]
	\arrow[""{name=1, anchor=center, inner sep=0}, shift right=3, from=1-1, to=1-3, hook']
	\arrow[from=1-3, to=1-1, two heads]
	\arrow[phantom, from=1, to=0, "\rotatebox{90}{\leq}"{description, fill=white, inner sep=2pt}]
 \end{tikzcd} \\
\[
\{-\}⨾\top_{(-)}=\mathbf{id}=(-)⨾\top_{(-)}.
\]

    % \caption{\label{fig:7}}
  \end{figure}

  It follows almost immediately from the definitions that \(\{-\}\) forms a strict natural 2-transformation and from \cref{■pn_:ofs} that \((-)\) forms an oplax one. But because each preorder in \(\mathbf{UBP}\) has a top, we can define a morphism
  \(\top_{(-)}{:} \mathbf{SOb}_{\overset{\hookrightarrow}{\mathbf{UBP}}}(P) \twoheadrightarrow P{:} ι ↦ \top_{ι}\) such that both \(\{-\}\) and \((-)\) are sections of
  \(\top_{(-)}\). Moreover, it also follows almost immediately from the definitions that \(\{-\}\) is a minimal section of \(\top_{(-)}\) and \((-)\) a maximal one, so that the three form a
  unity of opposites in the sense of \cite{lawvereUnityIdentityOpposites1996}. Note that the same does not hold for filters since we would need downward-boundedness.
\end{Remark}

Now let us define what it means for something to be diamond-over:

\begin{Definition}
  \label{dfn:4}
  An \emph{oplaxly pointed 2-category} \((\mathcal{B}, ★_{\mathcal{B}}) := \mathcal{B}_{★}\) is a 2-category with terminal object \(*\) and an oplax natural transformation
  \(★_{\mathcal{B}}{:} \mathbf{K}(*) → \mathbf{id}_{\mathcal{B}}\) from the constant functor \(\mathbf{K}(*)\) on the terminal object \(*\).

  We call a morphism
  \(f{:} X → Y\) in \(\mathcal{B}\) \emph{point-preserving} if \(★_{\mathcal{B}, X}⨾ f = ★_{\mathcal{B}, Y}\) and a natural transformation \emph{point-preserving} if each of its component-morphisms is.
\end{Definition}
\begin{Definition}[Diamond-over 2-category \(\mathcal{A}_{◆\mathcal{B}}\)]
  \label{def:diamond-over}
\index{diamond-over 2-category}
  Let \(\mathcal{A}\) be a sub-2-category of \(\mathbf{Cat}\), and let \(\mathcal{B}_{★}\) be an oplaxly pointed 2-category.
  % Assume we are given a fixed 2-functor
  % \[
  %   E {:} \mathcal{A} \longrightarrow \mathcal{B}.
  % \]
  The \emph{diamond-over 2-category} \(\mathcal{A}_{◆\mathcal{B}_{★}}\) \emph{of} \(\mathcal{A}\) \emph{over} \(\mathcal{B}_{★}\) is defined as follows.

  An object of \(\mathcal{A}_{◆\mathcal{B}_{★}}\) is a pair \((\mathcal{C},E)\) where:
  \begin{itemize}
  \item \(\mathcal{C}\) is an object of \(\mathcal{A}\);
  \item \(E {:} \mathcal{C} → \mathcal{B}\) is a 2-functor (necessarily trivial on 2-cells since \(\mathcal{C}\) is a category).
  \end{itemize}

  A 1-cell
  \[
    (F,ε) {:} (\mathcal{C},E) \longrightarrow (\mathcal{C}',E')
  \]
  in \(\mathcal{A}_{◆\mathcal{B}_{★}}\) is an oplaxly (or, less confusingly, covariantly) commuting triangle, so:
  \begin{itemize}
  \item a functor \(F {:} \mathcal{C} → \mathcal{C}'\) in \(\mathcal{A}\);
  \item a point-preserving 2-natural transformation
    \[
      ε {:} E \Rightarrow F⨾ E'
    \]
    in \(\mathcal{B}_{★}\),
  \end{itemize}

  \begin{figure}[H]
    \centering
    \begin{tikzcd}
      \mathcal{C} \ar[rr,"F"] \ar[dr,"E"', ""{name=E}] && \mathcal{C}' \ar[dl,"E'"] \\
      & \mathcal{B} \ar[Rightarrow, from=E, to=1-3, "ε"{below} description]
    \end{tikzcd}
    % \caption{\label{fi♭_:6}}
  \end{figure}

  A 2-cell
  \[
    (α,Φ) {:} (F,ε) \Rightarrow (F',ε')
  \]
  between 1-cells \((F,ε),(F',ε') {:} (\mathcal{C},E) → (\mathcal{C}',E')\) consists of:
  \begin{itemize}
  \item a natural transformation \(α {:} F \Rightarrow F'\) in \(\mathcal{A}\);
  \item a modification \(Φ{:} ε; (α⨾ E') \Rrightarrow ε'\).
  \end{itemize}

  \begin{figure}[H]
    \centering
    \begin{tikzcd}
      F⨾ E' \ar[rr, Rightarrow, "α⨾ E'"] && F'⨾ E' \\
      & E \ar[ul, Rightarrow, "ε"] \ar[ur, Rightarrow, "ε'"{below, yshift=-2pt, xshift=3pt}, ""{name=E'}] \ar[triple, to=E', from= 1-1, "Φ", shorten >=6pt, shorten <=6pt]
    \end{tikzcd}
    % \caption{\label{fi♭_:33}}
  \end{figure}

  Composition in \(\mathcal{A}_{◆\mathcal{B}_{★}}\) is defined by horizontal/vertical composition in \(\mathcal{A}\) and \(\mathcal{B}\) in the obvious way. Identities are given by
  \((\mathbf{id}_{\mathcal{C}},\mathbf{id}_E)\) on objects, and \((\mathbf{id}_F,\mathbf{id}_ε)\) on 1-cells. The 2-category axioms follow from those of \(\mathcal{A}\) and \(\mathcal{B}\).
\end{Definition}
\begin{Remark}
  \label{rmα_:19}
  This reduces to our previous definition: for \(\mathcal{A} = \overset{\hookrightarrow}{\mathbf{Cat}}\) and \(\mathcal{B} = \mathbf{UBP}\), the objects of \(\mathcal{A}_{◆\mathcal{B}_{★}}\) are representations \(F {:} \mathcal{C} → \mathbf{UBP}\); 1-cells are triangles \((H,ε)\) as in
  \cref{item:88}; 2-cells are natural transformations between the 1-cell functors over \(\mathbf{UBP}\) with a corresponding modifier (which has to be unique since \(\mathbf{UBP}\)
  consists of preorders). Thus, \(\mathcal{A}_{◆\mathcal{B}_{★}}\) is exactly \( \mathbf{\overset{\hookrightarrow}{Cat}}_{◆ \mathbf{UBP}}\).
\end{Remark}

Now, a key point is:

\begin{Proposition}
  \label{■pn_:7}
  Let \(\mathcal{A}\) be a sub-2‑category of \(\mathbf{Cat}\), \(\mathcal{B}_{★}, \mathcal{B}'_{★'}\) be oplaxly pointed 2‑categories, and \(F {:} \mathcal{B} → \mathcal{B}'\) a 2‑functor preserving the oplax point. Then \(F\) induces a 2‑functor \((-⨾ F) {:} \mathcal{A}_{◆\mathcal{B}_{★}} \longrightarrow \mathcal{A}_{◆\mathcal{B}_{★}'}\).
\end{Proposition}
\begin{proof}
  We define \((-⨾ F)\) by postcomposition with \(F\). On objects, we set
  \[
    (-⨾ F)(\mathcal{C}, H) := (\mathcal{C}, H⨾F).
  \]
  On a 1-cell \((G, ε) {:} (\mathcal{C}, H) \to (\mathcal{D}, K)\), where \(G {:} \mathcal{C} \to \mathcal{D}\) and \(ε {:} H \Rightarrow G⨾K\), we define
  \[
    (-⨾ F)(G, ε) := (G, F(ε)) {:} (\mathcal{C}, H⨾F) \to (\mathcal{D}, K⨾F),
  \]
  where \(F(ε)\) denotes the whiskering of the 2-natural transformation \(ε\) with \(F\), which is point-preserving since \(ε\) and \(F\) are.
  On a 2-cell \((α, Φ) {:} (G, ε) \Rightarrow (G', \delta)\), consisting of a 2-natural transformation
  \(α {:} G \Rightarrow G'\) and a modification \(Φ {:} ε; (α⨾K) \Rightarrow \delta\), we set
  \[
    (-⨾ F)(α, Φ) := (α, Φ⨾ F),
  \]
  where \(Φ⨾ F\) is the whiskering of \(Φ\) with \(F\).

  To verify that \((-⨾ F)\) is a 2-functor, consider composable 1-cells \((G, ε) {:} (\mathcal{C}, H) \to (\mathcal{D}, K)\) and
  \((L, η) {:} (\mathcal{D}, K) \to (\mathcal{E}, J)\). Their composite in \(\mathcal{A}_{◆\mathcal{B}_{★}}\) is given by \((G⨾L, \gamma)\), where \(\gamma\) is the horizontal composite
  \[
    \gamma {:} H \xRightarrow{ε} G⨾K \xRightarrow{G⨾η} G⨾L⨾J.
  \]
  Applying \(F\), we obtain
  \[
    \gamma⨾ F {:} H⨾F \xRightarrow{ε⨾ F} G⨾K⨾F \xRightarrow{G⨾ η⨾ F} G⨾L⨾J⨾F.
  \]
  Because \(F\) is a 2-functor, this \(F(\gamma)\) is exactly the structure map of the composite \((G, F(ε))⨾ (L, F(η))\) in \(\mathcal{A}_{◆\mathcal{B}_{★}'}\). The preservation of identity 1-cells follows
  identically from \(F(\mathbf{id}) = \mathbf{id}\).

  Finally, vertical and horizontal composition of 2-cells in \(\mathcal{A}_{◆\mathcal{B}_{★}}\) are computed componentwise in \(\mathcal{A}\) and \(\mathcal{B}\). Since \(F\) strictly preserves all compositions and
  identities of 2-cells, the assignment \((α, Φ) \mapsto (α, Φ⨾ F)\) is strictly functorial with respect to both vertical and horizontal composition. Thus, \((-⨾ F)\) is a strict 2-functor.
\end{proof}

Thus, we get for each oplaxly pointed 2-category \(\mathcal{B}_{★}\) and sub-2-category \(\mathcal{A}\) of \(\overset{\hookrightarrow}{\mathbf{Cat}}\) a 2-functor \((-⨾ \mathbf{SOb}_{\mathcal{B}}){:} \mathcal{A}_{◆ \mathcal{B}_{★}} → \mathcal{A}_{◆ \mathbf{UBP}}\). Now, if \(\mathcal{A}\) consists of monic
categories and monomorphism-preserving functors then all we really need to do is find a comparison with \(\mathbf{SOb}{:} \mathcal{A}_{◆ \mathbf{UBP}}\) and verify it checks out as a universal
property, and we're done. But if \(\mathcal{B}\) is 2-pseudomonic, then for each \(ι{:} A \hookrightarrow X\) in any \(\mathcal{C}{:} \mathcal{A}\) and any 2-functor \(F{:} \mathcal{C} → \mathcal{B}\) there exists an element \(F(ι) ∈ \mathbf{SOb}(F(X))\), so that we can
define for each \(X{:} \mathcal{C}\) an order-preserving function \(F_{*, X}{:} \mathbf{SOb}_{\mathcal{C}} → \mathbf{SOb}_{\mathcal{B}}, ι ↦ \top_{F(ι)}\) , which is unique since it is determined by its values on subobjects \(ι\) of \(X\).

\begin{Lemma}
  \label{le↘_:5}
  For each \(\mathcal{C}{:} \mathcal{A}\) and \(F{:} \mathcal{C} → \mathcal{B}_{★}\) with \(\mathcal{B}_{★}\) 2-pseudomonic and oplaxly pointed, \(F_{*}\) is a point-preserving 2-natural transformation \(\mathbf{SOb}_{\mathcal{C}} \Rightarrow (F⨾ \mathbf{SOb}_{\mathcal{B}}){:} \mathcal{C} → \mathbf{UBP}\).
\end{Lemma}
\begin{proof}
  Since \(\mathcal{C}\) is a 1-category, 2-naturality holds trivially. It suffices to verify that each \(F_{*, X}\) is monotone and natural in 1-cells.

  For monotonicity, let \(ι \leq \kappa\) in \(\mathbf{SOb}_{\mathcal{C}}(X)\), with representatives \(ι{:} A \hookrightarrow X\) and \(\kappa{:} B \hookrightarrow X\). By definition, there exists a (necessarily monic) 1-cell
  \(f{:} A \to B\) such that \(ι = f ⨾ \kappa\). Applying \(F\), we obtain a 1-cell \(F(f){:} F(A) \to F(B)\) in \(\mathcal{B}\). Since \(\mathcal{B}\) is pseudomonic, \(F(f)\) is pseudomonic, and thus
  \(F(ι) \leq F(\kappa)\) in \(\mathbf{SOb}_{\mathcal{B}}(F(X))\). Hence \(F_{*, X}\) is order-preserving. Point-preservation holds since \(\top_{F(ι)} = \mathbf{id}_{F(ι)}\).

  For naturality, let \(g{:} X \to Y\) be a 1-cell in \(\mathcal{C}\). We must show that the maps \(F_{*, X} ⨾ F(g)_{*}\) and
  \(g_{*} ⨾ F_{*, Y}\) from \(\mathbf{SOb}_{\mathcal{C}}(X)\) to \(\mathbf{SOb}_{\mathcal{B}}(F(Y))\) are equal. Evaluating both on an element \(ι{:} A \hookrightarrow X\), the left-hand side yields
  \[
    (F_{*, X} ⨾ F(g)_{*})(ι) = F_{*}(g)(F(ι)) = F(ι) ⨾ F(g),
  \]
  while the right-hand side yields
  \[
    (g_{*} ⨾ F_{*, Y})(ι) = F(ι ⨾ g).
  \]
  By the functoriality of \(F\), we have \(F(ι) ⨾ F(g) = F(ι ⨾ g)\), so the two composites agree pointwise. Thus, \(F_{*}\) is a natural transformation
  \(\mathbf{SOb}_{\mathcal{C}} \Rightarrow F ⨾ \mathbf{SOb}_{\mathcal{B}}\).
\end{proof}

Now, let us note that for each \(\mathcal{A} \hookrightarrow \mathbf{Cat}\) and each oplaxly pointed 2-category \(\mathcal{B}_{★}\), the 2-category \(\mathcal{A}_{◆ \mathcal{B}_{★}}\) is fibered over \(\mathcal{A}\) (in a general sense) through the 2-functor \(\mathbf{s}_{\mathcal{A}_{◆ \mathcal{B}_{★}}}\), which maps \(F{:} \mathcal{C} → \mathcal{B}\) to \(\mathcal{C}\), and
similar for higher cells.

\begin{Definition}
  \label{defn_:4}
  \index{diamond}
  Given an oplaxly pointed 2-category \(\mathcal{B}_{★}\) and a sub-2-category \(\mathcal{A} \hookrightarrow \mathbf{Cat}\), \emph{an \(\mathcal{A}\)-diamond in \(\mathcal{B}_{★}\)} is a section
  \(\mathcal{A} → \mathcal{A}_{◆ \mathcal{B}_{★}}\) of \(\mathbf{s}_{\mathcal{A}_{◆ \mathcal{B}_{★}}}\). We denote the 2-category of sections of \(\mathbf{s}_{\mathcal{A}_{◆ \mathcal{B}_{★}}}\) by \(\mathcal{A} ◆ \mathcal{B}_{★}\).
\end{Definition}
\begin{Remark}
  \label{rmα_:7}
  Thus, a diamond is a particular kind of reflection of \(\mathcal{A}\) in \(\mathcal{B}_{★}\), which assigns to each \(\mathcal{C}{:} \mathcal{A}\) a representation of \(\mathcal{A}\) in \(\mathcal{B}\) and to each functor and natural transformation in \(\mathcal{A}\)
  parameterized diagrams in \(\mathcal{B}\). So we have \(\mathcal{A}\) as a 2-category, so an object of \(2-\mathbf{Cat}\), but in a diamond we ``untangle'' \(\mathcal{A}\) into a diagram from which
  we define a coherent system out of. We can draw this like this:

  
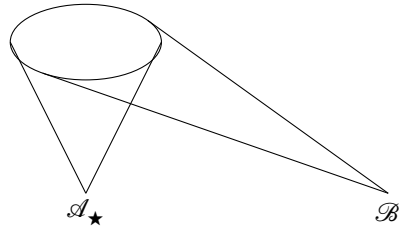
\begin{figure}[H]
    \centering
    \begin{tikzpicture}
    \draw[name=e] (0,0) ellipse (1cm and 0.5cm);
    \draw[black] (-1, 0) -- (0, -2);
    \draw[black] (1, 0) -- (0, -2);
    \node at (0, -2.25){\(\mathcal{A}_{★}\)};
    \node at (4, -2.25){\(\mathcal{B}\)};
    \draw[black] ({1/8 * (1 - sqrt(31))}, {1/16 * (-1 - sqrt(31))}) -- (4, -2);
    \draw[black] ({1/8 * (1 + sqrt(31))}, {1/16 * (sqrt(31) - 1)}) -- (4, -2);
  \end{tikzpicture}
    % \caption{\label{fig:4}}
\end{figure}

Where the ellipse contains \(\mathcal{A}\) as a diagram in \(\mathbf{2-Cat}\) and the cone from that ellipse to \(\mathcal{B}\) is actually a special kind of 2-cocone consisting of a 2-functor
\(D(\mathcal{C}){:} \mathcal{C} → \mathcal{B}\) for each category \(\mathcal{C}\) in \(\mathcal{A}\) and lax higher cells as described in \cref{def:diamond-over}. But in the following we will denote such diamonds in diagrams like
this:

\begin{figure}[H]
  \centering
  \begin{tikzcd}
   \mathcal{A} \ar[r, -{Triangle[length=9mm, width=3mm]}] & \mathcal{B}_{★}.
 \end{tikzcd}
  % \caption{\label{fig:6}}
\end{figure}
\end{Remark}
\begin{Remark}
  \label{rmk:3}
  By \cref{■pn_:7} we can postcompose a diamond with a 2-functor as if it was a functor. Since we use the \((-)_{*}\) notation quite a lot, we will just denote this as \(D⨾ F\).
\end{Remark}
\begin{Example}
  \label{xmp_:17}
  Let \(\mathbf{Cat}_{term}\) be the full sub-2-category of \(\mathbf{Cat}\) whose objects are small categories with a terminal object. Then slicing constitutes a
  \(\mathbf{Cat}_{term}\)-valued diamond (we spell out a proof below, but the procedure is the same as for \(\overset{\hookrightarrow}{\mathbf{Cat}}_{◆ \mathbf{UBP}}\)). Moreover, it seems like
  the proof of \cref{thm:1} can be straightforwardly adapted to \(\mathbf{Cat}_{◆ \mathbf{Cat}_{term}}\). Whether the currently going-on proof of \cref{thm:3} can also be adapted
  is less clear, but an analogue of it for non-monics should be true.
\end{Example}
\begin{proof}
  We construct a 2-functor \(D: \mathbf{Cat}_{term} \to (\mathbf{Cat}_{term})_{◆ \mathbf{Cat}_{term}}\) that is a section of the source 2-functor \(\mathbf{s}\). Let \(\mathcal{A} = \mathbf{Cat}_{term}\).

  \emph{Action on objects.}
  For each category \(\mathcal{C} {:} \mathcal{A}\), we assign \(D(\mathcal{C}) = (\mathcal{C}, E_{\mathcal{C}})\), where \(E_{\mathcal{C}}{:} \mathcal{C} \to \mathbf{Cat}_{term}\) is the strict 2-functor given by taking slice categories:
  \[
    E_{\mathcal{C}}(X) \;=\; \mathcal{C}_{/X}.
  \]
  For a morphism \(f {:} X \to Y\) in \(\mathcal{C}\), the functor \(E_{\mathcal{C}}(f) = f_{*}{:} \mathcal{C}_{/X} \to \mathcal{C}_{/Y}\) acts by postcomposition, sending an object \(a {:} A \to X\) to \(a ⨾ f {:} A \to Y\). Because \(\mathcal{C}_{/X}\) has a terminal object \(\mathbf{id}_X\), it belongs to \(\mathbf{Cat}_{term}\). The oplax point \(\star_{\mathcal{C}_{/X}}{:} \mathbf{1} \to \mathcal{C}_{/X}\) selects this terminal object.

  \emph{Action on 1-cells.}
  For a functor \(F {:} \mathcal{C} \to \mathcal{D}\) in \(\mathcal{A}\), we define \(D(F) = (F, \varepsilon_F)\), where \(\varepsilon_F{:} E_{\mathcal{C}} \Rightarrow F ⨾ E_{\mathcal{D}}\) is a point-preserving 2-natural transformation.
  For each \(X {:} \mathcal{C}\), the component \((\varepsilon_F)_X {:} \mathcal{C}_{/X} \to \mathcal{D}_{/F(X)}\) is the functor \(F_{/X}\) induced by postcomposition with \(F\):
  \[
    (A \xrightarrow{a} X) \quad \mapsto \quad (F(A) \xrightarrow{F(a)} F(X)).
  \]
  To see that \(\varepsilon_F\) is point-preserving, evaluate it on the point of \(\mathcal{C}_{/X}\), which is \(\mathbf{id}_X\):
  \[
    (\varepsilon_F)_X(\mathbf{id}_X) \;=\; F(\mathbf{id}_X) \;=\; \mathbf{id}_{F(X)},
  \]
  which is exactly the point of \(\mathcal{D}_{/F(X)}\).
  To verify 2-naturality, let \(f {:} X \to Y\) be a morphism in \(\mathcal{C}\). The following square of categories and functors must strictly commute:
  \[
    \begin{tikzcd}
      \mathcal{C}_{/X} \ar[r, "F_{/X}"] \ar[d, "f_{*}"'] & \mathcal{D}_{/F(X)} \ar[d, "F(f)_{*}"] \\
      \mathcal{C}_{/Y} \ar[r, "F_{/Y}"'] & \mathcal{D}_{/F(Y)}
    \end{tikzcd}
  \]
  For any object \(a {:} A \to X\) in \(\mathcal{C}_{/X}\), tracing right-then-down yields \(F(a) ⨾ F(f) = F(a ⨾ f)\), while down-then-right yields \(F(a ⨾ f)\). Thus, the square strictly commutes.

  \emph{Action on 2-cells.}
  For a natural transformation \(\alpha {:} F \Rightarrow G {:} \mathcal{C} \to \mathcal{D}\) in \(\mathcal{A}\), we define \(D(\alpha) = (\alpha, \Phi)\).
  By \cref{def:diamond-over}, we must provide a modification \(\Phi {:} \varepsilon_F ; (\alpha ⨾ E_{\mathcal{D}}) \Rrightarrow \varepsilon_G\).
  For each \(X {:} \mathcal{C}\), the transformation \((\alpha ⨾ E_{\mathcal{D}})_X {:} \mathcal{D}_{/F(X)} \Rightarrow \mathcal{D}_{/G(X)}\) is postcomposition with \(\alpha_X {:} F(X) \to G(X)\). We require a natural transformation \(\Phi_X {:} F_{/X} ⨾ (\alpha_X)_{*} \Rightarrow G_{/X}\) in \(\mathbf{Cat}_{term}\).
  For an object \(a {:} A \to X\) in \(\mathcal{C}_{/X}\), the source is the object \(F(a) ⨾ \alpha_X\) in \(\mathcal{D}_{/G(X)}\), and the target is \(G(a)\). By the naturality of \(\alpha\) in \(\mathcal{D}\), we have a morphism \(\alpha_A {:} F(A) \to G(A)\) satisfying \(\alpha_A ⨾ G(a) = F(a) ⨾ \alpha_X\).
  This makes \(\alpha_A\) a morphism in \(\mathcal{D}_{/G(X)}\) from \(F(a) ⨾ \alpha_X\) to \(G(a)\). We define \((\Phi_X)_a := \alpha_A\).
  The naturality of \(\Phi_X\) with respect to morphisms in \(\mathcal{C}_{/X}\), as well as the modification axioms for \(\Phi\), follow directly from the naturality of \(\alpha\).

  \emph{Section and 2-functor axioms.}
  By construction, the underlying object, 1-cell, and 2-cell of \(D(\mathcal{C})\), \(D(F)\), and \(D(\alpha)\) are \(\mathcal{C}\), \(F\), and \(\alpha\) respectively. Thus, \(D\) is a section of the source 2-functor \(\mathbf{s}{:} (\mathbf{Cat}_{term})_{◆ \mathbf{Cat}_{term}} \to \mathbf{Cat}_{term}\).
  Identities and compositions are strictly preserved because the slice and postcomposition operations define a strict 2-functor \(\mathbf{Cat} \to \mathbf{Cat}\) componentwise.
  Therefore, \(D\) constitutes a \(\mathbf{Cat}_{term}\)-valued diamond.
\end{proof}

Now we just have to put all diamond-structures into one nice package. However, there are still some pitfalls that have to be avoided. For one, note that if \(\mathcal{A}\) contains non-faithful functors
\(F{:} \mathcal{C} → \mathcal{C}'\), then the natural transformation \(F_{*}{:} \mathbf{SOb}_{\mathcal{C}} → \mathbf{SOb}_{\mathcal{C}'}\) will not be a monomorphism and thus not in \(\overset{\hookrightarrow}{\mathbf{UBP}}\) but in \(\mathbf{UBP}\). If
we only do comparisons over \(\mathbf{UBP}\) this is not a problem, but, of course, not every morphism in \(\mathbf{UBP}\) is monic, so \(\mathbf{UBP}\) is not a pseudomonic
2-category; but we really need to assume that \(\mathcal{B}\) is monic to get a \(\mathbf{SOb}\)-functor on it. We could \emph{try} to treat \(\mathcal{B}\) and \(\mathbf{UBP}\) on a different levelling, but
it seems more prudent to just assume that \(\mathcal{A}\) is pseudomonic itself i.e. consists of faithful isomorphism-reflecting functors, so that its subobject-functor factors through
\( \overset{\hookrightarrow}{\mathbf{UBP}}\). Note that this means \(\mathcal{A}\) is 2-pseudomonic since it consists of monic categories.

\begin{Remark}
  \label{rmα_:10}
  This restriction means that our aimed-for theorem will sadly not just be a direct generalization of \cref{thm:1}, and such a thing is probably impossible.
  % And this makes some sense: what
  % we are basically saying by providing a universal property for \(\overset{\hookrightarrow}{\mathbf{UBP}}\) is that it is the most purely monic thing ever: wherever we have subobjects it will appear, and all the
  % information of a category that is contained in subobjects is somewhere in \(\overset{\hookrightarrow}{\mathbf{UBP}}\), and this is true on all categorical levels. But if there is some non-monic part of a
  % category or collection of categories, then it will contain information that is not in its shadow in \(\overset{\hookrightarrow}{\mathbf{UBP}}\), so there might be information that a diamond structure in some
  % other ``coefficient object'' detects that never appears in \(\overset{\hookrightarrow}{\mathbf{UBP}}\), which keeps \(\overset{\hookrightarrow}{\mathbf{UBP}}\) from being universal with respect to these structures.
\end{Remark}

\begin{Definition}
  \label{defn_:5}
\index{diamonds!2-category of}
  Given a sub-2-category \(\mathcal{A} \hookrightarrow \overset{\hookrightarrow}{\mathbf{Cat}}\), the \emph{2-category of diamonds \(◆(\mathcal{A})\) over \(\mathcal{A}\)} is the 2-category:

  \begin{enumerate}
  \item \label{item:39} whose objects are pairs \((\mathcal{B}_{★}, D{:} \mathcal{A} ◆ \mathcal{B}_{★})\) with \(\mathcal{B}_{★}\) an oplaxly pointed 2-category:

\begin{figure}[H]
      \centering
      \begin{tikzcd}
      \mathcal{A} \ar[r, -{Triangle[length=9mm, width=3mm]}] & \mathcal{B},
    \end{tikzcd}
      % \caption{\label{fig:5}}
    \end{figure}
  \item\label{item:47} whose morphisms \(Δ{:} (\mathcal{B}_{★}, D{:} \mathcal{A} ◆ \mathcal{B}_{★}) → (\mathcal{B}_{★}', D'{:} \mathcal{A} ◆ \mathcal{B}_{★}')\) are given by pairs of
    \begin{enumerate}
    \item \label{item:48} a point-preserving 2-functor \(\overline{Δ}{:} \mathcal{B}_{★} → \mathcal{B}_{★}'\) (i.e. \(★_{\mathcal{B}_{★}}⨾ \overline{Δ} = ★_{\mathcal{B}_{★}'}\)),

    \item\label{item:49} a cell \(\underline{Δ}{:} D' → D⨾ \overline{Δ}{:} \mathcal{A} ◆ \mathcal{B}_{★}'\)
    \end{enumerate}

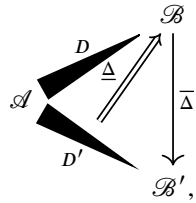
    \begin{figure}[H]
      \centering
      \begin{tikzcd}[column sep=large]
        & \mathcal{B} \ar[dd, "\overline{Δ}"] \\
        \mathcal{A} \ar[ur, -{Triangle[length=16mm, width=3mm]}, "D"{name=D}] \ar[dr, -{Triangle[length=15mm, width=3mm]}, "D'"'{name=E}] &\\
        & \mathcal{B}',
        \ar[Rightarrow, from=E, to=1-2, "\underline{Δ}", shorten <= 10pt]
    \end{tikzcd}
      % \caption{\label{fig:5}}
    \end{figure}

  \item\label{item:50} whose 2-morphisms \(I{:} Δ \Rightarrow Δ°\) are given by pairs of

    \begin{enumerate}
    \item\label{item:51} a point-preserving 2-natural transformation \(\overline{I}{:} \overline{Δ} \Rightarrow \overline{Δ}°{:} \mathcal{B}_{★} → \mathcal{B}_{★}'\) and

    \item \label{item:52} a 2-cell \(\underline{I}{:} \underline{Δ}° \Rightarrow \underline{Δ}; (D⨾ \overline{I}){:} \mathcal{A} ◆ \mathcal{B}_{★}'\).
    \end{enumerate}

    So a filling of this diagram opposite to \(\overline{I}\):

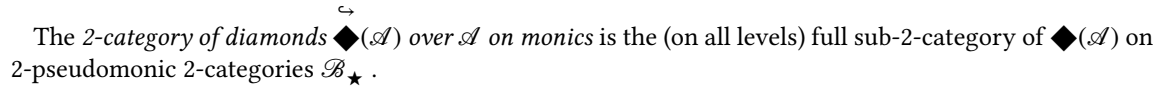
    \begin{figure}[H]
      \centering
      \begin{tikzcd}[every matrix/.append style={name=m}, column sep=large]
        & \mathcal{B} \ar[dd, "\overline{Δ}"{yshift=-14pt}, "\hphantom{\overline{Δ}}"{name=T}, bend left]
        \ar[dd, "\hphantom{\overline{Δ}°}"'{name=U}, "\overline{Δ}°"'{yshift=-14pt, xshift=3pt}, bend right]; \\
        \mathcal{A} \ar[ur, -{Triangle[length=16mm, width=3mm]}, "D"{name=D}]
        \ar[dr, -{Triangle[length=15mm, width=3mm]}, "D'"'{name=E}] &\\
        & \mathcal{B}'
        \ar[Rightarrow, from=T, to=U, "\overline{I}", shorten >=3pt, shorten <=3pt, yshift=-5pt]
\ar[from=E, "\underline{Δ}"'{yshift=14pt, xshift=12pt}, Rightarrow, -{Implies}, shorten <=5pt, shorten >=-7pt,
    to path={ (\tikztostart) -- ($(m-1-2)!0.3!(m-3-2)$) \tikztonodes }]
\ar[from=E, "\underline{Δ}°"{yshift=-5pt}, Rightarrow, -{Implies}, shorten <=5pt, shorten >=3pt,
to path={ (\tikztostart) -- ($(m-1-2)!0.1!(m-2-1)$) \tikztonodes }, shorten >=10pt]
  \end{tikzcd}
    % \caption{\label{fig:5}}
    \end{figure}
\end{enumerate}

  The \emph{2-category of diamonds \(\overset{\hookrightarrow}{◆}(\mathcal{A})\) over \(\mathcal{A}\) on monics} is the (on all levels) full sub-2-category of \(◆(\mathcal{A})\) on 2-pseudomonic 2-categories \(\mathcal{B}_{★}\) .
\end{Definition}

This definition is fairly complex, so let us disentangle it a bit: all cells go out of \(\mathcal{A}\). Morphisms
\(Δ{:} (\mathcal{B}_{★}, D) → (\mathcal{B}_{★}', D')\) between objects consist of one covariant ``external'' functor \(\overline{Δ}\) and one covariant ``internal'' morphism (which is a parameterized
point-preserving natural transformation)
\(\underline{Δ}\), which ``closes the square in \(\mathcal{B}\)''.

The more mysterious part are the 2-morphisms: so given \(Δ, Δ°\), a 2-morphism \(I{:} Δ \Rightarrow Δ°\) between them consists of a covariant part \(\overline{I}{:} \overline{Δ} \Rightarrow \overline{Δ}°\) ``on the outside'' and a
contravariant part \(\underline{I}{:} \underline{Δ}° \Rightarrow \underline{Δ}; (D⨾ \overline{I}){:}\) ``on the inside''. But let us consider how this looks in the particular case we are interested in, which is ensuring that the morphism from
\((\mathcal{B}_{★}, D)\) to \((\overset{\hookrightarrow}{\mathbf{UBP}}, \mathbf{SOb})\) given by \((\mathbf{SOb}_{\mathcal{B}}, D_{*})\) is terminal, where
\(D_{*}(\mathcal{C}){:} \mathbf{SOb}(\mathcal{C}) \Rightarrow D⨾ \mathbf{SOb}_{\mathcal{B}}\) is given by \(D(\mathcal{C})_{*}\) on each \(\mathcal{C}{:} \mathcal{A}\) and similar for higher cells. However, we have to interpret ``terminal'' in some
lax 2-categorical sense because it is not true on the nose. So the question is: \emph{how} is it true?
Well, given any 2-functor \(\overline{Δ}{:} \mathcal{B} → \overset{\hookrightarrow}{\mathbf{UBP}}\) and 2-cell \(\underline{Δ}{:} \mathbf{SOb} \Rightarrow D⨾ \overline{Δ}{:} \mathcal{A} ◆ \mathbf{UBP}\)
we have a top-preserving natural
transformation \(\overline{Δ}_{★}{:} \mathbf{SOb}_{\mathcal{B}} \Rightarrow \overline{Δ}\)  given by mapping \(ι{:} \mathbf{SOb}_{\mathcal{B}}(X)\) to the image top \(\top_{\overline{Δ}(ι)}\) of \(\overline{Δ}(ι)\)
so that \(D_{*}; (D⨾ \overline{Δ}_{★})\) goes from \(\mathbf{SOb}\) to \(D⨾ \overline{Δ}\) and maps \(κ{:} A \hookrightarrow X\) in \(\mathcal{C}{:} \mathcal{A}\) to \(\top_{\overline{Δ}(κ)}\). Now, since
\(\underline{Δ}\) is natural in each \(\mathcal{C}{:} \mathcal{A}\), we have a commuting diagram:

\begin{figure}[H]
  \centering
  \begin{tikzcd}
    (D⨾ \overline{Δ})(A) \ar[r, "(D⨾ \overline{Δ})(κ)"] & (D⨾ \overline{Δ})(X) \\
    \mathbf{SOb}_{\mathcal{C}}(A) \ar[u, "\underline{Δ}(\mathcal{C})_{A}"]\ar[r, "κ_{*}"'] & \mathbf{SOb}_{\mathcal{C}}(X) \ar[u, "\underline{Δ}(\mathcal{C})_{X}"']
\end{tikzcd}
\label{fig:2}
\end{figure}

And thus, \(D_{*}; (D⨾ \overline{Δ}_{★})(κ) = \top_{\overline{Δ}(κ)} \geq \underline{Δ}(κ)\).

\begin{Proposition}
  \label{■pn_:9}
\index{subobject functor!contravariant terminality}
  For any subcategory \(\mathcal{A} \hookrightarrow \mathbf{Cat}\) of monic categories and monomorphism-preserving functors,
  \(\mathbf{SOb}{:} \mathcal{A} → \overset{\hookrightarrow}{\mathbf{UBP}}\) is contravariantly terminal in \(\overset{\hookrightarrow}{◆}(\mathcal{A})\) (see \cref{defn_:9}).
\end{Proposition}
\begin{proof}
  \label{prf_:7}
  The constructions of \cref{■pn_:6}, \cref{■pn_:7}, and \cref{le↘_:5} together furnish the morphism
  \((\mathbf{SOb}_{\mathcal{B}}, D_{*}){:} (\mathcal{B}_{★}, D) \to (\overset{\hookrightarrow}{\mathbf{UBP}}, \mathbf{SOb})\), as was discussed in the preceding paragraphs. It remains to show that
  this is an initial object in the hom-category \(\overset{\hookrightarrow}{◆}(\mathcal{A})\bigl((\mathcal{B}_{★}, D),\, (\overset{\hookrightarrow}{\mathbf{UBP}}, \mathbf{SOb})\bigr)\). We have also for each other morphism
  \((\overline{Δ}, \underline{Δ}){:} (\mathcal{B}_{★}, D) → (\overset{\hookrightarrow}{\mathbf{UBP}}, \mathbf{SOb})\) constructed a 2-morphism \((\overline{Δ}_{★}, \leq)\). For uniqueness, note that for any top-preserving natural transformation \(\mathcal{I}{:} \mathbf{SOb}_{\mathcal{B}} \Rightarrow \overline{Δ}\) and \(ι{:} A \hookrightarrow X\) in \(\mathcal{B}_{★}\) it holds that the square

  \begin{figure}[H]
  \centering
  \begin{tikzcd}
    \overline{Δ}(A) \ar[r, "\overline{Δ}(ι)"] & \overline{Δ}(X) \\
    \mathbf{SOb}_{\mathcal{B}}(A) \ar[u, "\mathcal{I}_{A}"]\ar[r, "ι_{*}"'] & \mathbf{SOb}_{\mathcal{B}}(X) \ar[u, "\mathcal{I}_{X}"']
\end{tikzcd}
  % \caption{\label{fi♭_:8}}
\end{figure}

commutes, so that \(\mathcal{I}_{X}(ι) = \overline{Δ}_{★, X}(ι)\), since \(\mathcal{I}_{A}\) preserves tops. Thus, \(\mathcal{I} = \overline{Δ}_{★}\), so that \(\mathcal{I}\) is unique since 2-cells in \(\mathbf{UBP}\) are unique.
\end{proof}

Since contravariantly terminal objects are not universal, \cref{■pn_:9} is still not a universal characterization of \(\overset{\hookrightarrow}{\mathbf{UBP}}\). But let us assume that
\((\mathcal{B}_{★}, D)\) is some other contravariantly terminal diamond.

\begin{enumerate}
\item \label{item:61} The contravariant terminality of \((\overset{\hookrightarrow}{\mathbf{UBP}}, \mathbf{SOb})\) toward \((\mathcal{B}_{★}, D)\) is witnessed by the pair
  \((\mathbf{SOb}_{\mathcal{B}}, D_{*})\) with \(D_{*}{:} \mathbf{SOb} \Rightarrow (D⨾ \mathbf{SOb}_{\mathcal{B}}){:} ι ↦ D_{\mathcal{C}}(ι)\).

\item\label{item:65} The contravariant terminality of \((\overset{\hookrightarrow}{\mathbf{UBP}}, \mathbf{SOb})\) toward
  \((\overset{\hookrightarrow}{\mathbf{UBP}}, \mathbf{SOb}⨾ \mathbf{SOb}_{\overset{\hookrightarrow}{\mathbf{UBP}}})\) is witnessed by a morphism
  \((\mathbf{id}_{\overset{\hookrightarrow}{\mathbf{UBP}}}, \{-\}){:} \mathbf{SOb} \Rightarrow \mathbf{SOb}⨾ \mathbf{SOb}_{\overset{\hookrightarrow}{\mathbf{UBP}}}\).

\item\label{item:64} The contravariant terminality of \((\mathcal{B}_{★}, D)\) towards \((\overset{\hookrightarrow}{\mathbf{UBP}}, \mathbf{SOb})\) is witnessed by a pair
  \((\overline{\mathbb{I}}, \underline{\mathbb{I}})\) where \(\overline{\mathbb{I}}{:} \overset{\hookrightarrow}{\mathbf{UBP}} → \mathcal{B}\) is initial and
  \(\underline{\mathbb{I}}{:} D \Rightarrow \mathbf{SOb}⨾ \overline{\mathbb{I}}\) point-preserving and monic since \(\mathcal{B}_{★}\) is 2-pseudomonic.

\item\label{item:63} The initiality of \((\mathbf{SOb}_{\overset{\hookrightarrow}{\mathbf{UBP}}}, \{-\})\) towards
  \((\overline{\mathbb{I}}⨾ \mathbf{SOb}_{\mathcal{B}}, D_{*}; (\underline{\mathbb{I}}⨾ \mathbf{SOb}_{\mathcal{B}}))\) is witnessed by a 2-cell
  \((-)^{\mathbb{I}}{:} \mathbf{SOb}⨾ \mathbf{SOb}_{\overset{\hookrightarrow}{\mathbf{UBP}}} → \mathbf{SOb}⨾ \overline{\mathbb{I}}⨾ \mathbf{SOb}_{\mathcal{B}}\) which maps
  \(κ{:} \mathbf{SOb}_{\mathcal{C}}(X)\) to \(\overline{\mathbb{I}}(κ) ∈ \mathbf{SOb}_{\mathcal{B}}(\overline{\mathbb{I}}(\mathbf{SOb}_{\mathcal{C}}(X)))\).
\end{enumerate}

 All in all, this gives us a diagram of 2-cells of diamonds (so
basically natural transformations in \(\overset{\hookrightarrow}{\mathbf{UBP}}\)):

\begin{figure}[H]
  \centering
  \begin{tikzcd}
    D⨾ \mathbf{SOb}_{\mathcal{B}} \ar[r, "\underline{\mathbb{I}}_{*}"] & \mathbf{SOb}⨾ \overline{\mathbb{I}}⨾ \mathbf{SOb}_{\mathcal{B}} \\
    \mathbf{SOb} \ar[u, "{ι ↦ D_{\mathcal{C}}(ι)}"] \ar[r, "{ι ↦ \{ι\}}"']  & \mathbf{SOb}⨾ \mathbf{SOb}_{\overset{\hookrightarrow}{\mathbf{UBP}}} \ar[u, "{κ ↦ \overline{\mathbb{I}}(κ)}"'].
  \end{tikzcd}
  \caption{\label{fi♭_:10}}
\end{figure}

It follows from the definitions that this square commutes laxly, but a key step of our proof is to show that it commutes (weakly): from \cref{rmα_:9} we know that \(\{-\}\)
has a retraction \(\top_{(-)}{:} κ ↦ \top_{κ}\), so that the bottom edge in \cref{fi♭_:10} can also be transitioned along backwards, and so we also have a 2-cell \((-)^{D}\) from
\(\mathbf{SOb}⨾ \mathbf{SOb}_{\overset{\hookrightarrow}{\mathbf{UBP}}}\) to \(\mathbf{SOb}⨾ \overline{\mathbb{I}}⨾ \mathbf{SOb}_{\mathcal{B}}\) given by
\(κ ↦ \underline{\mathbb{I}}_{*}(ι_{κ})\), where \(ι_{κ}\) is the subobject of \(X\) corresponding to \(\top_{κ} ∈ \mathbf{SOb}_{\mathcal{C}}(X)\) for some
\(\mathcal{C}{:} \mathcal{A}\). This gives us two 2-cells \((-)^{\mathbb{I}}, (-)^{D}\) from \((\mathbf{SOb}_{\overset{\hookrightarrow}{\mathbf{UBP}}}, \{-\})\) to \((\overline{\mathbb{I}}⨾ \mathbf{SOb}_{\mathcal{B}}, \underline{\mathbb{I}})\)
where the modifications of both are pointwise inequalities in the diagram. But since \((\mathbf{SOb}_{\overset{\hookrightarrow}{\mathbf{UBP}}}, \{-\})\)
is initial, these two have to be (weakly) the same, so that for each
\(\mathcal{C}, ι{:} X \hookrightarrow Y\) in \(\mathcal{C}\) it holds that \(\underline{\mathbb{I}}_{*}(ι) = \overline{\mathbb{I}}(ι)\) in
\(\mathbf{SOb}_{\mathcal{B}}(\overline{\mathbb{I}}(\mathbf{SOb}_{\mathcal{C}}(X)))\). And since \(\overset{\hookrightarrow}{\mathbf{UBP}}\) is locally a preorder, the hom-category of diamonds between these objects
has at most one 2-cell, so the two 2-cells \((-)^{\mathbb{I}}\) and \((-)^{D}\) coincide.

Now, given \(X{:} \mathcal{C}\) we have that \(D_{\mathcal{C}}(\mathbf{id}_{X}) = \mathbf{id}_{D_{\mathcal{C}}(X)}\) and \(\overline{\mathbb{I}}(\mathbf{id}_{X}) = \mathbf{id}_{\overline{\mathbb{I}}}(X)\), so that by
definition of the covariant subobject preorder \(\underline{\mathbb{I}}_{*}(ι)\) has to be an isomorphism for
\(\underline{\mathbb{I}}_{*}(\mathbf{id}_{X}) = \overline{\mathbb{I}}(\mathbf{id}_{X})\) to hold in \(\mathbf{SOb}_{\mathcal{B}}(\overline{\mathbb{I}}(\mathbf{SOb}_{\mathcal{C}}(X)))\). Since
this holds for each \(X{:} \mathcal{C}\) we have \(D = \mathbf{SOb}⨾ \overline{\mathbb{I}}\) up to the canonical isomorphism \(\underline{\mathbb{I}}\). So each
contravariantly terminal diamond is induced by some 2-functor \(F{:} \overset{\hookrightarrow}{\mathbf{UBP}} → \mathcal{B}_{★}\) to some 2-category \(\mathcal{B}_{★}\).

Given now some other morphism \((G, η){:} (\overset{\hookrightarrow}{\mathbf{UBP}}, \mathbf{SOb}) → (\mathcal{B}_{★}, \mathbf{SOb}⨾ \overline{𝕀})\) such that for each \(\mathcal{C}{:} \mathcal{A}\) the 2-natural transformation
\(η_{\mathcal{C}}{:} \mathbf{SOb}_{\mathcal{C}}⨾ \overline{𝕀} \Rightarrow \mathbf{SOb}_{\mathcal{C}}⨾ G{:} \mathcal{C} → \mathcal{B}\) is an isomorphism, a 2-isomorphism
\((\overline{Σ}, \underline{Σ}){:} (G, η) \Rightarrow (\overline{𝕀}, \mathbf{id}_{\mathbf{SOb}_{\mathcal{C}}⨾ \overline{𝕀}})\) consists of a 2-natural transformation
\(\overline{Σ}{:} G \Rightarrow \overline{𝕀}\) and for each \(\mathcal{C}{:} \mathcal{A}\) a modification \(\mathbf{id}_{\mathbf{SOb}_{\mathcal{C}}⨾ \overline{𝕀}} \Rrightarrow η_{\mathcal{C}}; (\mathbf{SOb}_{\mathcal{C}}⨾ \overline{Σ})\). But if
\(\overline{Σ}\) is a 2-isomorphism, we can just use for each \(X{:} \mathcal{C}\) the 2-cell \(\mathbf{id}_{\mathbf{id}_{\mathbf{SOb}_{\mathcal{C}}⨾ \overline{𝕀}(X)}}\) which is equal to the
composite \(\mathbf{id}_{\mathbf{id}_{\mathbf{SOb}_{\mathcal{C}}⨾ \overline{Σ}(X)}⨾ η_{\mathcal{C}}, X}\) since \(η_{\mathcal{C}, X}\) is an isomorphism. Thus, it suffices to extract a 2-isomorphism
\(Σ{:} \overline{𝕀} \Rightarrow G\) from \(η_{\mathcal{C}}\).

Now note that each upward-bounded preorder \(P\) is a monic category, each preorder morphism \(f{:} P → Q\) a functor and each morphism
inequality \(f \leq g{:} P → Q\) a natural transformation, so \(\overset{\hookrightarrow}{\mathbf{UBP}}\) is a sub-2-category of \(\overset{\hookrightarrow}{\mathbf{Cat}}\). More precisely, there is a canonical embedding
\(ι_{\mathbf{UC}}{:} \overset{\hookrightarrow}{\mathbf{UBP}} \hookrightarrow \overset{\hookrightarrow}{\mathbf{Cat}}\) given by
treating each preorder as a category such that the 2-functor \(\mathbf{SOb}{:} \overset{\hookrightarrow}{\mathbf{Cat}} → \overset{\hookrightarrow}{\mathbf{UBP}}\) maps each \(P{:} \overset{\hookrightarrow}{\mathbf{UBP}}\)
to the principal ideal embedding \(ι_{(P)}{:} (P) \hookrightarrow \overset{\hookrightarrow}{\mathbf{UBP}}\). Intuitively, \(ι_{\mathbf{UC}}\) is a kind of unique section of \(\mathbf{SOb}\). In particular, if we restrict \(\mathbf{SOb}\) to the full sub-2-category
\(\overset{\hookrightarrow}{\mathbf{Cat}_{term}}\) of \(\overset{\hookrightarrow}{\mathbf{Cat}}\) on categories with a terminal object, then there exists a functor
\((-)_{*}{:} \overset{\hookrightarrow}{\mathbf{Cat}_{term}}_{◆ \mathbf{UBP}} → \overset{\hookrightarrow}{\mathbf{UBP}}\) that projects:

\begin{enumerate}
\item \label{item:26} each pair \((\mathcal{C}, F{:} \mathcal{C} → \overset{\hookrightarrow}{\mathbf{UBP}})\) to \(F(*)\),

\item\label{item:60} each 1-cell \((H, φ){:} (\mathcal{C}, F) → (\mathcal{D}, F')\) to the composite \(F_{*, \top}⨾ μ_{F(\top)}\) where \(μ_{F(\top)}{:} F(\top) → *\) is to unique morphism into the terminal
  object in \(\mathcal{D}\),

\item\label{item:84} each 2-cell \((α,Φ) {:} (F,ε) \Rightarrow (F',ε')\) to the inequality \(α_{\top}{:} F(\top) \leq F'(\top)\) in \(\mathbf{SOb}_{\mathcal{D}}(*)\).
\end{enumerate}

So

\[\overset{\hookrightarrow}{\mathbf{UBP}} \xrightarrow{ι_{\mathbf{UC}}} \overset{\hookrightarrow}{\mathbf{Cat}_{term}} \xrightarrow{\mathbf{SOb}} \overset{\hookrightarrow}{\mathbf{Cat}_{term}}_{◆ \mathbf{UBP}} \xrightarrow{(-)_{*}} \overset{\hookrightarrow}{\mathbf{UBP}}\]

is an equivalence \(σ\).

\begin{Lemma}
  \label{le↘_:6}
  \(ι_{\mathbf{UC}}\) is the terminal section of \(\mathbf{SOb}⨾ (-)_{*}{:} \overset{\hookrightarrow}{\mathbf{Cat}} → \overset{\hookrightarrow}{\mathbf{UBP}}\).
\end{Lemma}
\begin{proof}
  \label{prf_:8}
  We just showed that \(ι_{\mathbf{UC}}\) is a section. Let \(S{:} \overset{\hookrightarrow}{\mathbf{Cat}} → \overset{\hookrightarrow}{\mathbf{UBP}}\) be some other section. Then we can for each
  \(P{:} \overset{\hookrightarrow}{\mathbf{UBP}}\) define a functor \(μ_{S, P}{:} S(P) → ι_{\mathbf{UC}}(P)\) given by \(μ_{S, P}(X) := ι_{\mathbf{UC}}(μ_{X, *})(\top)\)
  for \(X{:} S(P)\), and for each \(f{:} X \hookrightarrow Y\) in \(S(P)\) by \(μ_{S, P}(f) {:} ι_{\mathbf{UC}}((\top_{f_{*}} \leq \top_{Y})⨾ μ_{Y, *})\). This is functorial in \(P\) since
  \(\mathbf{SOb}\) is functorial and \(ι_{\mathbf{UC}}\) a full embedding.
\end{proof}

Thus we can define \(Σ{:} \overline{𝕀} \Rightarrow G\) by \(ι_{\mathbf{UC}}⨾ η\), i.e. we define \(Σ_{P}{:} \overline{𝕀}(P) → G(P)\) by
\(η_{ι_{\mathbf{UC}}}(P)(\top)\).

\begin{Lemma}
  \label{le↘_:7}
  Assume that the image of \(\iota_{\mathbf{UC}}\) lies in \(\mathcal{A}\) and that \(\eta\) is a strict diamond 2-cell. For \(P {:} \overset{\hookrightarrow}{\mathbf{UBP}}\), identify
  \[
    \rho_P {:} P \cong \mathbf{SOb}_{P}(\top_P), \qquad a \mapsto (a \hookrightarrow \top_P),
  \]
  and let
  \[
    \theta_P {:} \overline{\mathbb I}(P) \longrightarrow G(P)
  \]
  be the component \(\eta_{\iota_{\mathbf{UC}}(P),\top_P}\), transported along \(\rho_P\). Define
  \[
    \Sigma_P := \theta_P^{-1}.
  \]
  Then \(\Sigma {:} G \Rightarrow \overline{\mathbb I}\) is a 2-natural isomorphism.
\end{Lemma}
\begin{proof}
  The identification \(\rho_P\) is natural in \(P\): for a monotone map \(f {:} P → Q\), the induced comparison \(\mathbf{SOb}_{P}(\top_P) → \mathbf{SOb}_{Q}(\top_Q)\) sends the subobject \(a \hookrightarrow \top_P\) to \(f(a) \hookrightarrow \top_Q\), using the unique inequality \(f(\top_P) \leq \top_Q\). Hence evaluation at terminal objects, after restriction along \(\iota_{\mathbf{UC}}\), carries the diamond 2-cell \(\eta\) to a 2-natural transformation
  \[
    \theta {:} \overline{\mathbb I} \Rightarrow G.
  \]
  Since \(\eta\) is assumed invertible, each component \(\theta_P\) is invertible; thus \(\theta\) is a 2-natural isomorphism.

  It remains to check that the inverse components assemble into a 2-natural transformation. Let \(f {:} P → Q\). The 2-naturality of \(\theta\) gives the equality of 1-cells
  \[
    \theta_P ⨾ G(f) \;=\; \overline{\mathbb I}(f) ⨾ \theta_Q .
  \]
  Using invertibility of \(\theta_P\) and \(\theta_Q\), we compute
  \[
  \begin{aligned}
    \theta_P^{-1} ⨾ \overline{\mathbb I}(f)
    &= \theta_P^{-1} ⨾ \overline{\mathbb I}(f) ⨾ \theta_Q ⨾ \theta_Q^{-1} \\
    &= \theta_P^{-1} ⨾ \theta_P ⨾ G(f) ⨾ \theta_Q^{-1} \\
    &= G(f) ⨾ \theta_Q^{-1}.
  \end{aligned}
  \]
  Thus \(\Sigma = \theta^{-1}\) is natural on 1-cells.

  For a 2-cell \(\alpha {:} f \Rightarrow g {:} P → Q\), the 2-naturality of \(\theta\) gives
  \[
    \overline{\mathbb I}(\alpha) ⨾ \theta_Q \;=\; \theta_P ⨾ G(\alpha).
  \]
  Whiskering by the inverses of \(\theta_P\) and \(\theta_Q\) yields
  \[
    \theta_P^{-1} ⨾ \overline{\mathbb I}(\alpha)
    \;=\;
    G(\alpha) ⨾ \theta_Q^{-1},
  \]
  which is precisely the 2-naturality of \(\Sigma\). Since every component \(\Sigma_P = \theta_P^{-1}\) is invertible, \(\Sigma\) is a 2-isomorphism.
\end{proof}

If \(\eta\) were only pseudo- or laxly natural, the same construction would yield only a pseudo- or lax inverse; the strict 2-isomorphism conclusion uses the strict diamond-cell structure adopted here.

So to summarize:

\begin{Theorem}
  \label{thm:3}
\index{universal characterization of UBP}
  Given a sub-2-category \(\mathcal{A} \hookrightarrow \overset{\hookrightarrow}{\mathbf{Cat}}\), such that \(ι_{\mathbf{UC}}\) factors through \(\mathcal{A}\), let \(\mathbf{CT}(\overset{\hookrightarrow}{◆}(\mathcal{A}))\) be the sub-2-category of \(\overset{\hookrightarrow}{◆}(\mathcal{A})\) on the
  contravariantly terminal objects, whose morphisms \(Δ\) are the pairs from \cref{item:47} such that \(\underline{Δ}\) is a 2-isomorphism. Then
  \((\overset{\hookrightarrow}{\mathbf{UBP}}, \mathbf{SOb})\) is initial in \(\mathbf{CT}(\overset{\hookrightarrow}{◆}(\mathcal{A}))\).
\end{Theorem}
  
And thus we have finally derived a universal characterization of \(\overset{\hookrightarrow}{\mathbf{UBP}}\).

\begin{Remark}
  \label{rmα_:24}
  It is surprisingly hard to tell which 2-functors out of \(\overset{\hookrightarrow}{\mathbf{UBP}}\) induce a contravariantly terminal diamond. Each point-preserving 2-functor
  \(F{:} \overset{→}{\mathbf{UBP}} → \mathcal{B}_{★}\) introduces a diamond \((\mathcal{B}_{★}, \mathbf{SOb}⨾ F)\) by \cref{■pn_:7} but to be contravariantly terminal, for each diamond
  \((\mathcal{D}, D)\) there has to exist an initial diamond-morphism from \((\mathcal{D}, D)\) to \((\mathcal{B}_{★}, \mathbf{SOb}⨾ F)\). Postcomposition with \(F\) induces a morphism given by
  \((\mathbf{SOb}_{\mathcal{D}}⨾ F, D_{*}⨾ F)\), but the question is: when is this morphism initial? For this, each diamond-morphism
  \(Δ{:} (\mathcal{D}, D) → (\mathcal{B}_{★}, \mathbf{SOb}⨾ F)\) has to induce a unique 2-morphism \(\mathbb{I}{:} (\mathbf{SOb}_{\mathcal{D}}⨾ F, D_{*}⨾ F) \Rightarrow Δ\). Unpacking the definitions we find that
  \(Δ\) consists of a 2-functor \(\overline{Δ}{:} \mathcal{D} → \mathcal{B}\) and a cell \(\underline{Δ}{:} (\mathcal{B}_{★}, \mathbf{SOb}⨾ F) → (\mathcal{B}_{★}, D⨾ \overline{Δ})\) in
  \(\mathcal{A} ◆ \mathcal{B}_{★}\). Unpacking further we find that \(\underline{Δ}\) consists for each \(\mathcal{C}{:} \mathcal{A}\) of a cell
  \(\underline{Δ}_{\mathcal{C}}{:} (\mathcal{C}, \mathbf{SOb}_{\mathcal{C}}⨾ F) → (\mathcal{C}, D(\mathcal{C})⨾ \overline{Δ})\) over \((\mathcal{C}, \mathbf{id}_{\mathcal{C}})\) and thus a 2-natural transformation
  \(\mathbf{SOb}_{\mathcal{C}}⨾ F \Rightarrow \mathbf{id}_{\mathcal{C}}⨾ D(\mathcal{C})⨾ \overline{Δ} = D(\mathcal{C})⨾ \overline{Δ}\) in \(\mathcal{B}^{\mathcal{C}}\). So \(\underline{Δ}\) has to be a 2-cell making this diagram commute:

\begin{figure}[H]
  \centering
  \begin{tikzcd}
    \overset{\hookrightarrow}{\mathbf{Cat}} \ar[dr, "\mathbf{SOb}"', ""{name=S}] \ar[r, "D"] & \mathcal{D} \ar[d, "\mathbf{SOb}_{\mathcal{D}}"{yshift=-6pt, xshift=1pt}] \ar[dr, "\overline{Δ}"] & \\
    & \overset{\hookrightarrow}{\mathbf{UBP}} \ar[r, "F"'] & \mathcal{B} \ar[Rightarrow, from=S, to=1-2, "D_{*}"']
  \end{tikzcd}
  % \caption{\label{fi♭_:15}}
\end{figure}

On the other hand, \(\mathbb{I}\) has to consist of a 2-natural transformation \(\overline{\mathbb{I}}{:} \mathbf{SOb}_{\mathcal{D}}⨾ F \Rightarrow \overline{Δ}\) and a modification
\(\underline{\mathbb{I}}{:} \underline{Δ} \Rrightarrow (D_{*}⨾ F); (D⨾ \overline{\mathbb{I}})\). So for each \(X{:} \mathcal{C}{:} \mathcal{A}\) the horn

\begin{figure}[H]
  \centering
  \begin{tikzcd}
    (\mathbf{SOb}⨾ F)(X) \ar[d, "(D_{*}⨾ F)(X)"'] \ar[dr, "\underline{Δ}_{X}"] & \\
    (D⨾ \mathbf{SOb}_{\mathcal{B}}⨾ F)(X) & (D⨾ \overline{Δ})(X)
  \end{tikzcd}
  % \caption{\label{fi♭_:16}}
\end{figure}

has to have a unique filling:

\begin{figure}[H]
  \centering
  \begin{tikzcd}
    (\mathbf{SOb}⨾ F)(X) \ar[d, "(D_{*}⨾ F)(X)"'] \ar[dr, "\underline{Δ}_{X}", ""{name=E}] & \\
    (D⨾ \mathbf{SOb}_{\mathcal{B}}⨾ F)(X) \ar[r, "(D⨾ \overline{\mathbb{I}})(X)"'] & (D⨾ \overline{Δ})(X) \ar[Rightarrow, from=E, to=2-1, "\underline{\mathbb{I}}"'{xshift=-5pt, yshift=-1pt}, shorten <= 6pt]
  \end{tikzcd}
  % \caption{\label{fi♭_:16}}
\end{figure}
\end{Remark}

% Now we just have to show that each for two 2-functors \(F, F'{:} \overset{\hookrightarrow}{\mathbf{UBP}} → \mathcal{B}\) and equivalence \(σ := (\overline{σ}, \underline{σ}){:} (\mathcal{B}, \mathbf{SOb}⨾ F) \overset{\sim}{→} (\mathcal{B}, \mathbf{SOb}⨾ F')\) of
% the induced diamonds \((\mathcal{B}, \mathbf{SOb}⨾ F), (\mathcal{B}, \mathbf{SOb}⨾ F')\) there is an equivalence \(\overline{σ}{:} F \Rightarrow F'\) and iso-modification
% \(Σ: σ \overset{\sim}{\Rrightarrow} \mathbf{SOb}⨾ \overline{σ}\), and similar for 2-natural transformations \(φ: F \Rightarrow F'\). For this we have to assume that
% \(\mathbf{SOb}: \mathcal{A} → \overset{\hookrightarrow}{\mathbf{UBP}}\) is sufficiently distinguishing and thus that \(\mathcal{A}\) is big enough. By definition, the maximal such \(\mathcal{A}\) is
% \(\overset{\hookrightarrow}{\mathbf{Cat}}\), which in particular contains the ideals of \(\overset{\hookrightarrow}{\mathbf{UBP}}\). By \cref{rmk:9} the 1-truncation of each ideal and thus in particular
% \(\mathopen{∣}(X)\mathclose{∣}_{1}\) of each principal ideal of each object \(X: \overset{\hookrightarrow}{\mathbf{UBP}}\) with its trivial 2-cell structure is mapped to itself under the \(\mathbf{SOb}\)-functor on
% itself, so that morphisms in \(\overset{\hookrightarrow}{\mathbf{UBP}}\) are. Thus:

\subsection{The \(n\)-Diamond Conjecture}
\label{sec:n-univ-conj}

Let us close by making the proof of \cref{thm:3} into a conjecture. For this, note that in \cref{xmp:4} and \cref{defn_:6} we use truncations for organizing the subobjects of
2-categories into a preorder. Accordingly, there is more than one minimal 2-functor \(\mathcal{B} → \mathbf{UBP}\). More precisely there are three of note: \(\mathbf{SOb}_{f, \mathcal{B}}\), the
preorder with dual 2-cells \(\mathbf{SOb}_{f^{op}, \mathcal{B}}\) in \cref{item:98} of \cref{defn_:6}, and the containment preorder \(\mathbf{SOb}_{\mathopen{∣}f\mathclose{∣}, \mathcal{B}}\) where the 2-cells of
\cref{item:98} have to be 2-isomorphisms. But it stands to reason that if we don't truncate, i.e. use a ``there exists''-clause in \cref{item:98}, that these lift to 3-functors
into the 3-category of a 2-upward-bounded 2-preorders \(2\text{-}\mathbf{UBP}\). Conversely, the right definition of a 2-upward-bounded 2-preorder, or similar for other \(n\), should become clear if we
inspect the subobject structure of a 2-category or \(n\)-category, and the \((n + 1)\)-functors out of \(n\text{-}\overset{\hookrightarrow}{\mathbf{UBP}}\) should again generate all contravariantly terminal
\(n\)-diamonds. So it seems that the pattern goes:

\begin{enumerate}
\item \(0\text{-}\mathbf{SOb}{:} \mathbf{Set} → \mathbf{Set}_{◆ \{∅ \hookrightarrow *\}}\): each part of an element of a set is either trivial or whole, so the functor \(0\text{-}\mathbf{SOb}_{S}\) just maps an
  element \(s ∈ S\) to \(*\) and the empty subobject to \(ι_{∅}{:} ∅ \hookrightarrow *\);

\item\label{item:2} \(1\text{-}\mathbf{SOb} = \mathbf{SOb}\) is the case we studied;

\item\label{item:6} \(n\text{-}\mathbf{SOb}\) works recursively like \(\mathbf{SOb}\) but trifurcated due to laxity and can be projected downward to
  \((n \text{-} 1)\text{-}\mathbf{UBP}\);

\item\label{item:74} \(n\text{-}\mathbf{UBP}\) receives initial representations from all sufficiently monic \(n\)-categories;

\item\label{item:73} \(n\text{-}\overset{\hookrightarrow}{\mathbf{UBP}}\) generates contravariantly terminal diamonds.
\end{enumerate}

\pagebreak

\section{Outlook}
\label{sec:conclusion}

The plainest germ can blossom into the most beautiful flower and bear the richest fruit, and this is also true regarding the relation of this work to successive
ones, as much of the material in it is set up so as to point beyond itself:

\begin{itemize}
\item We stopped the categorical development of the calculus in \cref{sec:inter-categ-size} where it bifurcates into Heyting and co-Heyting categories. In \emph{Closed Geometric
    Logic}, we investigate the closed part of this bifurcation specifically with regard to its logical properties and introduce \emph{subdivisions} and the smallness-reducedness
  decomposition hinted at in the introduction.

\item We took up to the theme of abstract polytopes in \cref{obs_:1} and \cref{xmp:21}. These are closely related to subdivisions, in a way that we will make fully explicit in
  \emph{Polytopes as Subdivisions}.

\item The effects of the equivalence-breaking of lax morphisms has to be examined particularly in light of the meaningfulness of the constructions introduced in
  \cref{sec:other-examples}.

\item A large part of these priors is necessary for the pursuit of one final goal: a general theory of composition.
\end{itemize}

A lot remains to be written.

\begin{quote}
  How can you see into my eyes like open doors?
Leading you down into my core where I've become so numb.
Without a soul my spirit's sleeping somewhere cold,
until you find it there and lead it back home.

  Wake me up inside, wake me up inside;
call my name and save me from the dark.
Bid my blood to run, before I come undone,
save me from the nothing I've become:

\textbf{Bring me to life!}
\end{quote}
 — \emph{Bring Me To Life}, Evanescence, shortened, with added punctuation.\cite{EvanescenceBringMe}

\pagebreak

\theendnotes

\printbibliography[
heading=bibintoc,
title={References}
]

\printindex
\end{document}